\documentclass[a4paper,10pt]{article}
 
\usepackage{amsmath, amsthm, amsfonts, amssymb,setspace}
\usepackage{xcolor}% for text color
\usepackage{mathrsfs,mathtools}
\mathtoolsset{showonlyrefs}
\RequirePackage[colorlinks,citecolor=blue,urlcolor=blue]{hyperref}
\usepackage{dsfont}
\usepackage{comment}
\usepackage{enumerate}
 
\usepackage{scalerel}

\usepackage{tikz} 

\usetikzlibrary{positioning, fit, calc} % used for the efficient working of the positioning system
\tikzset{block/.style={draw, thick, text width=2.6cm, minimum height=0.5cm, align=center}, % the align command is used to align the block diagram at the center
line/.style={-latex}   % the lesser the width the greater will be the diagram window
}

\renewcommand{\varpi}{\lambda_0}

\newcommand{\eps}{\ensuremath{\varepsilon}}
\newcommand{\sm}{s}%{s-}
\newcommand{\oldL}{\xi}%%L
\newcommand{\embed}{\hookrightarrow}

\newcommand\todown{\searrow}
\newcommand\toup{\nearrow}

\usepackage{todonotes}

\newcommand{\newL}{S}%%L
\newcommand{\newq}{\tilde{\mu}}%%q
\newcommand{\newp}{\mu}%%p
\newcommand{\newS}{L}%%L
\newcommand{\newH}{h}%%H
\newcommand{\newh}{h}%%p

\usepackage%[colorinlistoftodos]
{todonotes}

\mathtoolsset{showonlyrefs}
\usepackage{fullpage,color}

\newcommand{\divv}{\operatorname{div}}

\numberwithin{equation}{section}

\newtheorem{theorem}{Theorem}[section]
\newtheorem{corollary}[theorem]{Corollary}
\newtheorem{definition}[theorem]{Definition}
\newtheorem{remark}[theorem]{\it \rmfamily Remark}

\newtheorem{lemma}[theorem]{Lemma}
\newtheorem{notation}[theorem]{Notation}

\newtheorem{differences}[theorem]{Differences}
\newtheorem{hypothesis}{\it \rmfamily Hypothesis}

\newcounter{matriz}
\newenvironment{matriz}{\refstepcounter{matriz}\equation}{\tag{M\thematriz}\endequation}

\newtheorem{proposition}[theorem]{Proposition}

\newcommand{\beq}{\begin{equation}}
\newcommand{\eeq}{\end{equation}}
\newcommand{\bth}{\begin{theorem}}
\newcommand{\bpr}{\begin{proposition}}
\newcommand{\epr}{\end{proposition}}

\newcommand\Leb{\mathrm{Leb}}
\newcommand{\loc}{\mathrm{loc}}
\newcommand{\dist}{\mathrm{dist}}
\newcommand{\Skor}{\mathrm{Skor}}

\newcommand{\rc}{{\mathrm{c}}}

\allowdisplaybreaks[4]

\makeatletter
\def\mov@rlay#1#2{\leavevmode\vtop{%
   \baselineskip\z@skip \lineskiplimit-\maxdimen
   \ialign{\hfil$\m@th#1##$\hfil\cr#2\crcr}}}

\newcommand{\charfusion}[3][\mathord]{
    #1{\ifx#1\mathop\vphantom{#2}\fi
        \mathpalette\mov@rlay{#2\cr#3}
      }
    \ifx#1\mathop\expandafter\displaylimits\fi}

\def\Xint#1{\mathchoice
   {\XXint\displaystyle\textstyle{#1}}%
   {\XXint\textstyle\scriptstyle{#1}}%
   {\XXint\scriptstyle\scriptscriptstyle{#1}}%
   {\XXint\scriptscriptstyle\scriptscriptstyle{#1}}%
   \!\int}
\def\XXint#1#2#3{{\setbox0=\hbox{$#1{#2#3}{\int}$}
     \vcenter{\hbox{$#2#3$}}\kern-.5\wd0}}
\def\lint{\Xint\times}

\newcommand{\1}{\mathds{1}}

\newcommand\coma[1]{{\color{red} #1}}

\newcommand\dela[1]{}
\newcommand\delb[1]{} %Added by ZB on and after 02 Nov 2019

\newcommand\prm{\mu}
\newcommand{\id}{\mathrm{id}}
\newcommand\supp{\textrm{supp }}

\newcommand{\Nb}{\overline{\mathbb{N}}}
\newcommand{\R}{\mathbb R}
\newcommand\Q{{\mathbb Q}}
\newcommand\E{{\mathbb E }}

\newcommand\call[1]{{%\color{blue}
\mathscr #1}}
\newcommand\gen[1]{{%\color{blue}
\mathscr #1}}

\def\lb{\langle}
\def\rb{\rangle}
\newcommand{\nxi}{\phi}%\xi

\definecolor{darkblue}{rgb}{0.1,0.1,0.9}%%%COLOR FOR SIDE COMMENT

\def\hh{{\vskip 2.5 mm \noindent }}
\def\vv{{\vskip 1mm}}

\begin{document}

\title{Stochastic transport equation with L\'evy noise}

\author{Zdzis{\l}aw Brze\'{z}niak$^{a}$, Enrico Priola$^{b}$,  Jianliang Zhai$^{c}$ and  Jiahui Zhu$^d$\\
{\small{$a.$   Department of Mathematics, University of York,  York, UK}}\\
{\small{$b.$ Dipartimento di Matematica, Universit\`a di Pavia, via Ferrata 5, Pavia, Italy }}\\
{\small{ $c.$ School of Mathematical Sciences, University of Science and Technology of China, Hefei, China}}\\
{\small{ $d.$    School of Mathematical Sciences, Zhejiang University of Technology, Hangzhou, China}}
}
\maketitle

\begin{abstract}  
We   study the  stochastic  transport equation with globally $\beta$-H\"older continuous and bounded vector field driven by a non-degenerate pure-jump L\'evy noise of $\alpha$-stable type. Whereas the deterministic transport equation may lack uniqueness, we prove the existence and pathwise uniqueness of a weak solution in the presence of a multiplicative pure jump noise, assuming $\frac{\alpha}{2}+\beta>1$. Notably, our results cover the entire range $\alpha \in (0,2)$, including the supercritical regime $\alpha\in(0,1)$ where the
driving noise exhibits notoriously weak regularization.  A key step of our strategy is the development of   a   \emph{sharp} $C^{1+\delta}$-diffeomorphism 
   and new regularity results for the Jacobian determinant  of the stochastic flow associated to its stochastic characteristic equation. These novel probabilistic results are of independent interest and constitute a substantial component of our work.  Our results are the first full generalization of the celebrated paper by Flandoli, Gubinelli, and Priola [Invent. Math. 2010] from the Brownian motion to the pure jump L\'evy noise. To the best of our knowledge, this appears to be the first example of a  partial differential equation of fluid dynamics  where well-posedness is restored by  the influence of a non-degenerate pure-jump noise.
   
  \par\vspace{1ex}
\end{abstract}

\tableofcontents

\section{Introduction}\label{sec-introduction}
  
 In this article,   we study the following two main questions. Firstly, we study {\sl the well-posedness of  the following  singular  linear  transport equation with multiplicative gradient type noises of jump type}
 \begin{align}\label{eqn-transport-Markus}
 & \frac{\partial{u(t,x)}}{\partial t}+b(x)\cdot D u(t,x)\,+\sum_{i=1}^d D_i u(t-,x)\diamond \frac{d L_t^i}{dt}=0,~t\in[0,T], \;\;  x\in\mathbb{R}^d,
 \\
  & u(0,x)=u_0(x),\ \ x\in\mathbb{R}^d.
\end{align}
Here we assume that 
    $D_i$ is the weak partial derivative in the $i$-th direction, 
$b: \mathbb{R}^d \to \mathbb{R}^d$ is a vector field which is bounded and H\"older continuous with exponent  $\beta \in (0,1)$, i.e.,
\begin{equation}\label{eqn-g-beta}
 b \in C^\beta_{\mathrm{b}}(\mathbb{R}^d,\mathbb{R}^d);
\end{equation}
and 
$L=(L_t)_{t \geq 0}$ is a non-degenerate $\mathbb{R}^d$-valued 
pure-jump L\'evy  process defined on a probability space $(\Omega,\call{F},\mathbb{P})$ 
 of $\alpha$-stable type with  $\alpha \in (0,2)$, see also
  the monograph
 \cite{Sato_1999}.  Equation \eqref{eqn-transport-Markus} is written in the Markus form; for an equivalent formulation, see equation  \eqref{eqn-transport-strong-2} in Remark \ref{rem-defi}. In this work, we primarily deal with the so called   weak$^\ast$-$\mathrm{L}^{\infty}$-solutions, see Definition \ref{def-transport-weak-2}. Moreover, in Appendix  \ref{sec-no blow up}  we also consider the case of classical $C^1$-solutions.

Secondly,  we are interested in the {\sl regularity properties of the solution to the following stochastic ordinary differential equation (SDE),} with $x \in \R^d$ and $s\in [0,\infty)$, 
 driven by the process $L$: 
\begin{align}\label{eqn-SDE-intro} 
dX_t &= b(X_t)\,dt + dL_t,\; t>s;  \;\;\;
\\ 
X_s &=x,\nonumber
\end{align}
which is the stochastic   characteristics equation for equation \eqref{eqn-transport-Markus}.

  Note that when $L=0$, the corresponding transport equation may have infinitely many solutions, see the discussion before equation \eqref{ci2}. In contrast, we prove in Theorems \ref{thm-uniqueness-1} and \ref{thm-uniqueness} that  equation \eqref{eqn-transport-Markus}  becomes well-posed  provided the L\'evy noise $L$ is 
  (in suitable sense) non-degenerate. 
For clarity and ease of presentation, we present below  Theorems \ref{thm-uniqueness-1-intro} and \ref{thm-uniqueness-intro} which are simplified versions of the well posedness Theorems \ref{thm-uniqueness-1} and \ref{thm-uniqueness} to equation \eqref{eqn-transport-Markus}. 

  One of the reason we consider the transport equation \eqref{eqn-transport-Markus} in the Markus canonical form is that it has a solution represented in terms of the stochastic flow  associated with equation \eqref{eqn-SDE-intro}
$$
u(t,x,\omega):=u_0\big(\phi_{0,t}^{-1}(\omega)(x)\big), \quad t\in [0,T], x \in\mathbb{R}^d,
$$
where $\phi_{0,t}^{-1}$ is the inverse flow of equation \eqref{eqn-SDE-intro} with $s=0$.
Consequently, a thorough understanding of the \textit{fine properties} of this flow is essential. 
Let us point out that to the best of our knowledge, there is no  work that has  explored
the \emph{sharp} $C^{1+\delta}$-diffeomorphism property of the flows, leaving a significant gap between the existing results and the more refined outcomes we target. This motivates our second main objective and constitutes another central contribution of the present paper,  which is also of independent interest.

Our results concerning the first question demonstrate the so-called \emph{regularization by noise} phenomenon in the context of L\'evy processes. This represents, to the best of our knowledge,  the first example of a PDE of fluid dynamics that becomes well-posed under the influence of a pure-jump  L\'evy noise.

The analysis of the pure-jump case also leads us to develop new analytical and probabilistic techniques, which we believe will be useful beyond the present framework.

\paragraph{Motivation for  the study of stochastic transport equations.}
 Let us spend some time discussing the motivation for this problem.

  A transport equation is a basic partial differential equation (PDE) in fluid dynamics.
A vast literature has been devoted to the well-posedness of the deterministic transport equations
 with singular  time-dependent vector fields $b:[0,T]\times \mathbb{R}^{d}\rightarrow \mathbb{R}%
^{d}$. 
A remarkable example is a result  obtained by Di Perna and  Lions in \cite
{DiPernaLions},  who proved that, if the time dependent vector field  $b\in
L^{1}(0,T;W_{loc}^{1,1}(\mathbb{R}^{d},\mathbb{R}^{d}))$ satisfies the   linear growth condition with respect to the spatial variable  and ${\mathrm{div}\,}b\in L^{1}(0,T;L^{\infty }(\mathbb{R}%
^{d}))$, then for every given $u_{0}\in
L^{\infty }(\mathbb{R}^{d})$, there exists a unique $L^{\infty }([0,T]\times \mathbb{R}^{d})$
weak-$\ast $ solution that is continuous in time. Moreover, a generalized notion of a flow was  introduced and its existence and uniqueness were  proved.
For more information, see a partial review  \cite{AmbrCrippa}, the recent paper
\cite{Ciampa+Crippa+Spirito_2020} and the references therein.

On the other hand, it is well-known that when $L =0$, i.e., without  noise,  there exist  counterexamples to the uniqueness of weak-$\ast $ solutions even in dimension $d=1$. For instance, see Section 6.1 in \cite{FGP_2010-Inventiones}, for an arbitrary $\gamma \in \left( 0,1\right)$ and $R>0$, one can consider the following vector field $b$ on $\mathbb{R}$:
\begin{equation} \label{ci2}
b(x)=\frac{1}{1-\gamma }\,\mathrm{sign}\,(x)\left( |x|\wedge R\right)
^{\gamma }, \;\; x \in \mathbb{R}.
\end{equation}
These singular coefficients can lead to non-uniqueness of solutions to the corresponding deterministic transport equations.%, see Section 6.1 in \cite{FGP_2010-Inventiones}.

In contrast, by introducing suitable stochastic noise, the problem becomes well-posed, i.e., both the existence and the uniqueness of solutions hold, a phenomenon known as \emph{regularization by  noise}. 
This effect has attracted considerable interest from  both physicists and experts in stochastic analysis.

For the specific transport equation considered in this work, the foundational study can be traced to the pioneering work  \cite{FGP_2010-Inventiones}, by Flandoli, Gubinelli and the second named author,  which established the well-posedness results in the case of Brownian Motion. Subsequent research has significantly expanded this framework, investigating the Brownian noise scenario under various assumptions on the singularity of the drift coefficient $b$ and the initial data $u_0$; see, for example, \cite{BeckFlandoliGubinelliMaurelli, FedrizziOlivera, Galeati_2020,Koley25} and \cite{Zhang10}. Parallel developments have also been made for equations driven by fractional Brownian Motion, as documented in \cite{Amine,Catellier16},  and the references therein. It is important to note that all the aforementioned works are confined to the continuous noise case. In Remark \ref{zhang}, we will discuss a  random transport equation involving  L\'evy noise
 considered in a recent paper \cite{Hao+Zhang_2020}. It turns out that the  equation studied in that paper  is   different from our equation 
 \eqref{eqn-transport-Markus}.

Regarding the general phenomenon of regularization by noise, we refer the reader to the following literature: for results on Brownian motion-driven ordinary differential equations (ODEs), see \cite{Krylov-Rockner-2005,Zhang-2011,Zhang-2012}; for the stochastic partial differential equations (SPDEs) driven by Brownian motion, see \cite{Ref_SPDE_BM,BGM 2025,Flandoli-Luo 2021, HLL 202407}; and for ODEs driven by jump noise, see \cite{ABM 2020, Chen-Zhang-Zhao 2021, Pr12,Pr15,Pr18}; and \cite{Catellier-Duboscq2025} for other driving noises.

The present paper presents the first study for SPDEs driven by jump noise in this direction. Compared with the existing literature, our analysis of the pure-jump case necessitates the development of new analytical techniques, which will be introduced and discussed in detail later.

 We next  discuss the motivation for the use of a jump noise. Although  continuous type
noises have been regarded as a natural choice for modeling
random noise in continuous-time systems, numerous examples and evidence of jump type
noises have been discovered and documented in many real world complex systems. The  monograph
\cite{Birnir} presents  a phenomenological study of fully developed turbulence and intermittency;
 it proposes that  experimental observations  of these physical
characteristics of fluid dynamics can be modeled by stochastic PDEs  with L{\'e}vy noise.
We should mention here  two  recent papers about  the fluid dynamics driven by L\'evy processes, i.e., \cite{Brz+Peng+Zhai_2023} and \cite{Brz+Mot+Kosm+Razaf_2025}, where a fair attempt has been made to describe earlier contributions in that field.
The theory and several applications of
SPDEs with L\'evy noise have been considered in  the monograph \cite{Peszat+Zabczyk_2007}. 

These empirical and theoretical advances motivate the incorporation of pure jump
L\'evy noise into the transport problem. Such a formulation allows us to capture
anomalous transport and fluctuation effects that are beyond the scope of classical
diffusion models.

We also point out that the statistics of the L\'evy type noise turned out to be ubiquitous phenomena empirically observed in various areas including: physics (anomalous
diffusion, turbulent flows, nonlinear Hamiltonian dynamics \cite{[34 2025],[29-2025]}), biology (e.g., heartbeats \cite{[41]} and  firing of neural networks \cite{[42]}), seismology (e.g., recordings of seismic activity \cite{[43]}), electrical engineering (e.g., signal
processing \cite{[44],[46],[45]}), mathematical finance (\cite{Cont+Tankov_2004} and \cite{Barndorff-Nielsen_2001}), and economics  (e.g., financial time series \cite{[48],[49],[47]}).
Accordingly, an important problem is to understand the dynamics of
stochastic systems when the driving noises are L\'evy type. The present work represents a step toward this goal.

Let us now concentrate on  the motivation for working with equations in   the Marcus canonical form. It is well-known that the It\^o integral, despite its many remarkable features, fails to satisfy the classical chain rule and therefore It\^o SDEs are  not invariant under changes of coordinates. As a consequence, from a physical point of view, it does not provide a  natural geometric framework. In the Brownian motion driven systems, these drawbacks can often be overcome by using the Stratonovich integral instead of the It\^o one. Marcus canonical form is specifically designed as a natural extension of the Stratonovich calculus to the jump processes. Moreover,  stochastic differential equations written in the Marcus canonical form possess the geometric invariance under the change of variables, which yields a structural property that is crucial for our purposes. In particular, we prove that  under some regularity assumptions,  the unique weak$^\ast$-$\mathrm{L}^{\infty}$-solution to \eqref{eqn-transport-Markus} is given by a formula \eqref{eqn-def-u-intro0}, which involves the stochastic flow  generated by the equation \eqref{eqn-SDE-intro} (and associated with \eqref{eqn-transport-Markus}), see also  the discussion surrounding that formula.  This flow representation plays a central role in our uniqueness argument for the stochastic transport equation. On the other hand, the Marcus canonical integral can be interpreted as the limit of Wong-Zakai type approximations, where the jump noise is replaced by a sequence of smoother (piecewise smooth) processes. From physical and modeling  points of view, this perspective is also useful  for  numerical method approximating solutions of SDEs, see for instance \cite{KurtzPardouxProtter} and  \cite{PavlyukevichThipyarat25}, where SDEs with jumps with regular coefficients  are discussed. For studies on stochastic (partial) differential equations involving the Marcus integral and their background, we also direct readers to the seminal work \cite{Marcus 1978} and to the recent developments found in \cite{BLZ_2021,CP 2014, FPR 2025,HartmannPavlyukevich23,LT 2025}.

   \paragraph{Our first main results: The well-posedness results for \eqref{eqn-transport-Markus}.}
   
    In order to simplify the presentation of  this Introduction, here we only state our  results in  the case where the L\'evy process $L$ is a rotationally invariant  $\alpha$-stable process whose generator is the fractional Laplacian $(-\triangle)^{\frac{\alpha}{2}}$, i.e., when  $L$ satisfies our Hypothesis \ref{hyp-nondeg3}.  This is a special but  yet interesting case. 
  
More general L\'evy  processes  are considered in Section \ref{hy2}, where    we introduce three different assumptions on  $L$, see Hypotheses \ref{hyp-nondeg1}, \ref{hyp-nondeg2} and 
\ref{hyp-nondeg3}. These hypotheses   are  crucial for our main results. Roughly speaking, in the supercritical case  $\alpha \in (0,1)$, we have to assume that   $L$ is   rotationally invariant  but 
when $\alpha \in [1,2)$ we can consider more general non-degenerate $\alpha$-stable-type processes $L$.

 Let us first discuss the existence of solutions to \eqref{eqn-transport-Markus}, see   Theorem \ref{thm-transport equation} for more details, which holds under    more general assumptions on $L$.

  \begin{theorem}\label{thm-transport equation-rough} 
Assume  that  $\alpha \in (0,2)$ and 
\begin{equation}\label{eqn-alphabeta}
 \frac{\alpha}{2} +\beta  > 1.
\end{equation}
Assume that $L$ is  a rotationally invariant  $\alpha$-stable   L\'evy  process   whose 
 symbol $\psi$   is of the following form,   
\begin{gather} \label{stim4-2}
 \psi (u) = C_{\alpha} |u|^{\alpha},  \;\; u \in \R^d,
\end{gather}
%\eqref{stim4-2}.  
 for some    constant  $C_{\alpha}>0$, and  a vector field $b: \mathbb{R}^d \to \mathbb{R}^d$ satisfies conditions \eqref{eqn-g-beta} and 
  \begin{equation}\label{eqn-div b in L^1_loc} 
  \divv b \in L^1_{\loc}(\mathbb{R}^d). 
      \end{equation}
   Then for every    Borel measurable and essentially bounded function $u_0:\mathbb{R}^d \to \mathbb{R}$, the  stochastic process $u$ defined by the following formula 
\begin{equation}\label{eqn-def-u-intro0}
 u(t,x,\omega):=u_0\left(\phi_{0,t}^{-1}(\omega)(x)\right), \;\; t\in [0,T], x \in\mathbb{R}^d, \; \omega \in \Omega, 
\end{equation} %
where $(\phi_{0,t})_{t \geq 0}$ is the stochastic flow of diffeomorphisms associated to the SDE \eqref{eqn-SDE-intro} with $s=0$ and $\phi_{0,t}^{-1}$ is its inverse flow, 
is a weak$^\ast$-$\mathrm{L}^{\infty}$-solution to the problem \eqref{eqn-transport-Markus} in the sense of Definition \ref{def-transport-weak-2}. 
\end{theorem}

As we clearly see, the above result is not self-contained as  the stochastic flow $(\phi_{0,t})_{t \geq 0}$ appears in formula \eqref{eqn-def-u-intro0}. As stated before, it is related to the second main question of the paper, and it  will be addressed in Theorem \ref{ww1-1-intro}, which presents a refined version of the result.

  Now we state  two main uniqueness results. The first uniqueness result holds under the condition \eqref{eqn-alphabeta} but 
 it assumes that $\alpha \in [1,2)$. It can be proved  for more general L\'evy processes, and the complete statement can be found in Theorem \ref{thm-uniqueness-1}.
 The critical exponent $p\in(\alpha,2]$ is also included in this theorem. In contrast to the uniqueness result  for Brownian motion, in \cite[Theorem 6]{FGP-2012}, our critical range is independent of the spatial dimension $d$.

 \begin{theorem}\label{thm-uniqueness-1-intro}

Let us assume that $\alpha\in [1,2)$  and $\beta  \in (0, 1)  $ are such that \eqref{eqn-alphabeta} holds.
    
 Assume also that one of the following conditions holds. 
\begin{trivlist} \item[(1)] If    $d=1$, then $Db\in L^1_{loc}(\mathbb{R})$; if $d\geq 2$, then  $\divv b \in L^{p}( \mathbb{R}^{d})$ for some $p\geq2$.
 \item[(2)] If    $d=1$, then $Db\in L^1_{loc}(\mathbb{R})$; if $d\geq 2$, then $\beta >1-\alpha+\frac{\alpha}p$ and $\divv b \in L^{p}( \mathbb{R}^{d})$ for some $p\in(\alpha,2]$.
\end{trivlist}
Then, for every $u_{0} \in L^{\infty}(\mathbb{R}^{d})$, there exists a  unique weak$^\ast$-$\mathrm{L}^{\infty}$-solution to the 
%%problem \eqref{eqn-transport-Markus} 
transport equation. It is 
 of the form \eqref{eqn-def-u-intro0}.
\end{theorem}
The next uniqueness  result is quite unexpected because it also holds in the  \textit{supercritical case}  $\alpha \in (0,1)$. 
Recall that in this case the generator of the L\'evy process $L$ is the  fractional Laplacian      of order $\frac{\alpha}{2} <\frac{1}{2}$.  This   result shows regularization by noise even in the case of $\alpha \in (0,1)$. The complete statement of this result can be found in Theorem \ref{thm-uniqueness}.
\begin{theorem}\label{thm-uniqueness-intro}
 Assume that  $\alpha \in (0,2)$ and  $\beta  \in (0, 1)  $ are  such that 
\begin{align}\label{eqn-alphaover4+beta>1-intro}
    \frac{\alpha}4+\beta>1.
    \end{align} 
   Assume  also that condition \eqref{eqn-div b in L^1_loc} holds. % $\divv b \in L_{\mathrm{loc}}^{1}( \mathbb{R}^{d})$.
Then, for every $u_{0} \in L^{\infty}(\mathbb{R}^{d})$, there exists a unique weak$^\ast$-$\mathrm{L}^{\infty}$-solution to the transport equation.
 It is 
 of the form \eqref{eqn-def-u-intro0}.
\end{theorem} 

\begin{remark} \label{rem-fgp}

In a similar way,  we can treat the more general case of the time-dependent drift vector field $b(t,x)$. But this would substantially enlarge the paper and thus  we have decided
to postpone the treatment of time-dependent drift to a future publication.
\end{remark}

\begin{remark}
Theorems \ref{thm-uniqueness-1-intro} and \ref{thm-uniqueness-intro} differ primarily in two aspects. First, the function space for $\divv b$ differs: one requires
$L^p_{loc}(\R^d)$ with $p>1$, while the other assumes $L^1_{loc}(\R^d)$. Second, the relationship between the parameters $\alpha$ and $\beta$ is given by equations \eqref{eqn-alphabeta} and \eqref{eqn-alphaover4+beta>1-intro}, respectively. This distinction is natural: since Theorem \ref{thm-uniqueness-intro} imposes a weaker condition on $\divv b$, it correspondingly requires stronger assumptions elsewhere.
\end{remark}

In order to prove the above theorems, we need a couple of new  results of the flow associated with \eqref{eqn-SDE-intro}. 

\paragraph{Our second main result:  Existence and stability of stochastic  flow of diffeomorphisms.}
We obtain the following \textit{fine properties} of the flow associated with SDE \eqref{eqn-SDE-intro}, see  Theorems \ref{ww1} and  \ref{thm-stability} for more details, which  is not only essential for formulating the first problem  but also of great importance in itself.

\begin{theorem}\label{ww1-1-intro} 
  Let us assume that $T>0$,   $\alpha \in (0,2)$ and  $\beta \in (0,1) $ satisfy condition \eqref{eqn-alphabeta}. 
Then there exists a $\mathbb{P}$-full event $\Omega^{\prime\prime} \subset \Omega$ and a map 
\[
\phi:[0,T]\times [0,T] \times \mathbb{R}^d\times \Omega \to \mathbb{R}^d, \;\; 
\]
here $\Omega^{\prime\prime}$ is independent of $(s,t,x)\in[0,T]\times[0,T]\times\R^d$, such that,  for every $(s,x)\in [0,T] \times \R^d$, the process 
\[
\mbox{$[s,T] \times  \Omega \ni (t,\omega) \mapsto \phi_{s,t}(x, \omega) \in \mathbb{R}^d$, }
\]%
is the unique solution to  SDE
\eqref{eqn-SDE-intro}   and for all  $\omega \in \Omega^{\prime\prime}$,  $0\le s \le  t \le T$,
the mapping 
\begin{align}\label{flow-intro-1117}
\R^d \ni x \mapsto \phi_{s,t}(x) =\phi_{s,t}(x, \omega) \in \R^d
\end{align}
is  a surjective diffeomorphism of class $C^{1+\delta}$ for any $\delta\in [0,\frac\alpha2+\beta-1)$. More precisely, both $\phi_{s,t}(x)$ and its inverse $\phi_{s,t}^{-1}(x)$ satisfy the sharp c\`adl\`ag property  in $(s,t,x)$ separately:
\begin{trivlist}
\item[(scp)] for every $\omega \in \Omega''$,             
            the maps  
\[ (s,t,x)\mapsto \phi_{s,t}(x, \omega) \in \mathbb{R}^d,\quad\quad  (s,t,x)\mapsto \phi^{-1}_{s,t}(x, \omega) \in \mathbb{R}^d
\]
are c{\`a}dl{\`a}g w.r.t. $s$,  for all $(t,x)$  fixed, c{\`a}dl{\`a}g  w.r.t. $t$, for all  $(s,x)$ fixed,  and  continuous w.r.t. $x$,  for all $(s,t)$ fixed; 
\end{trivlist}
moreover, for every  $\omega \in \Omega^{\prime\prime}$ and $\delta \in (0,\frac{\alpha}2+\beta-1)$,
the derivative functions
\begin{align*}
&\mathbb{R}^d \ni x \mapsto D \phi_{s,t}(x)   \in \mathscr{L}(\R^d,\R^d) \\
& \mathbb{R}^d \ni x \mapsto  D (\phi_{s,t}^{-1} )(x) \in \mathscr{L}(\R^d,\R^d) 
\end{align*}
  are     locally $\delta$-H\"older continuous  uniformly in $(s,t) \in [0,T] \times [0,T]$;
  and 
   $ D \phi_{s,t}(x),  D (\phi_{s,t}^{-1} )(x) $
exhibit the \textit{sharp c\`adl\`ag property} in $(s,t,x)$ separately. 

Let $(b^{n})_{n=1}^\infty\subset C_{\mathrm{b}}^{\beta}(\mathbb{R}^d,\mathbb{R}^d)$
be a sequence of vector fields and $\phi^{n}$ the corresponding
stochastic flows. 
 Let $b^{n}\to b$ in $C_{\mathrm{b}}^{\beta}(\mathbb{R}^d,\mathbb{R}^d)$.
Then, for every $p\geq 1$ and any compact set $K \subset \R^d$, we have:
\begin{align}
&\lim_{n\to\infty}\sup_{x\in{\mathbb{R}}^{d}} \sup_{s\in [0,T]}
\mathbb{E}[ \sup_{t \in [s,T]} |\phi_{s,t}^{n}(x)-\phi_{s,t}(x)\vert^{p}]=0;\label{stability1-intro}\\
&\sup_{n\in\mathbb{N}}\sup_{x\in{\mathbb{R}}^{d}}\sup_{s\in [0,T]} \mathbb{E}[  \sup_{t \in [s,T]} \Vert D\phi_{s,t}^{n}(x)\Vert^{p}]
+
\sup_{x\in{\mathbb{R}}^{d}}\sup_{s\in [0,T]} \mathbb{E}[  \sup_{t \in [s,T]} \Vert D\phi_{s,t}(x)\Vert^{p}]
<\infty;\label{stability2-intro}\\
&   \,  \lim_{n\to\infty}\sup_{x\in K}\sup_{s\in [0,T]} \mathbb{E}[  \sup_{t \in [s,T]} \Vert D\phi_{s,t}^{n}(x)-D\phi_{s,t}(x)\Vert^{p}]=0.\label{stability23-intro}
\end{align}
and
\begin{align}
&\lim_{n\to\infty}\sup_{x\in{\mathbb{R}}^{d}} \sup_{t\in [0,T]}
\mathbb{E}[ \sup_{s \in [0,t]} |(\phi_{s,t}^{n})^{-1}(x)-\phi_{s,t}^{-1}(x)\vert^{p}]=0,\\
&\sup_{n\in\mathbb{N}}\sup_{x\in{\mathbb{R}}^{d}}\sup_{t\in [0,T]} \mathbb{E}[  \sup_{s \in [0,t]}\Vert D(\phi_{s,t}^{n})^{-1}(x)\Vert^{p}]
+
\sup_{x\in{\mathbb{R}}^{d}}\sup_{t\in [0,T]} \mathbb{E}[  \sup_{s \in [0,t]}\Vert D\phi_{s,t}^{-1}(x)\Vert^{p}]
<\infty,\\
&  \, \lim_{n\to\infty}\sup_{x\in K } \sup_{t\in [0,T]} \mathbb{E}[  \sup_{s \in [0,t]} \Vert D(\phi_{s,t}^{n})^{-1}(x)-D\phi_{s,t}^{-1}(x)\Vert^{p}]=0.\label{stability23-inverse-intro}
\end{align}

\end{theorem}

\paragraph{New contributions of the paper.}
In what follows,   we highlight \textit{our main new contributions} of the current  paper.   Items (i) and (ii) are related to  the second main result of the paper.
\begin{itemize}
     \item[(i)] The proof of  regularization property of the stochastic flow. We show that the flow associate to the characteristic equation \eqref{eqn-SDE-intro} is a  $C^{1+\delta}$-diffeomorphism for H\"older continuous coefficients, see Theorem \ref{ww1-1-intro}. The requirement that the $\mathbb{P}$-full event $\Omega^{\prime\prime}$ is independent of $(s,t,x)\in[0,T]\times[0,T]\times\R^d$ is critical, introducing inherent technical challenges.
    
    \item[(ii)] New regularity of the Jacobian determinant  of the solution to \eqref{eqn-SDE-intro} with respect to $x$.  This result requires delicate results on the Kolmogorov equations,  the use of stochastic integrals in Banach spaces of martingale type $2$, and sophisticated estimates in fractional Sobolev spaces (cf., Theorems \ref{Reg-est-Jaco} and \ref{Reg-est-Jaco-2}). 
    \item[(iii)]  
    The  linking between (i), (ii) and the above uniqueness theorems on the transport equation  is achieved via 
the crucial It\^o-Wentzell formula in the jump case. To the best of our knowledge, there is only one paper that rigorously applies   the It\^o-Wentzell formula in the jump case, see \cite{Leahy+Mik_2016}, which is  based on the seminal paper \cite{Mik_1983}. In our setting with singular coefficient $b$, a rigorous application and verification of the It\^o-Wentzell formula require a meticulous and careful calculation.

\item[(iv)] The pathwise uniqueness of weak$^\ast$-$\mathrm{L}^{\infty}$-solutions to the linear transport  equation \eqref{eqn-transport-Markus}, 
 with the initial data  $u_0$ being an \textsl{essentially bounded} Borel measurable function, are established for $\beta$-H\"older continuous drifts, $\beta\in(0,1)$, and $\alpha$-stable type processes, $\alpha\in(0,2)$, including the supercritical case, i.e., $\alpha\in(0,1)$. 
\end{itemize}

We now describe these main contributions in more detail.

Firstly, the method of constructing solutions analogous to the transport equation \eqref{eqn-transport-Markus} is referred to as the method of characteristics. This requires starting from the   stochastic characteristics equation  \eqref{eqn-SDE-intro} for the transport equation \eqref{eqn-transport-Markus}. In the stochastic case, it is essential to study the fine properties of the stochastic characteristic equation, i.e., the flow of equation \eqref{eqn-SDE-intro}, as detailed in Theorem \ref{ww1-1-intro}, which is foundational for our analysis. The established theory of stochastic flows for SDEs driven by jump processes, even for regular Lipschitz coefficients, is insufficient for our purposes. The latest existing results on this topic are summarized/discussed below:
\begin{itemize}
\item The \textit{sharp c\`adl\`ag property} (scp) for the flow and its inverse, which is crucial for defining our solutions via the method of characteristics, remains unaddressed in all prior works. For instance, Kunita in his recent monograph on jump diffusions \cite{Kunita_2019} considers possibly  degenerate SDEs with coefficients 
satisfying the Lipschitz conditions but does not prove,  or even mention the scp for the flow and its inverse.  
Meanwhile, \cite{Pr20} (corresponding to Theorem \ref{d32} in this paper) examines the case of singular drift  $b$ with  non-degenerate stable noise  but only addresses the scp of the flow itself, neglecting  the scp for the inverse flow. 
This scp is discussed in more detail in a recent paper  \cite{Bondi+Priola_2023} by Bondi and the second named author, and they proved the scp for the flow  by independent methods for  possibly degenerate SDEs with  Lipschitz coefficients, but they also do not study  the scp for the inverse flow, see  Remark \ref{rem-kunita}  and the comment at the end of page 7 in \cite{Bondi+Priola_2023} for more details.

\item When the vector $b$ is singular, the authors of \cite{CSZ18}, \cite{Pr12} and \cite{Pr15}  investigate the homeomorphism property and differentiability with respect to the spatial variable for solutions of equation \eqref{eqn-SDE-intro} for $\alpha \in [1, 2
)$ and $\alpha \in (0, 2)$, respectively, see Theorems \ref{uno} and \ref{uno1} in this paper. However, the existence of a $\mathbb{P}$-full set independent of $s,t$ and the $C^1$ and $C^{1+\delta}$-diffeomorphism property of the flow \eqref{flow-intro-1117} have not been studied. The requirement of a $s,t$-independent $\mathbb{P}$-full set is essential for resolving our problems.
\end{itemize}

\indent In contrast to these existing results, our result described in Theorem \ref{ww1-1-intro},  not only establishes the scp for the flow and its inverse but also demonstrates that their spatial derivatives exhibit the scp. A key contribution is  that we prove the existence of  a full-$\mathbb{P}$ event, independent of $s,t,x$, on which the stochastic flow $\phi_{s,t}(x)$ associated with our SDE \eqref{eqn-SDE-intro},  is a $C^1$-diffeomorphism. Furthermore, we establish that its derivatives $D\phi_{s,t}(x)$ and $D\phi^{-1}_{s,t}(x)$ are locally $\delta$-H\"{o}lder continuous in $x$ and exhibit sharp c\`adl\`ag properties in the time parameters separately.   
As highlighted earlier, there is a significant gap between the existing results and what is required for our purposes. Obtaining these precise and reliable results poses considerable challenges in our specific setting.  The principal obstacles/challenges are summarized as follows:
\begin{itemize}
\item 
 The $\rm It\hat{o}$-Tanaka approach, which is similar
to the Zvonkin transformation, plays a central role in the study of the  \textit{fine properties}  of the flow associated with \eqref{eqn-SDE-intro}. Heuristically, this approach transfers the regularization of the law of the diffusion associated with a corresponding parabolic/Kolmogorov equation to its sample paths. 
In our setting, the associated  parabolic (Kolmogorov) equation takes the form
 \begin{align}\label{eq-1125-2-intro}
 \lambda v -  \gen{A} v  - b \cdot Dv = f,\;\; \lambda>0,
 \end{align}
where $\gen{A}$ is the generator of the pure jump L\'evy process $L$. In this case $\gen{A}$ is a nonlocal operator. For the simplicity of presentation, we assume as before that $L$ is a rotationally invariant  $\alpha$-stable process whose generator is the fractional Laplacian $(-\triangle)^{\frac{\alpha}{2}}$. When $\alpha\in(1,2)$, the leading term in \eqref{eq-1125-2-intro} is $\gen{A} v$. Here, by leading term we mean the term that dominates the regularization effect in \eqref{eq-1125-2-intro}. When $\alpha=1$,  $\gen{A}v$ and $D v$ have the same order and both serve as leading terms, and it is necessary to take into account the interaction between these two terms. In contrast, when $\alpha\in(0,1)$, the leading term comes from $Dv$, while $\gen{A}v$, capturing the driving noise's regularization effect, plays a subordinate role. In this supercritical case, observing the regularization by noise phenomenon becomes extremely difficult. 

For comparison, when the noise is Brownian motion, $\gen{A}$ is the Laplacian (corresponding to $\alpha=2$). Clearly, the jump noise in our setting produces weaker regularization than the Brownian motion case. Additionally, the Brownian motion case involves a local operator, whereas our framework deals with a nonlocal operator. As a result, the analysis of all of the main results in our setting becomes considerably more challenging, especially for $\alpha\in(0,1)$. This explains why  obtaining the uniqueness result in the supercritical case is somewhat unexpected. Indeed,  the regularization by noise phenomenon in supercritical case still contains several open problems, see Remark \ref{nounique} for related discussions. Finally, we note that, due to these analytical difficulties, for $\alpha\in(0,1)$ our results yield only the local stability of the flow, see \eqref{stability23-intro}, \eqref{stability23-inverse-intro}, \eqref{eqn-conv-u_n to u 0}, Remarks \ref{cdd} and \ref{remark 2025-11-226}. Nevertheless, this is still sufficient to establish the uniqueness of the solutions to the transport equation \eqref{eqn-transport-Markus}. 

\item The proof of this diffeomorphism property is substantially more involved than its Brownian motion counterpart and requires new analytical tools. In order to establish the  \emph{sharp} $C^{1+\delta}$-diffeomorphism property of the stochastic flow $\phi_{s,t}$, the core of our analysis lies in a careful investigation of the properties of $\phi_{0,t}(x)$, its inverse mapping $\phi_{0,t}^{-1}(x)$ and their spatial derivatives. The flow and its inverse satisfy, respectively, the following stochastic differential equations:
\begin{align} 
&\nxi_{0,t}(x) =  x + \int_0^t b(\nxi_{0,r}(x)) dr +L_t,\quad t\geq 0;\label{eq-intr-1126-1}\\
&\nxi_{0,t}^{-1}(x) =  x - \int_0^t b(\nxi_{r,t}^{-1}(x)) dr -L_t, \;\;   t\geq 0.\label{eq-intr-1126-2}
\end{align}
Even though equations \eqref{eq-intr-1126-1} and \eqref{eq-intr-1126-2} appear structurally similar, a closer examination reveals that they rely on fundamentally different analytical structures. In equation \eqref{eq-intr-1126-1}, the drift term $ b(\nxi_{0,r}(x))$ depends on the forward flow $\nxi_{0,r}(x)$, which is the central subject of our study. This structure makes it possible to adapt  classical arguments; however, their application is nontrivial due to the singularity of 
$b$ and the limited regularization effect of the driving noise $L$. In contrast, equation \eqref{eq-intr-1126-2} is an SDE in the backward time direction, where $b(\nxi_{r,t}^{-1}(x))$ involves the inverse flow $\nxi_{r,t}^{-1}$, $r\in[0,t]$, not $\nxi_{0,r}^{-1}$. Such a structural difference makes the analysis of $\nxi_{0,t}^{-1}$ and, in particular, its derivative substantially more delicate. While one may formally write $b(\nxi_{r,t}^{-1}(x))$ as $b(\nxi_{0,r}\circ \nxi_{0,t}^{-1}(x))$, such a representation does not simplify the analysis, as $\phi_{0,r}(y), y\in \R^d$ is a family of stochastic processes. To overcome this difficulty, we introduce a new auxiliary stochastic process $K$ that characterizes the derivative of the inverse flow through a representation identity, see \eqref{eq-1126-2-K} and \eqref{inverse-identity-eq-101}. Importantly, the process $K$ possesses better regularity properties, making the analysis of $D\phi_{0,t}^{-1}$ considerably more tractable. This construction  is highly non-trivial, as it relies on a combination of the It\^o-Tanaka trick, the inverse mapping theorem and a strengthened version of the Kolmogorov test. By means of these analytical ingredients, we derive  a path-wise formulation for the derivative of the inverse flow $D\phi_{0,t}^{-1}$, see \eqref{eq 202412007 03}. This approach not only yields a rigorous representation of $D\phi_{0,t}^{-1}$ but also provides the key tool for proving the local H\"older continuity and sharp  c\`adl\`ag property of the derivative of the inverse stochastic flow i.e.,  the \textit{sharp c\`adl\`ag property} of $D\phi_{s,t}^{-1}(x)$ in $(s,t,x)$ separately.

\item Let us emphasize again  that the above properties hold on a full-$\mathbb{P}$ event  $\Omega''$ which is independent of $s,t$ and $x$.  This aspect gives rise to the third major challenge and highlights another key distinction between  our result and the existing results. The independence of $\Omega''$ from $s,t,x$ is crucial for our analysis and is repeatedly used throughout the entire proof. Establishing such $\mathbb{P}$-full set independent of $s,t$ is considerably more involved than dealing with fixed values of $s$ and $t$. It requires delicate analytical arguments and the development of new technical ingredients, see for instance, the proofs of Lemma \ref{ef}, Theorem \ref{thm-ww} and Theorem \ref{ww1}.

\end{itemize}

Secondly, with the non-regular vector fields $b$, our uniqueness result for  equation \eqref{eqn-transport-Markus}  is proved by means of the concept of renormalized solutions and commutator lemmas. This approach requires a detailed analysis of the Jacobian determinant of the stochastic flow, denoted by $J\phi_{0,t}$, which can be written as  %poses substantial technical challenges.
\begin{equation}\label{eq-intr-1126-10}
\log J \phi_{0,t}(x)=\int_{0}^{t} \divv b\left(\phi_{0,s}(x)\right) d s.
\end{equation}
Since the drift $b$ is only $\beta$-H\"older continuous with some  $\beta\in(0,1)$, we must explore the improved regularity of the integral term of the right-hand side in order to establish the desired properties of the Jacobian. We again employ the It\^o-Tanaka trick and  the regularizing effect of the associated nonlocal parabolic equation \eqref{eq-1125-2-intro}, see in particular Theorem \ref{TH-Lp-est}, which is based on results  by \cite{Zhang_2013}. In comparison with the Brownian motion case, see  \cite[Theorem 10]{FGP_2010-Inventiones}, where the $L^p$-parabolic estimates $\|v\|_{W^{2,p}(\mathbb{R}^d)}\leq C\|f\|_{L^p(\mathbb{R}^d)}$ is used, our nonlocal parabolic equation \eqref{eq-1125-2-intro} has poor regularity:  $
 \Vert v\Vert_{H^{\alpha,p}(\mathbb{R}^d)}
  \le C \Vert f\Vert_{L^p(\mathbb{R}^d)}$, with only  $\alpha\in[1,2)$. Here $v$ is  the unique solution  of equation \eqref{eq-1125-2-intro} with RHS $f$. Specifically, when $\alpha\in(1,2)$, our analysis is confined to fractional Sobolev spaces $H^{\alpha,p}(\R^d)$. Since $\alpha$ is not an integer, several classical embeddings fail. For instance, even for a bounded domain $\mathcal{O}\subset \mathbb{R}^d$, the embedding $W^{\delta, p}(\mathcal{O}) \subset W^{\delta, 1}(\mathcal{O})$, $p>1$ no longer holds when  $\delta$ is not an integer, see \cite{Mir+Sic}. Because of weaker regularizing effect of our nonlocal parabolic equation and the failure of certain Sobolev embeddings, we can only prove that  $J\phi_{0,t}$ belongs to  a local fractional-order Sobolev space $ W_r^{\delta,p}$ for some $\delta\in(0,1)$ and $p>1$. See Theorems \ref{Reg-est-Jaco} and \ref{Reg-est-Jaco-2} for  a complete description.   The analysis within our framework poses some substantial technical challenges.  Indeed, we develop new regularity results, i.e., Theorems \ref{Reg-est-Jaco} and \ref{Reg-est-Jaco-2}, concerning the fractional Sobolev‑type regularity of the Jacobian determinant. The proofs of these theorems rely on delicate analyses of Kolmogorov equations. Moreover, we employ a tool different from \cite{FGP_2010-Inventiones}, namely  estimates for stochastic integrals with respect to compensated Poisson random measures in Banach spaces of martingale type 2;  see \cite{Zhu+Brz+Liu-2019} and the references therein. Our argument also relies on sophisticated  estimates in fractional Sobolev spaces, see e.g., Lemma \ref{lem-Sob-space-comp}, \eqref{eq 20251108 01}, \eqref{eq 20251108 02},  \eqref{eq 20251108 03} and \eqref{eq-20251108-1}. These techniques do not arise in the Brownian motion case, as the Jacobian's behavior in jump settings introduces additional complexities that are absent in the Brownian Motion  case.

Thirdly, we bridge the existence of the  regular flow with the well-posedness of the transport equation via the It\^{o}-Wentzell formula in the jump case. While prior work \cite{Leahy+Mik_2016} provides a rigorous foundation, its application to our SPDEs context is non-trivial. Verifying its assumptions and reformulating it into a form suitable for our specific needs, as detailed in the appendices, constitutes a further contribution. To be more specific, we briefly recall the underlying problem and illustrate the formula in our framework. Let
 $$
 F(t,y):=u^\eps(t,y)= \int_{\mathbb{R}^d} \vartheta_\eps(y-x) u(t,x)\,dx,\quad\quad \xi_t=\phi_{0,t}(x),
  $$
where $\eps>0$ and  $\vartheta_\eps$ is a  standard mollifier, $u$ is a weak$^\ast$-$\mathrm{L}^{\infty}$-solution to the transport equation and $\phi_{0,t}$ is the associated stochastic flow. For each fixed $y \in \mathbb{R}^d$, the process $F(t, y)$ satisfies SDE \eqref{eqn-transport-Marcus-2}, and  $\xi_t$ satisfies \eqref{eqn-SDE-intro} with $s=0$. The It\^o-Wentzell formula investigates the conditions under which the composition process $F(t,\xi_t)$ is well defined and yields the explicit stochastic dynamics governing its evolution. A major difficulty lies in the random process $F(t,y)$, which satisfies the jump-driven SDE \eqref{eqn-transport-Marcus-2} for each fixed $y\in\mathbb{R}^d$, only on $y$-dependent $\mathbb{P}$-full set. To rigorously define the composition process $F(t,\xi_t)$, we must establish a universal $\mathbb{P}$-full set independent of $y$. This necessitates a stronger form of regularity of the random field  $F(t, y)$, $y\in\R^d$. Specifically, we must establish the existence of a $\mathbb{P}$-full event, on which the trajectory of $F(t, y)$, $y\in\R^d$ belongs to $D([0,T];C^2(\mathbb{R}^d;\mathbb{R}))$. In other words, on this $\mathbb{P}$-full event, for every $t\in[0,T]$, $F(t, y)$, $y\in\R^d$ is twice continuously differentiable in space and is a $C^2(\mathbb{R}^d;\mathbb{R})$-valued c\`adl\`ag function w.r.t. $ t \in [0, T ]$; see Lemma \ref{lem-u^eps satisfies assumptions ITWF} for the proof. In addition, one must carefully check the relation between jumps of the processes $[0,T] \ni t \mapsto F(t,y)$, $y \in \mathbb{R}^d$,  and those for $\xi_t$ such as establishing 
\begin{align}
	\sum_{s\leq T}\;\sup_{y\in \mathbb{K}}|D_yF(s,y)-D_yF(s-,y)||\xi_{s}-\xi_{s-}|<\infty,\quad \mathbb{K}=[-K,K]^d,
	\end{align}
    and analyzing sums of the form
    $$\sum_{s\leq t}\Big(\Delta F(s,\oldL_{s})-\Delta F(s,\oldL_{s-})\Big),\quad\quad here\ \Delta F(s,y):=F(s,y)-F(s-,y).$$
We refer to Appendix \ref{app}, in particular \eqref{eq-1029-1} and \eqref{s33}, for the details. 
Handling these interactions requires a careful analysis of the jump structures of c\`adl\`ag paths.
 Consequently, such a formula  is  a very  delicate issue. To clarify this point,   we first include  Appendix \ref{sec-IWF} in which we provide a simplified reformulation of Proposition 4.1 from the very nice paper \cite{Leahy+Mik_2016} by Leahy and Mikulevicius, as well as  correct some small its inaccuracies.  We would like to emphasize   that even the  simplified  form of  the It\^o-Wentzell formula 
given in  Appendix \ref{sec-IWF} 
   is    quite difficult to apply in concrete 
examples related to SPDEs. Indeed, as mentioned above, it is a challenging  task to verify  whether  the several  assumptions of the It\^o-Wentzell formula in the jump case  
 hold  and also to simplify such a formula in a convenient way. In a second step, in Appendix \ref{app}, we verify that  the assumptions of our result  from Appendix \ref{sec-IWF} are indeed satisfied in the case of 
 $u^{\eps}(t,y), y\in\R^d$ and $\phi_{0,t}(x)$.  This procedure plays a crucial role in establishing the uniqueness of the stochastic transport equation, see  the proofs of Theorems  \ref{thm-uniqueness-1} and \ref{thm-uniqueness}. Overall,  the  situation appears to be very different with respect to the one in  the well-known It\^o-Wentzell formula in  the continuous case, see e.g., the monograph \cite{Kunita_1990} and the proof of Theorem 20 in \cite{FGP_2010-Inventiones}.

 It is worth mentioning that even the definition of weak$^\ast$-$\mathrm{L}^{\infty}$-solution requires special attention in the present jump case, see in particular Section \ref{subsection-def-solution}, Remark \ref{d88} and Lemma \ref{lem-transport-weak-2-(i)}. As mentioned earlier,  we consider the so-called weak$^\ast$-$\mathrm{L}^{\infty}$-solutions  of the linear transport type equation \eqref{eqn-transport-Markus}, 
 with the initial data  $u_0$ being an \textsl{essentially bounded} Borel measurable function and  the stochastic integration is understood in the Marcus form, see e.g., \cite{Applebaum_2009} and \cite[page 82]{Protter_2004}. 
 The weak form of this equation is obtained by (formally)  integrating against test functions $\theta \in C_c^{\infty}(\R^d)$ assuming some integrability conditions on $\divv b$ as in \cite{FGP_2010-Inventiones}; see  Definition  \ref{def-transport-weak-1} in Section \ref{sec-existence}. We stress that our  definition of a weak solution already  requires overcoming some technical difficulties, which are not present in the Brownian Motion case considered  in \cite{FGP_2010-Inventiones}, see in particular the discussion at the beginning of Section \ref{sec-existence} and Differences \ref{differences-01}.  In particular, the Marcus form of equation \eqref{eqn-transport-Markus} can be rewritten using the It\^o  stochastic integrals with respect to compensated Poisson random measures, see Definition \ref{def-transport-weak-2} and compare this to the Stratonovich equation in  the Brownian Motion case, see  \cite{Brz+Elw_2000}.

Finally, after overcoming the above difficulties, the pathwise uniqueness of weak$^{\ast}$-$\mathrm{L}^{\infty}$-solution  is established under the condition $b\in C^\beta_{\mathrm{b}}(\mathbb{R}^d,\mathbb{R}^d)$ for some  $\beta\in(0,1)$ and $\divv b \in L^{p}_{loc}( \mathbb{R}^{d})$ for $p=1$ with $\alpha\in(0,2)$ and some $p\in (\alpha, 2)\cup [2,\infty)$ with $\alpha\in[1,2)$, including the critical interval $p\in(\alpha,2]$. 
 Our pathwise uniqueness results reveal a regularization by noise phenomenon in the setting of L\'evy-driven SPDEs, as without noise there are examples of non-uniqueness under the above weak conditions on $b$. To the best of our knowledge, this is the first rigorous demonstration of such an effect within the L\'evy-driven SPDEs framework.

Importantly, we also prove in Appendix \ref{perturb} that the definition of a weak$^{\ast}$-$\mathrm{L}^{\infty}$-solution can be formulated equivalently without using the Marcus canonical integral. The proof in our jump case is delicate. We believe that this result is of independent interest. It  is in the spirit of rough path theory. For recent developments of rough path theory for SDEs driven by  L\'evy processes, see \cite{KrempPerk} and the references therein. 
 In addition, we   establish  a no blow-up  result of $C^1$ solutions in Appendix \ref{sec-no blow up}. Its proof again requires the use of the It\^o-Wentzell formula in the jump case.
This result is  related to the one proved in \cite{FGP-2012} in the case of  Brownian motion, see also  \cite{fedrizzi}.

 It is an open problem  to extend the  two Wong-Zakai approximation results
proved in Appendix C of \cite{FGP_2010-Inventiones} to our jump case. 
 We plan to address this problem in future work.

 %----------

\paragraph{Plan of the paper and additional comments.} 

The paper is organized as follows. Section \ref{sec-preliminaries} is dedicated to preliminaries on L\'evy processes, function spaces, and our main assumptions. Section \ref{sec-regular flow} is devoted to the study of the stochastic flow, where the regularity results are established. Section \ref{sec-stability} presents stability estimates essential for subsequent analysis, with particular emphasis on the case $0<\alpha<1$. Section \ref{sec-existence} proves the main existence theorem for the transport equation,  with a rough version of that result  already formulated in Theorem \ref{thm-transport equation-rough}. Two main uniqueness results for the transport equation, see Theorems \ref{thm-uniqueness-1} and \ref{thm-uniqueness}, are proved in Section \ref{uni1}. To establish these uniqueness results,  we also derive  two new fractional Sobolev type regularity results concerning the Jacobian determinant of the stochastic flow, namely Theorems \ref{Reg-est-Jaco} and \ref{Reg-est-Jaco-2}.

The It\^o-Wentzell formula in the jump case   is presented in Appendix \ref{sec-IWF}. In Appendix \ref{app}, we verify  several  assumptions of the It\^o-Wentzell formula for $u^\varepsilon$ and $\phi_{0,t}$ and apply it to our problem.  We  prove  a result of no blow-up of $C^1$ solutions in Appendix \ref{sec-no blow up}. Appendix \ref{perturb} shows that the definition of weak$^\ast$-$\mathrm{L}^{\infty}$-solution can be formulated  equivalently without using the Marcus integration form.  Finally, Section \ref{sec-proofs of 3 lemmata}  contains the proof of some technical lemmata, and Section \ref{App-sec-Kolm-Th} reviews  the  Kolmogorov-Totoki-Cencov-type theorem
   in a  form we need in the paper.

\begin{remark}\label{zhang} 
We mention that  \cite{Hao+Zhang_2020} investigates the well-posedness of classical solutions to the following  random transport equation:
 \begin{align}\label{eqn-transport-strong3-Zhang}
 & \frac{\partial{u(t,x,\omega)}}{\partial t}+(b(t,x)+L_t(\omega))\cdot Du(t,x,\omega)\,=0,~t>0;\\
  & u(0,x)=u_0(x),\ \ x\in\mathbb{R}^d.
\end{align}
The authors assume that  $u_0 \in C^{1}_{\rm b}(\R^d)$, $L= (L_t)_{t\geq0}$ is an $\alpha$-stable type process with  $\alpha >1$, and 
$b: [0,T] \times \R^d$ is bounded
and $\beta$-H\"older continuous in $x$ uniformly in $t$ with $\beta \in  (\frac{2+ \alpha}{2(1+ \alpha)}, 1)$; this condition  implies $\frac{\alpha}{2}  + \beta >1$.  The previous equation is  different from our equation. 
   Similarly, in the case of Brownian motion $(W_t)_{t\geq0}$,  equation \eqref{eqn-transport-strong3-Zhang} with $L_t$ replaced by $W_t$
 is not the    equation considered in \cite{FGP_2010-Inventiones}, where
 equation  is formally like  
 \[\frac{\partial{u}(t,x,\omega)}{\partial t}+ (b(t,x)+  \frac{dW_t}{dt})\cdot Du(t,x,\omega)=0.
 \]
 \end{remark}

\begin{remark}\label{rem-kunita} 

Let us consider the following  possibly degenerate SDE with regular (Lipschitz)  time-dependent coefficients 
\begin{align}\label{eqn-SDE-regular}
\xi_{s,t}(x) &= x+ \int_s^t b(r, \xi_{s,r}(x))\,dr +   \int_s^t \int_{\mathbb{R}^d} g(r, \xi_{s,r-}(x))\tilde{N}(dr,dz), \;\;\;   t\geq s;  
\end{align}
see equation (3.4) and the proof of Theorem 3.1 in \cite{Kunita_2004}.  Here $\tilde{N}$ is a compensated Poisson random measure on $[0,\infty) \times \mathbb{R}^d$. 

 According to  our discussion after \eqref{eqn-SDE-intro},  the existence of a solution flow associated with this equation is of utmost importance. For instance, 
a natural question is whether there exists a function 
\[
\phi:[0,T]\times [0,T] \times \mathbb{R}^d\times \Omega^\prime \to \mathbb{R}^d
\]
where $\Omega'$ is  a $\mathbb{P}$-full event, such that 
 for every $(s,x)\in [0,T] \times \R^d$, the process   $\big ( \phi(s,t,x , \cdot ): t \in [s,T] \big)$
is a strong solution of Problem \eqref{eqn-SDE-regular} and the  sharp  c{\`a}dl{\`a}g
property (scp) is satisfied.
 Property  (scp) has not been stated  in the recent monograph \cite{Kunita_2019} by Kunita.		
 On the other hand, it has been  stated in the paper \cite{Kunita_2004} but    
 the proof in   \cite{Kunita_2004}  is not complete.
 
	 Indeed, on  \cite[page 353$_{4-3}$]{Kunita_2004},  it is stated, \textsl{but not proved}, that,  
  using  notation $\xi_t:=\xi_{0,t}$, the inverse flow process  $\xi_t^{-1}$ is a well-defined $C(\mathbb{R}^d , \mathbb{R}^d )$-valued c\`adl\`ag process.    If this assertion were true, then we would be able to define the process satisfying the sharp 
  c\`adl\`ag property as follows: 
  $\phi(s,t,x) =x$,    if  $t \le  s$ and     $\phi(s,t, x)=  [\xi_t \circ  \xi_s^{-1}](x)$,   if  $t \ge  s$.

In this way, the claims at the beginning of \cite[Page 354]{Kunita_2004} would be completely justified. 
However,  we have not 
found a proof of such a  statement in either \cite{Kunita_2004} or \cite{Kunita_2019}. On the other hand, 
the  sharp c\`adl\`ag property has been recently proved by independent methods in  a paper  \cite{Bondi+Priola_2023} 
by the second named author and Bondi for Lipschitz coefficients.  

\end{remark}

\section{Preliminaries}\label{sec-preliminaries}

%------------------------------

%-------------section{Notation}--------------------------

\subsection{Notation} By $\mathbb{N}$ we will denote the set of all natural numbers starting from $0$. By $\mathbb{N}^\ast$ we will denote the set of those natural numbers  which are $\geq 1$.
By $\mathbb{R}$ we will denote the set of all real numbers. By $\mathbb{R}^\ast$ we will denote the set of those real numbers which are different from $0$  and
by $\mathbb{R}^+$ we will denote the set of those real numbers which are larger than $0$. By a countable set we will understand an infinite countable set, i.e., a set which is bijective with the set $\mathbb{N}$. Let $\mathbb{Q}$ be the set of all rational numbers in $\mathbb{R}$.

If $(X,\mathscr{X},\mu)$ is a measure space, we call a subset $B$ of $X$, a $\mu$-full set if and only if there exists a set $A \in \mathscr{X}$ such that $A \subset B$ and $\mu(X\setminus A)=0$.

Suppose that $(Y,\vert \cdot\vert_{Y})$ is a normed vector space, e.g., $Y=\mathbb{R}^d$. If $Y=\mathbb{R}$, then we will omit the symbol $Y$, e.g., we will write $C_{\mathrm{b}}(\R^d)$ instead of $C_{\mathrm{b}}(\R^d,\R)$. A similar convention is also used for other function spaces.

Assume that $\beta\in (0,1)$. If $K \subset \R^d$, then  we define  semi-norms  $\Vert  \cdot \Vert_{0,K}$,  $[ \cdot ]_{\beta,K}$ and $\Vert  \cdot \Vert_{\beta,K}$, by, for  a function $f: K \to Y$, 
\begin{equation}\label{eqn-seminorm-0}
\Vert  f \Vert_{0,K}:= \sup \bigl\{  \vert f(x) \vert_{Y}:\; x \in K \bigr\},
\end{equation}
\begin{equation}\label{eqn-seminorm-beta}
[f]_{\beta,K}:= \sup \Bigl\{  \frac{\vert f(x)-f(y) \vert_{Y} }{\vert x - y \vert^\beta}: x,y  \in K\text{ with } x\not=y\Bigr\},
\end{equation}
and
 \begin{equation}\label{eqn-norm-beta}
\Vert  f \Vert_{\beta,K}:=
\Vert  f \Vert_{0,K}  +  [f]_{\beta,K}.\\
\end{equation}

 By $C_{\mathrm{b}}(K,Y)$,  we will denote the Banach space of all continuous and bounded functions $
f : K \to Y $ endowed with the supremum norm  \eqref{eqn-seminorm-0}.  
By  $C_{\mathrm{b}}^{0}(K,Y)$ we will denote the closed sub-space $C_{\mathrm{b}}(K,Y)$ consisting of all
  uniformly  continuous functions. 
If {$K$ is compact and }$Y$ is separable,  then also the spaces $C_{\mathrm{b}}(K,Y)$ and $C_{\mathrm{b}}^0(K,Y)$. If in addition $Y$ is complete,  then 
these two spaces are  separable  Banach space as well.

By $C_{\mathrm{b}}^{\beta}(K,Y)$ we will denote the non-separable  Banach space of all
  bounded and $\beta$-H\"older continuous functions $
f : K \to Y $ endowed with the
 norm   \eqref{eqn-norm-beta}.
Let us point out that the above definition makes sense also for  $\beta=1$ what  corresponds to the Lipschitz condition. However, 
the seminorms and normes in this case are defined differently, see below. 

Let us also observe that
\eqref{eqn-seminorm-beta} makes sense for $\beta=0$ and that
\[
[f]_{0,K} \leq 2 \Vert  f \Vert_{0,K}.
\]
If $K =\R^d$, then we will omit the subscript $K$ from the notation introduced above. 
 By $C_c^\infty(\R^d,Y)$  we will denote  the vector space of all infinitely differentiable functions from $\mathbb{R}^d$ to $Y$ with compact support. 
 Note that if $\beta\in (0,1)$ {and $Y$ is separable and complete}, then  the closure in $C_{\mathrm{b}}^{\beta}(\R^d,Y)$ of the space
$C_{c}^{\infty}(\R^d,Y)$, denoted by $c^\beta_{\mathrm{b}}(\R^d,Y)$, is  a separable Banach space.

  Similarly,  if   $n \geq 1$ is a  natural number, then for a suitable function $f: K \to Y $ we put 
\begin{equation}\label{eqn-norm-n}
\Vert  f \Vert_{n,K}:= \sum_{k \in \mathbb{N}^d: \vert k \vert \leq n}   \sup \,\bigl\{   \Vert D^k f(x) \Vert: \;x \in K\bigr\}
\end{equation}
and 
 \begin{equation}\label{eqn-norm-n+beta}
 \Vert f\Vert_{n+\beta,K}:=\Vert f\Vert_{n,K}
+ \sum_{k \in \mathbb{N}^d: \vert k \vert = n}
 [D^k f]_{\beta,K}.
\end{equation}
The last term on the RHS of \eqref{eqn-norm-n+beta} we will also denote by $\left[D^n f\right]_{\beta,K}$. As before, we omit the subscript $K$ if $K=\R^d.$ 

By  $C_{\mathrm{b}}^n(\R^d,Y)$  we will denote the separable  Banach space of all
  $C^n$-class  functions $ f : \R^{d} \to Y $  whose derivatives up to order $n$ are bounded and uniformly continuous,
    endowed with the
 norm  \eqref{eqn-norm-n}. 
 We also define \[C_{\mathrm{b}}^{\infty}(\mathbb{R}^d,Y)=\bigcap_{k =1}^\infty C_{\mathrm{b}}^k(\mathbb{R}^d,Y).\]
If   $n \geq 1$  is a  natural number  and  $ \beta \in(0,1)$, then we say that a function $f: \mathbb{R}^d \to  Y$ belongs to $C_{\mathrm{b}}^{n+\beta}(\mathbb{R}^d,Y)$ if $f \in C_{\mathrm{b}}^n(\mathbb{R}^d,Y)$ and,
for every $k \in \mathbb{N}^d$ such that $ \vert k \vert  = n$
 the partial derivatives $D^k f \in C_{\mathrm{b}}^\beta(\mathbb{R}^d ,Y)$. Note that $ C_{\mathrm{b}}^{n+\beta}(\mathbb{R}^d,Y)$ is a Banach space endowed with the norm 
 $\Vert \cdot \Vert_{n+\beta}$ defined  by \eqref{eqn-norm-n+beta}.

In our paper we will also need the following fundamental results, see \cite[Lemma 4.1]{Pr12} and inequality (4.16) therein. Our formulation of the first result is a bit more precise and a bit stronger. By $(\tau_x)_{x \in \mathbb{R}^d}$ we denote the group of translation operators on the space $C_{\mathrm{b}}^{0}(\R^d,Y)$ defined by
\begin{equation}
\label{rqn-tau_x}
(\tau_x f)(y):= f(x+y), \;\; y\in \R^d.
\end{equation}

\begin{lemma} \label{lem-Lemma 4.1}
Assume that $Y$ is a normed vector space and $\gamma \in (0, 1]$.
If   $f \in C^{1+ \gamma}_\mathrm{b}(\R^d,Y)$ and  $x, y, z \in \R^d$, then
 \begin{equation}\label{eqn-Lemma 4.1}
 | f(x+y) - f(y) - f(x+z) + f(z) | \le
 [ Df ]_{\gamma}
 \, \vert y- z| \, |x|^{\gamma}.
\end{equation}
If $\delta \in (0,\gamma]$,   $f \in C^{\gamma}_\mathrm{b}(\R^d,Y)$ and  $x \in \R^d$, then
 \begin{equation}\label{eqn-4.16-b}
 [ \tau_x f - f ]_{\delta} \leq  2 \vert x \vert^{\gamma-\delta} [f]_{\gamma},
 \end{equation}
i.e., for all $x, y, z \in \R^d$,
 \begin{equation}\label{eqn-4.16}
 \vert  f(x+y)- f(y)
 -  f(x+z)+ f(z)
 \vert \leq  2 \vert x \vert^{\gamma-\delta} [f]_{\gamma} \vert y-z\vert^{\delta}.
 \end{equation}
\end{lemma}

\begin{proof} The proof of the first part proceeds by  applying the Mean Value Theorem. 
The proof of the second part is based on two observations: (i) $\tau_x$ is a linear contraction in the space $C_{\mathrm{b}}^{0}(\R^d,Y)$; (ii) $\tau_x$ is a linear contraction in the space $ C^{\gamma}_\mathrm{b}(\R^d,Y)$ with respect to the seminorm $[\cdot]_{\gamma}$.  Then \eqref{eqn-4.16-b} follows from a simple interpolation inequality \eqref{eqn-interpolation inequality} stated in Lemma \ref{lem interpolatory ineq} below and inequality
\begin{align*}
[ \tau_xf-f]_{0}&\leq 2 \Vert \tau_xf-f \Vert_{0} \leq   2 \vert x \vert^{\gamma} [f ]_{\gamma}.
\end{align*}
\end{proof}
\begin{lemma}\label{lem interpolatory ineq}
Assume that $0 \leq \alpha_0 < \alpha_1 \leq 1$ and $\theta \in [0,1]$.  Then, for every function $f: \mathbb{R}^d \to Y$,
\begin{equation}\label{eqn-interpolation inequality}
[f]_{(1-\theta)\alpha_0+\theta \alpha_1}\leq [f]_{\alpha_0}^{1-\theta}[f]_{\alpha_1}^{\theta}.
\end{equation}
\end{lemma}
\begin{proof}[Proof of Lemma is simple and hence omitted.]

  \end{proof}

\subsection{L\'evy Process}\label{subsec-Levy process}

In this subsection, we will recall the probabilistic framework from the paper \cite{Pr18}, see also   \cite{Pr15}. The assumptions on the noise therein are weaker, i.e., more general,  than those  in the paper  \cite{Pr12}.

  Assume that $(\Omega,\call{F},\mathbb{P})$ is a probability space with a filtration $\mathbb{F}=\bigl(\call{F}_t\bigr)_{t\geq 0}$ satisfying the usual conditions.
  Let  $\call{P}$ be the  predictable $\sigma$-field
 on $\Omega \times [0,\infty)$, i.e., the   $\sigma$-field generated by all  left-continuous $\mathbb{F}$-adapted processes.

Let us recall that  an $\R^d$-valued  process  $L = (L_t)_{\,t\geq 0}$  is called a L\'evy process  if it has stationary and independent increments, $L_0=0$, $\mathbb{P}$-a.s. and all its sample paths are c\`adl\`ag. The last condition is often replaced by continuity in probability, see e.g. \cite[Definition 4.1 and Theorem 4.3]{Peszat+Zabczyk_2007}. Often it is only required that the sample paths are c\`adl\`ag  almost surely.  By the L\'evy-Khinchin formula, see \cite[Theorem 4.27]{Peszat+Zabczyk_2007} and \cite[Section 11]{Sato_1999}, a L\'evy process can be characterized by the formula
\begin{equation}\label{eqn-psi-LK decomposition}
\psi(\xi)=-{\rm{i}} \langle a, \xi\rangle +\frac12 \langle \xi, Q\xi\rangle  -\lint_{\mathbb{R}^d} (e^{{\rm{i}}\langle \xi,z\rangle}-1-{\rm{i}}\langle \xi,z\rangle \1_{B}(z))\,\nu(dz),\ \ \xi\in\mathbb{R}^d,
\end{equation}
where $\rm i=\sqrt{-1}$, $B$ is the unit ball in $\mathbb{R}^d$, i.e.,
\begin{equation}\label{eqn-B_1}
B=\{ z \in \mathbb{R}^d: |z|\leq 1\},
\end{equation}
$(a,Q,\nu)$ is a L\'evy triplet of $L$
 and $\psi$ is the symbol (called also exponent)  of $L$, i.e.,
 \begin{equation}\label{eqn-psi}
e^{-t\psi(\xi)}:=\mathbb{E}(e^{{\rm i}\langle \xi,L_t\rangle}), \;\; t\geq 0, \; \xi \in \mathbb{R}^d.
\end{equation}
A L\'evy triplet $(a,Q,\nu)$  consists of a vector  $a\in \mathbb{R}^d$,  called the drift, a symmetric non-negative matrix $Q \in \mathbb{R}^{d\times d}$ and a L\'evy measure $\nu$, i.e., a non-negative Borel measure on $\mathbb{R}^d \setminus \{0\}$ such that
\begin{equation}\label{eqn-Levy measure moment}
\lint_{\mathbb{R}^d} \min \{ 1,\vert z \vert^2\}\, \nu(dz) <\infty.
\end{equation}

In the above and in what follows, we use the notation we have learned about from  the  paper  \cite{Kuhn+Schilling_2019}:
\[
\lint_{C} f(z) \,\nu(dz):=\int_{C\setminus\{0\}} f(z) \, \nu(dz), \;\; C \in \mathscr{B}(\mathbb{R}^d).
\]

 Let us  recall the
  L\'evy-It\^o decomposition of the process $L $ with L\'evy triplet $(a,Q,\nu)$, see \cite[Theorem 2.4.16]{Applebaum_2009} or
 \cite[Theorem 2.7]{Kunita_2004}, according to which, there exists $a\in\R^d$, a Brownian motion $W_{Q}$ with covariance matrix $Q$ and an independent Poisson random measure $\mu$ on $\mathbb{R}^{+} \times(\mathbb{R}^d\setminus\{0\})$ such that, for each $t\geq 0$,
 \begin{equation} \label{eqn-Levy-Ito}
 L_t = a \, t + W_{Q}(t)+
 \int_0^t \lint_{B} z \tilde \mu(ds, d z) +
 \int_0^t \int_{B^{\mathrm{c}}} z  \mu(ds, d z), \;\; t \ge 0,
 \end{equation}
  where we adopt the notation as before $ \int_0^t \lint_{B} z \tilde \mu(ds, d z) = \int_0^t \int_{B \setminus\{0\}} z \tilde \mu(ds, d z) $.
Note that $\tilde{\mu} $ is  the corresponding compensated Poisson random measure, i.e.,
\begin{equation}\label{eqn-mu-tilde}
\tilde{\mu}:=\mu-\Leb \otimes \nu.
\end{equation}

In this paper, we will consider, as in \cite{Pr15},    a  pure jump  $\R^d$-valued  L\'evy process    $L =  (L_t)_{\,t\geq 0}$ without the drift term, i.e., with L\'evy triplet $(0,0,\nu)$. In other words, the L\'evy  process $L$ can be represented in the following form
\begin{align}\label{eqn-Levy-Ito-2}
 L_t =
 L(t)=\int_0^t \lint_{B} z \tilde{\mu}(ds,dz)+ \int_0^t \int_{B^\rc} z \mu(ds,dz), \;\; t\geq 0.
\end{align}

\begin{definition}\label{def-Markov semigroup}
The  Markov semigroup $(\call{R}_t)_{t\geq0}$ acting on $C_{\mathrm{b}}(\R^d)$ and associated to the process  $L$ is defined by
\begin{equation}\label{eq Markov semigroup}
\call{R}_tf(x)=\mathbb{E}(f(x+L_t)),\ t \geq 0,\,f\in C_{\mathrm{b}}(\R^d),\,x\in\R^d.
\end{equation}
The pre-generator $\gen{A}$ of the semigroup $(\call{R}_t)_{t\geq0}$ is defined by
\begin{equation}\label{eqn-generator L}
\begin{aligned}
\gen{A}g(x)&=\lint_{\R^d}(g(x+z)-g(x)-\1_{B}(z)\langle z,D g(x)\rangle)\,\nu(dz), \;\; g \in  \mathcal{D}(\gen{A}),
\\
\mathcal{D}(\gen{A})&=C_c^\infty(\R^d),
\end{aligned}
\end{equation}
where by   $D g(x)$ we denote the gradient of the function  $g$ at $x\in\R^d$, i.e., $D g(x)\in \R^d$ is the unique vector such that
\begin{equation}\label{eqn-gradient} 
  (D g(x),y)=(d_xg)(y), \;\; y \in \R^d,
\end{equation}
where $d_xg \in \gen{L}(\R^d,\R^d)$ is the Fr{\'e}chet derivative of $g$ at $x$.

\end{definition}

Obviously  $\call{R}_0=I$ and the words ``Markov semigroup" can be justified by showing that the family $\call{R}=\bigl(\call{R}_t)_{t\geq 0}$ is indeed a
positivity preserving strongly continuous contraction semigroup  on the Banach space $C_0(\R^d)$, which by definition is the (closed) subspace of  $C_{\mathrm{b}}(\R^d)$ consisting of those functions belonging to the latter space which vanish at $\infty$.
Note that we have not characterised the domain of the infinitesimal generator of this semigroup.

\begin{remark}\label{rem-H2}
 One can prove, see  the discussion at  the beginning of Section 4 of \cite{Pr15} or bottom of page 427 of \cite{Pr12}, that
 the RHS of formula \eqref{eqn-generator L}
 makes sense   for $g\in C_{\mathrm{b}}^{\alpha^\prime}(\mathbb{R}^d)$, if  $\alpha^\prime>\alpha$. Moreover, the resulting function belongs to $C_{\mathrm{b}}(\mathbb{R}^d)$ and it will also be denoted  by $\gen{A}g$.
\end{remark}

\subsection{Hypotheses on the noise in the transport equation}\label{hy2}

 We consider an $\R^d$-valued purely jump L\'evy process $L$ with a L\'evy measure $\nu$. Let us define the Blumenthal-Getoor index of the process $L$, see \cite{Blumenthal}, by
\begin{equation}\label{eqn-alpha}
\alpha:=\inf \Big\{\sigma>0: \lint_{B}|y|^\sigma \nu(d y)<\infty\Big\}.
\end{equation}
 We assume that $\alpha \in (0,2)$.

 We denote by  $e_{1},...,e_{d}$ the canonical basis of ${\mathbb{R}}^{d}$ and denote, cf. \cite{FGP_2010-Inventiones},
\begin{equation}\label{eqn-L}
L_{t}=(L_{t}^{1},...,L_{t}^{d}), \;\; t\geq 0.
\end{equation}
We consider   a vector field
\begin{equation}\label{eqn-b}
b : \R^{d} \to \R^d
\end{equation}
which is $\beta$-H\"older continuous and bounded, $\beta\in(0,1)$, i.e.,
\begin{equation}\label{eqn-b-C^beta_b}
 b \in C^\beta_{\mathrm{b}}(\mathbb{R}^d,\mathbb{R}^d).
\end{equation}
We assume  as in \cite{Pr12,Pr15,Pr18} and \cite{CSZ18}  that
\begin{equation}\label{eqn-beta+alpha half}
\frac{\alpha}{2} +\beta  > 1.
\end{equation}
In some parts of this paper we will make an additional assumption that
the  distributional divergence of $b$ exists and is locally integrable, i.e.,
\begin{equation}\label{eqn-b-divergence}
\divv b \in L^1_{\mathrm{loc}}(\mathbb{R}^d).
\end{equation}
In this paper we will
study  two objects. Firstly, we are interested in the following stochastic differential equation (SDE) driven by process $L$
\begin{align} \label{eqn-SDE-0}
dX_t &= b(X_t)\,dt + dL_t,\; t>0.\;\;\;
\end{align}
Secondly, we are interested in  the following  stochastic partial differential equation of transport type
\begin{align}\label{eqn-transport-strong}
 & \frac{\partial{u(t,x)}}{\partial t}+b(x)\cdot D u(t,x)\, +\sum_{i=1}^d D_i u(t-,x)\diamond \frac{d L_t^i}{dt}=0,~t\in[0,T];\\
  & u(0,x)=u_0(x),\ \ x\in\mathbb{R}^d, 
\end{align}
 where $D_iu=e_i\cdot  D u$ is the derivative in the $i$-th coordinate direction and the initial data  $u_0$ is a bounded Borel function and  the stochastic integration is understood in the Marcus form, see e.g.,  \cite{Applebaum_2009} and \cite[page 82]{Protter_2004}. Let us point out that equation \eqref{eqn-SDE-0} is the   characteristics equation  for the transport equation 
\eqref{eqn-transport-Markus} or \eqref{eqn-transport-strong}.

Following  \cite{Pr18} , see  also  \cite{Pr15}), we assume

\begin{hypothesis}  \label{hyp-nondeg1}
{\em
\begin{trivlist}
\item[\textbf{(H)}] The L\'evy process $(L_t)_{t\geq 0}$ has a Blumenthal-Getoor index $\alpha \in (0,2)$.
\item[\textbf{(R)}] The  Markov semigroup $(\call{R}_t)_{t\geq0}$ satisfies $\call{R}_t (C_{\mathrm{b}}(\R^d)) \subset
C^1_{\mathrm{b}}(\R^d)$, for every $t>0$ and there exists  $c= c_{\alpha} >0$  such that, for every $f \in C_{\mathrm{b}}(\R^d)$,
 \begin{align} \label{eqn-grad}
 \sup_{x \in \R^d}| D (\call{R}_t f)(x)| \le \frac{c}{t^{1/\alpha}} \; \sup_{x \in \R^d}| f(x)|,\;\;\; t \in (0,1] .
 \end{align}
 \item[\textbf{(T)}] There exists $\theta>0$ such that
\begin{align}\label{eqn-big jumps}
\int_{\{|z|>1\}}|z|^\theta \nu(d z)<\infty.
\end{align}
\end{trivlist}
}
\end{hypothesis}

 Hypothesis \ref{hyp-nondeg1} is the most general assumption we impose on  the noise, compared with Hypotheses \ref{hyp-nondeg2} and \ref{hyp-nondeg3} below. Note that we always have $\alpha\in[0,2]$.
  In particular, assumptions \textbf{(H)} and \textbf{(R)} are equivalent to the assumptions in  \cite{Pr15}  when $\alpha \in (0,2)$.
  Assumption \textbf{(T)} will be used in the whole paper starting from  Theorem \ref{d32}. For examples of L\'evy processes that satisfy Hypothesis \ref{hyp-nondeg1}, please refer to \cite{Pr18}.

 In the paper we will consider Hypothesis \ref{hyp-nondeg1} in the case of   $ \alpha \in [1,2)$. Indeed it is even not known if pathwise uniqueness holds for the corresponding SDE
  under Hypothesis \ref{hyp-nondeg1} when $\alpha \in (0,1)$ and condition \eqref{eqn-beta+alpha half} holds;  cf.  Remarks  \ref{esempi} and \ref{nounique}.

 Hypothesis \ref{hyp-nondeg1} in particular will imply existence and uniqueness of weak solution to the transport equation when $\alpha \ge 1$, see Theorem \ref{thm-uniqueness}, which only requires $b\in C_{\mathrm{b}}^{\beta}(\mathbb{R}^d,\mathbb{R}^d)$,  provided that $1>\beta>1-\frac{\alpha}4>\frac12$  and $\divv b \in L_{\mathrm{loc}}^{1}( \mathbb{R}^{d})$ when $d\geq 1$.

\indent For the first   crucial  uniqueness  results, see Theorem  \ref{thm-uniqueness-1} in
Section \ref{uni1},   we will need the following  stronger assumption on the noise. In Theorem  \ref{thm-uniqueness-1}, we assume weaker H\"older regularity condition $b\in C_{\mathrm{b}}^{\beta}(\mathbb{R}^d,\mathbb{R}^d)$ with $1-\frac{\alpha}2<\beta\leq \frac12$, in addition an $L^p$ integrability condition on the divergence $\divv b \in L^{p}( \mathbb{R}^{d})$ for some $p\geq2$ when $d\geq 1$.

\begin{hypothesis} \label{hyp-nondeg2}
The L\'evy  process {\em  $(L_t)_{t\geq 0}$ is  a symmetric $\alpha$-stable  process, which is non-degenerate, i.e.,
\beq \label{stim}
\mbox{ there exists 
  $C_{\alpha}>0$: }\;\;  \psi (u) \ge C_{\alpha} |u|^{\alpha}, \;\; u \in \R^d.
\eeq
}
\end{hypothesis}
Again, we always consider Hypothesis \ref{hyp-nondeg2} with $\alpha \in [1,2)$. We will use this hypothesis as  needed to apply  $L^p$-estimates,  see in particular
Theorem \ref{TH-Lp-est}.  Hypothesis \ref{hyp-nondeg2} is a special case of the assumptions used in \cite{Zhang_2013}. Indeed   we do not know if the results from  \cite{Zhang_2013} hold when  $0< \alpha <1$.

\vskip 2mm We have a final assumption which is the  strongest one.  We will apply it in Theorem \ref{thm-uniqueness} to show the existence (and even the uniqueness)  of a weak solution when  $0<\alpha <1$.

\begin{hypothesis} \label{hyp-nondeg3}
The L\'evy  process {\em  $(L_t)_{t\geq 0}$ is  a symmetric $\alpha$-stable process, which is  rotationally invariant, i.e., there exists a positive
 constant
  $C_{\alpha}$ such that, for any $u \in \R^d$,
\beq \label{stim4}
 \psi (u) = C_{\alpha} |u|^{\alpha}.
\eeq
}
\end{hypothesis}

\begin{remark}\label{esempi} 
 In paper \cite{Pr15},  the second named author  proved that
 if   $\alpha \in [1,2)$,  $L =  (L_t)_{t\geq 0}$
  verifies  \textbf{(H)} and \textbf{(R)} in  Hypothesis \ref{hyp-nondeg1}
  and
 $b \in C_{\mathrm{b}}^{\beta}(\R^d,\R^d)$ with $\beta$ satisfying condition  \eqref{eqn-beta+alpha half},
  then the existence and the strong  uniqueness  holds for the following  SDE with $x \in \R^d$ and $s\in [0,\infty)$,
\begin{align} \label{eqn-SDE}
dX_t &= b(X_t)\,dt + dL_t,\; t>s;\;\;\;
\\
\label{eqn-SDE-IC}
X_s &=x.
\end{align}

 In  the earlier paper
 \cite{Pr12}, he proved similar statements under an   additional assumptions that
  $L = (L_t)_{t\geq 0}$ verifies Hypothesis \ref{hyp-nondeg2} with $\alpha \in [1,2)$.
 Indeed all the results from   \cite{Pr12} have been generalized in \cite{Pr15}
 by introducing new weaker assumptions  \textbf{(H)} and \textbf{(R)} of Hypothesis \ref{hyp-nondeg1}.

 A sufficient condition in order that condition  \textbf{(R)} in Hypothesis \ref{hyp-nondeg1} holds  is the following one, see also \cite{SSW}.
\begin{trivlist}
\item[\textbf{Condition (R')}]
There exists a  L\'evy measure $\nu_1$  on $\mathbb{R}^d$ and there exist   $C_1,C_2,M>0$ such that   symbol $\psi_1$ corresponding to $\nu_1$, i.e.,  defined by
\[
\psi_1(\xi)= -\lint_{\mathbb{R}^d} (e^{{\rm i}\langle \xi,z\rangle}-1-{\rm i}\langle \xi,z\rangle \1_{B}(z))\,\nu_1(dz), \;\; \xi \in  \mathbb{R}^d,
\]
 satisfies the following inequality
\begin{equation} \label{df44}
C_1 |\xi\vert^{\alpha} \le Re\, \psi_1 (\xi) \le C_2 |\xi\vert^{\alpha},\quad  |\xi|>M,
\end{equation}
and such that
\[
\nu(A)\geq \nu_1(A), \;\;A\in\call{B}(\mathbb{R}^d).\]

\end{trivlist}
Note that in paper \cite{Pr18} a stronger result is proven but under an additional  assumption \textbf{(T)},  see  Theorem \ref{d32} below.
\end{remark}

%------------------------

 Before we state the next important  inequality, let us add some comments regarding notations we use in our paper.

\begin{proposition}\label{prop-Ito integral}
 Let $C \in \call{B}(\R^d \setminus \{ 0\})$ and consider    a $\call{P} \otimes \call{B}(C)/\call{B}(\R^d)$-measurable mapping
 \begin{equation}\label{eqn-function F}
 F :
 [0, \infty) \times \Omega\times C  \to \R^d.
 \end{equation}
  If
  \begin{equation}\label{eqn-assumption function F}
  \mathbb{E} \int_0^T \int_C |F(s,z )\vert^2 \nu (dz ) ds< \infty, \mbox{ for every }T>0,
  \end{equation}
    then one can define
  the stochastic integral
 \begin{equation}\label{eqn-integral of  F}
 Z_t =  \int_0^t \int_{C} F(s,z)  \tilde{\prm}(ds, dz), \;\; t \in [0,\infty),
 \end{equation}
and  the process $Z = (Z_t)_{t\geq0}$ is an $L^2$-martingale with a c\`adl\`ag modification.
\end{proposition}

The construction of
the It\^o stochastic integral can  also be done by defining the stochastic integral first for predictable integrands, and then extending
the stochastic integral  to progressively measurable integrands using the fact that the dual predictable
projection of a Poisson random measure,  is absolutely continuous with respect
to the Lebesgue measure (with respect to the time variable).

We need  the following fundamental  $L^p$-inequality \eqref{eqn-Burkholder inequality} for stochastic integrals, see \cite[Theorem  2.11]{Kunita_2004} for more details. This inequality is usually called the Burkholder inequality.

\begin{proposition}\label{prop-Burkholder inequality}
For every $p \ge 2$,
 there exists $c(p) >0$ such that for every process $F$ as in Proposition \ref{prop-Ito integral} and for every $T\geq 0$, the following inequality holds
  \begin{equation}  \label{eqn-Burkholder inequality}
  \begin{aligned}
\mathbb{E} \Big[\sup_{0 \le s \le T} |Z_s\vert^p\Big]
     \le  c(p)
   \mathbb{E} \left[ \Bigl(\int_0^T \int_{C}|F(s,z)\vert^2  \nu(d z)ds \Bigr)^{p/2}
 +
 \,  \int_0^T \int_{C}|F(s,z)\vert^p  \nu(dz)ds\right].
 \end{aligned}
 \end{equation}
\end{proposition}
Of course, the above result is interesting only if the RHS of inequality \eqref{eqn-Burkholder inequality} is finite.

\vskip 0.3cm

\section{A regular stochastic flow}\label{sec-regular flow}

In this section, we consider a  vector field $b$ as in \eqref{eqn-b}  which only satisfies  the condition \eqref{eqn-b-C^beta_b}.
We will not use the assumption \eqref{eqn-b-divergence} about the divergence of $b$.

Let us first recall the following result from  \cite[Theorem 1.1]{Pr15}, see also \cite{Pr12}. Note that these  results hold also in the  case when $b: \R^d \to \R^d$ is Lipschitz and bounded, i.e., when $\beta=1$. 

We comment on the difference between the next theorem and 
Theorem \ref{ww1-1-intro} in Remark \ref{se4}.  
  \begin{theorem}
\label{uno}
Assume that  $\alpha \in [1,2)$ and $\beta \in (0,1)$  satisfy condition
\eqref{eqn-beta+alpha half}. Assume that $b\in C_{\mathrm{b}}^{\beta}(\mathbb{R}^d,\mathbb{R}^d)$.
Assume that   $L = (L_t)_{t\geq 0}$ is a L\'evy  process  satisfying \textbf{(H)} and \textbf{(R)} in Hypothesis \ref{hyp-nondeg1}.
Then the global existence and the pathwise uniqueness hold for Problem      \eqref{eqn-SDE}-\eqref{eqn-SDE-IC}. \\
 Moreover, if $X^x = (X_t^x)_{t\geq 0}$,  $x \in\mathbb{R}^{d}$, denotes
   the unique solution of Problem \eqref{eqn-SDE}-\eqref{eqn-SDE-IC} with $s=0$, then the following conditions are satisfied:

\hh (i) for all  $T>0$ and $p \ge 1$, there exists
  a constant $C(p,T)>0$ (depending also on $\alpha, \beta, d$,  the L\'evy measure $\nu$ and  the norm  $\Vert b \Vert_{\beta}$)  such that
 \begin{equation} \label{ciao7}
\mathbb{E}\Big[\, \sup_{0 \le s \le T} |X_s^x - X_s^y\vert^p\, \Big] \, \le C(p,T) \,\,
 |x-y\vert^p,\;\;\; x,\, y \in \R^d;
  \end{equation}
(ii) for every $t \ge 0$, the mapping
\begin{equation}\label{eqn-1.6}
\R^d \ni x \mapsto X_t^x \in  \R^d
\end{equation}
 is a surjective homeomorphism, $\mathbb{P}$-a.s.;

\hh(iii) for every $t \ge 0$, the mapping
\[ \R^d \ni  x \mapsto X_t^x \in \R^d\]
 is  a $C^1$-function 
   $\mathbb{P}$-a.s..
\end{theorem}

When $\alpha \in (0,1)$ a related result follows from Theorem 1.2 and Corollary 1.4 in  \cite{CSZ18}. Indeed, we have:

\begin{theorem} \label{uno1}
  Assume that  $\alpha \in (0,1)$ and $\beta \in (0,1)$ satisfy condition
\eqref{eqn-beta+alpha half}.
Assume that  $b\in C_{\mathrm{b}}^{\beta}(\mathbb{R}^d,\mathbb{R}^d)$.
Assume that   $L = (L_t)_{t\geq 0}$ is a L\'evy  process  satisfying Hypothesis \ref{hyp-nondeg3}.
Then assertions (i), (ii) and (iii) of Theorem \ref{uno} continue to hold.
\end{theorem}

\begin{remark} \label{nounique}    Note that the assumptions of  Theorem 1.2 in \cite{CSZ18} are more general than
those of Theorem \ref{uno1}.  On the other hand, it remains  an \textbf{Open Problem} if the pathwise uniqueness holds for Problem \eqref{eqn-SDE}-\eqref{eqn-SDE-IC} with $\alpha \in (0,1)$ and $\beta \in (0,1)$  satisfying  condition
\eqref{eqn-beta+alpha half} and    $L = (L_t)_{t\geq 0}$ is a L\'evy  process  satisfying     \textbf{(H)} and \textbf{(R)} in the general Hypothesis \ref{hyp-nondeg1} (cf. Remark \ref{esempi}). On this respect we mention \cite{ButkovskyDereGer}  which also contains 
    pathwise uniqueness results  when $0<\alpha <1$ under additional hypotheses. 
\end{remark}

\begin{remark} \label{se4}   The proofs of Theorem  \ref{uno}   given in \cite{Pr15} and  Theorem  \ref{uno1}   given in \cite{CSZ18}
show  that the assertions (ii) and (iii) of Theorems  \ref{uno} and \ref{uno1}  hold on a $\mathbb{P}$-full set which can {\it possibly depend}  on $t \ge 0$.
We could rewrite  these conditions (ii) and (iii) as follows. \\
 For any $t \ge 0$, there exists a $\mathbb{P}$-full  event
\begin{equation}\label{eqn-full set-2}
\Omega_t^\prime \in \call{F}:\;\;\; \mathbb{P}(\Omega_t^\prime)=1,
\end{equation}
such that, for every $\omega
\in \Omega_t^\prime$,  the map
\begin{equation}\label{eqn-1.6-omega}
\R^d \ni x \mapsto X_t^x(\omega) \in  \R^d
\end{equation}
is a surjective homeomorphism
 of   $C^1$-class. \\
  However, the $C^1$-class diffeomorphism property of  the map from \eqref{eqn-1.6}  has not been investigated.

According to the  assertions in Theorem \ref{ww1-1-intro}  relevant  goals 
of the present Section \ref{sec-regular flow}  are  to show that $\Omega_t^\prime$ can be chosen  independent of $t$, see  Theorem \ref{thm-homeomorphism},  and to
 demonstrate the  $C^1$-class diffeomorphism property of the map starting from \eqref{eqn-1.6}, see Theorem \ref{ww1}.
\end{remark}

%---------------------
\subsection{The homeomorphism property}\label{subsec-homeomorphism property}

\smallskip
We start with    \cite[Theorem 6.8]{Pr18}
 and  \cite[Theorem 6.10]{Pr18} which are consolidated in the next theorem (a related result which covers degenerate SDEs is \cite[Theorem 3.5]{Pr20}).

Note that in \cite{Pr18} the vector field $b$ was time-dependent and the time $t$ was restricted to a bounded interval $[0,T]$. Since we consider the vector field $b$ here to be  time-independent, we can work on the whole unbounded interval $[0,\infty)$. We only have to pay attention to estimates which are uniform only on bounded subsets of the set $[0,\infty) \times \mathbb{R}^d$.

In the next result we need all the assumptions of Hypothesis \ref{hyp-nondeg1} (including \textbf{(T)}).

\begin{theorem} \label{d32} Assume that $\beta \in (0,1)$  satisfies  condition
\eqref{eqn-beta+alpha half} with  $b\in C_{\mathrm{b}}^{\beta}(\mathbb{R}^d,\mathbb{R}^d)$. 
Moreover,  if   $\alpha \in [1,2)$, then  we assume   that   $L = (L_t)_{t\geq 0}$ is a L\'evy  process  satisfying  Hypothesis \ref{hyp-nondeg1}, i.e.,
\textbf{(H)}, \textbf{(R)} and \textbf{(T)}. If $\alpha \in (0,1)$, then  we assume that $L$ satisfies
Hypothesis \ref{hyp-nondeg3}. \\
  Then, there  exists  a  $\call{B}([0,\infty) \times [0,\infty) \times {\mathbb R}^d) \otimes   \call{F}/\call{B}({\mathbb R}^d) $-measurable function
\begin{equation} \label{nuova}
\phi : [0,\infty) \times [0,\infty) \times {\mathbb R}^d \times \Omega \to {\mathbb R}^d,
\end{equation}
such that for every $(s,x)\in [0,\infty)\times \R^d$, the process   $\big ( \phi(s,t,x , \cdot ) \big)_{t \in [s,\infty)}$
is a strong solution of Problem \eqref{eqn-SDE}-\eqref{eqn-SDE-IC}.

Moreover, there exists an almost sure event $\Omega^\prime$
 such that
{\sl the following assertions hold for every $\omega \in \Omega^\prime$.}

\hh
(i) For any $x \in {\mathbb R}^d$,  the mapping
 \[ [0,\infty) \ni s \mapsto   \phi (s,t,x, \omega) \in \R^d \] is c\`adl\`ag locally uniformly in $t$ and $x$, i.e.,
  if  $s \in (0,\infty)$ and   $(s_k)$ is  a sequence such that   $s_k \toup s$,
   then, for all $T,M>0$, the following hold
   \begin{align} \label{eqn-sg}
\lim_{k \to \infty} \sup_{|x| \le M}\sup_{t \in [0,T]} |\phi (s_k,t,x, \omega) - \phi (s-,t,x, \omega) | =0,
\end{align}
and if $s \in [0,\infty)$  and
 $(s^n) $  is sequence such that    $s^n \todown s$, then, for all $T,M>0$, the following holds
\begin{align} \label{sd}
 \lim_{n \to \infty} \sup_{|x| \le M}\sup_{t \in [0,T]} |\phi (s^n,t,x, \omega) - \phi (s,t,x, \omega) | =0.
\end{align}
\hh (ii) For all  $x \in {\mathbb R}^d$ and $s \in [0,\infty)$,
\begin{equation}
  \phi (s,t,x, \omega) =x,\quad  \mbox{ if } \, 0 \le t \le s,
  \end{equation}
   and
\begin{equation} \label{d566}
 \phi (s,t,x, \omega) = x + \int_{s}^{t}b\left( \phi (s,r,x, \omega) \right) dr + L_t(\omega) - L_s(\omega),\;\;\; t \in [s,\infty).
\end{equation}

\hh (iii) For any $s \in [0,\infty)$, the function \[\R^d \ni x \mapsto \phi (s,t,x, \omega) \in \R^d\]
is continuous in $x$,  locally uniformly with respect to  $t$.\\
 Moreover,
for every  integer $n > 2d$ and every  $T>0$, there exists  a $\call{B}([0,T])\otimes \call{F}/\call{B}([0,\infty])$-measurable function  $V_n(\cdot,\cdot;T) : [0,T] \times \Omega  \to [0,\infty]$ such that
\begin{equation} \label{eqn-V_n-L^1}
\int_0^T V_n(s , \omega;T)\, ds < \infty
\end{equation}
 and
\begin{equation}
\label{toro}
\begin{aligned}
\hspace{-6truecm}
\sup_{t \in [0,T]} &|\phi (s,t,x, \omega)  - \phi (s,t,y, \omega)|
\\
&\le V_n(s, \omega;T) \,  |x-y\vert^{\frac{n - 2d}{n}} [(|x| \vee \vert y|)^{\frac{2d + 1}{n}} \vee 1],\;\;\; x,y
\in {\mathbb R}^d,  \; s \in [0,T].
\end{aligned}
\end{equation}
(iv) For any $0 \le s <  r \le t < \infty$, $x \in {\mathbb R}^d$, we have
\begin{equation} \label{fl2}
 \phi (s ,t,x, \omega) = \phi (r ,t, \phi (s ,r,x, \omega), \omega),
\end{equation}
i.e.,
\begin{equation} \label{eqn-semiflow}
 \phi_{s,t}  = \phi_{r,t} \circ \phi_{s,r},
\end{equation}
where
\begin{equation} \label{eqn-semiflow-def}
  \phi_{s,t} (x, \omega):= \phi(s ,t,x, \omega).
  \end{equation}
  (v)
 If for some  $s_0 \in [0,\infty)$, $\tau = \tau(\omega) \in (s_0, \infty)$  and  $x \in {\mathbb R}^d$,   a c\`adl\`ag  function
 $g:  [s_0,\tau ) \to  {\mathbb R}^d$ solves the integral equation
\begin{equation} \label{integral}
g(t) = x + \int_{s_0}^{t}b\left( g(r)\right)  dr
 +  L_{t} (\omega) - L_{s_0}(\omega),\;\;\; t \in [s_0,\tau),
\end{equation}
then we have $g(r) = \phi(s_0,r,x, \omega)$, for $r \in [s_0,\tau)$.
\end{theorem}

In the present paper  we prove  new results which extend the results from    Theorem \ref{d32}, see e.g. Lemma \ref{ef} and Theorems  \ref{thm-homeomorphism},
\ref{thm-ww} and \ref{ww1}.

\vskip 1mm  In the sequel $\Omega^\prime$ (or $\bar{\Omega}^\prime$) will denote a $\mathbb{P}$-full event that can be obtained starting from the set of full measure given in Theorem \ref{d32}. For each $\omega \in \Omega^\prime$ and $x\in\mathbb{R}^d$, see Theorem \ref{d32}, we   set
\begin{equation}\label{eqn-inverse}
\nxi_{s,t}(x)(\omega) :=\nxi_{s,t}(x,\omega) = \phi(s,t,x, \omega),\;\;\;  s , t \in [0,\infty).
\end{equation}
We can extend the random variables $\nxi_{s,t}(x)$ from $\Omega^\prime$ to the whole of $\Omega$, by defining them to be equal to  $\id_{\mathbb{R}^d}(x):=x$  when $\omega \in \Omega \setminus \Omega^\prime$.

 By  Theorems \ref{uno} and \ref{uno1} we know that, for every $t \in [0,\infty)$, there exists a $\mathbb{P}$-full event possibly depending on $t$ such that  on this event
 we have that $\nxi_{0,t}(x)$ is a homeomorphism.
We extend this property
   showing  that
\begin{lemma} \label{ef}    Under the assumption of Theorem \ref{d32},
there exists an almost sure event $\bar{\Omega}^\prime$ such that
for every  $\omega \in \bar{\Omega}^\prime$ and every   $ t \in [0,\infty)$,
the mapping

\[ \R^d \ni x \mapsto \nxi_{0,t}(x)(\omega)\in \R^d\]
 is a surjective homeomorphism.

Moreover,  on $\bar{\Omega}^\prime$, for every $T\ge 0$,   we have:
\begin{equation}
 \nxi_{0,t}^{-1}(x) =  \nxi_{0,T}^{-1} ( \nxi_{t,T}(x)),\;\;   \, x \in \R^d,\ t\in[0,T].
\end{equation}
Here and in  the remainder of this paper, we denote by $\nxi_{0,t}^{-1}$  the inverse of $\nxi_{0,t}$. Finally, on $\bar{\Omega}^\prime$, for every $x \in \R^d$, the process
\begin{equation} \label{cadd}
 [0,\infty) \ni  t  \mapsto \nxi_{0,t}^{-1}(x) \in \R^d
\end{equation}
is c\`adl\`ag.
 \end{lemma}
\begin{proof}
The definition of $\nxi_{0,t}$, $t\geq0$ implies that the process $X=(X_t:=\nxi_{0,t})_{t\geq0}$ satisfies conditions (i), (ii), and (iii) in Theorem \ref{uno}.
In order to go further  we
fix $T \in \mathbb{N}^*$.

According to assertions (ii) and (iii) of Theorem \ref{uno}, the mapping $\mathbb{R}^d\ni x\rightarrow \nxi_{0,T}(x)\in\mathbb{R}^d$ is a surjective homeomorphism and a $C^1$-function on a $\mathbb{P}$-full event $\Omega^T$ which is only dependent on $T$. Meanwhile, by assertion (iii) of  Theorem \ref{d32}, for any $t\geq 0$, the function $\nxi_{t,T}(x)$ is continuous w.r.t. $x$ on some $\mathbb{P}$-full event $\Omega'$, which is introduced in Theorem \ref{d32} and  independent of $t\in[0,\infty)$. Hence, we have that on  $\mathbb{P}$-full event  $\Omega_0^T=\Omega^T\cap \Omega'$, for every $t \in [0,T]$, $x \in \R^d,$ the
 random variable
\[ \eta_{0,t}(x) :=  \nxi_{0,T}^{-1} (\nxi_{t,T}(x)),
\]
 is well defined, depends continuously on $ x \in \R^d$. Furthermore, by using assertion (i) of Theorem \ref{d32} together with the continuity of $\R^d\ni x \mapsto \nxi_{0,T}^{-1}(x)\in \R^d $,  we infer that the process
\begin{equation}\label{cadd 1}
 [0,T] \ni  t  \mapsto \eta_{0,t}(x) \in \R^d
\end{equation}
is c\`adl\`ag.

We also have, using \eqref{fl2},
\begin{equation}
\eta_{0,t}( \nxi_{0,t}(x)) =  \nxi_{0,T}^{-1} ( \nxi_{t,T} ( \nxi_{0,t}(x))) =x.
\end{equation}
 This shows that $\nxi_{0,t} : \R^d \to \R^d$ is {\it one to one} and that $\nxi_{0,t}^{-1}=\eta_{0,t}$.

 Let us prove the {\it surjectivity} of $\nxi_{0,t}$.  Let us choose and fix $y \in \R^d$ and $t\in [0,T)$. By Theorem \ref{uno}, we  infer that for every  $q \in  [0,T] \cap \Q$ there exists a $\mathbb{P}$-full event $\Omega_q$ and $x_q $ such that
\begin{equation} \label{222}
\nxi_{0,q}( x_q) =y,\quad \text{on }\Omega_q.
\end{equation}
Set
 \begin{equation} \bar{\Omega}'_T=\Omega^T_0\cap(\mathop{\cap}\limits_{q \in  [0,T] \cap \Q}\Omega_q).\end{equation}
Clearly $\mathbb{P}(\bar{\Omega}'_T)=1$.
Let us choose a sequence $(t_n) \subset [0,T] \cap \Q$ such that $t_n \todown t $  {\it from the right}
and write $x_{t_n} = x_n$. Note that on $\bar{\Omega}'_T$ we have, see (ii) in Theorem \ref{d32},
\begin{equation}
\nxi_{0,  t_n} (x_n) = x_{n} + \int_0^{t_n} b(\nxi_{0,s} (x_{n}))\, ds + L_{t_n}.
\end{equation}
By the boundedness of $b$, it follows that, for every ${n} \ge 1$, $|x_{n}| \le M_T(\omega) + \vert y|$, since $\nxi_{0,t_n}(x_n)=y$ and $L$ is c\`{a}dl\`{a}g.

Hence, by possibly passing to a subsequence we may assume that there exists $z\in\R^d$ such that $x_{n} \to z$. Passing to the limit in \eqref{222}, we obtain using also assertions (ii) and (iii) in Theorem \ref{d32},
\begin{equation}
|\nxi_{0,{t_n}} (x_{n}) - \nxi_{0,{t}}(z)| \le
|\nxi_{0,{t_n}} (x_{n}) - \nxi_{0,{t_n}}(z)|
 +
|\nxi_{0,{t_n}}( z) - \nxi_{0,{t}}(z)|
\to 0 \;\; \text{as} \; {n} \to \infty.
\end{equation}
Hence
 $\nxi_{0, t}( z) = y.$
This shows that $\nxi_{0,t}$ is onto for $t\in[0,T)$.

Now set $\bar{\Omega}'=\cap_{T\in\mathbb{N}^*}\bar{\Omega}'_T$. {Since
 by the previous arguments, for each $T\in \mathbb{N}^*$, $\nxi_{0,t}^{-1}$ is well define for all $t\in[0,T]$ on the $\mathbb{P}$-full event $\bar{\Omega}_T'$, it follows that for every $t\geq 0$ (which lies in some interval $[0,T]$, $T\in\mathbb{N}^*$), $\nxi_{0,t}^{-1}$ is well defined on   $\bar{\Omega}'$.} Now let us fix an arbitrary $T>0$, and define {on $\bar{\Omega}'$} a new process $Z = (Z_t)_{t \in [0,T]}$ by the following formula
\begin{equation*}
Z_t = \nxi_{0,T}^{-1} \circ\nxi_{t,T},\;\; t \in [0,T].
\end{equation*}
{Note that on the $\mathbb{P}$-full event $\bar{\Omega}'\cap \Omega_0^T$, $Z_t=\eta_{0,t}$, for every $t\in [0,T]$.} 
Arguing as the first part, one can prove that $Z_t$ is the inverse of $\nxi_{0,t}$ on $\bar{\Omega}'$, for every $t \in [0,T]$, and it is c\`adl\`ag. 

Finally,  assertion (iii) of  Theorem \ref{d32} implies that
for every  $\omega \in \bar{\Omega}^\prime$ and every   $ t \in [0,\infty)$,
the mapping

\[ \R^d \ni x \mapsto \nxi_{0,t}(x)(\omega)\in \R^d\]
 is continuous.

The proof of this lemma is complete.
 \end{proof}

\begin{remark} \label{ci9}{
In view of the applications  to the stochastic transport equation, we note that, by the previous result,  the function:
\begin{equation}
u(t,x, \omega) = u_0(\nxi_{0,t}^{-1}(x)(\omega))
\end{equation}
is well defined for every $t \in [0,\infty)$, $x \in \R^d$ and $\omega \in \bar{\Omega}^\prime$.
}
\end{remark}

In the sequel, we also set
\begin{equation}
\nxi_{0,t}(x) = \nxi_{t}x = \nxi_t(x),\;\;\; t \in [0,\infty),\; x \in \R^d.
\end{equation}

\begin{theorem}\label{thm-homeomorphism}
Under  the assumptions of Theorem \ref{d32},  there exists an $\mathbb{P}$-full event $\bar{\Omega}^\prime$ such that,
for each $\omega \in \bar{\Omega}^\prime$,  $s,  t\ge 0$,
the mapping: \[\R^d \ni  x \mapsto \nxi_{s,t}(x)(\omega) \in \R^d\] is a
surjective homeomorphism.

Moreover, on  $\bar{\Omega}^\prime$,  for every $y \in \R^d$,  we have
\begin{equation} \label{de}
\nxi_{s,t}^{-1}(y) =  y - \int_s^t b(\nxi_{r,t}^{-1}(y)) dr - (L_t - L_s), \;\;  0\le s \le t,
\end{equation}
and
$$
\nxi_{s,t}^{-1}(y) =  y, \;\;  0\le t \le s.
$$
Finally, for every $x \in \R^d,$ $t \in [0,\infty)$, the mapping: $s  \mapsto \nxi_{s,t}^{-1}(x)$  is c\`adl\`ag on $[0,\infty)$ and,
 for every $s \in [0,\infty)$, the mapping: $t   \mapsto \nxi_{s,t}^{-1}(x)$  is c\`adl\`ag on $[0,\infty)$.
\end{theorem}
Before we embark on the proof of this result, let us formulate a simple but useful corollary of it.
\begin{corollary}\label{cor-homeomorphism}
Under the assumptions of Theorem \ref{d32},   there exists an almost sure event $\Omega^{\prime\prime}\subset \bar{\Omega}^\prime$ such that,
for each $\omega \in \Omega^{\prime\prime}$ and $T>0$  there exists $C(\omega)>0$ (may dependent on $T$) such that  the following inequality holds
\begin{equation} \label{ineq-inverse}
\vert \nxi_{s,t}^{-1}(y) -  y \vert \leq   \Vert  b \Vert_{0} \vert t- s\vert  + C(\omega), \;\; y \in \R^d, \; 0\le s \le t \le T.
\end{equation}
\begin{proof} It is enough to observe that since almost surely the trajectory $L: [0,T]\to \R^d$ is c\`adl\`ag,
there exists a $\mathbb{P}$-full event $\Omega^{\prime\prime} \subset \bar{\Omega}^\prime$ such that for $\omega\in \Omega^{\prime\prime}$,
$C(\omega):=\sup\{ \vert L_t(\omega) - L_s(\omega) \vert: 0\le s \le t \le T \}$  is finite.
\end{proof}

\end{corollary}

\begin{proof}[Proof of Theorem \ref{thm-homeomorphism}]
It is easy to see that the case with $0\le t\le s$ holds. In the following we prove the case with $0\le s\le t$.

 By Lemma \ref{ef},  $\nxi_{s}^{-1}, s\ge 0$ are well defined random variables on the same $\mathbb{P}$-full event  $\bar{\Omega}^\prime$ and 
on $\bar{\Omega}^\prime$, for every $s \in [0,t]$, we have
\begin{equation} \label{dd}
\nxi_{s}^{-1}=\nxi_{0,t}^{-1}\circ
\nxi_{s,t}=
 \nxi_{t}^{-1}\circ \nxi_{s,t}.
\end{equation}
It follows that  $\nxi_{s,t} : \R^d \to \R^d$ is also a
surjective homeomorphism and
\begin{equation} \label{xi}
\nxi_{s,t} = \nxi_{t}\circ \nxi_{s}^{-1},
\;\;\;\;
\nxi_{s,t}^{-1} = \nxi_{s}\circ \nxi_{t}^{-1}.
\end{equation}
Finally, since by Theorem \ref{d32}  on $\Omega^\prime(\supset \bar{\Omega}^\prime)$ we have:
\begin{equation}
\nxi_{s,t}(x)=x+ \int_{s}^{t}b(\nxi_{s,r}( x) )\,dr+ (L_{t}-L_{s}),
\end{equation}
setting $x = \nxi_{s,t}^{-1}(y)$ we find
\begin{equation}
y =  \nxi_{s,t}^{-1}(y) + \int_{s}^{t}b(\nxi_{s,r}\circ \nxi_{s,t}^{-1}(y) )\,dr+ (L_{t}-L_{s}),
\end{equation}
 and using that
\begin{equation}
\nxi_{s,r}\circ \nxi_{s,t}^{-1} = [\nxi_{r}\circ \nxi_{s}^{-1}]
 \circ[\nxi_{s} \circ  \nxi_{t}^{-1}]
 = \nxi_{r} \circ  \nxi_{t}^{-1} = \nxi_{r,t}^{-1},
\end{equation}
 we get \eqref{de}. The regularity properties of $\nxi_{s,t}^{-1}(x)$ can be deduced by \eqref{xi}.
\end{proof}

\subsection {The differentiability property}\label{subsection-diff}

Here we prove the existence of a $\mathbb{P}$-full event $ \Omega^\prime$ (independent of $t$) related to the differentiability of the solutions with respect to $x$, see part  (iii) of  Theorem \ref{uno} and  Theorem \ref{uno1}.

In below we use the following notation. If $u=(u^k)_{k=1}^d$, then we put
 \begin{equation}\label{eqn-}
 b \cdot Du:=(\sum_{j=1}^d b_jD_ju^k)_{k=1}^d \mbox{ and } \gen{A} u:= ( \gen{A} u^k)_{k=1}^d.
 \end{equation}
Recall that $\gen{A}$ is introduced in Definition \ref{def-Markov semigroup}.

In what follows we  will  use the  following  result, which is a special case of \cite[Corollary 2.54]{Bahouri+Chemin+Danchin_2011}, see also \cite[Corollary 2.2]{Brz+Millet_2014}.
\begin{lemma} If $\theta\in[0,1]$, then the following inequality (with constant $1$ on the RHS) holds
\begin{align}\label{eq zhai prio 1}
\Vert fg\Vert_\theta\leq \Vert f\Vert_0\Vert g\Vert_\theta+\Vert g\Vert_0[f]_\theta,\ \ \forall f,g\in C_{\mathrm{b}}^\theta(\mathbb{R}^d).
\end{align}
\end{lemma}

Theorem \ref{reg} below holds under weaker assumptions than Theorem \ref{uno}. Indeed, if condition  \eqref{eqn-beta+alpha half} is satisfied, then there exists $\beta^\prime\in (1-\frac{\alpha}{2},\beta]$ such that condition \eqref{eqn-alpha+beta} with $\beta$ replaced by $\beta^\prime$ is satisfied as well.
 This result is a slight reformulation of  \cite[Theorem 4.3]{Pr15} and \cite[Theorem 6.7 and  6.11]{Pr18}. Recall that the proof of \cite[Theorem 6.11]{Pr18} is based on \cite{Si} (an extension of \cite{Si} is given in \cite{Chaudru+Menozzi+Priola_2020}). Assertion \eqref{eq zhai v0 control} follows from  Theorem 4.3 in \cite{Pr15},  whose  proof is the same as of   \cite[Proposition 3.2 and Theorem 3.4]{Pr12}.

\vv

Please note that  our assertion about the constant $c(\varpi,r) $ is new which can be proven by carefully examining the proof and checking the constants given in \cite[Theorem 6.7 and  6.11]{Pr18}.

\begin{theorem} \label{reg}
 Assume that  $\alpha \in  [1,2)$ or $\alpha \in (0,1)$. Assume that  $L = (L_t)_{t\geq 0}$ is a L\'evy  process  satisfying  \textbf{(H)} and \textbf{(R)} in   Hypothesis \ref{hyp-nondeg1} in the former case or Hypothesis \ref{hyp-nondeg3} in the latter case.  Assume also that  $\beta\in(0,1]$ is such that
 \begin{equation}\label{eqn-alpha+beta}
1< \alpha + \beta <2 .
 \end{equation}
 Assume that   $b\in C_{\mathrm{b}}^{\beta}(\mathbb{R}^d,\mathbb{R}^d)$.

 Then the following assertions hold.

\begin{itemize}
\item[(I)]
For every
 $\lambda\geq 1$ and  $f \in C^{\beta}_{\mathrm{b}} (\R^d,\R^d)$,
   there exists a unique
  $u= u_{\lambda}(f) \in C^{\alpha + \beta}_{\mathrm{b}}
  (\R^d,\R^d)$ which is a   solution to the following uncoupled system of linear equations
  \begin{equation} \label{eqn-wee}
 \lambda u -  \gen{A} u  - b \cdot Du = f \;\;  \mbox{ on }\R^d.
 \end{equation}
These solutions have the following two properties
\begin{align}
  \label{eq zhai v0 control}
\lambda\Vert u_\lambda\Vert_0\leq \Vert f\Vert_0,
\end{align}
and
\begin{align}
  \label{eqn-limit lambda infty}
  &\lim_{ \lambda \to \infty} \Vert Du_{\lambda} \Vert_{0} =0.
    \end{align}
\item[(II)]\ For every $\varpi \geq 1$,
  there exists $r_0>0$ and a continuous   function
 \begin{matriz}
(0,\infty) \ni r \mapsto c(\varpi,r ) \in (0,\infty)
 \end{matriz}
which is constant  on $(0,r_0]$ and increasing on $[r_0,\infty)$
    such that  for every   $\lambda \ge \varpi$, every $r>0$ and every
    $b \in C^{\beta}_{\mathrm{b}} (\R^d,\R^d)$: $\Vert b \Vert_\beta \leq r $,    the unique element $u_{\lambda}$  introduced in Part I above satisfies
     \begin{equation} \label{eqn-sch4}
 \lambda \Vert u_{\lambda}(f)\Vert_0 + [Du_{\lambda}(f)]_{\alpha + \beta - 1}
  \le c(\varpi,r ) \Vert f\Vert_{\beta}, \;\; f \in C^{\beta}_{\mathrm{b}} (\R^d,\R^d).
   \end{equation}

\item[(III)]

For every $\eps>0$ and every $M>0$ there exist constants $\varpi{}_{,M}\geq1$ and  $C_M>0$ such that if  $\lambda \geq \varpi{}_{,M}$ and  a vector field  $b\in C^{\beta}_{\mathrm{b}} (\R^d,\R^d) $   satisfies  $\Vert b\Vert_{\beta}\leq M$,
 the unique solution   $ u_\lambda \in C_{\mathrm{b}}^{ \alpha+\beta}(\R^d, \R^d) $ of the equation  \eqref{eqn-wee}
satisfies the following two conditions
\begin{align}\label{eqn-D less 13}
&\| Du_\lambda\|_0 \leq \eps \Vert f\Vert_{\beta},
\\
\label{eqn ub sup norm}
&\| u_\lambda\|_{\alpha+\beta} \leq C_M \Vert f\Vert_{\beta}.
\end{align}

\end{itemize}
  \end{theorem}

\begin{remark}\label{rem-generator domain}
If $b=0$, the previous result implies that for every
$f \in C^{\beta}_{\mathrm{b}} (\R^d,\R^d)$,
the solution    $u_\lambda$  of equation \eqref{eqn-wee} belongs to the space $ C^{\alpha + \beta}_{\mathrm{b}}
  (\R^d,\R^d)$.  This means that the resolvent operator $(\lambda I -\gen{A})^{-1}$ maps the space $C^{\beta}_{\mathrm{b}} (\R^d,\R^d)$ into
  the space $ C^{\alpha + \beta}_{\mathrm{b}}
  (\R^d,\R^d)$ (and as such is bounded). But we do not claim that this operator is surjective. In fact, we do not know whether this is true.

\end{remark}

\begin{remark}\label{rem-regularity of b and f}  Note that the Schauder estimates of the previous theorem require that $f$ and $b$ have the same regularity.  
We do not  know  whether the above results can be formulated in terms of Besov spaces,  see the recent paper \cite{Chaudru+Menozzi_2022} for a related result when $\alpha>1$. 
  \end{remark}

\begin{remark}\label{rem-alpha+beta=2}
 One may wonder why  do we assume that $\alpha+\beta<2$ and exclude the case when   $\alpha+\beta \leq 2$? Indeed, the statement of the theorem  would still make sense.
However, in that case we would need to consider the class $ C^{2}(\R^d)$ in the Zygmund sense, i.e.,  there exists $c>0$ such that  for every $\gamma \in \mathbb{N}^d$ with $\vert \gamma \vert=2$,
\begin{equation}\label{eqn-Zygmund class}
  \vert D^\gamma g(x + h) + D^\gamma g(x- h)  - 2D^\gamma g(x)   \vert  \leq  c\vert h \vert^2, \;\; x,h\in \mathbb{R}^d.
\end{equation}
 This case requires further analysis. For our purposes, the statement as above is completely sufficient.
  \end{remark}

The following result is taken  from  \cite[Lemma 5.2]{Pr15}.
\begin{lemma} \label{lem-Ito formula Tanaka trick} %{dee old version}
 {Under the assumptions of Theorem \ref{reg},
if $\lambda \ge 1$ } and
\[
u = u_{\lambda}\in C_{\mathrm{b}}^{ \alpha +\beta}(\R^d, \R^d) \]
 is the unique solution
 of equation \eqref{eqn-wee} with the RHS $f=b$, then
the following It\^o formula holds for all $t \in [0,\infty)$, $x \in \R^d$, $\mathbb{P}$-a.s.,
\begin{equation}
\label{eqn-Ito-iii}%{eqn-iii}
\begin{aligned}
 u (\phi_{t}(x)) -  u (x ) =&   L_t -\bigl(  \phi_{t}(x)-x \bigr)   + \lambda  \int_0^t  u(\phi_{s}(x))\, ds
\\
& +  \int_0^{t} \lint_{{\mathbb R}^d} [  u(\phi_{s-}(x) + z) -  u(\phi_{s-}(x))]
   \tilde{\prm}(ds, dz),
   \end{aligned}
\end{equation}
where  $\phi_t(x)$, $t\geq 0$,  is  the solution of \eqref{eqn-SDE} with $s=0$, i.e., with vector field $b$,  starting at  $x \in \R^d$.
\end{lemma}
The It\^o formula  can also be written in the following form, with $I:\mathbb{R}^d \to \mathbb{R}^d$ being the identity map,
\begin{equation}
\label{eqn-Ito-iv}%{eqn-iii}
\begin{aligned}
 (u+I) (\phi_{t}(x)) -  (u+I) (x ) =&   L_t   + \lambda  \int_0^t  u(\phi_{s}(x))\, ds
\\
& +  \int_0^{t} \lint_{{\mathbb R}^d} [  u(\phi_{s-}(x) + z) -  u(\phi_{s-}(x))]
   \tilde{\prm}(ds, dz).
   \end{aligned}
\end{equation}

\begin{theorem} \label{thm-approximation}%{dee}
 Suppose that  the assumptions of Theorem \ref{reg}  are satisfied. Assume that $(b^{n})_{n=0}^\infty$ and $(f^{n})_{n=0}^\infty$ are two  $ C_{\mathrm{b}}^{\beta}(\R^d,\R^d)$-valued
sequences of vector fields such that
\begin{equation}\label{eqn-b^n to b^0, f^n to f^0}
b^{n}\to b^0  \mbox{ and } f^{n}\to f^0 \mbox{ in } C_{\mathrm{b}}^{\beta}(\R^d,\R^d).
\end{equation}
Let, for all  $n\geq 0$ and $\lambda\geq 1$,   $ u^n_{\lambda}\in C_{\mathrm{b}}^{\alpha+\beta}(\R^d,\R^d) $  be  the unique solution to equation \eqref{eqn-wee} with the vector field $b$ replaced by the vector field  $b^n$
and $f$ replaced by $f^n$, i.e.,
\begin{align}\label{eqn-wee-n}
\lambda u^n_{\lambda} -  \gen{A} u^n_{\lambda}  - b^n \cdot Du^n_{\lambda} = f^n \;\;  \mbox{ on }\R^d.
\end{align}
 Then, there exists $\lambda_0\geq1$ and $R>0$  such that
\begin{equation}\label{eqn ub sup norm-sup}
\sup_{\lambda\geq \lambda_0}\sup_{n\geq0}\Vert u^n_{\lambda}\Vert_{\alpha+\beta} \leq R \quad \mbox{ and }\quad \sup_{\lambda\geq \lambda_0} \sup_{n\geq0}\Vert Du^n_{\lambda}\Vert_0\leq \frac{1}{3},
\end{equation}
\begin{equation}\label{eqn-conv-u_n to u 0 20251007}
 \lim_{n\to \infty}\Vert u^n_{\lambda}-u^0_{\lambda}\Vert_{0}=0,
\end{equation}
and the following results hold: 
for any compact set $K \subset \R^d$ and $0< \beta' <\beta$,
\begin{equation}\label{eqn-conv-u_n to u 0}
 \lim_{n\to \infty}\Vert u^n_{\lambda}-u^0_{\lambda}\Vert_{\alpha+\beta',K}=0.
\end{equation}
\end{theorem}

Note that \eqref{eqn-conv-u_n to u 0} implies that, for any compact set $K \subset \R^d$,
\begin{equation}\label{grad}
 \lim_{n\to \infty} \sup_{x \in K} | Du^n_{\lambda}(x)-Du^0_{\lambda}(x) |=0.
\end{equation}
This convergence will be useful in the sequel.

 \begin{theorem} \label{thm-approximation-Ito}
   Assume the framework of Theorem \ref{thm-approximation}.
    If $f^n=b^n$ for some $n \in \mathbb{N}$,  then
 $u^n_\lambda$ satisfies the following  It\^o formula.
 For all  $t \in [0,\infty)$, $x \in \R^d$, $\mathbb{P}$-a.s., we have,
\begin{equation}
\label{eqn-Ito-iii-n}%{eqn-iii-n}
\begin{aligned}
 u^n_{\lambda} (\phi^{n}_{t}(x)) -  u^n_{\lambda} (x ) =&  x + L_t - \phi^{n}_{t}(x)   + \lambda  \int_0^t  u^n_{\lambda}(\phi^{n}_{s}(x))\, ds
\\
& +  \int_0^{t} \lint_{{\mathbb R}^d} [  u^n_{\lambda}(\phi^{n}_{s-}(x) + z) -  u^n_{\lambda}(\phi^{n}_{s-}(x))]
   \tilde{\prm}(ds, dz),
   \end{aligned}
\end{equation}
where  $\phi^n_t(x)$, $t\geq 0$  is  the solution of \eqref{eqn-SDE}  with vector field $b^n$,  starting at  $x \in \R^d$ and $s=0$.
\end{theorem}
%-----------------
\begin{proof}[Proof of Theorem \ref{thm-approximation}] 
 We only need to prove claims \eqref{eqn ub sup norm-sup}, \eqref{eqn-conv-u_n to u 0 20251007} and \eqref{eqn-conv-u_n to u 0}.
 %and \eqref{eqn-conv-u_n to u 01}.  
 Since by assumption we can find $M>0$ such that
\begin{equation}\label{eqn-b^n-bound}
\sup_{n\in \mathbb{N}}\Vert b^n\Vert_{\beta}+\sup_{n\in \mathbb{N}}\Vert f^n\Vert_{\beta}\leq M.
\end{equation}
Hence  assertion \eqref{eqn ub sup norm-sup} follows from Theorem \ref{reg}.  Thus it remains  to prove \eqref{eqn-conv-u_n to u 0 20251007} and \eqref{eqn-conv-u_n to u 0}.
%and \eqref{eqn-conv-u_n to u 01}. 

Let us fix  $\lambda_0\geq 1$ and $R>0$ from \eqref{eqn ub sup norm-sup}.  Taking the difference between the equation \eqref{eqn-wee-n} for $u^n_\lambda$  and $u^0_\lambda$ we get, for every $n \in \mathbb{N}$,
\begin{align}\label{eqn-wee-difference}
(\lambda I-\gen{A}) ( u^n_{\lambda}-u^0_{\lambda})
-b^n \cdot D ( u^n_{\lambda}-u^0_{\lambda})  = f^n-f^0+ (b^n-b^0) \cdot D u^0_{\lambda}.
\end{align}
 By the maximum principle \eqref{eq zhai v0 control} we know that
\begin{equation} \label{max}
\lambda \| u^n_{\lambda}-u^0_{\lambda} \|_{0} \le \|  f^n-f^0+ (b^n-b^0) \cdot D u^0_{\lambda} \|_{0} \to 0
\end{equation}
as $n \to \infty$. Then we have \eqref{eqn-conv-u_n to u 0 20251007}.

Note that when $\alpha <1$ we cannot deduce  that the RHS of \eqref{eqn-wee-difference} converges to $0$
in the space  $C_{\mathrm{b}}^{\beta}(\R^d,\R^d)$. This is because in general the term  $(b^n-b^0) \cdot D u^0_{\lambda}$  only belongs to $C_{\mathrm{b}}^{\alpha + \beta -1}(\R^d,\R^d)$;  in general this term is not in  $C_{\mathrm{b}}^{\beta}(\R^d,\R^d)$ , see  also Remark \ref{cdd}.

 Now we use a useful  property of H\"older spaces , see  page 37 in \cite{KrHolder} and page 431 in \cite{Pr12}) together with  \eqref{eqn ub sup norm-sup} and \eqref{max}.  
Let us fix $\lambda \ge \lambda_0$ and a compact set $K \subset \R^d$.  Let us consider the H\"older space
$C^{\alpha + \beta'} (K,\R^d)$ for any    $0< \beta' <\beta$. 
 Let us consider any subsequence $u^k_{\lambda} $
 of $u^n_{\lambda}$. By  \eqref{eqn ub sup norm-sup}, we have
 $$
 \sup_{\lambda\geq \lambda_0}\sup_{n\geq0}\Vert u^n_{\lambda}\Vert_{\alpha+\beta} \leq R. 
 $$ 
Hence, 
 there exists a function $ v_{\lambda}\in C^{\alpha + \beta}_b (\R^d,\R^d)$  and 
 a further subsequence of  $u^k_{\lambda}$ still denoted by $u^k_{\lambda}$, possibly depending on $\lambda$ and $K$, such that 
 $u^k_{\lambda}$ converges to  $v_{\lambda}$
  in $C^{\alpha + \beta'} (K,\R^d)$. This function $v_{\lambda} $ coincides with  $u^0_{\lambda}$ by \eqref{max}. This implies that the whole sequence $u^n_{\lambda}$ converges to $u^0_{\lambda}$ in $C^{\alpha + \beta'} (K,\R^d)$. 
This proves \eqref{eqn-conv-u_n to u 0}. The proof  of the theorem  is complete. 
 \end{proof}
 
\begin{remark} \label{cdd} Note that
if $\alpha\in[1,2)$, then we have the following stronger result:
\begin{equation}\label{eqn-conv-u_n to u1}
 \lim_{n\to \infty}\Vert u^n_{\lambda}-u^0_{\lambda}\Vert_{\alpha+\beta}=0.
\end{equation}
The proof follows  using \eqref{eqn-wee-difference}  since in this case $(b^n-b^0) \cdot D u^0_{\lambda}$   belongs to   $C_{\mathrm{b}}^{\beta}(\R^d)$.
  By assumptions \eqref{eqn-b^n to b^0, f^n to f^0} and inequality \eqref{eq zhai prio 1}, we deduce that the RHS of \eqref{eqn-wee-difference} converges to $0$
in the space  $C_{\mathrm{b}}^{\beta}(\R^d)$. In view of the  bound  \eqref{eqn-b^n-bound},  by  part (III) of Theorem
 \ref{reg} we infer  \eqref{eqn-conv-u_n to u1}. 
\end{remark}

%-----------------

In the next result, we will  prove a stronger assertion than assertions  (iii) in Theorem \ref{uno}  and (iii) in  Theorem \ref{uno1}, see also Remark \ref{se4}.
This  result seems to be new even for Lipschitz $b$, see  \cite{Kunita_2004}. Let us emphasize that the results in Kunita's paper \cite{Kunita_2004}  are  not sufficient for our purposes because he didn't prove the existence of a ``universal" $\mathbb{P}$-full set  $\Omega^\prime$ as we did. See also Remark \ref{rem-kunita}.

\begin{theorem}\label{thm-ww}%{ww}
  {Under the assumptions of Theorem \ref{d32},}
 there exists a $\mathbb{P}$-full event $\Omega^\prime$ such that,
for  every $\omega \in \Omega^\prime$
\begin{trivlist}
\item[ (i)] for every $t\in [0,\infty)$,
 the mapping \eqref{eqn-1.6} , i.e.,
\begin{equation}\label{eqn-flow}
\R^d \ni
x \mapsto \phi_{t}(x)
\in \R^d
\end{equation}
 is differentiable;
\item[ (ii)]
the derivative  function
\[[0,\infty) \times \R^d \ni (t,x) \mapsto
D \phi_t (x) \in \mathscr{L}(\R^d,\R^d) \]
 is continuous;

\item[ (iii)]
for any $0<\delta<\frac{\alpha}2+\beta-1$ ,
   the derivative  function
\begin{equation}\label{eqn-derivative flow}
\R^d \ni
x \mapsto      D \phi_t (x)  \in \mathscr{L}(\R^d,\R^d)
\end{equation}
is locally $\delta$-H\"older continuous, locally  uniformly in $t\in [0,\infty)$.
\end{trivlist}
\end{theorem}
\begin{proof}[Proof of Theorem \ref{thm-ww}]

Since  Theorem \ref{d32} is given,  see  also Remark \ref{rem-kunita},  we
 argue, with important modifications,  as in Section 3 of Kunita  \cite{Kunita_2004}.
  Indeed, the It\^o formula
  \eqref{eqn-Ito-iii} is similar to  a classical   SDE with regular coefficients    considered
in    \cite[Section 3.3]{Kunita_2004}. However, applying directly the results from \cite{Kunita_2004}, we can only assert that
the map \eqref{eqn-flow} is  differentiable on $\R^d$ on some almost sure event $\Omega_t^\prime$ which may depend on $t$.

 Let us first define an auxiliary set $\mathscr{O}$ by
\begin{equation}
\label{eqn-set O}
\mathscr{O}:= \R^d \times ([-1,0) \cup (0,1]).
\end{equation}
Let us next define an auxiliary random field $N$
\begin{align}\label{eqn-N-nardom field}
N_t (x, \lambda) &= \frac{\phi_t(x + \lambda e_i) - \phi_t(x) }{\lambda}, \quad (t,x, \lambda) \in [0,\infty) \times \mathscr{O}.
 \end{align}
  By looking at our initial SDE \eqref{eqn-SDE}, see  also \eqref{d566}, we deduce that
 \begin{equation}
\begin{aligned}
\phi_t(x + \lambda e_i) - \phi_t(x) &= \lambda e_i + \int_0^t [b(\phi_s(x + \lambda e_i)) -b (\phi_s(x ) )] ds + L_t - L_t
\\
&= \lambda e_i + \int_0^t [b(\phi_s(x + \lambda e_i)) -b (\phi_s(x ) )] ds ,
\end{aligned}
\end{equation}
from which we deduce that the paths on $\Omega'$ (from Theorem \ref{d32})
\begin{equation} \label{serve}
[0,\infty) \ni t \mapsto N_{t} (x, \lambda) \;\; \text{are  continuous}.
\end{equation}
This fact will be  important in the sequel.  Here the situation is different from the more general case considered in \cite{Kunita_2004} in which  the process $N_{\cdot} (x, \lambda)$   had only a.s. c\`adl\`ag paths.

Thus  we  can view  $N(x,\lambda)$ as a random variable with values in the separable Fr\'echet space $C([0,\infty), \R^d)$ whereas the author  of \cite{Kunita_2004} was unable to  view  $N(x,\lambda)$ as a random variable with values in the  Fr\'echet space of all c\`adl\`ag paths endowed with the local supremum norm, because such  space is not separable.

\smallskip

 We choose and fix  numbers $T>0$.
Let us define an auxiliary parameter
\begin{equation}\label{eqn-gamma}
\gamma = \alpha + \beta -1.
\end{equation}
Let us observe that due to assumtion \eqref{eqn-beta+alpha half}, there exists $\delta$ satisfying the condition below.
\begin{equation}\label{eqn-delta}
0<\delta<\frac{\alpha}2+\beta-1<\gamma.
\end{equation}
In what follows we choose and fix such  parameter $\delta$.
Moreover,   condition  \eqref{eqn-beta+alpha half}  implies that  $ \gamma > \frac{\alpha}2$ and the number
\begin{equation}\label{eqn-gamma'}\gamma^\prime  :=\gamma-\delta  \in (\frac{\alpha}2, \gamma).
\end{equation}
Therefore, by  Hypothesis  \textbf{(H)}  and the definition of the Blumenthal-Getoor index  we infer that
\begin{align}\label{eqn-finitness of integral}
 \lint_{B} \vert z\vert^{p\gamma^\prime} \nu (dz) < \infty,\quad \mbox{ for } p \ge 2.
\end{align}
Note that in particular inequality \eqref{eqn-finitness of integral} holds for $p=2$  and we use this fact.

To continue with the proof of Theorem \ref{thm-ww}, following the approach used in proof of  \cite[Theorem 3.3]{Kunita_2004}, let us first formulate two important results that are essential to the proof.
\begin{proposition}
\label{prop-moment estimates}
Assume that  $N$ is the random field defined in
\eqref{eqn-N-nardom field}.
 For every $p \ge 2$, there exists  $C_p= C(T, \alpha, \beta, \nu, p, d)>0$  such that
\begin{align} \label{eqn-do}%{do}
\mathbb{E}[\sup_{s \in [0,T]} \vert N_s (x, \lambda)\vert^p] &\leq  C_p, \;\;\; \mbox{for all $(x,\lambda) \in \mathscr{O}$}.
\end{align}
\end{proposition}

\begin{proposition}
\label{prop-moment estimates-difference}
Assume that  $N$ is the random field defined in
\eqref{eqn-N-nardom field} and  $\delta$ is as in \eqref{eqn-delta}. Then, for every $p \ge 2$, there exists  $C_p= C(T, \alpha, \beta, \nu, p, d)>0$  such that
for all $(x,\lambda),(x^\prime,\lambda^\prime) \in \mathscr{O}$
\begin{align}
\label{do1}
\mathbb{E}[\sup_{s \in [0,T]} |N_s (x, \lambda) - N_s (x^\prime, \lambda^\prime) \vert^p] & \leq C_p (
|x -x^\prime\vert^{\delta p}  + |\lambda - \lambda^\prime\vert^{\delta p}).
\end{align}
\end{proposition}

Assuming for the time being validity of Proposition \ref{prop-moment estimates-difference}, %\eqref{do1} holds,
we deduce that
\begin{equation} \label{ew}
\E[ \Vert N_{\cdot} (x, \lambda) - N_{\cdot} (x^\prime, \lambda^\prime) \Vert^p_{C([0,T],\R^d)}] \leq C_p (
|x -x^\prime\vert^{\delta p}  + |\lambda - \lambda^\prime\vert^{\delta p}),\quad (x,\lambda), (x^\prime,\lambda^\prime) \in \mathscr{O}.
\end{equation}
Choosing $p$ large enough and then applying  the Kolmogorov Test, see e.g.  \cite[Appendix]{Kunita_2004},  we deduce
that there exists an $C([0,T],\R^d)$-valued continuous random field $\tilde{N}$ defined on the closure $\bar{\mathscr{O}}=\R^d \times [-1,1]$ of the set $\mathscr{O}$,
such that for every $(x, \lambda) \in \mathscr{O}$, $N (x, \lambda)=\tilde{N} (x, \lambda)$, $\mathbb{P}$-almost surely. Hence we obtain
the differentiability with respect to  $x$ of the process $\phi_t(x)$ uniformly in $t\in [0,T]$   with $D_i \phi_t (x)=\tilde{N}_t(x,0) $. Moreover,
we deduce that the process    $D_i \phi_t (x) $ is continuous w.r.t. $(t,x) \in [0,T] \times \R^d$. More precisely,
 we get that, $\mathbb{P}$-a.s., for every $x \in \R^d$, $i=1, \ldots, d,$
\begin{equation} \label{22}
 \lim_{\lambda \to 0}
 \big \Vert\frac{ \phi_\cdot (x + \lambda e_i) - \phi_\cdot(x)}{\lambda}  - D_i \phi_\cdot (x) \big \Vert_{C([0,T]; \R^d)} =0.
\end{equation}
Moreover, since $\delta$ can be any element in $(0,\frac\alpha2+\beta-1)$,  $D_i \phi_\cdot (x)$ is locally $\delta$-H\"older continuous w.r.t. $x$ with values in $C([0,T]; \R^d)$, for any $0<\delta =\gamma-\gamma^\prime<\frac\alpha2+\beta-1$.

To finish the proof of Theorem \ref{thm-ww}, it remains to prove Propositions \ref{prop-moment estimates} and \ref{prop-moment estimates-difference}.% \eqref{eqn-do} and \eqref{do1}. We proceed in two steps.

In the following proofs the constants $C$ or $c$  can change values from line to line. They may only depend on $\alpha$, $\beta$, $\nu$, $p,$ $d$ and $T$.
We  also choose and fix $\varpi>\varpi{}_{,M}$ and $\eps=\frac{1}{3M}$, where $M=\Vert b \Vert_{\beta}$ and $\varpi{}_{,M}$ is introduced in (III) of Theorem \ref{reg}. Set $u=u_{\lambda_0}$ be the unique solution
 of equation \eqref{eqn-wee} with the RHS $f=b$ and $\lambda=\lambda_0$. Hence, (III) of Theorem \ref{reg} implies that
\begin{equation}\label{eq P29 star}
    \|Du\|_0\leq \frac{1}{3}\text{ and } \|u\|_{\alpha+\beta}\leq C_MM<\infty.
\end{equation}

\begin{proof}[Proof of Proposition \ref{prop-moment estimates}] %\eqref{eqn-do}
Let us choose and fix $p \in [2,\infty)$ and  $(x,\lambda) \in \mathscr{O}$ and $i \in \{1,\cdots, d\}$. Set
\begin{equation}\label{eqn-Y}
Y (t;x,\lambda) = N_t(x, \lambda)=\frac{\phi_t(x + \lambda e_i) - \phi_t(x) }{\lambda} .
\end{equation}
For simplicity of presentation, we write $Y(t)$ for $Y (t;x,\lambda)$.
Let us recall that by  \eqref{eqn-Ito-iii},  with $u=u_{\varpi}$  we have
\begin{align}\label{xi-u-eq-1-1}
\phi_{t}(x+ \lambda e_i)
=&
 x+ \lambda e_i+ L_t - [ u (\phi_{t}(x+ \lambda e_i)) -  u (x+ \lambda e_i )]
 + \varpi  \int_0^t  u(\phi_{s}(x+ \lambda e_i))\, ds\nonumber\\
 &+  \int_0^{t} \lint_{{\mathbb R}^d} [  u(\phi_{s-}(x+ \lambda e_i) + z) -  u(\phi_{s-}(x+ \lambda e_i))]
   \tilde{\prm}(ds, dz),
\end{align}
and
\begin{align}\label{xi-u-eq-1-2}
\phi_{t}(x)
=&
 x+ L_t - [ u (\phi_{t}(x)) -  u (x )]
 + \varpi  \int_0^t  u(\phi_{s}(x))\, ds\nonumber\\\
 &+  \int_0^{t} \lint_{{\mathbb R}^d}  [  u(\phi_{s-}(x) + z) -  u(\phi_{s-}(x))]
   \tilde{\prm}(ds, dz).
\end{align}
Thus  we have
\begin{align}
Y(t;x,\lambda)
= &e_i  - \frac{1}{\lambda}[ u (\phi_{t}(x + \lambda e_i)) -  u (x + \lambda e_i) - u (\phi_{t}(x )) +  u (x )]
\\
 &+ \frac{\varpi}{\lambda}  \int_0^t  [u(\phi_{s}(x + \lambda e_i)) - u(\phi_{s}(x)) ]  ds
 +  \int_0^{t} \lint_{{\mathbb R}^d}
  Z(s,x,z,\lambda) \,\tilde{\prm}(ds, dz),
\end{align}
where, for $  s\geq 0$  and  $(x,\lambda) \in \mathscr{O}$,
\begin{align}\label{eqn-Z}
Z(s,x,z,\lambda):=\frac{1}{\lambda}\Big[u(\phi_{s-}(x + \lambda e_i) \!+\! z) \!-\!  u(\phi_{s-}(x \!+\! \lambda e_i))
\!-\! u(\phi_{s-}(x) \!+\! z) \!+\!  u(\phi_{s-}(x))\Big].
\end{align}
Since by \eqref{eq P29 star}  $\Vert Du \Vert_0 \le \frac13$,  i.e., the Lipschitz constant of $u$ is $\leq \frac13$, we infer that
\begin{align}
& \Bigl\vert \frac{1}{\lambda}[ u (\phi_{t}(x + \lambda e_i)) -  u (x + \lambda e_i) - u (\phi_{t}(x )) +  u (x )] \Bigr\vert
\\
& \leq
\Bigl\vert \frac{1}{\lambda}[ u (\phi_{t}(x + \lambda e_i)) -  u (\phi_{t}(x ))]\Bigr\vert  +   \Bigl\vert  \frac{1}{\lambda}[ u (x + \lambda e_i) -   u (x )] \Bigr\vert\le \frac{1}{3} (|Y(t)| + 1).
\end{align}
 It follows, since $p\geq 2$,  that for every  $t\in[0,T]$,
\begin{align}
 & |Y(t)\vert^p
 \leq c_p + c_p \int_0^t  |Y(r)\vert^p dr+ c_p \Bigl\vert \int_0^{t} \lint_{{\mathbb R}^d}
    Z(s,x,z,\lambda)    \tilde{\prm}(ds, dz) \Big\vert^p.
\end{align}
Now, we estimate the stochastic integral  as in the proof of Theorem 4.3 in \cite{Pr12}. Using \eqref{eqn-Burkholder inequality} and treating separately the small jumps part and the large jumps part, we obtain
\begin{align*}
   & \ \mathbb{E}\Big[ \sup_{0\leq s\leq t}  \Bigl\vert \int_0^{s} \int_{B^c}
    Z(r,x,z,\lambda)
   \tilde{\prm}(dr, dz) \Big\vert^p     \Big]\\
   &\leq C_p \mathbb{E} \Big[ \Big(   \int_0^t\int_{B^c}|Z(s,x,z,\lambda) |^2 \nu(dz)d s\Big)^{\frac{p}2}         \Big]
   +C_p \mathbb{E} \Big[   \int_0^t\int_{B^c}|Z(s,x,z,\lambda) |^p \nu(dz)d s       \Big],
\end{align*}
and
\begin{align*}
   &  \mathbb{E}\Big[ \sup_{0\leq s\leq t}  \Bigl\vert \int_0^{s} \lint_{B}
    Z(r,x,z,\lambda)
   \tilde{\prm}(dr, dz) \Big\vert^p     \Big]\\
   &\leq C_p \mathbb{E} \Big[ \Big(   \int_0^t\lint_{B}|Z(s,x,z,\lambda) |^2 \nu(dz)d s\Big)^{\frac{p}2}         \Big]
   +C_p \mathbb{E} \Big[   \int_0^t\lint_{B}|Z(s,x,z,\lambda) |^p \nu(dz)d s       \Big].
\end{align*}
In view of the fact that $|Z(s,x,z,\lambda)|\leq 2\|Du\|_0|Y(s-)| \leq \frac23|Y(s-)|$, we have
\begin{align*}
   & \mathbb{E}\Big[ \sup_{0\leq s\leq t}  \Bigl\vert \int_0^{s} \int_{B^c}
    Z(r,x,z,\lambda)
   \tilde{\prm}(dr, dz) \Big\vert^p     \Big]\\
   &\leq C_p(1+t^{\frac{p}2-1})\Big[\int_{B^c} \nu(dz)+ \Big(  \int_{B^c} \nu(dz)     \Big)^{\frac{p}2}\Big]\int_0^t\mathbb{E}\sup_{0\leq s\leq r}|Y(s)|^pdr,
\end{align*}
where we also used the H\"older inequality.

By applying Lemma \ref{lem-Lemma 4.1}, we infer $|Z(s,x,z,\lambda)|\leq \|u\|_{1+\gamma}|Y(s-)||z|^{\gamma}$ with $\gamma=\alpha+\beta-1$, see \eqref{eqn-gamma},  and hence we get
\begin{align*}
   & \mathbb{E}\Big[ \sup_{0\leq s\leq t}  \Bigl\vert \int_0^{s} \lint_{B}
    Z(r,x,z,\lambda)
   \tilde{\prm}(dr, dz) \Big\vert^p     \Big]\\
   &\leq C_p \|u\|_{1+\gamma}^p \Big[ \mathbb{E} \Big(   \int_0^t\lint_{B  }|Y_s |^2|z|^{2\gamma} \nu(dz)d s\Big)^{\frac{p}2}         +\mathbb{E}  \int_0^t\lint_{B  }|Y_s |^p|z|^{\gamma p} \nu(dz)d s       \Big]\\
   &\leq C_p (1+t^{\frac{p}2-1})\|u\|_{1+\gamma}^p\Big[ \Big(   \lint_{B  }|z|^{2\gamma} \nu(dz)\Big)^{\frac{p}2}  +  \lint_{B }|z|^{p\gamma} \nu(dz)      \Big] \int_0^t \mathbb{E}\sup_{0\leq s\leq r}|Y(s)|^pdr,
\end{align*}
where $\lint_{B}|z|^{2 \gamma} \nu(d z)+\lint_{B}|z|^{p \gamma} \nu(d z)<+\infty$, since $2\gamma>\alpha$ and $p\gamma>\alpha$.
Combining the above estimates, finally, we obtain
\begin{equation}\label{dd-2}
\mathbb{E} [\sup_{0 \le s \le t}|Y(s)\vert^p] \le c_p + c_p(1+t^{\frac{p}2-1}) \int_0^t  \mathbb{E} [\sup_{0 \le r \le s} |Y(r)\vert^p] ds,\;\;\; \; t \in [0,T],
\end{equation}
for some constant $c_p>0$ independent of $x$ and $\lambda$. Applying the Gronwall lemma we obtain  \eqref{eqn-do}.
Hence the proof of Proposition \ref{prop-moment estimates} is complete.
\end{proof}

\begin{proof}[Proof of Proposition \ref{prop-moment estimates-difference}] %\eqref{do1}}]
Let us choose and fix $ p \in [2,\infty)$ and $(x,\lambda),(x^\prime,\lambda^\prime)\in \mathscr{O}$.
 We  consider
 a continuous process $\hat X$ defined by
\begin{equation}\label{eqn-X-hat}
\hat X(t)=\hat X(t;x, \lambda,x^\prime, \lambda^\prime) = N_t(x, \lambda) - N_t (x^\prime, \lambda^\prime),\;\;\; t \geq 0,
\end{equation}
where $N$ is the random field defined in \eqref{eqn-N-nardom field}.\\
Then, by  \eqref{xi-u-eq-1-1} and \eqref{xi-u-eq-1-2},  we have
\begin{align*}
\hat X(t)=J_1(t) + J_2(t) + J_{31}(t)+J_{32}(t),\;\; t\geq 0,
\end{align*}
where
\begin{align*}
J_1(t)&= \! -   \frac{1}{\lambda}[ u (\phi_{t}(x + \lambda e_i)) -  u (x + \lambda e_i) - u (\phi_{t}(x )) +  u (x )]
\\
&\quad\quad\quad+\frac{1}{\lambda^\prime}[ u (\phi_{t}(x^\prime + \lambda^\prime e_i)) -  u (x^\prime + \lambda^\prime e_i) - u (\phi_{t}(x^\prime )) +  u (x^\prime )],
\\
J_2(t)&=   \varpi  \int_0^t \Bigl[
 \frac{u(\phi_{s}(x + \lambda e_i)) - u(\phi_{s}(x)) }{\lambda}
 -
\frac{u(\phi_{s}(x^\prime + \lambda^\prime e_i)) - u(\phi_{s}(x^\prime)) }{\lambda^\prime}
\Big ]ds,
\\
J_{31}(t)&=   \int_0^{t} \int_{B^{\rc}} H(s-,z,x,x^\prime, \lambda, \lambda^\prime)
   \tilde{\prm}(ds, dz),
\\
J_{32}(t)&=   \int_0^{t} \lint_{B} H(s-,z,x,x^\prime, \lambda, \lambda^\prime)
   \tilde{\prm}(ds, dz),
\end{align*}
and where
\begin{align*}
&H(s,z,x,x^\prime, \lambda, \lambda^\prime)\\
&=
\frac{1}{\lambda} [  u(\phi_{\sm}(x + \lambda e_i) + z) -  u(\phi_{s}(x + \lambda e_i))
- u(\phi_{\sm}(x) + z) +  u(\phi_{\sm}(x))]
 \\ &\quad-  \frac{1}{\lambda^\prime}
[  u(\phi_{\sm}(x^\prime
+ \lambda^\prime e_i) + z) -  u(\phi_{\sm}(x^\prime + \lambda^\prime e_i))- u(\phi_{\sm}(x^\prime) + z) +  u(\phi_{\sm}(x^\prime))]
\\
&= \frac{1}{\lambda} \bigl[ u(\phi_{\sm}(x + \lambda e_i)+z)- u(\phi_{\sm}(x)+z)    \bigr] - \frac{1}{\lambda} \bigl[ u(\phi_{\sm}(x + \lambda e_i))- u(\phi_{\sm}(x))  \bigr]\\
&\quad - \frac{1}{\lambda^\prime} \bigl[ u(\phi_{\sm}(x^\prime + \lambda^\prime e_i) + z) - u(\phi_{\sm}(x^\prime) + z) \bigr]+ \frac{1}{\lambda^\prime} \bigl[ u(\phi_{\sm}(x^\prime + \lambda^\prime e_i)) -u(\phi_{\sm}(x^\prime)) \bigr].
\end{align*}
%------
With $r_1=\phi_{\sm}(x + \lambda e_i)+z$ and $r_2=\phi_{\sm}(x)+z$, so that   by  \eqref{eqn-N-nardom field}
\begin{align}
r_1-r_2&=\phi_{\sm}(x + \lambda e_i) -\phi_{\sm}(x)=\lambda N_{\sm}(x,\lambda),\\
r_2+\theta (r_1-r_2)&= \phi_{\sm}(x)+z+ \theta \lambda N_{\sm}(x,\lambda).
\end{align}
Hence by the  mean value theorem  we have
\begin{align*}
&\frac{1}{\lambda} \bigl[ u(\phi_{\sm}(x + \lambda e_i)+z)- u(\phi_{\sm}(x)+z)    \bigr]
=
\frac{1}{\lambda} \bigl[u(r_1)-u(r_2)]
\\
&=\frac{1}{\lambda} \int_0^1 Du(r_2+\theta (r_1-r_2))(r_1-r_2)\, d\theta
=
\int_0^1 Du(\phi_{\sm}(x)+z+ \theta \lambda N_{\sm}(x,\lambda))(N_{\sm}(x,\lambda))\, d\theta.
\end{align*}

Analogously we have
\begin{align*}
\frac{1}{\lambda^\prime} \bigl[ u(\phi_{\sm}(x^\prime + \lambda^\prime e_i)+z)- u(\phi_{\sm}(x^\prime)+z)    \bigr]
=
\int_0^1 Du(\phi_{\sm}(x^\prime)+z+ \theta \lambda^\prime N_{\sm}(x^\prime,\lambda^\prime))(N_{\sm}(x^\prime,\lambda^\prime))\, d\theta,
\end{align*}
and 
\begin{align}
\frac{1}{\lambda} \bigl[ u(\phi_{\sm}(x + \lambda e_i))- u(\phi_{\sm}(x))    \bigr]
&=
\int_0^1 Du(\phi_{\sm}(x)+ \theta \lambda N_{\sm}(x,\lambda))(N_{\sm}(x,\lambda))\, d\theta,\label{eq P32 eq 1}\\
\frac{1}{\lambda^\prime} \bigl[ u(\phi_{\sm}(x^\prime + \lambda^\prime e_i))- u(\phi_{\sm}(x^\prime))    \bigr]
&=
\int_0^1 Du(\phi_{\sm}(x^\prime)+ \theta \lambda^\prime N_{\sm}(x^\prime,\lambda^\prime))(N_{\sm}(x^\prime,\lambda^\prime))\, d\theta.\label{eq P32 eq 2}
\end{align}
Therefore,
\begin{align}
&H(s,z,x,x^\prime, \lambda, \lambda^\prime)\nonumber\\
&=
\int_0^1 Du(\phi_{\sm}(x)+z+ \theta \lambda N_{\sm}(x,\lambda))(N_{\sm}(x,\lambda))\, d\theta
-\int_0^1 Du(\phi_{\sm}(x)+ \theta \lambda N_{\sm}(x,\lambda))(N_{\sm}(x,\lambda))\, d\theta\nonumber\\
&\quad -\int_0^1 Du(\phi_{\sm}(x^\prime)+z+ \theta \lambda^\prime N_{\sm}(x^\prime,\lambda^\prime))(N_{\sm}(x^\prime,\lambda^\prime))\, d\theta
\nonumber\\
&\quad +
\int_0^1 Du(\phi_{\sm}(x^\prime)+ \theta \lambda^\prime N_{\sm}(x^\prime,\lambda^\prime))(N_{\sm}(x^\prime,\lambda^\prime))\, d\theta
\nonumber\\
&=
\int_0^1 \Bigl[ Du(\phi_{\sm}(x)+z+ \theta \lambda N_{\sm}(x,\lambda))- Du(\phi_{\sm}(x)+ \theta \lambda N_{\sm}(x,\lambda)) \Bigr]
(N_{\sm}(x,\lambda)-N_{\sm}(x^\prime,\lambda^\prime))\, d\theta
\nonumber\\
&\quad +
\int_0^1 \Bigl[ Du(\phi_{\sm}(x)+z+ \theta \lambda N_{\sm}(x,\lambda))-
Du(\phi_{\sm}(x)+ \theta \lambda N_{\sm}(x,\lambda))
\label{eqn-H} \\
&
\quad\quad\quad\quad- Du(\phi_{\sm}(x^\prime)+z+ \theta \lambda^\prime N_{\sm}(x^\prime,\lambda^\prime))
+Du(\phi_{\sm}(x^\prime)+ \theta \lambda^\prime N_{\sm}(x^\prime,\lambda^\prime))
\Bigr] (N_{\sm}(x^\prime,\lambda^\prime))\, d\theta.\nonumber
\end{align}

In order to find estimates for $J_{31}(t)$ and $J_{32}(t)$, we will proceed as in the proof of   \cite[Theorem 4.3]{Pr12} and  use the Burkholder inequality \eqref{eqn-Burkholder inequality} from Proposition \ref{prop-Burkholder inequality}.

We will first   deal with the term $J_{32}$. Let us choose and fix an arbitrary  $z \in B$.
Using the fact that $Du$ is $\gamma$-H\"older continuous  and the notation \eqref{eqn-X-hat},  we infer that
   the 1st term on the RHS of \eqref{eqn-H} is bounded by
\[
C \vert z\vert^{\gamma} |\hat X(s)|, \;\; s \geq 0.
\]

In order to estimate the 2nd term  on the RHS of \eqref{eqn-H} we will argue as on   \cite[pages 445 and 446]{Pr12}.
By   inequality \eqref{eqn-4.16} in Lemma \ref{lem-Lemma 4.1} and \eqref{eq P29 star}, we deduce that  there exists a constant $c$ depending on $Du$, $\gamma$ and $\gamma^\prime$ such that
for all  $y,y' \in \R^d$,
\begin{equation}\label{rr4}
 |D  u (y+ z)  -  D  u (y) - D  u (y' + z) + D  u (y') | \le c \vert z\vert^{\gamma^\prime} \, |y-y'\vert^{\gamma - \gamma^\prime}.
\end{equation}
Applying \eqref{rr4} we find that the 2nd term on the RHS of \eqref{eqn-H} is bounded by
 \begin{align*}
 & C| N_s(x^\prime, \lambda^\prime )| \, \vert z\vert^{\gamma^\prime} \, \int_0^1  |
 \phi_{s}(x) +  \theta \lambda N_s(x, \lambda)
 -
 \phi_{s}(x^\prime) -  \theta \lambda^\prime N_s(x^\prime, \lambda^\prime) \vert^{\gamma - \gamma^\prime} d \theta
\\
&\leq   C \vert z\vert^{\gamma^\prime}  | N_s(x^\prime, \lambda^\prime )| \,  \, \Big[  |
 \phi_{s}(x)   -  \phi_{s}(x^\prime)\vert^{\gamma - \gamma^\prime}
 +
|\phi_s (x + \lambda e_i) - \phi_s(x^\prime + \lambda^\prime e_i)\vert^{\gamma - \gamma^\prime}
 \Bigr].
\end{align*}
Therefore, we deduce that
\begin{align*}
&|H(s,z,x,x^\prime, \lambda, \lambda^\prime)|\\
&\leq
C \vert z\vert^{\gamma}   |\hat X(s)| + C \vert z\vert^{\gamma^\prime}  | N_s(x^\prime, \lambda^\prime )|
 \Big[  |
 \phi_{s}(x)   -  \phi_{s}(x^\prime)\vert^{\gamma - \gamma^\prime}
 +
|\phi_s (x + \lambda e_i) - \phi_s(x^\prime + \lambda^\prime e_i)\vert^{\gamma - \gamma^\prime}
 \Bigr],
\end{align*}
and so,
 we get
\begin{align*}
&|H(s,z,x,x^\prime, \lambda, \lambda^\prime)\vert^p
\leq C_p \vert z\vert^{p\gamma^\prime} |\hat X(s)\vert^p\\
&+
  C_p | N_s(x^\prime, \lambda^\prime )\vert^p \, \vert z\vert^{p\gamma^\prime} \Big[  |
 \phi_{s}(x)   -  \phi_{s}(x^\prime)\vert^{p(\gamma - \gamma^\prime)}
+
|\phi_s (x + \lambda e_i) - \phi_s(x^\prime + \lambda^\prime e_i)\vert^{p(\gamma - \gamma^\prime)}
\Bigr],
 \end{align*}
which is a useful estimate.  Thus, since $\delta = \gamma - \gamma^\prime$,  we infer that
\begin{align} \label{wq}
&\mathbb{E}\Big[\sup_{s\in[0,t]}\vert J_{32}(s)\vert^p \Big] \leq C_p \mathbb{E}\Big[\int_0^t  |\hat X(s)\vert^p ds\Big]
\\   &+ C_p  \mathbb{E}\Big[\int_0^t \!\! | N_s(x^\prime, \lambda^\prime )\vert^p \,   \Big[  |
 \phi_{s}(x)   -  \phi_{s}(x^\prime)\vert^{p \delta}
 +
|\phi_s (x + \lambda e_i) - \phi_s(x^\prime + \lambda^\prime e_i)\vert^{p \delta }
  \Bigr]
 ds\Big].\nonumber
\end{align}
We have  used \eqref{eqn-Burkholder inequality} and \eqref{eqn-finitness of integral} to obtain the above estimate.

Let us treat $J_{31}(t)$. Since $Du$ is $\gamma$-H\"older continuous and \eqref{eq P29 star},  by
\eqref{eqn-interpolation inequality}, we have
\begin{equation} \label{2ww}
 |D  u (y )  - D  u (y')  | \le c  |y-y'\vert^{\gamma - \gamma^\prime} =
  c  |y-y'\vert^{\delta}, \;\; y,y' \in \R^d.
\end{equation}
We obtain by \eqref{eq P29 star}, \eqref{eqn-H} and \eqref{2ww}, when $\vert z| >1$,
\begin{align*}
\hspace{-0.8truecm}| H(s,z,x,x^\prime, \lambda, \lambda^\prime)|
\le c |\hat X(s)|
 \!+\! c |N_s(x^\prime, \lambda^\prime)|   [
 |\phi_s (x) \!-\! \phi_s(x^\prime)\vert^{\delta} \!+\! |\phi_s (x + \lambda e_i) \!-\! \phi_s(x^\prime + \lambda^\prime e_i)\vert^{\delta} ],
\end{align*}
and  so we arrive at, for $p \ge 2,$
\begin{align} \label{wq1}
 \mathbb{E}\Big[\sup_{s\in[0,t]}|J_{31} (s)|^p\Big]
 \leq&  C_p \mathbb{E}\Big[\int_0^t  |\hat X(s)\vert^p \,ds
\\ & + \int_0^t |N_s(x^\prime, \lambda^\prime)\vert^p \,     ( |\phi_s (x) - \phi_s(x^\prime)\vert^{p\delta} + |\phi_s (x + \lambda e_i) - \phi_s(x^\prime + \lambda^\prime e_i)\vert^{p\delta } )\, ds\Big].\nonumber
\end{align}

 Now  we consider  $J_1$. By \eqref{eqn-X-hat}, \eqref{eq P32 eq 1} and \eqref{eq P32 eq 2}, $J_1$ can be expressed  as
\begin{align*}
J_1(t)=& - \int_0^1 D  u (\phi_{t}(x) + \theta \lambda N_t(x, \lambda))\,[ \hat X(t)]
  d \theta+  \int_0^1 [D  u (x + \theta \lambda e_i)-D  u (x^\prime + \theta \lambda^\prime e_i)] e_i d \theta
\\
& +  \int_0^1\!\! [ D  u (\phi_{t}(x^\prime) \!+\! \theta [\phi_t(x^\prime \!+\! \lambda^\prime e_i) \!-\! \phi_t(x^\prime)]) \!-\! D  u (\phi_{t}(x) \!+\! \theta
[\phi(x \!+\! \lambda e_i) \!-\! \phi(x)] )]
 N_t(x^\prime, \lambda^\prime) d \theta.
\end{align*}
Recalling \eqref{eq P29 star} and \eqref{2ww}, it follows that
\begin{align*}
 |J_1(t)|
 \leq&\,
  \frac{1}{3}  |\hat X(t)| + C |\lambda - \lambda^\prime\vert^{\delta} + C |x - x^\prime\vert^{\delta}
 \\
&+
   C  |N_t(x^\prime, \lambda^\prime)| \,  [
 |\phi_t (x) - \phi_t(x^\prime)\vert^{\delta} + |\phi_t (x + \lambda e_i) - \phi_t(x^\prime + \lambda^\prime e_i)\vert^{\delta} ].
  \end{align*}
 Arguing as above, we get
\begin{align*}
J_2(t)=& \varpi  \int_0^t  \int_0^1 D  u (\phi_{s}(x) + \theta \lambda N_s(x, \lambda))
  N_s(x, \lambda) d\theta\,  ds
\\
&- \varpi \int_0^t  \int_0^1 D  u (\phi_{s}(x^\prime) + \theta \lambda^\prime N_s(x^\prime, \lambda^\prime))
  N_s(x^\prime, \lambda^\prime) d\theta\, ds
\\
=&
\varpi  \int_0^t  \int_0^1 D  u (\phi_{s}(x) + \theta \lambda N_s(x, \lambda))
 [ \hat X(s) ]d \theta  \,ds
\\
&+ \varpi \int_0^t \int_0^1 [D  u (\phi_{s}(x) + \theta \lambda N_s(x, \lambda)) -  D  u (\phi_{s}(x^\prime) + \theta \lambda^\prime N_s(x^\prime, \lambda^\prime))]
  N_s(x^\prime, \lambda^\prime) d \theta\, ds.
   \end{align*}
It follows that
\begin{equation}
|J_2(t)|
 \le c \int_0^t  |\hat X(s)| ds + c \int_0^t |N_s(x^\prime, \lambda^\prime)| \,  [
 |\phi_s (x) - \phi_s(x^\prime)\vert^{\delta} + |\phi_s (x + \lambda e_i) - \phi_s(x^\prime + \lambda^\prime e_i)\vert^{\delta} ]\,  ds.
\end{equation}

  Note that the following inequality is
 useful estimates. We can use \eqref{ciao7} and \eqref{eqn-do} to obtain, for $p \ge 2,$  by the H\"older inequality
\begin{align} \label{e44}
&\mathbb{E} \Big[\sup_{0 \le t \le T} \Big(|N_t(x^\prime, \lambda^\prime)\vert^p  \, \big[
 |\phi_t (x) - \phi_t(x^\prime)\vert^{p \delta} + |\phi_t (x + \lambda e_i) - \phi_t(x^\prime + \lambda^\prime e_i)\vert^{p\delta} \big] \Big)\Big]\nonumber\\
&\le C_p [|\lambda - \lambda^\prime\vert^{p\delta} + |x - x^\prime\vert^{p\delta}],
\end{align}
for some constant $C_p >0$ independent of $x,x^\prime, \lambda, \lambda^\prime$.

  By  the previous  estimates, 
 we easily obtain
 \begin{equation} \label{1ww}
\mathbb{E} [\sup_{0 \le s \le t}|\hat X(s)\vert^p] \le  C \int_0^t  \mathbb{E} [\sup_{0 \le r \le s} |\hat X(r)\vert^p]\, ds +   C |\lambda - \lambda^\prime\vert^{\delta p} +
 C|x - x^\prime\vert^{\delta p},
\end{equation}
where $C$ depends on $\alpha$, $\beta$, $\nu$, $p$, $d$ and $T$. From  \eqref{1ww} we deduce easily \eqref{do1} using the Gronwall Lemma. This completes the proof of Proposition \ref{prop-moment estimates-difference}.
 \end{proof}

The proof of Theorem \ref{thm-ww} is now complete.
\end{proof}

As a byproduct of  the previous proof we also get the following result.

\begin{corollary}\label{cor-s11}  For every $p \ge 2$ and $T>0$ there exists $C_p=C_p(T) >0$ such that for every    $x \in \R^d,$
\begin{equation} \label{s11}
 \mathbb{E} [\sup_{t \in [0,T]} |D \phi_t(x)\vert^p ] \le C_p,
\end{equation}
\end{corollary}

\subsection{The diffeomorphism property }

  To prove Theorem \ref{thm-ww}, we have applied the standard Kolmogorov test in order to  obtain
nice regularity properties for $D\phi_t(x)$. However, to establish the diffeomorphism  property of  
$\phi_{s,t}(x)$, it seems that we cannot use directly such  Kolmogorov test because of the presence of both  $s$ and $t$. More precisely, when analyzing the inverse flow $\phi^{-1}_{s,t}(x)$, see \eqref{de}, the presence of a backward term requires a different approach and a stronger version of the Kolmogorov test is necessary (cf. Theorem \ref{thm-KTC-ZB}). This is why in the next Theorem \ref{ww1} we will work differently to study    $D\phi_{s,t}(x)$.

\begin{theorem}
\label{ww1} {Under the assumptions of Theorem \ref{d32},}
there exists a $\mathbb{P}$-full event $\Omega^{\prime\prime}$ such that,
for all  $\omega \in \Omega^{\prime\prime}$ and $T>0$,  $0\le s \le  t \le T$,
the mapping: \[\R^d \ni x \mapsto \phi_{s,t}(x) \in \R^d\] is  a surjective diffeomorphism of $C^1$-class and,
for any $\delta \in (0,\frac{\alpha}2+\beta-1)$,
the function
\[
\mathbb{R}^d \ni x \mapsto D \phi_{s,t}(x)   \in \mathscr{L}(\R^d,\R^d) \]
is locally $\delta$-H\"older continuous  uniformly in $(s,t) \in [0,T] \times [0,T]$.

Moreover,  for every  $\omega \in \Omega^{\prime\prime}$ and  for any $\delta \in (0,\frac{\alpha}2+\beta-1)$,
the  function
\[ \mathbb{R}^d \ni x \mapsto  D (\phi_{s,t}^{-1} )(x) \in \mathscr{L}(\R^d,\R^d) \]
  is     locally $\delta$-H\"older continuous  uniformly in $(s,t) \in [0,T] \times [0,T]$.

  Finally, for every  $\omega \in \Omega^{\prime\prime}$  and for every $x \in \R^d, $  the functions
  \begin{align}
  & [0,T]\times [0,T] \ni (s,t) \mapsto D \phi_{s,t}(x)
   \\&
    [0,T]\times [0,T] \ni (s,t) \mapsto  D (\phi_{s,t}^{-1} )(x)
  \end{align}
 are separately c\`adl\`ag, i.e., for each $s \in [0,T]$,
  \[ [0,T] \ni t \mapsto  D \phi_{s,t}(x) \in \R^d \mbox{ is c\`adl\`ag}\] and  for each $t \in [0,T]$, the map
  \[ [0,T] \ni s \mapsto  D \phi_{s,t}(x)\in \R^d  \mbox{ is c\`adl\`ag}.\]
  The same holds for $ D (\phi_{s,t}^{-1} )(x)$.
\end{theorem}
{
\begin{remark}\label{rem-ext-infinite-time}
 By applying an argument analogous to that used in Lemma \ref{ef}, one can extend the above result to the infinite time interval and construct a universal $\mathbb{P}$-full set (independent of $T$) on which the flow $\phi_{s,t}(x), s,t \in[0, \infty)$, is a $C^1$-diffeomorphism whose derivatives are locally $\delta$-H\"older continuous.
\end{remark}
}
%-------------------------

\begin{proof}
 Consider a  function
 \[
 \psi: \mathbb{R}^d \ni x \mapsto x+u(x) \in \mathbb{R}^d.
 \]
 Here $u=u_{\lambda_0}$ and $\lambda_0$ appear above \eqref{eq P29 star}.
According to \eqref{eqn-Ito-iv}, we have
 \begin{align}\label{eq phi-y-1}
 \psi(\phi_{t}(x))=Y_{t}(\psi(x)),\quad x\in \R^d,\; t\geq0,
 \end{align}
where
\begin{align}\label{eq Yy-1}
     Y_{t}(y)=&y+\int_{0}^{t}\varpi{}_{} u(\psi^{-1}(Y_{s}(y))) \,d s\\
     &+\int_{0}^{t}\lint_{\R^{d}}u(\psi^{-1}(Y_{s-}(y))+z)-u(\psi^{-1}(Y_{s-}(y)))\tilde{\mu}(d s,dz)+L_{t}.
\end{align}
Some simple calculations lead to
\begin{align}\label{eq-Dphi-nonsingular}
[(D\psi)(\phi_t(x)) ] D\phi_t(x)=& D\psi(x)+\int_0^t \lambda_{0} Du(\phi_s(x)) D\phi_s(x) ds\\
&+\int_0^t \lint_{\R^{d}}\big[ Du(\phi_{s-}(x)+z)-Du(\phi_{s-}(x))\big] D\phi_{s-}(x)\tilde{\mu}(ds,dz),
\end{align}
which is comparable with equality (21) in \cite{FGP_2010-Inventiones}.

Before proceeding with the proof, we present several results and facts, derived from Theorem \ref{thm-approximation}, collected from \eqref{eq P29 star}, and supplemented by some simple calculations, that will be  used in the subsequent analysis:
\begin{itemize}
    \item[(F1)]  $D\psi(y)=I+Du(y) $, $\forall y\in \mathbb{R}^d$.

    \item[(F2)] $\Vert D u\Vert_0\leq\frac13$.

    \item[(F3)] $\Vert u\Vert_{\alpha+\beta}<\infty$.

    \item[(F4)] $\frac23 \leq \| D\psi\|_0=\|I+Du \|_0\leq \frac43$.

    \item[(F5)] $\frac34\leq  \| (D\psi)^{-1}\|_0=  \| (I+Du)^{-1}\|_0 \leq \frac{3}2$.

    \item[(F6)] By the definition of $\Vert Du\Vert_{\gamma}$ and (F3), for any $s\geq0$ and $x\in\mathbb{R}^d$,
\begin{align}
\Vert D u(\phi_s(x)  +  z)  - D u(\phi_s(x))\Vert\leq c(|z\vert^{\gamma}\1_{B}(z)+\1_{B^c}(z)),\ \ \ \forall z\in \R^{d}\setminus\{0\}.
\end{align}
Here and in the following $\gamma=\alpha+\beta-1$. Keep in mind that \eqref{eqn-beta+alpha half} implies that $2\gamma>\alpha$.

\item[(F7)] For any $s\geq0$ and $x,y\in\mathbb{R}^d$,
\begin{equation}
    \phi_s(x) -\phi_s(y)=\int_0^1 D\phi_s(y+\theta (x-y))(x-y)d\theta.
\end{equation}

\item[(F8)] For any $d\times d$ real invertible matrices $A$ and $B$,
\begin{equation}\label{eq matrix 20241130}
A^{-1}-B^{-1}=A^{-1}(B-A)B^{-1}.
\end{equation}

\item[(F9)] For any $p>\alpha$,  $\lint_B|z|^{p}\nu(dz)+\nu(B^c)<\infty$.

\item[(F10)] By (F1) (F3) (F5) and (F8), for any $x,y\in\mathbb{R}^d$,
\begin{equation}
    \|(D\psi(x))^{-1}-(D\psi(y))^{-1}\|
    \leq
    \frac{9}{4}\|u\|_{\alpha+\beta}|x-y|^{\gamma}
    \leq
    c|x-y|^{\gamma}.
\end{equation}

\end{itemize}
Consider the $\mathscr{L}(\mathbb{R}^d, \mathbb{R}^d)$-valued solution to the following  SDE, for every fixed $x \in \mathbb{R}^d$,
\begin{align}
  & K_t(x)=(D\psi(x))^{-1} -\lambda_{0}\int_0^t K_s(x)Du(\phi_s(x))((D\psi)(\phi_s(x)) )^{-1} \,d s\label{eq-1126-2-K}\\
   &-\int_0^t\lint_{\R^d} K_{s-}(x)  \big[ Du(\phi_{s-}(x)+z)-Du(\phi_{s-}(x))\big] ((D\psi)(\phi_{s-}(x)) )^{-1} \tilde{\mu}(ds,dz)\nonumber\ \\
   &+\int_0^t\lint_{\R^d} K_{s-}(x) \big[ Du(\phi_{s-}(x)+z)-Du(\phi_{s-}(x))\big]\nonumber\\
   &\quad\quad\quad\times((D\psi)(\phi_{s-}(x)) )^{-1} \big[ Du(\phi_{s-}(x)+z)-Du(\phi_{s-}(x))\big] ((D\psi)(\phi_{s-}(x)+z) )^{-1} \mu(ds,dz).\nonumber
\end{align}
By (F2), (F5), (F6) and (F9), classical results imply that there exist unique solutions to the above SDEs. Using the It\^o   product formula, c.f. \cite[Theorem 4.4.13]{Applebaum_2009}, we have
\begin{align*}
   &  d\big(K_t(x)\big( (D\psi)(\phi_t(x))  D\phi_t(x)\big) \big)\\
    &=K_{t-}(x)  d\big((D\psi)(\phi_t(x))  D\phi_t(x) \big)+(dK_t(x)) (D\psi)(\phi_{t-}(x))  D\phi_{t-}(x) 
     + d \big[K_\cdot(x),  (D\psi)(\phi_\cdot(x))  D\phi_\cdot(x)\big)\big]_t\\
    & = \lambda_{0}  K_{t}(x)  Du(\phi_t(x)) D\phi_t(x) dt\\
&\quad+\lint_{\R^{d}}K_{t-}(x)\big[ Du(\phi_{t-}(x)+z)-Du(\phi_{t-}(x))\big] D\phi_{t-}(x)\tilde{\mu}(dt,dz)\\
&-\lambda_{0}K_{t}(x)Du(\phi_t(x))((D\psi)(\phi_t(x)) )^{-1}(D\psi)(\phi_t(x))  D\phi_t(x) \,d t\\
   &\quad-\lint_{\R^d} K_{t-}(x)  \big[ Du(\phi_{t-}(x)+z)-Du(\phi_{t-}(x))\big] ((D\psi)(\phi_{t-}(x)) )^{-1}(D\psi)(\phi_{t-}(x))  D\phi_{t-}(x) \tilde{\mu}(dt,dz) \\
   &\quad+\lint_{\R^d} K_{t-}(x) \big[ Du(\phi_{t-}(x)+z)-Du(\phi_{t-}(x))\big]((D\psi)(\phi_{t-}(x)) )^{-1} \big[ Du(\phi_{t-}(x)+z)-Du(\phi_{t-}(x))\big] \\
   &\quad\quad\quad ((D\psi)(\phi_{t-}(x)+z) )^{-1}(D\psi)(\phi_{t-}(x))  D\phi_{t-}(x)  \; \mu(dt,dz)\\
   &\quad-\lint_{\R^d} K_{t-}(x)  \big[ Du(\phi_{t-}(x)+z)-Du(\phi_{t-}(x))\big] ((D\psi)(\phi_{t-}(x)) )^{-1}\big[ Du(\phi_{t-}(x)+z)-Du(\phi_{t-}(x))\big]\\
   &\quad\quad\quad \ D\phi_{t-}(x)\; \mu(dt,dz)\\
   &\quad+\lint_{\R^{d}} K_{t-}(x) \big[ Du(\phi_{t-}(x)+z)-Du(\phi_{t-}(x))\big]((D\psi)(\phi_{t-}(x)) )^{-1} \big[ Du(\phi_{t-}(x)+z)-Du(\phi_{t-}(x))\big] \\
   &\quad\quad\quad ((D\psi)(\phi_{t-}(x)+z) )^{-1} \big[ Du(\phi_{t-}(x)+z)-Du(\phi_{t-}(x))\big] D\phi_{t-}(x)\; \mu(dt,dz)\\
   &=0,
\end{align*}
where we applied (F1) to cancel out the last three terms.
This proves that for any $x\in\R^d$, $\mathbb{P}$-a.s.,
\begin{align}
    K_t(x) [(D\psi)(\phi_t(x)) ] D\phi_t(x)=I,\quad \forall t\in[0,T].\label{inverse-identity-eq-101}
\end{align}
Now we shall prove that $x\mapsto K_t(x)$ is continuous uniformly in $t\in[0,T]$, $\mathbb{P}$-a.s.
 Prior to that, we need to prepare some preliminary results. In particular, there exists some $M_{p,T}>0$ such that  
\begin{align}
    \sup_{x\in\R^d}\E\sup_{t\in[0,T]}\|K_t(x)\|^p \leq M_{p,T}.\label{eq-bounded-kt-1}
\end{align}

By using (F2), (F5), (F6) and applying the Burkholder inequality twice (c.f. \cite[Theorem 2.11]{Kunita_2004}), we can derive that for any $x\in\R^d$, $p\geq 2$ and $T>0$,
\begin{align*}
 & \E\sup_{t\in[0,T]}\|K_t(x)\|^p \\
 &\leq C_p\big\| (D\psi)^{-1}\big\|_0^p+
  C_p  \E\sup_{t\in[0,T]}\Big\| \int_0^t \lambda_{0} K_s(x)Du(\phi_s(x))[(D\psi)(\phi_s(x)) ]^{-1} \,d s\Big\|^p\\
  &\quad + C_p\E\sup_{t\in[0,T]}\Big\| \int_0^t \lint_{\R^d} K_{s-}(x)  \big[ Du(\phi_{s-}(x)+z)-Du(\phi_{s-}(x))\big] [(D\psi)(\phi_{s-}(x)) ]^{-1} \tilde{\mu}(ds,dz) \Big\|^p\\
 &\quad + C_p\E\sup_{t\in[0,T]}\Big\|\int_0^t\lint_{\R^d} K_{s-}(x) \big[ Du(\phi_{s-}(x)+z)-Du(\phi_{s-}(x))\big][(D\psi)(\phi_{s-}(x)) ]^{-1}\\
 &\quad\quad\quad\quad\quad\quad\times\big[ Du(\phi_{s-}(x)+z)-Du(\phi_{s-}(x))\big][(D\psi)(\phi_{s-}(x)+z) ]^{-1} \mu(ds,dz)  \Big\|^p\\
  &\leq C_p\big\| (D\psi)^{-1}\big\|_0^p+  C_p\lambda_{0}^p T^{p-1} \| Du\|_0^p\big\| (D\psi)^{-1}\big\|_0^p \int_0^T  \E\sup_{s\in[0,t]}\|K_s(x)\|^p \,d t \\
 &\quad +C_p\E \Big(\int_0^T \lint_{\R^d}\Big\| K_s(x)  \big[ Du(\phi_{s}(x)+z)-Du(\phi_{s}(x))\big] [(D\psi)(\phi_s(x)) ]^{-1}\Big\|^2 \nu(dz)ds\Big)^{\frac{p}2}\\
 &\quad+C_p\E \Big(\int_0^T \lint_{\R^d}\Big\| K_s(x)  \big[ Du(\phi_{s}(x)+z)-Du(\phi_{s}(x))\big] [(D\psi)(\phi_s(x)) ]^{-1}\Big\|^p \nu(dz)ds\Big)\\
 &\quad+C_p\E\Big(\int_0^T  \lint_{\R^d} \|K_{s-}(x)\|  \| u\|^2_{\alpha+\beta}(\1_{B}(z)|z\vert^{2\gamma}+\1_{B^c}(z)) \big\| (D\psi)^{-1}\big\|_0^2\, \mu(ds,dz)\Big)^p\\
  &\leq C_p\big\| (D\psi)^{-1}\big\|_0^p+ C_p \lambda_{0}^p T^{p-1} \| Du\|_0^p\big\|(D\psi)^{-1}\big\|_0^p \int_0^T  \E\sup_{s\in[0,t]}\|K_s(x)\|^p \,d t \\
 &\quad+ C_p\|u\|^p_{\alpha+\beta}\big\| (D\psi)^{-1}\big\|_0^pT^{\frac{p-2}{2}}   \int_0^T \E\sup_{0\leq s\leq t}\| K_s(x)\|^p d t\Big(  \lint_B|z\vert^{2\gamma}\nu(dz) +\nu(B^c) \Big)^{\frac{p}2}\\
  &\quad+ C_p\|u\|^p_{\alpha+\beta}\big\| (D\psi)^{-1}\big\|_0^p  \int_0^T \E\sup_{0\leq s\leq t}\| K_s(x)\|^p d t\Big(  \lint_B|z\vert^{p\gamma}\nu(dz) +\nu(B^c) \Big)\\
  &\quad+C_p \| u\|^{2p}_{\alpha+\beta} \big\| (D\psi)^{-1}\big\|_0^{2p}\;\E\Big| \int_0^T  \lint_{\R^d} \|K_{s-}(x)\| (\1_{B}(z)|z\vert^{2\gamma}+\1_{B^c}(z)) \tilde{\mu}(ds,dz)\Big|^{p}\\
  &\quad+C_p \| u\|^{2p}_{\alpha+\beta} \big\| (D\psi)^{-1}\big\|_0^{2p}\;\E\Big|\int_0^T  \lint_{\R^d} \|K_{s}(x)\| (\1_{B}(z)|z\vert^{2\gamma}+\1_{B^c}(z))\, \nu(dz)ds\Big|^p\\
  & \leq C_p\big\| (D\psi)^{-1}\big\|_0^p+ C_p \lambda_{0}^p T^{p-1} \| Du\|_0^p\big\|(D\psi)^{-1}\big\|_0^p \int_0^T  \E\sup_{s\in[0,t]}\|K_s(x)\|^p \,d t \\
& \quad+ C_p\|u\|^p_{\alpha+\beta}\big\| (D\psi)^{-1}\big\|_0^pT^{\frac{p-2}{2}}   \int_0^T \E\sup_{0\leq s\leq t}\| K_s(x)\|^p d t\Big(  \lint_B|z\vert^{2\gamma}\nu(dz) +\nu(B^c) \Big)^{\frac{p}2}\\
  & \quad+ C_p\|u\|^p_{\alpha+\beta}\big\| (D\psi)^{-1}\big\|_0^p  \int_0^T \E\sup_{0\leq s\leq t}\| K_s(x)\|^p d t\Big(  \lint_B|z\vert^{p\gamma}\nu(dz) +\nu(B^c) \Big)\\
 & \quad+ C_p\|u\|^{2p}_{\alpha+\beta}\| \big\| (D\psi))^{-1}\big\|_0^{2p} M_T  \int_0^T \E\sup_{0\leq s\leq t}\| K_s(x)\|^p d t,
\end{align*}
where $M_T:=T^{\frac{p-2}2}\Big(  \lint_B|z\vert^{4\gamma}\nu(dz) +\nu(B^c) \Big)^\frac{p}2+\Big(  \lint_B|z\vert^{2p\gamma}\nu(dz) +\nu(B^c) \Big)+T^{p-1}\Big(  \lint_B|z\vert^{2\gamma}\nu(dz) +\nu(B^c) \Big)^p<\infty$ by (F6) and (F9).
By the facts listed above and using the Gronwall inequality, we obtain \eqref{eq-bounded-kt-1}. 

Applying \eqref{eq-bounded-kt-1}, Corollary \ref{cor-s11}, (F7), and the H\"older inequality,
for any $q,r,p>0$ with $p\geq1$ and $rp\geq1$, there exists some $C_{q,r,p,T}>0$ such that for any $x,y\in\R^d$
\begin{align}\label{eq 2024 1205}
    \E\Big(\sup_{t\in[0,T]}\|K_t(x)\|^q\Big(\int_0^T|\phi_s(x)-\phi_s(y)|^rds\Big)^p\Big)
    \leq
    C_{q,r,p,T}|x-y|^{rp}.
\end{align}

Now it is in a position to  prove that $x\mapsto K_t(x)$ is continuous uniformly in $t\in[0,T]$, $\mathbb{P}$-a.s. 
Let $x,y\in\R^d$. Consider
\begin{align}\label{eq 20241223 06}
   K_t(x)-K_t(y)=&(D\psi(x))^{-1}-(D\psi(y))^{-1}-\lambda_{0}\int_0^t \big(K_s(x)-K_s(y)\big)Du(\phi_s(x))((D\psi)(\phi_s(x)) )^{-1} ds\nonumber\\
   &-\lambda_{0}\int_0^tK_s(y)\big(Du(\phi_s(x))((D\psi)(\phi_s(x)) )^{-1}-Du(\phi_s(y))((D\psi)(\phi_s(y)) )^{-1}\big) ds\nonumber\\
   &-\int_0^t\lint_{\R^d}\big( K_{s-}(x) -K_{s-}(y)\big)G(s-,x,z)  \tilde{\mu}(ds,dz) \\
     &-\int_0^t\lint_{\R^d} K_{s-}(y)\big(G(s-,x,z)-G(s-,y,z)\big)  \tilde{\mu}(ds,dz) \\
   &+\int_0^t\lint_{\R^d} \big( K_{s-}(x) -K_{s-}(y)\big)  H(s-,x,z) \mu(ds,dz)\\
   &+\int_0^t\lint_{\R^d} K_{s-}(y)\big(H(s-,x,z)-H(s-,y,z)\big)\mu(ds,dz)\\
   :=&I^0(x,y)+I^1_t(x,y)+I^2_t(x,y)+I^3_t(x,y)+I^4_t(x,y)+I^5_t(x,y)+I^6_t(x,y),
\end{align}
where
$$
I^0(x,y):=(D\psi(x))^{-1}-(D\psi(y))^{-1},
$$
$$G(s,x,z):=\big[ Du(\phi_{s}(x)+z)-Du(\phi_{s}(x))\big] ((D\psi)(\phi_s(x)) )^{-1},$$ and
\begin{align*}
H(s,x,z):=&\big[ Du(\phi_{s}(x)+z)-Du(\phi_{s}(x))\big]((D\psi)(\phi_{s}(x)) )^{-1} \big[ Du(\phi_{s}(x)+z)-Du(\phi_{s}(x))\big] \\
&((D\psi)(\phi_{s}(x)+z) )^{-1}.
\end{align*}
By using (F2) and (F5), we infer, for any $p\geq2$,
\begin{align}\label{eq 20241223 I1}
        \mathbb{E}\sup_{0\leq t\leq T}\|I^1_t(x,y)\|^p\leq C_pT^{p-1}\lambda_{0}^p \int_0^T \mathbb{E} \sup_{0\leq t\leq s} \|K_t(x)-K_t(y)\|^p ds.
\end{align}
By the definition of $[Du]_{\gamma}$, (F1), (F3), and (F8),
\begin{align}
      \|Du(\phi_s(x))-Du(\phi_s(y))\|\leq c|\phi_s(x) -\phi_s(y)|^{\gamma},\label{inver-theo-holder-eq-101}
\end{align}
and
\begin{align}
       &\| ((D\psi)(\phi_s(x) ))^{-1} -((D\psi)(\phi_s(y)) )^{-1}   \|\\
       &= \| ((D\psi)(\phi_s(x)) )^{-1} \big((D\psi)(\phi_s(y)) -(D\psi)(\phi_s(x)) \big) [(D\psi)(\phi_s(y)) ]^{-1}\nonumber \|\\
       &\leq \frac94 \|Du(\phi_s(x)) - Du(\phi_s(y)) \|\nonumber \\
       &\leq c|\phi_s(x) -\phi_s(y)|^{\gamma}.\label{inver-theo-holder-eq-102}
\end{align}

By \eqref{eqn-4.16} and (F3), we can deduce that for any $ z\in \R^{d}\setminus\{0\}$
\begin{align}
      &\|Du(\phi_s(x)+z)-Du(\phi_s(x)) - Du(\phi_s(y)+z)+Du(\phi_s(y))\|\nonumber\\
      &\leq c[Du]_{\gamma} |\phi_s(x)-\phi_s(y)|^{\delta}|z|^{\gamma^\prime}\1_{B}(z)+c[Du]_{\gamma}|\phi_s(x)-\phi_s(y)|^{\gamma}\1_{B^c}(z)\nonumber\\
      &\leq c |\phi_s(x)-\phi_s(y)|^{\delta}|z|^{\gamma^\prime}\1_{B}(z)+c|\phi_s(x)-\phi_s(y)|^{\gamma}\1_{B^c}(z),\label{inver-theo-holder-eq-103}
      \end{align}
with $\delta\in(0,\frac{\alpha}2+\beta-1)$ and $\gamma^\prime:=\gamma-\delta\in (\frac{\alpha}2, \gamma)$.
It follows from (F5),  (F6), \eqref{inver-theo-holder-eq-102}, and \eqref{inver-theo-holder-eq-103}  that  for any $ z\in \R^{d}\setminus\{0\}$
\begin{align}\label{eq 20241222 01}
\|G(s,x,z)\|\leq  c\Big(|z|^{\gamma}\1_{B}(z)+\1_{B^c}(z)\Big),
\end{align}
and
\begin{align}
&\|G(s,x,z)-G(s,y,z)\|\nonumber\\
&\leq \|Du(\phi_s(x)+z)-Du(\phi_s(x)) - Du(\phi_s(y)+z)+Du(\phi_s(y))\|\|((D\psi)(\phi_s(x)))^{-1}\|\nonumber\\
&\quad + \| Du(\phi_s(y)+z)-Du(\phi_s(y))\|\|[(D\psi)(\phi_s(x))]^{-1}-((D\psi)(\phi_s(y)))^{-1}\|\label{Dphi-inverse-eq-201}\\
&\leq c |\phi_s(x)-\phi_s(y)|^{\delta}|z|^{\gamma^\prime}\1_{B}(z)+c|\phi_s(x)-\phi_s(y)|^{\gamma}\1_{B^c}(z)
+c |\phi_s(x)-\phi_s(y)|^{\gamma}|z|^{\gamma}\1_{B}(z).\nonumber
\end{align}
Similarly, for any $ z\in \R^{d}\setminus\{0\}$, we have
\begin{align}\label{eq 20241223 01}
\|H(s,x,z)\|
\leq c|z|^{2\gamma} \1_{B}(z)+c\1_{B^c}(z),
\end{align}
and
\begin{align}\label{eq 20241223 02}
&\|H(s,x,z)-H(s,y,z)\|\\
&\leq c\|Du(\phi_s(x)+z)-Du(\phi_s(x)) - Du(\phi_s(y)+z)+Du(\phi_s(y))\|\|Du(\phi_s(x)+z)-Du(\phi_s(x))\|\\
&+c\|Du(\phi_s(y)+z)-Du(\phi_s(y))\|\| ((D\psi)(\phi_s(x) )^{-1} -((D\psi)(\phi_s(y)) )^{-1}   \|\|Du(\phi_s(x)+z)-Du(\phi_s(x))\|\\
&+c\|Du(\phi_s(y)+z)-Du(\phi_s(y))\|\|Du(\phi_s(x)+z)-Du(\phi_s(x)) - Du(\phi_s(y)+z)+Du(\phi_s(y))\|\\
&+c\|Du(\phi_s(y)+z)-Du(\phi_s(y))\|^2\| [((D\psi)(\phi_s(x)+z ))^{-1} -((D\psi)(\phi_s(y)+z) )^{-1}   \|\\
&\leq c |\phi_s(x)-\phi_s(y)|^{\delta}|z|^{\gamma^\prime+\gamma}\1_{B}(z)+c|\phi_s(x)-\phi_s(y)|^{\gamma}\1_{B^c}(z)
+c |\phi_s(x)-\phi_s(y)|^{\gamma}|z|^{2\gamma}\1_{B}(z).
\end{align}

Applying \eqref{inver-theo-holder-eq-101}, \eqref{inver-theo-holder-eq-102} and \eqref{eq 2024 1205}, we have, for any $p\geq \frac{1}{\gamma}\vee 2$,
\begin{align}\label{eq 20241223 I2}
        &\mathbb{E}\sup_{0\leq t\leq T}\|I^2_t(x,y)\|^p\\
       & \leq C_p \mathbb{E}\Big[\!\sup_{0\leq s\leq T} \!\!\|K_s(y)\|^p 
        \Big(\!\int_0^T\!\!\!   \big\| Du(\phi_s(x))-Du(\phi_s(y)) \big\|\!+\!\big\|((D\psi)(\phi_s(x)) )^{-1}\!-\!((D\psi)(\phi_s(y)) )^{-1}\big\| ds\Big)^p\Big]\\
        &\leq  C_p \mathbb{E}\Big(\sup_{0\leq s\leq T}\|K_s(y)\|^p\Big(\int_0^T |\phi_s(x)-\phi_s(y)|^\gamma ds\Big)^p\Big)
        \leq C_{p,\gamma,T}|x-y|^{\gamma p}
        .
\end{align}
Using the $\rm It\hat{o}$ formula and \eqref{eq 20241222 01}, we infer
\begin{align}\label{eq 20241223 I3}
        \mathbb{E}\sup_{0\leq t\leq T}\|I^3_t(x,y)\|^p\leq& C_p\E \Big( \int_0^T\lint_{\R^d}\big\| K_s(x) -K_s(y)\big\|^2 \|G(s,x,z)\|^2  \nu(dz)ds\Big)^{\frac{p}{2}}\\
        &+
        C_p\E \Big( \int_0^T\lint_{\R^d}\big\| K_s(x) -K_s(y)\big\|^p \|G(s,x,z)\|^p  \nu(dz)ds\Big)\\
        \leq& C_{p,T} \mathbb{E} \int_0^T\sup_{0\leq s\leq t} \|K_s(x)-K_s(y)\|^p  dt,
\end{align}
where $C_{p,T}= C_pT^{\frac{p-2}{2}}\big(\lint_{B}|z|^{2\gamma} \nu(dz)+\nu(B^c)\big)^{\frac{p}{2}}+C_p\big(\lint_{B}|z|^{p\gamma} \nu(dz)+\nu(B^c)\big)<\infty$ by (F6) and (F9).

For the fourth term $I^4$, applying (F9), \eqref{Dphi-inverse-eq-201} and \eqref{eq 2024 1205}, we have, for any $p\geq\frac{1}{\delta}\vee\frac{1}{\gamma}\vee 2$,
\begin{align}\label{eq 20241223 I4}
        \mathbb{E}\sup_{0\leq t\leq T}\|I^4_t(x,y)\|^p
        &\leq C_p \E \Big( \int_0^T\lint_{\R^d}\big\| K_s(y)\big\|^2 \|G(s,x,z)-G(s,y,z)\|^2  \nu(dz)ds\Big)^{\frac{p}{2}}\nonumber \\
        &\quad +
        C_p \E \Big( \int_0^T\lint_{\R^d}\big\| K_s(y)\big\|^p \|G(s,x,z)-G(s,y,z)\|^p  \nu(dz)ds\Big)\nonumber \\
        &\leq C_p\Big(\lint_B|z|^{2\gamma^\prime}+|z|^{2\gamma}\nu(dz)+\nu(B^c)\Big)^{\frac{p}{2}}
        \\
   &\hspace{1truecm}   \cdot  \mathbb{E}\Big(\sup_{0\leq s\leq T}\|K_s(y)\|^p\Big(\int_0^T |\phi_s(x)-\phi_s(y)|^{2\delta}+|\phi_s(x)-\phi_s(y)|^{2\gamma} ds\Big)^{\frac{p}{2}}\Big)\nonumber \\
        &\quad+
C_p\Big(\lint_B|z|^{p\gamma^\prime}+|z|^{p\gamma}\nu(dz)+\nu(B^c)\Big)
\\&
\hspace{1truecm}\cdot \mathbb{E}\Big(\sup_{0\leq s\leq T}\|K_s(y)\|^p\Big(\int_0^T |\phi_s(x)-\phi_s(y)|^{p\delta}+|\phi_s(x)-\phi_s(y)|^{p\gamma} ds\Big)\Big)\nonumber \\
        &\leq C_{p,\delta,T,\gamma,\gamma^\prime}\Big(
        |x-y|^{\delta p}+|x-y|^{\gamma p}
        \Big).
\end{align}

Using  the $\rm It\hat{o}$ formula and \eqref{eq 20241223 01},   we infer
\begin{align}\label{eq 20241223 I5}
        \mathbb{E}\sup_{0\leq t\leq T}\|I^5_t(x,y)\|^p\leq& C_{p,\gamma}\E \Big( \int_0^T\lint_{\R^d}\big\| K_s(x) -K_s(y)\big\|^2 \|H(s,x,z)\|^2  \nu(dz)ds\Big)^{\frac{p}{2}}\\
        &+C_{p,\gamma}\E \Big( \int_0^T\lint_{\R^d}\big\| K_s(x) -K_s(y)\big\|^p \|H(s,x,z)\|^p  \nu(dz)ds\Big)\\
        &+C_{p,\gamma}\E \Big( \int_0^T\lint_{\R^d}\big\| K_s(x) -K_s(y)\big\| \|H(s,x,z)\|  \nu(dz)ds\Big)^{p}\\
        \leq& C_{p,\gamma,T} \mathbb{E}\int_0^T \sup_{0\leq s\leq t} \|K_s(x)-K_s(y)\|^p  dt.
\end{align}
where \[
\begin{aligned}
C_{p,\gamma,T}=&C_p\Big[\Big(\lint_{B}|z|^{4\gamma}\nu(dz)+\nu(B^c)\Big)^{\frac{p}{2}}T^{\frac{p-2}{2}}
\\
&+\lint_{B}|z|^{2p\gamma}\nu(dz)+\nu(B^c)+\Big(\lint_{B}|z|^{2\gamma}\nu(dz)+\nu(B^c)\Big)^{p}T^{p-1}\Big]<\infty.
\end{aligned}
\]
By using  (F9), \eqref{eq 20241223 02}, and condition  $\gamma^\prime+\gamma>\alpha$, see \eqref{eqn-beta+alpha half}, we infer that, for any $p\geq \frac{1}{\delta}\vee\frac{1}{\gamma}\vee2$,  %{eqn-beta}%{eqn-beta+alpha half})  and \eqref{eq 2024 1205}.
\begin{align}\label{eq 20241223 I6}
        \mathbb{E}\sup_{0\leq t\leq T}\|I^6_t(x,y)\|^p
        \leq& C_p\E\Big( \int_0^T\lint_{\R^d}\big\|K_s(y)\big\|^2 \|H(s,x,z)-H(s,y,z)\|^2  \nu(dz)ds\Big)^{\frac{p}{2}}\\
        &+C_p\E \Big( \int_0^T\lint_{\R^d}\big\|K_s(y)\big\|^p \|H(s,x,z)-H(s,y,z)\|^p  \nu(dz)ds\Big)\\
        &+C_p\E \Big( \int_0^T\lint_{\R^d}\big\|K_s(y)\big\| \|H(s,x,z)-H(s,y,z)\|  \nu(dz)ds\Big)^{p}\\
        \leq& C_{p}\E\Big(
            \sup_{s\in[0,T]}\big\|K_s(y)\big\|^p
            \Big(
              \int_0^T |\phi_s(x)-\phi_s(y)|^{2\delta}+|\phi_s(x)-\phi_s(y)|^{2\gamma} ds
            \Big)^{\frac{p}{2}}
        \Big)\\
        &+
        C_{p}\E\Big(
            \sup_{s\in[0,T]}\big\|K_s(y)\big\|^p
            \Big(
              \int_0^T |\phi_s(x)-\phi_s(y)|^{p\delta}+|\phi_s(x)-\phi_s(y)|^{p\gamma} ds
            \Big)
        \Big)\\
        &+
        C_{p}\E\Big(
            \sup_{s\in[0,T]}\big\|K_s(y)\big\|^p
            \Big(
              \int_0^T |\phi_s(x)-\phi_s(y)|^{\delta}+|\phi_s(x)-\phi_s(y)|^{\gamma} ds
            \Big)^p
        \Big)\nonumber \\
        \leq& C_{p,\delta,T,\gamma,\gamma^\prime}\Big(
        |x-y|^{\delta p}+|x-y|^{\gamma p}
        \Big).
\end{align}
Here
\begin{align}
&C_{p,\delta,T,\gamma,\gamma^\prime}\\
=&C_p\Big[\Big(\lint_{B}|z|^{4\gamma}\nu(dz)+\lint_{B}|z|^{2(\gamma^\prime+\gamma)}\nu(dz)+\nu(B^c)\Big)^{\frac{p}{2}}T^{\frac{p-2}{2}}
\\
&+\lint_{B}|z|^{2p\gamma}\nu(dz)+\lint_{B}|z|^{p(\gamma^\prime+\gamma)}\nu(dz)+\nu(B^c)+\Big(\lint_{B}|z|^{2\gamma}\nu(dz)+\lint_{B}|z|^{\gamma^\prime+\gamma}\nu(dz)+\nu(B^c)\Big)^{p}T^{p-1}\Big]
\\<&\infty.
\end{align}

Combining the above estimates ((F10), \eqref{eq 20241223 06}, \eqref{eq 20241223 I1}, \eqref{eq 20241223 I2}, \eqref{eq 20241223 I3}, \eqref{eq 20241223 I4}, \eqref{eq 20241223 I5}, \eqref{eq 20241223 I6}) and applying the Gronwall inequality, we obtain, for any $p\geq2\vee\frac{1}{\delta}\vee\frac{1}{\gamma}$,
\begin{align*}
    \E\sup_{t\in[0,T]}\|K_t(x)-K_t(y)\|^p \leq &C_{p,\delta,T,\gamma,\gamma^\prime}\Big(
        |x-y|^{\delta p}+|x-y|^{\gamma p}, \;\; x ,y \in \mathbb{R}^d
        \Big).
\end{align*}
Note that $\delta<\gamma$ by \eqref{eqn-delta}. This fact plays an important r\^ole below.

In order to continue, let us recall useful notation. By $D_{\Skor}([0,T];\mathscr{L}(\mathbb{R}^d,\mathbb{R}^d))$  we denote the space $D([0,T];\mathscr{L}(\mathbb{R}^d,\mathbb{R}^d))$ of $\mathscr{L}(\mathbb{R}^d,\mathbb{R}^d))$-valued c\`adl\`ag functions endowed with  the Skorokhod metric while by
$ D_{\mathrm{sup}}([0,T];\mathscr{L}(\mathbb{R}^d,\mathbb{R}^d))$  we denote the same space endowed with the sup norm.

Next we apply the Kolmogorov criterion in the form of Theorem \ref{thm-KTC-ZB} and deduce that there exists a family
$\{K^\prime(x),  x\in\mathbb{R}^d\}$ of $D_{\Skor}([0,T];\mathscr{L}(\mathbb{R}^d,\mathbb{R}^d))$-valued random variables such that
\begin{trivlist}
\item[(o)] $K^\prime$ is a modification of $K$, i.e., for any $x\in\mathbb{R}^d$, $\mathbb{P}(K(x)\neq K^\prime(x))=0$,

\item[(i)] for every $\omega \in \Omega$, the function
\begin{equation}\label{eqn-cont-K'}
  \mathbb{R}^d \ni x \mapsto  K^\prime(x) \in D_{\mathrm{sup}}([0,T];\mathscr{L}(\mathbb{R}^d,\mathbb{R}^d))
\end{equation}
is continuous,
\item[(ii)] if
 $\delta^\prime < \delta-\frac{d}{p}$, then the function  $K^\prime$ from \eqref{eqn-cont-K'} is locally
 $\delta'$-H\"older continuous;
\item[(iii)] in particular, for every $\omega\in \Omega$ and every $R>0$,  there exists a constant $J_R(\omega)$, such that
\begin{equation}\label{eq 202412007 02}
 \sup_{|x|\leq R}\sup_{t\in[0,T]}\|K^\prime_t(x,\omega)\|
    \leq
    J_{R}(\omega).
\end{equation}
\end{trivlist}

Since $\delta$ can be any element in the open interval $(0,\frac{\alpha}{2}+\beta-1)$ and $p$ can be any sufficiently large constant, (ii) implies that 
\begin{trivlist}
\item[(ii')]
for any 
 $\delta\in (0,\frac{\alpha}{2}+\beta-1)$,  the function  $K^\prime$ from \eqref{eqn-cont-K'} is locally
 $\delta$-H\"older continuous. That is, for every $\omega\in \Omega$, $\delta\in (0,\frac{\alpha}{2}+\beta-1)$ and every $R>0$,  there exists a constant $J_{R,\delta}(\omega)$, such that
\begin{equation}\label{eq 202412007 02-2025}
 \sup_{|x|\vee |y|\leq R}\sup_{t\in[0,T]}\|K^\prime_t(x,\omega)-K^\prime_t(y,\omega)\|
    \leq
    J_{R,\delta}(\omega)|x-y|^\delta.
\end{equation}
\end{trivlist}
Therefore, in view of \eqref{inverse-identity-eq-101}, (iii) of Theorem \ref{d32} and (iii) of Theorem \ref{thm-ww}, there exists a $\mathbb{P}$-full measure set $\Omega^0$ independent of $x,t$ such that
\begin{equation} \label{che}
K^\prime_t(x) [(D\psi)(\phi_t(x)) ] D\phi_t(x)=I  \mbox{ for every } x\in\R^d, \;\;t\in[0,T] \mbox{ on } \Omega^0.
\end{equation}

%--------------------------

By using the Inverse Mapping Theorem, see e.g. \cite[Theorem 6.2.8]{Kran+Park}, \cite[Appendix A]{Robert+Vos},  we infer that  on $\Omega^0$, the function $\phi^{-1}_t$ is also a $C^1$-class map and
\begin{equation}\label{eq 202412007 03}
    D\phi^{-1}_t(x)=((D\phi_t)(\phi_t^{-1}(x)))^{-1}=K_t^\prime(\phi^{-1}_t(x))[D\psi(x)],\ \forall (x,t)\in\mathbb{R}^d\times[0,T].
\end{equation}
 By Theorems \ref{d32}, \ref{thm-homeomorphism} and \ref{thm-ww},  on some $\mathbb{P}$-full subset of $\Omega'$, for any $x \in {\mathbb R}^d$,
 \begin{itemize}
     \item the mapping $[0,T]\ni t \mapsto   \phi_t(x) \in \R^d$ is c\`adl\`ag,

     \item the mapping $[0,T]\ni t \mapsto \nxi_{t}^{-1}(x)\in \R^d$  is c\`adl\`ag,

     \item the mapping
\[[0,T] \times \R^d \ni (t,x) \mapsto
D \phi_t (x) \in \mathscr{L}(\R^d,\R^d) \]
 is continuous.
 \end{itemize}
    Let $\Omega''=\Omega'\cap\Omega^0$.
Hence, we infer that on $\Omega''$ for every $x\in\R^d$, $D\phi^{-1}_t(x)$ is c\`adl\`ag on $[0,T]$.

Now we prove that,   for every  $\omega \in \Omega''$ and  any $\delta \in (0,\frac{\alpha}2+\beta-1)$,
the  function
\[ \mathbb{R}^d \ni x \mapsto  D \phi_{t}^{-1} (x) \in \mathscr{L}(\R^d,\R^d) \]
  is     locally $\delta$-H\"older continuous  uniformly in $t \in [0,T]$.

Fix $\omega \in \Omega''$ and $N\in\mathbb{N}$. By Corollary \ref{cor-homeomorphism},  there exists a constant $M:=C_{\omega,N,T}$ such that, for any $|y|\leq N$,
$$
\sup_{t\in[0,T]}|\phi^{-1}_t(y)|\leq M.
$$

 Hence, by (ii'), \eqref{eq 202412007 02}, \eqref{eq 202412007 03}, and (F1)(F3)(F4), for any $t\in[0,T]$ and $x,y\in\R^d$ with
 $|x|\vee |y|\leq N$,
 \begin{align}\label{eq 20241223 holder 01}
     & \|D\phi_t^{-1}(x)-D\phi_t^{-1}(y)\|\\
     &=
     \|K_t^\prime(\phi^{-1}_t(x))[D\psi(x)]-K_t^\prime(\phi^{-1}_t(y))[D\psi(y)]\|\\
     &\leq
     C\|K_t^\prime(\phi^{-1}_t(x))-K_t^\prime(\phi^{-1}_t(y))\|
     +
     \|K_t^\prime(\phi^{-1}_t(y))\|\|[D\psi(x)]-[D\psi(y)]\|\\
     &\leq
    CJ_{M,\delta}|\phi^{-1}_t(x)-\phi^{-1}_t(y)|^\delta
     +
     J_M\|u\|_{\alpha+\beta}|x-y|^\gamma\\
     &\leq
CJ_{M,\delta}\Big(\int_0^1\|D\phi^{-1}_t(y+\theta(x-y))\|d\theta\Big)^\delta|x-y|^\delta
     +
     J_M\|u\|_{\alpha+\beta}|x-y|^\gamma\\
     &\leq
CJ_{M,\delta}\Big(\int_0^1\|K_t^\prime(\phi^{-1}_t(y+\theta(x-y)))[D\psi(y+\theta(x-y))]\|d\theta\Big)^\delta|x-y|^\delta
     +
     J_M\|u\|_{\alpha+\beta}|x-y|^\gamma\\
     &\leq
CJ_{M,\delta}J_M^\delta|x-y|^\delta
     +
     J_M\|u\|_{\alpha+\beta}|x-y|^\gamma,
 \end{align}
which proves that $D\phi_t^{-1}$ is  locally $\delta$-H\"older continuous uniformly in $t\in[0,T]$ on $\Omega''$.

Recall that from \eqref{xi} we have
 \begin{equation} \label{xi1-1}
   \phi_{s,t} = \phi_{t}\circ \phi_{s}^{-1},
\;\;\;\;
\phi_{s,t}^{-1} = \phi_{s}\circ \phi_{t}^{-1},\;\;\; 0 \le s \le t \le T.
 \end{equation}
By \eqref{xi1-1} we can differentiate:
 \begin{align}
 D \phi_{s,t}(x) &= D\phi_{t}\, (\phi_{s}^{-1}(x)) D\phi_{s}^{-1}(x),\;\;\; 0 \le s \le t \le T, \, x \in \R^d\label{eq 1-Dphi-st},\\
 D \phi_{s,t}^{-1}(x) &= D\phi_{s}\, (\phi_{t}^{-1}(x)) D\phi_{t}^{-1}(x),\;\;\; 0 \le s \le t \le T, \, x \in \R^d.\label{eq 2-Dphi-st}
\end{align}
By (iii) in Theorem \ref{thm-ww} and the above results on $D\phi_t^{-1}$, using similar arguments as proving \eqref{eq 20241223 holder 01}, $D \phi_{s,t}(x)$ and $D \phi_{s,t}^{-1}(x)$ are also locally $\delta$-H\"older continuous uniformly in $(s,t) \in [0,T] \times [0,T]$ on some $\mathbb{P}$-full event. From \eqref{eq 1-Dphi-st} and \eqref{eq 2-Dphi-st}, it is easy to conclude that for every $x \in \R^d, $  the functions
  \begin{align}
  & [0,T]\times [0,T] \ni (s,t) \mapsto D \phi_{s,t}(x)
   \\&
    [0,T]\times [0,T] \ni (s,t) \mapsto  D \phi_{s,t}^{-1} (x)
  \end{align}
 are separately c\`adl\`ag.
\end{proof}

\section{Stability}\label{sec-stability}
In this section, we use all the notations and assumptions introduced in the previous Section \ref{sec-regular flow}, and let  $T>0$ be any fixed constant.

Below we denote by  $\Vert\cdot\Vert$  the Hilbert-Schmidt norm on $d\times d$ matrices.  In the next result we  recall  the assumptions of Theorem \ref{d32}.

\begin{theorem}\label{thm-stability}
Assume that  $\alpha \in  [1,2)$ or $\alpha \in (0,1)$. Assume that $L = (L_t)_{t\geq 0}$ is a L\'evy  process  satisfying  Hypotheses \ref{hyp-nondeg1} in the former case or Hypothesis \ref{hyp-nondeg3} in the latter case.  Assume that $\beta\in(0,1)$ satisfies condition \eqref{eqn-beta+alpha half} and $b\in C_{\mathrm{b}}^{\beta}(\mathbb{R}^d,\mathbb{R}^d)$.
  Let $(b^{n})_{n=1}^\infty\subset C_{\mathrm{b}}^{\beta}(\mathbb{R}^d,\mathbb{R}^d)$
be a sequence of vector fields and $\phi^{n}$ the corresponding
stochastic flows. 
 Let $b^{n}\to b$ in $C_{\mathrm{b}}^{\beta}(\mathbb{R}^d,\mathbb{R}^d)$.

Then, for every $p\geq 1$ and any compact set $K \subset \R^d$, we have:
\begin{align}
\label{stability1}
&\lim_{n\to\infty}\sup_{x\in{\mathbb{R}}^{d}} \sup_{s\in [0,T]}
\mathbb{E}[ \sup_{t \in [s,T]} |\phi_{s,t}^{n}(x)-\phi_{s,t}(x)\vert^{p}]=0.\\
\label{stability2-uniform}
&\sup_{n\in\mathbb{N}}\sup_{x\in{\mathbb{R}}^{d}}\sup_{s\in [0,T]} \mathbb{E}[  \sup_{t \in [s,T]} \Vert D\phi_{s,t}^{n}(x)\Vert^{p}]
+
\sup_{x\in{\mathbb{R}}^{d}}\sup_{s\in [0,T]} \mathbb{E}[  \sup_{t \in [s,T]} \Vert D\phi_{s,t}(x)\Vert^{p}]
<\infty.\\
\label{stability2}
&   \,  \lim_{n\to\infty}\sup_{x\in K}\sup_{s\in [0,T]} \mathbb{E}[  \sup_{t \in [s,T]} \Vert D\phi_{s,t}^{n}(x)-D\phi_{s,t}(x)\Vert^{p}]=0.
\end{align}
\end{theorem}
\begin{remark} \label{okk} The previous result remains valid with the same proof under the following weaker assumption. \\
    Assume that numbers  $0 <\beta^\prime \le \beta$  satisfy 
\begin{equation}
\frac{\alpha}{2} + \beta^\prime >1.    
\end{equation}
  Assume also that   $(b^{n})_{n=1}^\infty$ is 
 a sequence of vector fields from the set  $C_{\mathrm{b}}^{\beta^\prime}(\mathbb{R}^d,\mathbb{R}^d)$
such that $b^{n}\to b$ in $C_{\mathrm{b}}^{\beta^\prime}(\mathbb{R}^d,\mathbb{R}^d)$. 
 \end{remark}

Before we embark on the proof of the above important result, let us  first formulate the following
essential corollary in which we use notation $B_M=\{y \in \R^d: \vert y|\leq M\}$.

\begin{corollary}\label{cor-civuole}Under the same assumptions as in Theorem \ref{thm-stability}.
There exists a subsequence of the sequence $\phi^{n}$,   still denoted by $\phi^{n}$, and a  $\mathbb{P}$-full set $\Omega' \subset \Omega$    such that for all $\omega \in \Omega'$, $p\geq1$ and $M>0$,
\begin{equation}
\label{stability-final-000}
\lim_{n\to\infty} \int_{B_M}\sup_{t \in [0,T] } \, \Bigl[ |\phi_{0,t}^{n}(x)-\phi_{0,t}(x)\vert^p +  \Vert D\phi_{0,t}^{n}(x)-D\phi_{0,t}(x) \Vert^p \Bigr]dx=0,
\end{equation}
and, for any $n\in\mathbb{N}$,
\begin{equation}
\label{stability-final-0001}
 \int_{B_M}\sup_{t \in [0,T] } \, \Bigl[\Vert D\phi_{0,t}^{n}(x) \Vert^p+\Vert  D\phi_{0,t}(x) \Vert^p \Bigr]dx<\infty.
\end{equation}

There exists a subsequence of the sequence $\phi^{n}$,   still denoted by $\phi^{n}$, and a  $\Leb\otimes \mathbb{P}$-full set $A \subset \R^d\times \Omega$    such that for all $(x,\omega) \in A$,
\begin{equation}
\label{stability-final}
\lim_{n\to\infty} \sup_{t \in [0,T] } \, \Bigl[ |\phi_{0,t}^{n}(x)-\phi_{0,t}(x)| +  \Vert D\phi_{0,t}^{n}(x)-D\phi_{0,t}(x) \Vert \Bigr]=0.
\end{equation}

\end{corollary}
\begin{proof}[Proof of Corollary \ref{cor-civuole}] This follows immediately from the Fubini Theorem and Theorem \ref{thm-stability}.

\end{proof}

\begin{remark}\label{rem Sec 5 000}
According to Theorem \ref{thm-homeomorphism}, the inverse flow $(\phi_{s,t})^{-1}$ is the unique solution of the backward stochastic equation \eqref{de}, analogous to  the original equation \eqref{eqn-SDE}, with the only difference being that the drift has an opposite sign.
Therefore, based on Theorem \ref{ww1}, Theorem \ref{thm-stability} and Corollary \ref{cor-civuole} remain valid
when $\phi_{s,t}^{n}(x)$ and $\phi_{s,t}(x)$ are replaced respectively by $(\phi_{s,t}^{n})^{-1}(x)$ and $(\phi_{s,t})^{-1}(x)$ , with slight adjustments to the proofs. However, it is important to note that the backward equation governs paths in reverse time $(\phi_{s,t})^{-1},s\in[0,t]$, unlike the forward equation $\phi_{s,t},t\in[s,T]$. Consequently, the formulation of the results differs accordingly:
\begin{align}
\label{stability1-001}
&\lim_{n\to\infty}\sup_{x\in{\mathbb{R}}^{d}} \sup_{t\in [0,T]}
\mathbb{E}[ \sup_{s \in [0,t]} |(\phi_{s,t}^{n})^{-1}(x)-\phi_{s,t}^{-1}(x)\vert^{p}]=0,\\
\label{stability2-uniform-001}
&\sup_{n\in\mathbb{N}}\sup_{x\in{\mathbb{R}}^{d}}\sup_{t\in [0,T]} \mathbb{E}[  \sup_{s \in [0,t]}\Vert D(\phi_{s,t}^{n})^{-1}(x)\Vert^{p}]
+
\sup_{x\in{\mathbb{R}}^{d}}\sup_{t\in [0,T]} \mathbb{E}[  \sup_{s \in [0,t]}\Vert D\phi_{s,t}^{-1}(x)\Vert^{p}]
<\infty,\\
\label{stability2-001}
&  \, \lim_{n\to\infty}\sup_{x\in K } \sup_{t\in [0,T]} \mathbb{E}[  \sup_{s \in [0,t]} \Vert D(\phi_{s,t}^{n})^{-1}(x)-D\phi_{s,t}^{-1}(x)\Vert^{p}]=0,
\end{align}
where $K \subset \R^d$ is any compact set. 
\end{remark}

\begin{remark}\label{remark 2025-11-226} By Remark \ref{cdd},
if $\alpha\in[1,2)$, we could obtain the following global stability of the flow:
\begin{align}
\label{stability1-alpha in 1-2}
&  \lim_{n\to\infty}\sup_{x\in \R^d}\sup_{s\in [0,T]} \mathbb{E}[  \sup_{t \in [s,T]} \Vert D\phi_{s,t}^{n}(x)-D\phi_{s,t}(x)\Vert^{p}]=0,\\
&  \lim_{n\to\infty}\sup_{x\in \R^d } \sup_{t\in [0,T]} \mathbb{E}[  \sup_{s \in [0,t]} \Vert D(\phi_{s,t}^{n})^{-1}(x)-D\phi_{s,t}^{-1}(x)\Vert^{p}]=0.
\end{align}

\end{remark}

\begin{proof}[Proof of Theorem \ref{thm-stability} ]

Let us fix $p \ge 2.$ In view of Theorem \ref{d32}, 
it is enough to verify that 
\eqref{stability1}, \eqref{stability2-uniform} and \eqref{stability2} of Theorem \ref{thm-stability} hold for $\phi_{0,t}^n=\phi_{t}^n$ and $\phi_{0,t}=\phi_{t}$, i.e.,
for any compact set $K \subset \R^d$, 
\begin{equation}
\label{stability101}
\lim_{n\to\infty} \sup_{x\in{\mathbb{R}}^{d}} \mathbb{E}[ \sup_{s \in [0,T]} |\phi_{s}^{n}(x)-\phi_{s}(x)\vert^{p}]=0,
\end{equation}%
\begin{equation}
\label{stability2-uniform-1}
\sup_{n\in\mathbb{N}}\sup_{x\in{\mathbb{R}}^{d}}  \mathbb{E}[  \sup_{s \in [0,T]} \Vert D\phi_{s}^{n}(x)\Vert^{p}]
+
\sup_{x\in{\mathbb{R}}^{d}}   \mathbb{E}[  \sup_{s \in [0,T]} \Vert D\phi_{s}(x)\Vert^{p}]
<\infty,
\end{equation}
\begin{equation}
\label{stability2-1} 
\lim_{n\to\infty}\sup_{x\in K}  \mathbb{E}[  \sup_{s \in [0,T]} \Vert D\phi_{s}^{n}(x)-D\phi_{s}(x)\Vert^{p}]=0.
\end{equation}
 We will prove the assertions stated above in several installments.

\begin{proof}[\textbf{Proof of assertions \eqref{stability101}.}]
According to Lemma \ref{lem-Ito formula Tanaka trick} and  Theorem \ref{thm-approximation}, there exists $\varpi{}_{,M}\geq 1$ such that, $\mathbb{P}$-a.s.,
\begin{equation} \label{itok}
\begin{aligned}
 u(\phi_{t}(x)) - u(x)
=&  x- \phi_{t}(x) + L_t  + \varpi{}_{,M} \int_0^t u(\phi_{s}(x))\, ds
 \\
 &+ \int_0^t \lint_{\R^d  } [ u(\phi_{s-}(x) + z) - u(\phi_{s-}(x))]
   \tilde \mu( ds,dz),
   \end{aligned}
\end{equation}
and
\begin{equation} \label{itok-n}
\begin{aligned}
 u^n(\phi_{t}^n(x)) - u^n(x)
=&  x- \phi_{t}^n(x) + L_t  + \varpi{}_{,M} \int_0^t u^n(\phi_{s}^n(x))\, ds
  \\&+ \int_0^t \lint_{\R^d  } [ u^n(\phi^n_{s-}(x) + z) - u^n(\phi^n_{s-}(x))]
   \tilde \mu(ds,dz).
   \end{aligned}
\end{equation}
Recall that by  Theorem \ref{thm-approximation} and \eqref{grad} we have, for any compact set $K \subset \R^d$,
\begin{align}
&\sup_{n\geq 0}\Vert u^n\Vert_{\alpha+\beta}<\infty \; \text{which implies} \; \sup_{n\geq 0}[Du^n]_{\alpha+\beta-1}<\infty,\label{eq tianjian 000-1}\\
 & \sup_{n\geq 0}\Vert Du^n\Vert_0\leq \frac{1}{3},
\quad \lim_{n\to \infty}\Vert u^n-u^0\Vert_{0}= 
 \lim_{n\to \infty} \sup_{x \in \R^d} |u^n(x) - u^0(x) |=0,
 \label{eq tianjian 000} 
\\
 \text{and } \quad &\lim_{n\to \infty} \sup_{x \in K} \Vert Du^n(x) - Du^0(x)\Vert=0. \label{eq tianjian 000 20251008} 
\end{align}
Here $u^n$ is the solution to \eqref{eqn-wee-n} with $\lambda=\varpi{}_{,M}$ and $f^n=b^n$, 
and $u=u^0$ is the solution to  equation \eqref{eqn-wee} with the RHS $f=b$ and $\lambda=\varpi{}_{,M}$.\\
Then, we have
$\mathbb{P}$-a.s.,
 \begin{align} \label{ser}
 \phi_t(x) - \phi^n_t(x) =&   [u(x) - u^n(x)] + [u^n( \phi^n_t(x))- u( \phi_t(x))] + \varpi{}_{,M}   \int_0^t   [u(\phi_s(x))- u^n(\phi^n_s(x))] ds
\\ &+
 \int_0^t \lint_{\R^d } [ u( \phi_{s-}(x) + z) - u( \phi_{s-}(x))
  - u^n( \phi^n_{s-}(x) + z) + u^n( \phi^n_{s-}(x))]
   \tilde \mu( ds,dz).\nonumber
\end{align}
The above equality can be rewritten as, $\mathbb{P}$-a.s.,
 \begin{align} \label{ser-02}
  \phi_t(x) - \phi^n_t(x) =&   [u^n(\phi^n_t(x))- u^n(\phi_t(x))] + \varpi{}_{,M}   \int_0^t   [u^n(\phi_s(x))- u^n(\phi^n_s(x))]\, ds
\\ &+
 \int_0^t \lint_{\R^d } [ u^n(\phi_{s-}(x) + z) - u^n(\phi_{s-}(x))
- u^n(\phi^n_{s-}(x) + z) + u^n(\phi^n_{s-}(x))]
   \tilde \mu(ds,dz)\nonumber
 \\
 &
+  [u(x) - u^n(x)]  + [u^n(\phi_t(x))- u(\phi_t(x))]+\varpi{}_{,M}   \int_0^t   [u(\phi_s(x))- u^n(\phi_s(x))] ds
\nonumber\\
&+
 \int_0^t \lint_{\R^d } [ u(\phi_{s-}(x) + z) - u(\phi_{s-}(x))
  - u^n(\phi_{s-}(x) + z) + u^n(\phi_{s-}(x))]
   \tilde \mu(ds,dz).
\nonumber
\end{align}
Since
 $\Vert Du^n\Vert_0 \le \frac13 $, we have
 \begin{align}\label{se-est01}
 |u^n(\phi_t(x))- u^n(\phi^n_t(x))| \le \frac{1}{3} |\phi_t(x)  - \phi_t^n(x)|.
 \end{align}
Combining \eqref{ser-02} and \eqref{se-est01} leads us to the estimate
 \[
 | \phi_t(x) - \phi^n_t(x)| \le \frac{3}{2} \Lambda^n_1(t) +
 \frac{3}{2}\Lambda^n_2(t) + \frac{3}{2}\Lambda^n_3(t)
 + \frac{3}{2}\Lambda^n_4(t),
 \]
  where
\begin{align*}
 \Lambda^n_1(t) &=  \Bigl\vert
 \int_0^t \int_{ B^c } [ u^n(\phi_{s-}(x) + z) - u^n(\phi_{s-}(x))
  - u^n(\phi^n_{s-}(x) + z) + u^n(\phi^n_{s-}(x))]
   \tilde \mu(ds,dz) \Bigr\vert,
\\
\Lambda^n_2(t) &=  \varpi{}_{,M}   \int_0^t   |u^n(\phi_s(x))- u^n(\phi^n_s(x))| ds,
\\
\Lambda^n_3(t) &=  \Bigl\vert \int_0^t \lint_{ B} [ u^n(\phi_{s-}(x) +
z) - u^n(\phi_{s-}(x))
 - u^n(\phi^n_{s-}(x) + z) + u^n(\phi^n_{s-}(x)) ]
   \tilde \mu( ds,dz) \Big|,
\\
 \Lambda^n_4(t) &= \Big|[u(x) - u^n(x)]  + [u^n(\phi_t(x))- u(\phi_t(x))]+ \varpi{}_{,M}   \int_0^t   [u(\phi_s(x))- u^n(\phi_s(x))] ds\Bigr\vert
\\
&\quad+
 \Bigl\vert\int_0^t \lint_{\R^d } [ u(\phi_{s-}(x) + z) - u(\phi_{s-}(x))
  - u^n(\phi_{s-}(x) + z) + u^n(\phi_{s-}(x))]
   \tilde \mu( ds,dz)\Bigr\vert.
\end{align*}
Now set $p\geq 2$. Note 
that, $\mathbb{P}$-a.s.,
\begin{equation*}
\sup_{0 \le s \le t } |\phi_s(x) - \phi^n_s(x)\vert^p \le
 C_p\sum_{j=1}^4 \sup_{0 \le s \le t } \, \Lambda^n_j(s)^p.
\end{equation*}
The main difficulty is to estimate $\sup_{0 \le s \le t } \, \Lambda^n_4(s)^p$. Let us first
consider the other terms. By \eqref{se-est01} and the H\"older inequality, we see
\begin{equation}
\sup_{0 \le s \le t } \Lambda^n_2(s)^p \le c_1 (p) \, t^{p-1} \int_0^t
\sup_{0 \le s \le r }  |\phi_s(x) - \phi^n_s(x)\vert^p \, dr.
\end{equation}
By \eqref{eqn-Burkholder inequality} with $C =B^c$
 we find
\begin{align*}
& \mathbb{E}[\sup_{0 \le s \le t } \Lambda^n_1(s)^p]
\\ \le&  c(p)
 \mathbb{E} \Bigl[ \Big ( \int_0^t   \int_{ B^c } | u^n(\phi_{s}(x) + z)
  - u^n(\phi^n_{s}(x) + z) + u^n(\phi^n_{s}(x))- u^n(\phi_{s}(x)) \vert^2 \nu(dz)  ds\Big)^{p/2}
  \Big ]
\\
&+ \, c(p)
 \mathbb{E}  \int_0^t  \int_{B^c } | u^n(\phi_{s}(x) + z)
  - u^n(\phi^n_{s}(x) + z) + u^n(\phi^n_{s}(x))- u^n(\phi_{s}(x))\vert^p \nu(dz) ds.
\end{align*}
Using the following inequality, by \eqref{eq tianjian 000},
\begin{align*} | u^n(\phi_{s}(x) + z)
  - u^n(\phi_{s}^n(x) + z) + u^n(\phi_{s}^n(x))- u^n(\phi_{s}(x)) |
  \le \frac{2}{3} |\phi_{s}(x)
  -  \phi_{s}^n(x)|,
  \end{align*}
   and the H\"older inequality, we get
\begin{align}\label{eq Gamma 01}
\mathbb{E}[\sup_{0 \le s \le t } \Lambda_1(s)^p]
\le & C_1(p) \, (1+ t^{p/2
-1})
 \Bigl(
 \int_{B^c } \nu(dz) \,  +  \big(\int_{B^c }
\nu(dz) \big)^{p/2} \Bigr) \nonumber\\
 &\cdot \, \int_0^t \mathbb{E}[
\sup_{0 \le s \le r } |\phi_{s}(x)
  -  \phi_{s}^n(x)\vert^p ] dr.
\end{align}
Let us treat  $\Lambda^n_3(t)$. Set $\gamma=\alpha+\beta-1$. Note that \eqref{eqn-beta+alpha half} implies that $2 \gamma >
\alpha$. By using \eqref{eqn-Burkholder inequality} with
 $C = B\setminus\{0\}$ and Lemma \ref{lem-Lemma 4.1}, we get
\begin{align*}
\mathbb{E}[\sup_{0 \le s \le t } \Lambda^n_3(s)^p]
 \le c(p) \Vert u^n\Vert_{1+ \gamma}^p
 \, \mathbb{E} \Bigl[
 \Big ( \int_0^t \lint_{B } |\phi_{s}(x) - \phi^n_s(x)\vert^2
 |z\vert^{2\gamma}
 \nu(dz)ds \Big )^{p/2} \Big ]
\\
 + \, c(p) \Vert u^n\Vert_{1+ \gamma}^p
 \, \mathbb{E}  \int_0^t \lint_{B } |\phi_{s}(x) - \phi^n_s(x)\vert^p
 |z\vert^{\gamma p}
 \nu(dz) ds .
\end{align*}
We obtain
\begin{align}\label{eq Gamma 03}
\mathbb{E}[\sup_{0 \le s \le t } \Lambda^n_3(s)^p]\le&
C_2 (p) \, (1+ t^{p/2
-1}) \,
 \Vert u^n\Vert^p_{1+ \gamma} \, \cdot \Bigl(
 \big( \lint_{B } |z\vert^{2 \gamma} \nu(dz) \big)^{p/2} +
  \lint_{B } |z\vert^{\gamma p}
\nu(dz) \Bigr)\\
  &\;\times \int_0^t \mathbb{E}[\sup_{0 \le s \le r }
|\phi_{s}(x) - \phi^n_s(x)\vert^p] \, dr,
\end{align}
where $\lint_{B } |z\vert^{2 \gamma} \nu(dz)+\lint_{B }  \, |z\vert^{p \gamma}
   \nu(dz) < +\infty  $,
 since $p \ge 2$ and $2 \gamma > \alpha$. By \eqref{eq tianjian 000-1},
 \begin{align}
 \sup_{n\geq1}\Vert u^n\Vert_{\gamma+1}=\sup_{n\geq1}\Vert u^n\Vert_{\alpha+\beta}\leq C<\infty.
 \end{align}
 Collecting the previous estimates, we arrive at
\begin{align*}
&\quad\ \mathbb{E}[ \sup_{0 \le s \le t} |\phi_s(x) - \phi^n_s(x)\vert^p] \\
&\le C_p \,
 \mathbb{E}[  \sup_{0 \le s \le t } \Lambda^n_4(s)^p]\,+\, C_3 (p) \, (1+t^{p/2-1}+t^{p -1}) \cdot\int_0^t \mathbb{E}[\sup_{0 \le
s \le r }  |\phi_s(x) - \phi^n_s(x)\vert^p] \, dr.
\end{align*}
Applying the   Gronwall Lemma we obtain, for any $t\geq0$,
  \begin{equation} \label{ciao211}
\mathbb{E}[ \sup_{0 \le s \le t} |\phi_s(x) - \phi^n_s(x)\vert^p] \le C(t) \, C_p  \mathbb{E}[  \sup_{0 \le s \le t } \Lambda^n_4(s)^p],
  \end{equation}
  with $C(t) = \exp \big( C_3 (p) \, (t+ t^{p/2}+t^{p })\big)$.
It remains to prove that
  \begin{equation} \label{eqn-remains}
\lim_{n\to \infty}\sup_{x\in\mathbb{R}^d}  \mathbb{E}[ \sup_{0 \le s \le T }\Lambda^n_4(s)^p] = 0.
\end{equation}
By applying the Burholder inequality, see  \eqref{eqn-Burkholder inequality} and the H\"older inequality, we get
\begin{align}
&\hspace{-2.6cm}\E\sup_{0\leq s\leq T} |\Lambda^n_4(s)\vert^{p} \\
&\hspace{-2.6cm}\leq C_{4}(p)|u(x) - u^n(x)\vert^{p}  +C_{4}(p)\E \sup_{0\leq s\leq T}|u^n(\phi_s(x))- u(\phi_s(x))\vert^{p}
\nonumber
\\
\label{eq control 0}
&\hspace{-2.6cm}+ \varpi{}_{,M} C_4(p)  T^{p-1} \E \int_{0}^{T}\sup_{0\leq s\leq r} |u(\phi_s(x))- u^n(\phi_s(x))\vert^{p} dr
\\
\nonumber
&\hspace{-2.6cm}\lefteqn{+ C_{4}(p)\E \left(
 \int_0^T \lint_{\R^d} |u(\phi_{s}(x) + z) - u(\phi_{s}(x))
  - u^n(\phi_{s}(x) + z) + u^n(\phi_{s}(x))\vert^{2}
   \nu(d z)\,d s\right)^{\frac p2}}
 \\
 \nonumber
 &\hspace{-2.6cm}\lefteqn{+ C_{4}(p)\E
 \int_0^T \lint_{\R^d } |u(\phi_{s}(x) + z) - u(\phi_{s}(x))
  - u^n(\phi_{s}(x) + z) + u^n(\phi_{s}(x))\vert^{p}
   \nu(d z)\, ds .}
\end{align}
By using \eqref{eq tianjian 000} we deduce that
\begin{align}
\nonumber
\lim_{n\to \infty}\sup_{x\in\mathbb{R}^d}&\Big(
   C_{4}(p)|u(x) - u^n(x)\vert^{p}  +C_{4}(p)\E \sup_{0\leq s\leq T}|u^n(\phi_s(x))- u(\phi_s(x))\vert^{p}\nonumber\\
   &\ \ \ \ \ \ \ \ \ \ \ \ \ \ \ \ \ \ \ \ \ +
   \varpi{}_{,M}   C_{4}(p)  T^{p-1} \E \int_{0}^{T}\sup_{0\leq s\leq r} |u(\phi_s(x))- u^n(\phi_s(x))\vert^{p} dr\Big)\nonumber\\
   \label{eq ctrong  000}
   &\leq
   \lim_{n\to \infty}C_p(1+T^p)\Vert u^n-u\Vert_0^p =0.
   \end{align}
By applying again \eqref{eq tianjian 000-1} and \eqref{eq tianjian 000}, there exists a constant $C$ such that, for any $s\in[0,T]$ and $|z|>1$,
\begin{align}\label{eq control 4}
\sup_{x\in\mathbb{R}^d}|u(\phi_s(x) + z) - u(\phi_s(x)) - u^n(\phi_s(x) + z) + u^n(\phi_s(x))\vert^{2}\leq C<\infty,
\end{align}
and,  for any $s\in[0,T]$ and $|z|\leq1$,
\begin{align}\label{eq control 5}
&\sup_{x\in\mathbb{R}^d}|u(\phi_s(x)  +  z)  -  u(\phi_s(x))  -  u^n(\phi_s(x)  +  z)  +  u^n(\phi_s(x))\vert^{2}\\
& \leq  C(\Vert Du\Vert_0^2 + \Vert Du^n\Vert_0^2)|z\vert^2\leq  C|z\vert^2.\nonumber
\end{align}
We also have
\begin{align}\label{eq control 6}
\hspace{-1cm}\lim_{n\to \infty}\sup_{x\in\mathbb{R}^d}|u(\phi_s(x) + z) - u(\phi_s(x)) - u^n(\phi_s(x) + z) + u^n(\phi_s(x))|\leq
\lim_{n\to \infty}2\Vert u-u^n\Vert_0=0.
\end{align}
In view of \eqref{eq control 4}, \eqref{eq control 5}, \eqref{eq control 6} and the Dominated Convergence Theorem, we infer
\begin{align}\label{eq control 7}
&\quad\ \lim_{n\to \infty}\sup_{x\in\mathbb{R}^d}
\E \Big(
 \int_0^T \lint_{\R^d } |u(\phi_s(x)  + z) - u(\phi_s(x) )
  - u^n(\phi_s(x)  + z) + u^n(\phi_s(x) )\vert^{2}
   \nu(d z)\,d s\Big)^{\frac p2}\nonumber\\
&\leq
\lim_{n\to \infty}\E \Big(
 \int_0^T \lint_{\R^d } \sup_{x\in\mathbb{R}^d}|u(\phi_s(x) + z) - u(\phi_s(x))
  - u^n(\phi_s(x) + z) + u^n(\phi_s(x))\vert^{2}
   \nu(d z)\,d s\Big)^{\frac p2}\nonumber\\
   &=0.
\end{align}
Using the similar argument as \eqref{eq control 7}, we have
\begin{align}\label{eq control 8}
&\lim_{n\to \infty}\!\sup_{x\in\mathbb{R}^d}
\!\E
 \int_0^T \lint_{\R^d } |u(\phi_s(x) + z) - u(\phi_s(x))
  - u^n(\phi_s(x) + z) + u^n(\phi_s(x))\vert^{p}
   \nu(d z)\, ds\nonumber\\
 &  =0.
\end{align}
By \eqref{eq control 0}, \eqref{eq ctrong  000}, \eqref{eq control 7}, and \eqref{eq control 8}, we infer \eqref{eqn-remains}. Hence the proof of \eqref{stability101} is complete.
\end{proof}

\vskip 0.2cm

\begin{proof}[\textbf{Proof of assertion \eqref{stability2-uniform-1}  and \eqref{stability2-1}.}]

Define functions $\psi$ and $\psi^n$ by
\[\psi(x)=x+u(x)\quad \mbox{ and }\quad \psi^n(x)=x+u^n(x)\quad \mbox{ for } x\in\mathbb{R}^{d}.\]
Since $\Vert D u\Vert_0\leq\frac13$ and $\Vert D u^n\Vert_0\leq\frac13$, the classical
 Hadamard theorem, see \cite[p.330]{Protter_2004}, implies that $\psi$ and $\psi^n$ are $C^1$-diffeomorphism from $\mathbb{R}^d$ onto $\mathbb{R}^d$ and
 \begin{align}\label{eq stab 0}
  \Vert D\psi\Vert_0\leq4/3,\ \Vert D(\psi^n)\Vert_0\leq4/3.
 \end{align}
Noting that for any $x\in\mathbb{R}^d$,
\begin{align}
D\psi^{-1}(x)&=\big(I+Du(\psi^{-1}(x))\big)^{-1}=\sum_{k\geq0}(-Du(\psi^{-1}(x)) )^k,\label{eq-D-psi-1}\\
 D(\psi^n)^{-1}(x)& =\big(I+Du^n((\psi^n)^{-1}(x))\big)^{-1}=\sum_{k\geq0}(-Du^n((\psi^n)^{-1}(x)) )^k,\label{eq-D-psi-2}
\end{align}
we have
  \begin{align}\label{eq stab 1}
  \Vert D\psi^{-1}\Vert_0\leq3/2\text{, and }\Vert D(\psi^n)^{-1}\Vert_0\leq3/2.
 \end{align}
 It follows that for any $x\in\mathbb{R}^d$,
\begin{align*}
  &\quad\ |(\psi^n)^{-1}(x)-\psi^{-1}(x)|\\
&=
  |(\psi^n)^{-1}\circ \psi\circ \psi^{-1}(x)-(\psi^n)^{-1}\circ \psi^n\circ \psi^{-1}(x)|\\
&\leq
  \Vert D (\psi^n)^{-1}\Vert_0|\psi\circ \psi^{-1}(x)-\psi^n\circ \psi^{-1}(x)|\leq
  \Vert D (\psi^n)^{-1}\Vert_0\Vert\psi-\psi^n\Vert_0\leq 3/2\Vert u-u^n\Vert_0,
\end{align*}
and hence
\begin{align}\label{lim psi n 01}
  \lim_{n\to \infty}\Vert(\psi^n)^{-1}-\psi^{-1}\Vert_0=0.
\end{align}

 Let us fix any compact set $K \subset \R^d$. By using \eqref{eq-D-psi-1}, \eqref{eq-D-psi-2} and the fact that for any $d\times d$ real matrices $A$ and $B$ with $(I+A)$ and $(I+B)$ being  invertible,
\begin{align}\label{invers-ind}
(I+A)^{-1}-(I+B)^{-1}=(I+A)^{-1}(B-A)(I+B)^{-1},
\end{align}
we deduce that, for any $x\in K$,  
\begin{align}
   &\quad\ \Vert D\psi^{-1}(x)-D(\psi^n)^{-1}(x)\Vert \\
   &=\big\Vert\big(I+Du(\psi^{-1}(x))\big)^{-1}\big(Du^n((\psi^n)^{-1}(x))-Du(\psi^{-1}(x))\big) \big(I+Du^n((\psi^n)^{-1}(x))\big)^{-1}\big\Vert\\
&\leq\Vert D\psi^{-1}\Vert_0\Vert Du(\psi^{-1}(x))-Du^n((\psi^n)^{-1}(x))\Vert \Vert D(\psi^n)^{-1}\Vert_0\\
&\leq
\frac{9}{4}\Big(\Vert Du(\psi^{-1}(x))-Du((\psi^n)^{-1}(x))\Vert+\Vert Du((\psi^n)^{-1}(x))-Du^n((\psi^n)^{-1}(x))\Vert\Big)\\
&\leq
\frac{9}{4}\Big([Du]_\gamma\Vert\psi^{-1}-(\psi^n)^{-1}\Vert_0^\gamma+ \Vert Du((\psi^n)^{-1}(x))-Du^n((\psi^n)^{-1}(x) \Vert\Big), 
\end{align}
where $[Du]_\gamma<\infty$ according to \eqref{eq tianjian 000-1}. 
We have to study the term $(\psi^n)^{-1}(x)$. Note that $K \subset B_M$ for some $M>0$, where $B_M$ is the closed ball as before. We write, for any $n \ge 1$,
\begin{gather} \label{ciao2}
|(\psi^n)^{-1}(x)| \le |(\psi^n)^{-1}(x)-(\psi^n)^{-1}(0)  |
+ |(\psi^n)^{-1}(0)| \le 3M/2  +|(\psi^n)^{-1}(0)|,
\end{gather}
by \eqref{eq stab 1}.  Let $y_n = (\psi^n)^{-1}(0)$. We have
$$
\psi^n(y_n) = y_n + u^n(y_n) =0. 
$$
Hence $|y_n|  \le  |u^n(y_n)| \le \sup_{n \ge 1}\| u_n \|_0
= C < \infty$ by \eqref{eq tianjian 000-1}. It follows that
$$
|(\psi^n)^{-1}(x)| \le  3M/2  +C =:M', \;\; n \ge 1,\, x \in K.
$$
By \eqref{eq tianjian 000 20251008} and \eqref{lim psi n 01}, we obtain
$$
 \sup_{x \in K}  \Vert D\psi^{-1}(x)-D(\psi^n)^{-1}(x)\Vert
 \le \frac{9}{4}\Big([Du]_\gamma\Vert\psi^{-1}-(\psi^n)^{-1}\Vert_0^\gamma + \sup_{y \in B_{M'} }\Vert Du(y)-Du^n(y)\Vert \Big)\to 0
$$
as $n \to \infty$.  Therefore, we get
 %This estimate, together with \eqref{eq tianjian 000} and %\eqref{lim psi n 01} implies
\begin{align}\label{lim psi n 02}
   \lim_{n\to \infty} \sup_{x \in K}\Vert D\psi^{-1}(x)-D(\psi^n)^{-1}(x)\Vert=0.
\end{align}
Using again \eqref{invers-ind}, we also have the following a priori estimates:
\begin{align}
\Vert D\psi^{-1}(x)-D\psi^{-1}(x')\Vert 
&\leq
    \Vert D\psi^{-1}\Vert_0\Vert Du(\psi^{-1}(x))-Du(\psi^{-1}(x'))\Vert \Vert D\psi^{-1}\Vert_0\\
&\leq
    9/4[Du]_\gamma\Vert D\psi\Vert_0^\gamma|x-x'\vert^\gamma,\quad\ \forall x,x'\in\mathbb{R}^d,
\end{align}
which gives
\begin{align}\label{eq Dpsi gamma}
[D\psi^{-1}]_\gamma
\leq
    \frac{9}{4}(\frac{4}{3})^\gamma[Du]_\gamma.
\end{align}
Similarly,
\begin{align}\label{eq Dpsin gamma}
[D(\psi^n)^{-1}]_\gamma
\leq
    \frac{9}{4}(\frac{4}{3})^\gamma[Du^n]_\gamma.
\end{align}
 We know from \cite[p. 442 and (4.14)]{Pr12} that
 \begin{align}\label{eq phi-y}
 \psi(\phi_{t}(x))=Y_{t}(\psi(x)),
 \end{align}
where
\begin{align}\label{eq Yy}
     Y_{t}(y)=y+\int_{0}^{t}\tilde{b}(Y_{s}(y))\,d s+\int_{0}^{t}\lint_{\R^{d}}h(Y_{s-}(y),z)\tilde{\mu}(d s,dz)+L_{t},
\end{align}
with $\tilde{b}(x)=\varpi{}_{,M} u(\psi^{-1}(x))$ and $h(x,z)=u(\psi^{-1}(x)+z)-u(\psi^{-1}(x))$;
and
\begin{align}\label{eq phi-yn}
\psi^n(\phi^n_{t}(x))=Y^n_{t}(\psi^n(x)),
\end{align}
where
\begin{align}\label{eq Yny}
     Y^n_{t}(y)=y+\int_{0}^{t}\tilde{b}^n(Y^n_{s}(y))\,d s+\int_{0}^{t}\lint_{\R^{d}}h^n(Y^n_{s-}(y),z)\tilde{\mu}(d s,dz)+L_{t},
\end{align}
with $\tilde{b}^n(x)=\varpi{}_{,M} u^n((\psi^n)^{-1}(x))$ and $h^n(x,z)=u^n((\psi^n)^{-1}(x)+z)-u^n((\psi^n)^{-1}(x))$.\\
\indent Observe that by \eqref{eq phi-y} and \eqref{eq phi-yn}, we have
\begin{align}\label{eq-Dphi}
D\phi_{t}(x)=(D\psi^{-1})(Y_{t}(\psi(x)) DY_{t}(\psi(x)) D\psi(x),
\end{align}
and
\begin{align}\label{eq-Dphi-n}
D\phi^n_{t}(x)=(D(\psi^n)^{-1})(Y^n_{t}(\psi^n(x))) DY^n_{t}(\psi^n(x)) D\psi^n(x).
\end{align}
Thanks to the above observations, to establish \eqref{stability2-uniform-1} and \eqref{stability2-1}, it is sufficient to prove the following formulae (recall that  $K \subset \R^d$ can be any compact set)
\begin{align}
      &\lim_{n\to\infty}\sup_{y\in\R^{d}}\E\sup_{0\leq t\leq T}|Y^{n}_{t}(y)-Y_{t}(y)\vert^{p}=0,\label{eq Part 01 Z}\\
&\mathbb{E}\Big(
    \sup_{t\in[0,T]}|Y_{t}(y)-Y_{t}(y')\vert^p\Big)
\leq
C_T \vert y-y'\vert^p,\;\;\forall\; y,y'\in\mathbb{R}^d,\label{eq Yy Yy' 03}\\
  & \sup_{n\geq1}\sup_{y\in\R^{d}}\E\sup_{0\leq t\leq T}\Vert DY^{n}_{t}(y)\Vert^{p}+\sup_{y\in\R^{d}}\E\sup_{0\leq t\leq T}\Vert DY_{t}(y)\Vert^{p}<\infty,\label{eq zhai sup 01}\\
  &
  \lim_{n\to\infty}  \, \sup_{y\in K}\E\sup_{0\leq t\leq T}\Vert DY^{n}_{t}(y)-DY_{t}(y)\Vert^{p}=0,\label{eq DY DYn lim}
  \\
&\E\sup_{t\in[0,T]}\Vert DY_{t}(y)-DY_{t}(y')\Vert^p
\leq
  C\Big(\vert y-y'\vert^{p\gamma}+\vert y-y'\vert^{\delta p}\Big),\;\ \forall\; y,y'\in\mathbb{R}^d.\label{eq DYy DYy' 03}
\end{align}
Here $\delta>0$ is a constant, which we will determine later.
Indeed,  \eqref{eq-Dphi}-\eqref{eq-Dphi-n} combined with \eqref{eq stab 0}, \eqref{eq stab 1} and \eqref{eq zhai sup 01} implies
\begin{align*}
& \quad\ \sup_{x\in\R^{d}}\E\sup_{t\in [0,T]}\Vert D\phi_{t}(x)\Vert^{p}+\sup_{n\geq1}\sup_{x\in\R^{d}}\E\sup_{t\in [0,T]}\Vert D\phi^{n}_{t}(x)\Vert^{p}\\
 & \leq  \Vert(D\psi^{-1})\Vert^p_0\Vert D\psi\Vert_0^p \sup_{y\in\R^{d}}\E\sup_{t\in [0,T]} \Vert DY_{t}(y)\Vert^p+\sup_{n\geq 1}\Vert(D(\psi^n)^{-1})\Vert^p_0\Vert D\psi^n\Vert_0^p \sup_{y\in\R^{d}}\E\sup_{t\in [0,T]} \Vert DY_{t}^n(y)\Vert^p \\
 &<\infty,
\end{align*}
which establisheds  \eqref{stability2-uniform-1}.\\
\indent Next, we will show that the above five inequalities imply assertion \eqref{stability2-1}.
%---------------------
Owing to \eqref{eq-Dphi}-\eqref{eq-Dphi-n} and \eqref{eq Dpsi gamma}, we deduce  for $x \in K$,
\begin{align}
\E&\sup_{t\in [0,T]}\Vert D\phi_{t}(x)-D\phi^{n}_{t}(x)\Vert^{p}\\
\leq&
C_p\,\E\sup_{t\in [0,T]}\big\Vert \big((D\psi^{-1})(Y_{t}(\psi(x)))
                       -
                       (D\psi^{-1})(Y_{t}(\psi^n(x)))\big)\cdot DY_{t}(\psi(x))\cdot D\psi(x)\big\Vert^p\\
&+
C_p\,\E\sup_{t\in [0,T]}\big\Vert\big(D\psi^{-1})(Y_{t}(\psi^n(x)))
                       -
                       (D\psi^{-1})(Y^n_{t}(\psi^n(x)))\big)\cdot DY_{t}(\psi(x))\cdot D\psi(x)\big\Vert^p\\
&+
C_p\,\E\sup_{t\in [0,T]}\big\Vert\big((D\psi^{-1})(Y^n_{t}(\psi^n(x)))
                       -
                       (D(\psi^n)^{-1})(Y^n_{t}(\psi^n(x)))\big)\cdot DY_{t}(\psi(x))\cdot D\psi(x)\big\Vert^p\\
&+
C_p\,\E\sup_{t\in [0,T]}\big\Vert(D(\psi^n)^{-1})(Y^n_{t}(\psi^n(x)))\cdot\big( DY_{t}(\psi(x))                       -
                        DY_{t}(\psi^n(x))\big)\cdot D\psi(x)\Vert^p\\
&+
C_p\,\E\sup_{t\in [0,T]}\big\Vert \big((D(\psi^n)^{-1})(Y^n_{t}(\psi^n(x)))\cdot \big(DY_{t}(\psi^n(x))
                       -
                        DY^n_{t}(\psi^n(x))\big)\cdot D\psi(x)\big\Vert^p\\
&+
C_p\,\E\sup_{t\in [0,T]}\big\Vert(D(\psi^n)^{-1})(Y^n_{t}(\psi^n(x)))\cdot DY^n_{t}(\psi^n(x))\cdot \big(D\psi(x)
                       -
                       D\psi^n(x)\big)\big\Vert^p\\
                       \leq&
C_p[Du]_\gamma^p\Vert D\psi\Vert_0^p\,\E\sup_{t\in [0,T]}\Big(|Y_{t}(\psi(x))
                       -
                       Y_{t}(\psi^n(x))\vert^{\gamma p}\Vert DY_{t}(\psi(x))\Vert^p\Big)\\
&+
C_p[Du]_\gamma^p\Vert D\psi\Vert_0^p\,\E\sup_{t\in [0,T]}\Big(|Y_{t}(\psi^n(x))
                       -
                       Y^n_{t}(\psi^n(x))\vert^{\gamma p}\Vert DY_{t}(\psi(x))\Vert^p\Big)\\
&+  
C_p\Vert D\psi\Vert_0^p\,\,\E\sup_{t\in [0,T]}\Big(\big\Vert(D\psi^{-1})(Y^n_{t}(\psi^n(x)))
                       -
                       (D(\psi^n)^{-1})(Y^n_{t}(\psi^n(x)))\big\Vert^p
                       \big\Vert DY_{t}(\psi(x))\big\Vert^p\Big)\\
&+ 
C_p\Vert D(\psi^n)^{-1}\Vert_0^p\Vert D\psi\Vert_0^p\,\E\sup_{t\in [0,T]}\Big(\Vert DY_{t}(\psi(x))-DY_{t}(\psi^n(x))\Vert^p\Big)\\
&+
C_p\Vert D(\psi^n)^{-1}\Vert_0^p\Vert D\psi\Vert_0^p\,\E\sup_{t\in [0,T]}\Big(\Vert DY_{t}(\psi^n(x))-DY^n_{t}(\psi^n(x))\Vert^p\Big)\\
&+ \,
C_p\Vert D(\psi^n)^{-1}\Vert_0^p\Vert D\psi (x)-D\psi^n (x)\Vert^p\,\E\sup_{t\in [0,T]}\Big(\Vert DY^n_{t}(\psi^n(x))\Vert^p\Big).
\end{align}
Note that  by the definition of $\psi^n$ and \eqref{eq tianjian 000-1}  we have
\begin{equation} \label{ciao3}
\sup_{n \ge 1}\sup_{x \in K} |\psi^n(x)| =: M_K < \infty. 
\end{equation}
By choosing $q>1$ large enough such that $\gamma pq>2$, and setting $\frac{1}{q}+\frac{1}{q'}=1$, we have
\begin{align*}
& \, \sup_{x\in K}\E\sup_{t\in [0,T]}\Vert D\phi_{t}(x)-D\phi^{n}_{t}(x)\Vert^{p}\\
&\leq C_p\,
[Du]_\gamma^p\Vert D\psi\Vert_0^p\Big(\sup_{y\in\mathbb{R}^d}\Big(\E\sup_{t\in [0,T]}\Vert DY_{t}(y)\Vert^{pq'}\Big)\Big)^{\frac{1}{q'}}
                       \Big(
                        \sup_{x\in\mathbb{R}^d}\E\Big(\sup_{t\in [0,T]}|Y_{t}(\psi(x))
                       -
                       Y_{t}(\psi^n(x))\vert^{\gamma pq}\Big)\Big)^{\frac{1}{q}}\\
&\quad+C_p\,
[Du]_\gamma^p\Vert D\psi\Vert_0^p
\Big(\sup_{y\in\mathbb{R}^d}\Big(\E\sup_{t\in [0,T]}\Vert DY_{t}(y)\Vert^{pq'}\Big)\Big)^{\frac{1}{q'}}
                       \Big(
                        \sup_{y\in\mathbb{R}^d}\E\Big(\sup_{t\in [0,T]}|Y_{t}(y)
                       -
                       Y^n_{t}(y)\vert^{\gamma pq}\Big)\Big)^{\frac{1}{q}}\\
&\quad+C_p\Vert D\psi\Vert_0^p\,\,\Big(\sup_{y\in\mathbb{R}^d}\Big(\E\sup_{t\in [0,T]}\Vert DY_{t}(y)\Vert^{2p}\Big)\Big)^{\frac{1}{2}}\\
                       &\quad\quad\quad\cdot\Big(\sup_{x\in K}\E\sup_{t\in [0,T]}\Big(\big\Vert(D\psi^{-1})(Y^n_{t}(\psi^n(x)))
                       -
                       (D(\psi^n)^{-1})(Y^n_{t}(\psi^n(x)))\big\Vert^{2p}
                       \Big)\Big)^{\frac{1}{2}}\\
&\quad  \, +C_p\,
\Vert D(\psi^n)^{-1}\Vert_0^p\Vert D\psi\Vert_0^p\sup_{x\in \R^d}\E\sup_{t\in [0,T]}\Big(\Vert DY_{t}(\psi(x))-DY_{t}(\psi^n(x))\Vert^p\Big)\\
& \quad+C_p\, 
\Vert D(\psi^n)^{-1}\Vert_0^p\Vert D\psi\Vert_0^p\sup_{y\in B_{ M_K}}  \E\sup_{t\in [0,T]}\Big(\Vert DY_{t}(y)-DY^n_{t}(y)\Vert^p\Big)\\
&\quad+C_p\, 
\Vert D(\psi^n)^{-1}\Vert_0^p \sup_{x \in K}\Vert Du(x)-Du^n(x)\Vert^p\sup_{y\in\mathbb{R}^d}\E\sup_{t\in [0,T]}\Big(\Vert DY^n_{t}(y)\Vert^p\Big).
\end{align*}
Since by \eqref{eq tianjian 000} $\lim_{n\to \infty}\Vert\psi^n-\psi\Vert_0=0$,  \eqref{eq Yy Yy' 03} and \eqref{eq DYy DYy' 03} imply that 
\begin{align*}
&\lim_{n\to \infty}\sup_{x\in\mathbb{R}^d}\E\Big(\sup_{t\in [0,T]}|Y_{t}(\psi(x))
                       -
                       Y_{t}(\psi^n(x))\vert^{\gamma pq}\Big)=0,\\
&\lim_{n\to \infty}
\sup_{x\in\mathbb{R}^d}\E\Big(\sup_{t\in [0,T]}\Vert DY_{t}(\psi(x))-DY_{t}(\psi^n(x))\Vert^p\Big)=0.
\end{align*}
And once we prove that
\begin{align}\label{eq 20251008-0001}
    \lim_{n\rightarrow\infty}\sup_{x\in K}\E\sup_{t\in [0,T]}\Big(\big\Vert(D\psi^{-1})(Y^n_{t}(\psi^n(x)))
                       -
                       (D(\psi^n)^{-1})(Y^n_{t}(\psi^n(x)))\big\Vert^{2p}
                       \Big)=0,
\end{align}
then,  using \eqref{eq tianjian 000-1}, \eqref{eq tianjian 000 20251008}, \eqref{eq stab 0}, \eqref{eq stab 1},  \eqref{eq zhai sup 01}, and
\eqref{eq DY DYn lim}, we conclude that
\begin{align*} 
\lim_{n\to \infty}\sup_{x \in K}\E\sup_{t\in [0,T]}\Vert D\phi_{t}(x)-D\phi^{n}_{t}(x)\Vert^{p}=0.
\end{align*}
which implies  \eqref{stability2-1}. 

Now it remains to prove \eqref{eq 20251008-0001}. By \eqref{eq phi-yn},
\begin{align}\label{eq phi-yn-20251008}
Y^n_{t}(\psi^n(x))=\psi^n(\phi^n_{t}(x))=\phi^n_{t}(x)+u^n(\phi^n_{t}(x)).
\end{align}
Combining \eqref{eq tianjian 000-1} and (ii) in Theorem \ref{d32},
\begin{align}
    \sup_{n\geq1}\sup_{x\in K}\sup_{t\in[0,T]}|Y^n_{t}(\psi^n(x))(\omega)|
    &\leq
   \sup_{x\in K}|x|+\sup_{n\geq1}\|b^n\|_0T+ \sup_{t\in[0,T]}|L_t(\omega)|+\sup_{n\geq1}\|u^n\|_0\\
   &=:C_{K,T}(\omega)<\infty, \ \ \mathbb{P}\text{-}a.s.\  \omega\in\Omega.
\end{align}
For any $J>0$, set $A_J=\{\omega\in\Omega:C_{K,T}(\omega)\leq J\}$. Then we have
\begin{align}\label{eq 20251008-0001 001}
    &\sup_{x\in K}\E\sup_{t\in [0,T]}\Big(\big\Vert(D\psi^{-1})(Y^n_{t}(\psi^n(x)))
                       -
                       (D(\psi^n)^{-1})(Y^n_{t}(\psi^n(x)))\big\Vert^{2p}
                       \Big)\\
    =&
    \sup_{x\in K}\E\sup_{t\in [0,T]}\Big(\big\Vert(D\psi^{-1})(Y^n_{t}(\psi^n(x)))
                       -
                       (D(\psi^n)^{-1})(Y^n_{t}(\psi^n(x)))\big\Vert^{2p}\1_{A_J}
                       \Big)\\
    &+
    \sup_{x\in K}\E\sup_{t\in [0,T]}\Big(\big\Vert(D\psi^{-1})(Y^n_{t}(\psi^n(x)))
                       -
                       (D(\psi^n)^{-1})(Y^n_{t}(\psi^n(x)))\big\Vert^{2p}\1_{A^c_J}
                       \Big)\\
\leq&
\sup_{y\in B_J}\Vert(D\psi^{-1})(y)
                       -
                       (D(\psi^n)^{-1})(y)\big\Vert^{2p}
    +
    C_p\Big(\sup_{n\geq1}\Vert D(\psi^n)^{-1}\Vert_0^{2p}+\Vert D\psi^{-1}\Vert_0^{2p}\Big)\mathbb{P}(A^c_J).
\end{align}
Then by \eqref{eq stab 1} and \eqref{lim psi n 02}, we have
\begin{align}\label{eq 20251008-0001 001-1}
    \limsup_{n\rightarrow\infty}\sup_{x\in K}\E\sup_{t\in [0,T]}\Big(\big\Vert(D\psi^{-1})(Y^n_{t}(\psi^n(x)))
                       -
                       (D(\psi^n)^{-1})(Y^n_{t}(\psi^n(x)))\big\Vert^{2p}
                       \Big)
\leq
    C_p\mathbb{P}(A^c_J).
\end{align}
The arbitrariness of $J$ implies \eqref{eq 20251008-0001}.

The proof of assertion \eqref{stability2-uniform-1}  and \eqref{stability2-1} is complete.
\end{proof}
Now it remains to prove \eqref{eq Part 01 Z}-\eqref{eq DYy DYy' 03}. The more difficult assertion will be 
 \eqref{eq DY DYn lim}. 
 
 \begin{proof}[Proof of
\eqref{eq Part 01 Z}]
Notice that by Page 444 and (4.15) of \cite{Pr12},
\begin{align}\label{eq SSSS 01}
\Vert Dh^n(\cdot,z)\Vert_0\leq c \Vert Du^n\Vert_{\gamma}(\1_{B}(z)|z\vert^{\gamma}+\1_{B^c}(z)),\ \ \ \forall z\in \R^{d}\setminus\{0\}, \;  \gamma = \alpha + \beta -1,
\end{align}
and
\begin{align}\label{eq 4.44}
\Vert Dh(\cdot,z)\Vert_0\leq c \Vert Du\Vert_{\gamma}(\1_{B}(z)|z\vert^{\gamma}+\1_{B^c}(z)),\ \ \ \forall z\in \R^{d}\setminus\{0\}.
\end{align}
By \eqref{eq tianjian 000-1}, \eqref{eq tianjian 000} and \eqref{eq stab 1}, it is easy to see that
 \begin{align*}
   \sup_{n}\Vert\tilde{b}^n\Vert_0+\Vert\tilde{b}\Vert_0\leq \varpi{}_{,M} C<\infty,
\end{align*}
and
\begin{align}
\begin{aligned}\label{eq Db 0 pro}
&\Vert D\tilde{b}^n\Vert_0\!\leq\!\varpi{}_{,M}\Vert Du^n\Vert_0\Vert D(\psi^n)^{-1}\Vert_0\!\leq\!\frac{\varpi{}_{,M}}{2},\\
& \Vert D\tilde{b}\Vert_0\!\leq\! \varpi{}_{,M}\Vert Du\Vert_0\Vert D\psi^{-1}\Vert_0\!\leq\! \frac{\varpi{}_{,M}}{2}.
\end{aligned}
\end{align}
For any $x\in\mathbb{R}^d$,
 \begin{align*}
   |\tilde{b}^n(x)-\tilde{b}(x)|
   &\leq
   \varpi{}_{,M}\Big(|u^n((\psi^n)^{-1}(x))-u((\psi^n)^{-1}(x))|+|u((\psi^n)^{-1}(x))-u(\psi^{-1}(x))|\Big)\\
   &\leq
   \varpi{}_{,M}\Big(\Vert u^n-u\Vert_0+\Vert Du\Vert_0\Vert(\psi^n)^{-1}-\psi^{-1}\Vert_0\Big),
\end{align*}
and, for any $x \in K$, 
 \begin{align*}
   \|D\tilde{b}^n(x)-D\tilde{b}(x)\|
   =&
   \varpi{}_{,M}\|Du^n((\psi^n)^{-1}(x))D(\psi^n)^{-1}(x)-Du(\psi^{-1}(x))D\psi^{-1}(x)\|\\
   \leq&
   \varpi{}_{,M}\|Du^n((\psi^n)^{-1}(x))D(\psi^n)^{-1}(x)-Du((\psi^n)^{-1}(x))D(\psi^n)^{-1}(x)\|\\
   &+
   \varpi{}_{,M}\|Du((\psi^n)^{-1}(x))D(\psi^n)^{-1}(x)-Du(\psi^{-1}(x))D(\psi^n)^{-1}(x)\|\\
   &+
   \varpi{}_{,M}\|Du(\psi^{-1}(x))D(\psi^n)^{-1}(x)-Du(\psi^{-1}(x))D\psi^{-1}(x)\|\\
   \leq& 
   \varpi{}_{,M} \sup_{y \in B_{M'}}\| Du^n(y) -Du(y)\| \, \Vert D(\psi^n)^{-1}\Vert_0
 \\   &  + 
   \varpi{}_{,M}[Du]_\gamma\Vert(\psi^n)^{-1}-\psi^{-1}\Vert_0^\gamma\Vert D(\psi^n)^{-1}\Vert_0
  \\ & +
   \varpi{}_{,M}\Vert Du\Vert_0  \sup_{x \in K}\| D(\psi^n)^{-1}(x)-D\psi^{-1}(x)\|,
\end{align*}
here $M'$ is introduced in \eqref{ciao3}.
Hence, \eqref{eq tianjian 000-1}, \eqref{eq tianjian 000}, \eqref{eq tianjian 000 20251008}, \eqref{eq stab 1}, \eqref{lim psi n 01} and \eqref{lim psi n 02} imply that
 \begin{align}\label{eq b Db lim} 
   \lim_{n\to \infty}\Big[\Vert\tilde{b}^n-\tilde{b}\Vert_0+ \sup_{x \in K} \|D\tilde{b}^n(x)-D\tilde{b}(x)\|\Big]=0.
\end{align}
Similar arguments imply that, for any $z\in \mathbb{R}^{d}\setminus\{0\}$,
 \begin{align}\label{Eq hn h 01} 
   \lim_{n\to \infty}\Big[\Vert{h}^n(\cdot,z)-{h}(\cdot,z)\Vert_0+  \sup_{x \in K}\| D{h}^n(x,z)-D{h}(x,z)\|\Big]=0.
\end{align}
 Let $y \in \R^d$. 
By \eqref{eq Yy} and \eqref{eq Yny},
\begin{align*}
      |Y^{n}_{t}(y)-Y_{t}(y)|
\leq&
    \int_{0}^{t}|\tilde{b}^n(Y^n_{s}(y))-\tilde{b}(Y_{s}(y))|\,d s
    +
    \Big|\int_{0}^{t}\lint_{\R^{d}}h^n(Y^n_{s-}(y),z)-h(Y_{s-}(y),z)\tilde{\mu}(d s,dz)\Big|\\
\leq&
    \int_{0}^{t}|\tilde{b}^n(Y^n_{s}(y))-\tilde{b}(Y^n_{s}(y))|\,d s
    +
    \int_{0}^{t}|\tilde{b}(Y^n_{s}(y))-\tilde{b}(Y_{s}(y))|\,d s\\
   &+
    \Bigl\vert\int_{0}^{t}\lint_{\R^{d}}h^n(Y^n_{s-}(y),z)-h(Y^n_{s-}(y),z)\tilde{\mu}(d s,dz)\Big|\\
   &+
    \Bigl\vert\int_{0}^{t}\lint_{\mathbb{R}^{d}}h(Y^n_{s-}(y),z)-h(Y_{s-}(y),z)\tilde{\mu}(d s,dz)\Big|\\
=:&
   I_1(t)+I_2(t)+I_3(t)+I_4(t).
\end{align*}
For $I_1(t)+I_2(t)$, we have
\begin{align}\label{eq J3 0001}
  I_1(t)+I_2(t)
\leq
  \Vert\tilde{b}^n-\tilde{b}\Vert_0t
  +
  \Vert D\tilde{b}\Vert_0\int_{0}^{t}|Y^n_{s}(y)-Y_{s}(y)|\,d s.
\end{align}
By \eqref{eq tianjian 000-1}, \eqref{eq 4.44} and \eqref{eqn-Burkholder inequality}, for any $t\in[0,T]$,
\begin{align}\label{eq J3 01000}
\mathbb{E}\Big(
    \sup_{s\in[0,t]}|I_4(s)\vert^p\Big)
\leq&
C_{p,T}\Vert Du\Vert_\gamma^p
\Big(
\lint_{\R^{d}}\1_{B}(z)|z\vert^{\gamma p}+\1_{B^c}(z)\,\nu(dz)
+
   \Big(
       \lint_{\R^{d}}\1_{B}(z)|z\vert^{2\gamma}+\1_{B^c}(z)\,\nu(dz)
   \Big)^{p/2}
\Big)\nonumber\\
 &\cdot\mathbb{E}
    \Big(\int_{0}^{t}\sup_{r\in[0,s]}|Y^n_{r}(y)-Y_{r}(y)\vert^pds
           \Big)\nonumber\\
\leq&
C_{p,T}\mathbb{E}
    \Big(\int_{0}^{t}\sup_{r\in[0,s]}|Y^n_{r}(y)-Y_{r}(y)\vert^pds
           \Big).
\end{align}
The Gronwall lemma with the above three inequalities implies that
\begin{align}\label{eq J3 01}
      \E\sup_{0\leq t\leq T}|Y^{n}_{t}(y)-Y_{t}(y)\vert^{p}
      \leq
      C_{p,T}\Big(
        \Vert\tilde{b}^n-\tilde{b}\Vert_0^p
        +
        \E(\sup_{0\leq t\leq T}|I_3(t)\vert^p)
      \Big).
\end{align}
By applying inequality \eqref{eqn-Burkholder inequality}, we deduce that
\begin{align}
\mathbb{E}\Big(
    \sup_{t\in[0,T]}|I_3(t)\vert^p
           \Big)\leq\;&
 C_p\mathbb{E}\Big(
         \Bigl\vert\int_{0}^{T}\lint_{\mathbb{R}^{d}}|h^n(Y^n_{s}(y),z)-h(Y^n_{s}(y),z)\vert^2\nu(dz)ds\Big\vert^{p/2}
           \Big)\\
    &+
    C_p\mathbb{E}\Big(
         \int_{0}^{T}\lint_{\mathbb{R}^{d}}|h^n(Y^n_{s}(y),z)-h(Y^n_{s}(y),z)\vert^p\nu(dz)ds
           \Big).
\end{align}
By \eqref{Eq hn h 01} and using similar arguments as \eqref{eq control 7} and \eqref{eq control 8}, we may invoke the Dominated Convergence Theorem to get
\begin{align}
\lim_{n\to \infty}\sup_{y\in\mathbb{R}^d}\mathbb{E}\Big(
    \sup_{t\in[0,T]}|I_3(t)\vert^p
           \Big)
=
0.
\end{align}
Hence, we have \eqref{eq Part 01 Z}.
\end{proof}

Next we move to prove \eqref{eq Yy Yy' 03}.
\begin{proof}[Proof of \eqref{eq Yy Yy' 03}]
By \eqref{eq Yy}, we have, for any $y,y'\in\mathbb{R}^d$,
\begin{align}
     |Y_{t}(y)-Y_{t}(y')|
\leq&
     \vert y-y'| + \Big|\int_{0}^{t}\tilde{b}(Y_{s}(y))-\tilde{b}(Y_{s}(y'))\,d s\Big|\nonumber\\
     &+
     \Big|\int_{0}^{t}\lint_{\R^{d}}h(Y_{s-}(y),z)-h(Y_{s-}(y'),z)\tilde{\mu}(d s,dz)\Big|.\label{eq Yy Yy' 00}
\end{align}
A similar argument to that used for the proof of \eqref{eq J3 0001} and \eqref{eq J3 01000} shows that
\begin{align}\label{eq Yy Yy' 01}
\Big|\int_{0}^{t}\tilde{b}(Y_{s}(y))-\tilde{b}(Y_{s}(y'))\,d s\Big|
\leq
\Vert D\tilde{b}\Vert_0\int_{0}^{t}|Y_{s}(y)-Y_{s}(y')|\,d s,
\end{align}
and
\begin{align}\label{eq Yy Yy' 02}
\mathbb{E}\Big(
    \sup_{l\in[0,t]}\Big|\int_{0}^{l}\lint_{\R^{d}}h(Y_{s-}(y),z)-h(Y_{s-}(y'),z)\tilde{\mu}(d s,dz)\Big|^p\Big)
\leq
C_{p,T} \mathbb{E}
    \Big(\int_{0}^{t}\sup_{r\in[0,s]}|Y_{r}(y)-Y_{r}(y')\vert^pds
           \Big).
\end{align}
Applying \eqref{eq Yy Yy' 00}--\eqref{eq Yy Yy' 02} and Gronwall's inequality, we obtain
\begin{align}\label{eq Yy Yy' 03-1}
\mathbb{E}\Big(
    \sup_{t\in[0,T]}|Y_{t}(y)-Y_{t}(y')\vert^p\Big)
\leq
C_{p,T} \vert y-y'\vert^p,
\end{align}
which gives  \eqref{eq Yy Yy' 03}.
\end{proof}

\begin{proof}[Proof of \eqref{eq zhai sup 01}]

Let us observe that \cite[Theorem 3.4]{Kunita_2004}, see also Page 445 of \cite{Pr12},  implies  that  the following equation for $DY_{t}(y)$ and $DY^n_{t}(y)$ holds,
\begin{align}\label{eq DY 01}
   DY_{t}(y)=I&+ \int_{0}^{t}D\tilde{b}(Y_{s}(y))DY_{s}(y)\, ds+ \int_{0}^{t}\lint_{\R^{d}}Dh(Y_{s-}(y),z)DY_{s-}(y)\tilde{\mu}(ds,dz),\quad
\end{align}
and
\begin{align}\label{eq DYn 01}
   DY^n_{t}(y)=I&+ \int_{0}^{t}D\tilde{b}^n(Y^n_{s}(y))DY^n_{s}(y)\, ds
    + \int_{0}^{t}\lint_{\R^{d}}Dh^n(Y^n_{s-}(y),z)DY^n_{s-}(y)\tilde{\mu}(ds,dz).\quad\quad
\end{align}
It is not difficulty to prove that, by \eqref{eq tianjian 000-1}, \eqref{eq SSSS 01}, \eqref{eq 4.44} and \eqref{eq Db 0 pro},
\begin{align}\label{eq zhai sup 01-1}
  & \sup_{n\geq1}\sup_{y\in\R^{d}}\E\sup_{0\leq t\leq T}\Vert DY^{n}_{t}(y)\Vert^{p}+\sup_{y\in\R^{d}}\E\sup_{0\leq t\leq T}\Vert DY_{t}(y)\Vert^{p}<\infty.
\end{align}
which implies \eqref{eq zhai sup 01}.
\end{proof}

\begin{proof}[Proof of \eqref{eq DY DYn lim}]
 We will prove that (recall that $K \subset \R^d$ is any compact set)
\begin{align} \label{eq DY DYn lim-1}
  &  \lim_{n\to\infty}\sup_{y\in K}\E\sup_{0\leq t\leq T}\Vert DY^{n}_{t}(y)-DY_{t}(y)\|^{p}=0.
\end{align}
By \eqref{eq DY 01} and \eqref{eq DYn 01}, 
\begin{align}\label{eq 20220118 00}
\Vert DY^{n}_{t}(y)-DY_{t}(y)\Vert \leq&
  \int_{0}^{t}\Vert D\tilde{b}^n(Y^n_{s}(y))DY^n_{s}(y)-D\tilde{b}(Y_{s}(y))DY_{s}(y)\Vert\, ds\nonumber\\
    &+ \Big\Vert\int_{0}^{t}\lint_{\R^{d}}Dh^n(Y^n_{s-}(y),z)DY^n_{s-}(y)-Dh(Y_{s-}(y),z)DY_{s-}(y)\tilde{\mu}(ds,dz)\Big\Vert\nonumber\\
\leq& 
  \int_{0}^{t}\Vert D\tilde{b}^n(Y^n_{s}(y))DY^n_{s}(y)-D\tilde{b}^n(Y_{s}(y))DY^n_{s}(y)\Vert\, ds\nonumber\\
   &+
  \int_{0}^{t}\Vert D\tilde{b}^n(Y_{s}(y))DY^n_{s}(y)-D\tilde{b}(Y_{s}(y))DY_{s}^n(y)\Vert\, ds\nonumber\\
    &+
    \int_{0}^{t}\Vert D\tilde{b}(Y_{s}(y))DY_{s}^n(y)-D\tilde{b}(Y_{s}(y))DY_{s}(y)\Vert\, ds\nonumber
    \\
    & + \Big\Vert\int_{0}^{t}\lint_{\R^{d}}Dh^n(Y^n_{s-}(y),z)DY^n_{s-}(y)-Dh^n(Y_{s-}(y),z)DY^n_{s-}(y)\tilde{\mu}(ds,dz)\Big\Vert\nonumber\\
    &+ \Big\Vert\int_{0}^{t}\lint_{\R^{d}}Dh^n(Y_{s-}(y),z)DY^n_{s-}(y)-Dh(Y_{s-}(y),z)DY^n_{s-}(y)\tilde{\mu}(ds,dz)\Big\Vert\nonumber\\
    &+ \Big\Vert\int_{0}^{t}\lint_{\R^{d}}Dh(Y_{s-}(y),z)DY^n_{s-}(y)-Dh(Y_{s-}(y),z)DY_{s-}(y)\tilde{\mu}(ds,dz)\Big\Vert\nonumber 
    \\
=:&\sum_{i=1}^6J_i(t,y,n).
\end{align}
Before proceeding, we note that
 by  \eqref{eq tianjian 000-1}, \eqref{eq Dpsin gamma}, \eqref{eq stab 1} and Page 444 Lines 9-24 of \cite{Pr12},
\begin{align}\label{new1}
\sup_{n \ge 1}[D\tilde{b}^n]_{\gamma}
\leq \sup_{n \ge 1} \, 
\varpi{}_{,M} \Big(\Vert Du^n\Vert_0[ D(\psi^n)^{-1}]_\gamma+[Du^n]_\gamma\Vert D(\psi^n)^{-1}\Vert_0^{1+\gamma}\Big)
\leq C<\infty,
\end{align}
with $\gamma = \alpha + \beta -1.$   
Similarly, we have
\begin{align}\label{eq b gamma pro}
 [D\tilde{b}]_{\gamma}
\leq
\varpi{}_{,M} \Big(\Vert Du\Vert_0[ D\psi^{-1}]_\gamma+[Du]_\gamma\Vert D\psi^{-1}\Vert_0^{1+\gamma}\Big)
\leq C<\infty.
\end{align}
For $J_6(t,y,n)$, using the Burkholder inequality \eqref{eqn-Burkholder inequality}, we have for any $0 \le S \le T$, $y \in \R^d$,
\begin{align}\label{eq J5 DYn DY 000}
& \E\sup_{t\in[0,S]}\vert J_6(t,y,n)\vert^p\\
&\leq
   C_p  \E\Big(
         \int_{0}^{S}\lint_{\R^{d}}\Vert Dh(Y_{s}(y),z)\Vert^2\Vert DY^n_{s}(y)-DY_{s}(y)\Vert^2\nu(dz)ds
      \Big)^{p/2}\nonumber\\
      &\quad +
      C_p \E\Big(
         \int_{0}^{S}\lint_{\R^{d}}\Vert Dh(Y_{s}(y),z)\Vert^p\Vert DY^n_{s}(y)-DY_{s}(y)\Vert^p\nu(dz)ds
      \Big)\nonumber\\
&\leq
   C_p\E\Big(
         \int_{0}^{S} \Vert DY^n_{s}(y)-DY_{s}(y)\Vert^2\lint_{\R^{d}}\Vert Dh(\cdot,z)\Vert^2_0\nu(dz)ds
      \Big)^{p/2}\nonumber\\
      &\quad+
      C_p\E\Big(
         \int_{0}^{S} \Vert DY^n_{s}(y)-DY_{s}(y)\Vert^p\lint_{\R^{d}}\Vert Dh(\cdot,z)\Vert^p_0\nu(dz)ds
      \Big)\nonumber\\
&\leq
      C_{p,T}\E\Big(
         \int_{0}^{S} \Vert DY^n_{s}(y)-DY_{s}(y)\Vert^pds
      \Big)
      \Big[\Big(\lint_{\R^{d}}\Vert Dh(\cdot,z)\Vert^2_0\nu(dz)\Big)^{p/2}+
         \lint_{\R^{d}}\Vert Dh(\cdot,z)\Vert^p_0\nu(dz)
      \Big]\nonumber\\
&\leq
      C_{p,T}\E\Big(
         \int_{0}^{S} \Vert DY^n_{s}(y)-DY_{s}(y)\Vert^pds
      \Big).
\end{align}
For the last inequality, we have used the fact that
 \eqref{eq 4.44} implies that
\begin{align}\label{eq Dh pro}
\lint_{\R^{d}}\Vert Dh(\cdot,z)\Vert^2_0\nu(dz)+\lint_{\R^{d}}\Vert Dh(\cdot,z)\Vert^p_0\nu(dz)<\infty.
\end{align}
The term $J_3(t,y,n)$ can be easily estimated  using the bound \eqref{eq Db 0 pro}. We get a similar estimate:
\begin{equation}
  \E
  \sup_{t\in[0,S]}\vert J_3(t,y,n)\vert^p   \le M_{p,T}\E
         \int_{0}^{S} \Vert DY^n_{s}(y)-DY_{s}(y)\Vert^pds. 
        \end{equation} 
In order to estimate $J_1(t,y,n)$, we use  \eqref{new1}. We find 
\begin{align}
  \E\sup_{t\in[0,T]}\vert J_1(t,y,n)\vert^p &\le  \E \Big(\int_{0}^{T}\Vert D\tilde{b}^n(Y^n_{s}(y))-D\tilde{b}^n(Y_{s}(y))\Vert \Vert DY^n_{s}(y)\Vert 
  \, ds \Big)^p  \\ 
  &\le C\E
\Big(\int_{0}^{T} |Y^n_{s}(y))- Y_{s}(y)|^{\gamma}  \Vert DY^n_{s}(y)\Vert 
  \, ds \Big)^p
\\
&\le C_T
\Big(\E 
   \sup_{s\in[0,T]}\Vert DY_{s}^n(y)\Vert^{2p}
   \Big)^{\frac{1}{2}}
\Big(\E
   \sup_{s\in[0,T]}|Y^n_{s}(y)-Y_{s}(y)\vert^{2\gamma p}\Big)^{\frac{1}{2}}.
\end{align}
 Hence, by \eqref{eq Part 01 Z} and \eqref{eq zhai sup 01},
$$
 \sup_{y \in \mathbb{R}^d}  \E\sup_{t\in[0,T]}\vert J_1(t,y,n)\vert^p \le M_{p,T} \sup_{y \in \mathbb{R}^d} \Big(\E
   \sup_{s\in[0,T]}|Y^n_{s}(y)-Y_{s}(y)\vert^{2\gamma p}\Big)^{\frac{1}{2}} \to 0,\quad\text{as }n \to \infty.
$$
\indent Let us estimate $J_{4}(t,y,n)$.  For any $y\in\mathbb{R}^d$,
\begin{align}\label{eq J4}
   &\E\sup_{s\in[0,T]}|J_4(s,y,n)\vert^p\\
&\leq
  C_p\E\Big(\int_{0}^{T}\lint_{\R^{d}}\Vert Dh^n(Y^n_{s}(y),z)-Dh^n(Y_{s}(y),z)\Vert^2\Vert DY^n_{s}(y)\Vert^2\nu(dz)ds\Big)^{p/2}\\
   &\quad +
   C_p\E\Big(\int_{0}^{T}\lint_{\R^{d}}\Vert Dh^n(Y^n_{s}(y),z)-Dh^n(Y_{s}(y),z)\Vert^p\Vert DY^n_{s}(y)\Vert^p\nu(dz)ds\Big)\nonumber\\
&\leq
    C_p\Big(\E\Big(\sup_{s\in[0,T]}\Vert DY^n_{s}(y)\Vert^{2p}\Big)\Big)^{\frac{1}{2}}
       \Big(\E\Big(\int_{0}^{T}\lint_{\R^{d}}\Vert Dh^n(Y^n_{s}(y),z)-Dh^n(Y_{s}(y),z)\Vert^2\nu(dz)ds\Big)^{p}\Big)^{\frac{1}{2}}\nonumber\\
   &\quad+
   C_p\Big(\E\Big(\sup_{s\in[0,T]}\Vert DY^n_{s}(y)\Vert^{2p}\Big)\Big)^{\frac{1}{2}}
       \Big(\E\Big(\int_{0}^{T}\lint_{\R^{d}}\Vert Dh^n(Y^n_{s}(y),z)-Dh^n(Y_{s}(y),z)\Vert^p\nu(dz)ds\Big)^{2}\Big)^{\frac{1}{2}}\nonumber\\
&\leq
    C_p
       \Big(\E\Big(\int_{0}^{T}\lint_{\R^{d}}\Vert Dh^n(Y^n_{s}(y),z)-Dh^n(Y_{s}(y),z)\Vert^2\nu(dz)ds\Big)^{p}\Big)^{\frac{1}{2}}\nonumber\\
   &\quad+
   C_p
       \Big(\E\Big(\int_{0}^{T}\lint_{\R^{d}}\Vert Dh^n(Y^n_{s}(y),z)-Dh^n(Y_{s}(y),z)\Vert^p\nu(dz)ds\Big)^{2}\Big)^{\frac{1}{2}}.\nonumber
\end{align} 
For the last inequality, we have used \eqref{eq zhai sup 01}.
Note that
\begin{equation}
Dh^n(x,z)=Du^n((\psi^n)^{-1}(x)+z)D(\psi^n)^{-1}(x)-Du^n((\psi^n)^{-1}(x))D(\psi^n)^{-1}(x).
\end{equation}
Arguing as in (4.15) in \cite{Pr12}, we obtain  that
there exists some $\delta>0$ and $\mathcal{K}(z)$ with $\int_{\R^{d}\setminus\{0\}}\mathcal{K}^p(z)\,\nu(dz)<\infty$ for any  $p \ge 2$ such that, for any $n \ge 1$, 
\begin{equation}
[Dh^n(\cdot,z)]_\delta\leq \mathcal{K}(z),\ \ \ \forall z\in\R^{d}\setminus\{0\}.
\end{equation}
Hence, for any $y\in\mathbb{R}^d$,
\begin{align}\label{eq J42}
   \E\sup_{s\in[0,T]}|J_4(s,y,n)\vert^p
   \le& 
  C_p
       \Big(\E\Big(\int_{0}^{T}\lint_{\R^{d}} |Y^n_{s}(y)-Y_{s}(y) |^{2 \delta} \, \mathcal{K}^2(z)\nu(dz)ds\Big)^{p}\Big)^{\frac{1}{2}}\nonumber\\
   &+
   C_p
       \Big(\E\Big(\int_{0}^{T}\lint_{\R^{d}} |Y^n_{s}(y)-Y_{s}(y) |^{p \delta} \, \mathcal{K}^p(z)\nu(dz)ds\Big)^{2}\Big)^{\frac{1}{2}}\\ \le& 
  C_{p,T}
       \Big(\E
   \sup_{s\in[0,T]}|Y^n_{s}(y)-Y_{s}(y)\vert^{2p\delta }\Big)^{\frac{1}{2}}\nonumber.
\end{align} 
Using \eqref{eq Part 01 Z} it follows easily that 
\begin{align}\label{eq J422}
  \sup_{y \in \mathbb{R}^d} \E\sup_{s\in[0,T]}|J_4(s,y,n)\vert^p \le
 C_{p,T} \Big(\sup_{y \in \R^d} \E
   \sup_{s\in[0,T]}|Y^n_{s}(y)-Y_{s}(y)\vert^{2p\delta }\Big)^{\frac{1}{2}} \to 0
 \nonumber
\end{align} 
 as $n \to \infty$.  

To estimate the other terms, we consider the compact set $K$. Let us estimate $J_2(t,y,n)$.  For any $y\in\mathbb{R}^d$,
\begin{align}
  \E\sup_{t\in[0,T]}\vert J_2(t,y,n)\vert^p &\le  \E \Big(\int_{0}^{T}\Vert D\tilde{b}^n(Y_{s}(y))-D\tilde{b}(Y_{s}(y))\Vert \Vert DY^n_{s}(y)\Vert 
  \, ds \Big)^p  
\\
&\le C_T
\Big(\E 
   \sup_{s\in[0,T]}\Vert DY_{s}^n(y)\Vert^{2p}
   \Big)^{\frac{1}{2}}
\Big(\E
   \sup_{s\in[0,T]}\Vert D\tilde{b}^n(Y_{s}(y))-D\tilde{b}(Y_{s}(y))\Vert^{2p}\Big)^{\frac{1}{2}}\\
&\le C_T\Big(\E
   \sup_{s\in[0,T]}\Vert D\tilde{b}^n(Y_{s}(y))-D\tilde{b}(Y_{s}(y))\Vert^{2p}\Big)^{\frac{1}{2}}.
\end{align}
 For the last inequality, we have used \eqref{eq zhai sup 01}. Note that we  have
$$
\sup_{y \in K}\E
   \sup_{s\in[0,T]}\Vert D\tilde{b}^n(Y_{s}(y))-D\tilde{b}(Y_{s}(y))\Vert^{2p} \le \E
   \sup_{s\in[0,T]} \sup_{y \in K}\Vert D\tilde{b}^n(Y_{s}(y))-D\tilde{b}(Y_{s}(y))\Vert^{ 2p}.
$$
Since we know that $\sup_{n\geq1}\Vert D\tilde{b}^n \Vert_0 + \Vert D\tilde{b} \Vert_0$ is bounded, see  \eqref{eq Db 0 pro}, we can prove that
$$
  \E
   \sup_{s\in[0,T]} \sup_{y \in K}\Vert D\tilde{b}^n(Y_{s}(y))-D\tilde{b}(Y_{s}(y))\Vert^{ 2p} \to 0
$$
as $n \to \infty$ by the dominated convergence theorem. To this purpose, we shall apply the fact that, for $\mathbb{P}$-a.s $\omega$,  
\begin{equation} \label{ma13}
\sup_{y \in K} \sup_{t \in [0,T]}|Y_{t}(y)(\omega)| = M(\omega)< \infty,
\end{equation}
and
$$
\lim_{n \to \infty} \sup_{|r| \le M(\omega)} \Vert D\tilde{b}^n(r)-D\tilde{b}(r)\Vert^{ }
=0,
$$
see  \eqref{eq b Db lim}. It remains to prove \eqref{ma13}.
 We note that that by \eqref{eq phi-y} we have, $\mathbb{P}$-a.s., for any $t \in [0,T]$, $y \in \R^d$,
\begin{align}\label{ma0}
 \psi(\phi_{t}(\psi^{-1}(y)))=Y_{t}(y).
 \end{align}
Recall that $\phi_{t}(x)$ solves the SDE \eqref{d566} with $s=0$. Since $b$ is globally bounded, we get that for $\mathbb{P}$-a.s. $\omega$, for any compact set $F \subset \R^d$, there exists $M_{F,T}(\omega)\in(0,\infty)$ such that
$$
\sup_{x \in F} \sup_{t \in [0,T]} |\phi_{t}(x)(\omega)| = M_{F,T}(\omega).
$$
Using also the continuity of $\psi$ and $\psi^{-1}$ we obtain \eqref{ma13}.   Finally we get
 \begin{equation}
\sup_{y \in K} \E\sup_{t\in[0,T]}\vert J_2(t,y,n)\vert^p \le C_T\Big(\E
   \sup_{s\in[0,T]}\sup_{y\in K}\Vert D\tilde{b}^n(Y_{s}(y))-D\tilde{b}(Y_{s}(y))\Vert^{2p}\Big)^{\frac{1}{2}}
 \to 0     
 \end{equation}
 as $n \to \infty.$

We still have to deal with $J_5(t,y,n)$. By \eqref{eq zhai sup 01}, we have, for any $y \in K,$ 
\begin{align}\label{eq J5}
   \E\sup_{s\in[0,T]}|J_5(s,y,n)\vert^p
\leq&
  C_p\E\Big(\int_{0}^{T}\lint_{\R^{d}}\Vert Dh^n(Y_{s}(y),z)-Dh(Y_{s}(y),z)\Vert^2\Vert DY^n_{s}(y)\Vert^2\nu(dz)ds\Big)^{p/2}\\
   &+
   C_p\E\Big(\int_{0}^{T}\lint_{\R^{d}}\Vert Dh^n(Y_{s}(y),z)-Dh(Y_{s}(y),z)\Vert^p\Vert DY^n_{s}(y)\Vert^p\nu(dz)ds\Big)\nonumber\\
\leq&
    C_p \Big(\E\Big(\int_{0}^{T}\lint_{\R^{d}}\sup_{y \in K}\Vert Dh^n(Y_{s}(y),z)-Dh(Y_{s}(y),z)\Vert^2\nu(dz)ds\Big)^{p}\Big)^{\frac{1}{2}}\nonumber\\
   &+ C_p 
    \Big(\E\Big(\int_{0}^{T}\lint_{\R^{d}}\sup_{y \in K}\Vert Dh^n(Y_{s}(y),z)-Dh(Y_{s}(y),z)\Vert^p\nu(dz)ds\Big)^{2}\Big)^{\frac{1}{2}}.\nonumber
\end{align} 
In order to get  $\lim_{n \to \infty }\sup_{y \in K} \E\sup_{s\in[0,T]}|J_5(s,y,n)\vert^p =0$ we need to apply the dominated convergence theorem in the previous formula. To this purpose, we first note that by \eqref{Eq hn h 01} for any compact set $K' \subset \R^d$ we have
\begin{align}\label{Eq hn h 01 V2 2025} 
   \lim_{n\to \infty} \sup_{x \in K'}\| D{h}^n(x,z)-D{h}(x,z)\|=0,\;\; z \not = 0.
\end{align}
And by \eqref{eq tianjian 000-1}, \eqref{eq SSSS 01} and \eqref{eq 4.44}, there exists $\mathcal{L}(z)$ such that
$$
\sup_{n\geq1}\sup_{x\in\mathbb{R}^d}\Vert Dh^n(x,z) \Vert + 
\sup_{x\in\mathbb{R}^d}\Vert Dh(x,z) \Vert
  \leq \mathcal{L}(z),\ \ \ \forall z\in\R^{d}\setminus\{0\},
$$
with $\int_{\R^{d}\setminus\{0\}}\mathcal{L}^p(z)\,\nu(dz)<\infty$ for any $p \ge 2$. Therefore, we get, $\mathbb{P}$-a.s., $n \ge 1,$ $z \not =0,$ $s \in [0,T]$, $y \in \R^d,$
$$
\Vert Dh^n(Y_{s}(y),z)-Dh(Y_{s}(y),z)\Vert^2
\le C\mathcal{L}^2(z).
$$
Recall \eqref{ma13} that for $\mathbb{P}$-a.s. $\omega$,
\begin{equation} 
\sup_{y \in K} \sup_{t \in [0,T]}|Y_{t}(y)(\omega)| = M(\omega) < \infty.
\end{equation}
 Hence, for $\mathbb{P}$-a.s. $\omega$, $z \not =0,$ $s \in [0,T]$,  we get
$$
\sup_{y \in K}\Vert Dh^n(Y_{s}(y)(\omega),z)-Dh(Y_{s}(\omega)(y),z)\Vert^2 \le \sup_{x \in B_{M(\omega)}}\Vert Dh^n(x,z)-Dh(Y_{s}(x,z)\Vert^2 \to 0
$$
as $n \to \infty$. We can apply the dominated convergence theorem and get
$$
\lim_{n \to \infty} \E\Big(\int_{0}^{T}\lint_{\R^{d}} \sup_{y \in K}\Vert Dh^n(Y_{s}(y),z)-Dh(Y_{s}(y),z)\Vert^2\nu(dz)ds\Big)^{p} =0.
$$
 Similarly, we get 
$$
\lim_{n \to \infty} \E\Big(\int_{0}^{T}\lint_{\R^{d}}\sup_{y \in K}\Vert Dh^n(Y_{s}(y),z)-Dh(Y_{s}(y),z)\Vert^p\nu(dz)ds\Big)^{2}  =0.
$$ 
We arrive at
\begin{equation} \label{c55}
    \lim_{n \to \infty }\sup_{y \in K} \E\sup_{s\in[0,T]}|J_5(s,y,n)\vert^p =0.
\end{equation}
 Finally let us come back to \eqref{eq 20220118 00}. 
By the previous estimates on $J_i(t,y,n)$ and the Gronwall lemma,   we obtain assertion \eqref{eq DY DYn lim}. 
   \end{proof}

\begin{proof}[Now let us prove \eqref{eq DYy DYy' 03}]
By \eqref{eq DY 01}, we have, for any $y,y'\in\mathbb{R}^d$,
\begin{align}\label{eq DYy DYy' 00}
     &\Vert DY_{t}(y)-DY_{t}(y')\Vert\\
     &\leq
     \Vert\int_{0}^{t}D\tilde{b}(Y_{s}(y))DY_{s}(y)-D\tilde{b}(Y_{s}(y'))DY_{s}(y')\,d s\Vert\nonumber\\
     &\quad +
    \big \Vert\int_{0}^{t}\lint_{\R^{d}}Dh(Y_{s-}(y),z)DY_{s-}(y)-Dh(Y_{s-}(y'),z)DY_{s-}(y')\tilde{\mu}(d s,dz)\big\Vert\nonumber\\
&\leq
     \big\Vert\int_{0}^{t}D\tilde{b}(Y_{s}(y))DY_{s}(y)-D\tilde{b}(Y_{s}(y'))DY_{s}(y)\,d s\big\Vert\nonumber\\
     &\quad +
     \big\Vert\int_{0}^{t}D\tilde{b}(Y_{s}(y'))DY_{s}(y)-D\tilde{b}(Y_{s}(y'))DY_{s}(y')\,d s\big\Vert\nonumber\\
     &\quad +
     \big\Vert\int_{0}^{t}\lint_{\R^{d}}Dh(Y_{s-}(y),z)DY_{s-}(y)-Dh(Y_{s-}(y'),z)DY_{s-}(y)\tilde{\mu}(d s,dz)\big\Vert\nonumber\\
     &\quad +
     \big\Vert\int_{0}^{t}\lint_{\R^{d}}Dh(Y_{s-}(y'),z)DY_{s-}(y)-Dh(Y_{s-}(y'),z)DY_{s-}(y')\tilde{\mu}(d s,dz)\big\Vert\nonumber\\
&\leq
     T[D\tilde{b}]_\gamma\sup_{s\in[0,T]}\Vert DY_{s}(y)\Vert \sup_{s\in[0,T]}|Y_{s}(y)-Y_{s}(y')\vert^\gamma+
     \Vert D\tilde{b}\Vert_0\int_{0}^{t}\Vert DY_{s}(y)-DY_{s}(y')\Vert\,d s\nonumber\\
     &\quad +
     \big\Vert\int_{0}^{t}\lint_{\R^{d}}Dh(Y_{s-}(y),z)DY_{s-}(y)-Dh(Y_{s-}(y'),z)DY_{s-}(y)\tilde{\mu}(d s,dz)\big\Vert\nonumber\\
     &\quad  +
     \big\Vert\int_{0}^{t}\lint_{\R^{d}}Dh(Y_{s-}(y'),z)DY_{s-}(y)-Dh(Y_{s-}(y'),z)DY_{s-}(y')\tilde{\mu}(d s,dz)\big\Vert.
\end{align}
A similar argument to that used for the proof of \eqref{eq J5 DYn DY 000} and \eqref{eq J42} shows that
\begin{align}\label{eq DYy DYy' 01}
&\E\sup_{t\in[0,T]} \big\Vert\int_{0}^{t}\lint_{\R^{d}}Dh(Y_{s-}(y'),z)DY_{s-}(y)-Dh(Y_{s-}(y'),z)DY_{s-}(y')\tilde{\mu}(d s,dz)\big\Vert^p\\
&\leq
C
 \E\Big(
         \int_{0}^{T}\Vert DY_{s}(y)-DY_{s}(y')\Vert^pds
      \Big),
\end{align}
where \eqref{eq Dh pro} has been used, and there exists $\delta>0$ such that
\begin{align}
    \label{eq DYy DYy' 02}
&\E\sup_{t\in[0,T]} \big\Vert\int_{0}^{t}\lint_{\R^{d}}Dh(Y_{s-}(y),z)DY_{s-}(y)-Dh(Y_{s-}(y'),z)DY_{s-}(y)\tilde{\mu}(d s,dz)\big\Vert^p\nonumber\\
&\leq
C_T
\Big(\E
   \sup_{s\in[0,T]}|Y_{s}(y)-Y_{s}(y')\vert^{2\delta p}\Big)^{\frac{1}{2}}.
\end{align}

\indent
Combining \eqref{eq zhai sup 01}, \eqref{eq Db 0 pro}, \eqref{eq b gamma pro}, \eqref{eq Yy Yy' 03}, and \eqref{eq DYy DYy' 00}, \eqref{eq DYy DYy' 01}, \eqref{eq DYy DYy' 02}, and choosing $q>1$ such that $\delta pq>2$, the Gronwall inequality and the $\rm H\ddot{o}lder$ inequality show that
\begin{align}\label{eq DYy DYy' 03-1}
\E\sup_{t\in[0,T]}\Vert DY_{t}(y)-DY_{t}(y')\Vert^p
\leq
  C\Big(\vert y-y'\vert^{p\gamma}+\vert y-y'\vert^{\delta p}\Big).
\end{align}
Here the constant $C$ is independent on $y$ and $y'$.\\
\indent The proof of \eqref{eq DYy DYy' 03} is complete.

\end{proof}

The proof of Theorem \ref{thm-stability} is complete.
\end{proof}

If $\eta:\mathbb{R}^d\to \mathbb{R}^d$ is a $C^1$-diffeomorphism, we will denote by $J\eta(x)=\det[D\eta(x)]$ its Jacobian determinant.
\begin{lemma}
\label{lemma change} Assume ${\divv\,}b(\cdot)=0$ in the sense of distributions. Then the
stochastic flow $\phi$ is $\mathbb{P}$-a.s. measure-preserving, i.e.,
$J\phi_{s,t}(x)=1$, for all $0\leq s\leq t\leq T$ and all
$x\in{\mathbb{R}}^{d}$, $\mathbb{P}$-a.s..
\end{lemma}

\begin{proof}This proof is adapted from the proof of \cite[Lemma 9]{FGP_2010-Inventiones} and \cite{FGP-2008}.
Let $b_{n}$ be $C_{\mathrm{b}}^{\infty}(\mathbb{R}^d,\mathbb{R}^d)$ vector fields that converges to
$b$ in $C_{\mathrm{b}}^{\beta}(\mathbb{R}^d,\mathbb{R}^d)$ and such that
${\divv\,}b_{n}=0$. The functions $b_{n}$ can be
constructed as in (6) of \cite{FGP_2010-Inventiones}.  Let $\phi^{n}$ be the associated
smooth diffeomorphism. Applying \cite[Theorem 3.4]{Kunita_2004}  and the
well known Liouville theorem, we get that the diffeomorphism
$\phi^{n}$ preserves the Lebsegue measure since $b_{n}$ is a
divergenceless vector field.

Then $\det [D\phi_{s,t}^{n}(x)]=1$ for all $0\leq s<t\leq T$ and
all $x\in{\mathbb{R}}^{d}$, $\mathbb{P}$-a.s.. Fix $x\in{\mathbb{R}}^{d}$ and $s \in[0,T]$. By
\eqref{stability2}, there exists a subsequence (possibly depending on
$x,s$ and still denoted by $D\phi_{s,t}^{n}$) such that $\mathbb{P}$-a.s.
\[
\sup_{s\leq t\leq T}\Vert D\phi_{s,t}^{n}(x)-D\phi_{s,t}(x)\Vert
^{2}\to0,\;\;\mbox{as}\;\;n\to\infty.
\]
We find that $J\phi_{s,t}(x)=1$, for every $t\in[s,T]$, $\mathbb{P}$-a.s.. Hence, by Theorem \ref{ww1}, $\phi$ is a  measure-preserving diffeomorphism, $\mathbb{P}$-a.s..
\end{proof}

\section {Definition and existence  of  weak solution to the stochastic transport equation}\label{sec-existence}

  Let us fix $T>0.$ Before  introducing  the   definition of solution $u$ of our transport equation, we point out two differences 
with respect to the  definition of solution given in  \cite[Definition 12]{FGP_2010-Inventiones}.

These differences are due to the fact that here we need to consider stochastic integrals with respect to compensated Poisson random measures of the following form
 \begin{equation}\label{s55}
 \int_0^t\lint_{B} \Bigl(\int_{\mathbb{R}^d}u(s-,x)\,[\theta(x+z)-\theta(x)]\, dx\Bigr)\tilde{\prm}(ds,dz),
 \end{equation}
  where $ \theta\in C_c^{\infty}(\mathbb{R}^d)$; see the next formula  \eqref{eqn-transport-Marcus}.

 Let us  clarify these  two differences.

\begin{differences}\label{differences-01}
\begin{trivlist}
\item[(i)] Instead of requiring, as in \cite{FGP_2010-Inventiones}, that  $u \in L^{\infty}([0,T]\times\mathbb{R}^d\times \Omega)$,
where $L^{\infty}([0,T]\times\mathbb{R}^d\times \Omega)$ is the set of equivalence classes of { all $ dt\otimes dx\otimes \mathbb{P}(d\omega)$ }essentially bounded functions, we only require that 
 $$
 u :   [0,T] \times \mathbb{R}^d \times \Omega \to \mathbb{R}
 $$
  is a $ {\mathcal B}([0,T]) \otimes {\mathcal B}(\mathbb{R}^d)\otimes {\call F}/\mathcal{B}(\mathbb{R})$-measurable function which is essentially bounded in the following sense (this is slightly stronger than requiring the  usual essential boundedness). 
  We are assuming that there exists $M>0 $ such
  \begin{align}\label{eq-esb0}
    \sup_{t\in[0,T]}\| u( t, \cdot,\omega) \|_{L^\infty(\mathbb{R}^d)} \leq M,\quad\mathbb{P}\text{-a.s.}
    \end{align}
   This means that there exists $M>0$ and 
a $\mathbb{P}$-full set $\Omega^\prime \subset \Omega$ such that for any $t \in [0,T]$, $\omega \in \Omega^\prime,$
\begin{align}\label{eq-esb03}
|u(t , x,\omega)| \le M, \;\;  \forall x \in A_{t, \omega},
\end{align}
where $A_{t, \omega}$ is a Borel set of $\mathbb{R}^d$ such that $\Leb( A_{t, \omega}^c)=0$.

\item[(ii)] We  require that,
for every test function $\theta\in C_c^{\infty}(\mathbb{R}^d)$,  the process $u_{\cdot}(\theta): \Omega \times [0,T] \to \mathbb{R}$ defined by 
 \begin{equation}\label{utt}
u_{t}(\theta)=\int_{\mathbb{R}^d}\theta(x)u(t,x)\,dx
\end{equation}
 is  $\mathbb{F}$-adapted and it has c\`{a}dl\`{a}g paths,  ${\mathbb P}$-a.s. Recall that $\mathbb{F}:=(\mathscr{F}_t)_{t\geq 0}$.

 This is different from  \cite[Definition 12]{FGP_2010-Inventiones} where $u_{\cdot}(\theta)$
  is an equivalent class in $L^{\infty}(\Omega\times[0,T])$ which
  has an $\mathbb{F}$-adapted c\`{a}dl\`{a}g modification (possibly depending on $\theta$).\\ 
  But note that the exceptional set in item (ii) above may depend on the test function $\theta$. 
\end{trivlist}

\end{differences}

\begin{remark}\label{rem-def-transport-weak-2}
    It should be not difficult to prove that item (ii) above implies the following stronger one with 
        $\theta \in L^{1}(\mathbb{R}^d)$. 
\begin{trivlist}
\item[(ii')] for every test function  $\theta\in L^{1}(\mathbb{R}^d)$, the process $\int_{\mathbb{R}^d}\theta(x)u(t,x)\,dx$, $t\in[0,T]$ is $\mathbb{F}$-adapted and with  c\`{a}dl\`{a}g paths, ${\mathbb P}$-a.s..
\end{trivlist}
This is quite standard, thanks to \eqref{eq-esb0}.  
 If $\theta\in L^{1}(\mathbb{R}^d)$ and  $(\theta_n)$ is a $C_c^{\infty}(\mathbb{R}^d)$-valued 
    sequence convergent to $\theta$ in $ L^{1}(\mathbb{R}^d)$, then  ${\mathbb P}$-a.s, 
    $u_{\cdot}(\theta_n) \to u_{\cdot}(\theta)$ uniformly on $[0,T]$. Therefore, the process $u_{\cdot}(\theta)$ is $\mathbb{F}$-adapted, see Corollary (i) to Theorem 1.14 in \cite{Rudin_RCA_1987} and
    it has c\`{a}dl\`{a}g paths  ${\mathbb P}$-a.s. 
Such a stronger version is needed because of, for instance, the second term in the second line of equality \eqref{eqn-transport-Marcus}.
    \end{remark}

 Beside items (i) and (ii)  stated above, we also use the following two lemmata. 

  \begin{lemma}\label{lem-transport-weak-2-(i)}
Assume that  a measurable and essentially bounded process
\begin{equation}\label{eqn-u1}
u:
[0,T]\times\mathbb{R}^d\times \Omega \to \mathbb{R},
\end{equation}
see  \eqref{eq-esb0},
satisfies the following condition, see  (ii) in Differences \ref{differences-01},   
for any  $\theta\in C_c^{\infty}(\mathbb{R}^d)$,  the process $u_{\cdot}(\theta):  [0,T]\times \Omega \to \mathbb{R}$
  is  $\mathbb{F}$-adapted and it has c\`{a}dl\`{a}g paths,  ${\mathbb P}$-a.s..

Then there exists a $\mathbb{P}$-full set $\Omega^\prime \subset \Omega$ such that for every $\omega \in \Omega^\prime$,
 for every test function  $\theta\in L^1(\mathbb{R}^d)$,
\begin{trivlist}
\item[(iii)] 
 the function
  \begin{equation}
\label{eqn-cadlag-001j}
u_\cdot(\theta)(\omega): [0,T]  \ni t \mapsto  \int_{\mathbb{R}^d}\theta(x)u( t,x,\omega)\,dx
\end{equation}
is c\`{a}dl\`{a}g.
\end{trivlist}
\end{lemma} 

\begin{proof}   Since $L^1(\mathbb{R}^d)$ is separable we find that $C_c^{\infty}(\mathbb{R}^d )\subset L^1(\mathbb{R}^d)$ is separable as well.  Now we argue similarly to  Remark \ref{rem-def-transport-weak-2}. 

Let us choose  a countable dense  subset $\mathcal{A}$ of $C_c^{\infty}(\mathbb{R}^d)$ with respect to the $L^1$-norm. Thus for  any $\eta \in L^1(\mathbb{R}^d) $ there exists
 {$(\theta_n)_{n=1}^\infty \subset \mathcal{A}$} such that
\begin{equation}\label{app1}
\lim_{n \to \infty} \int_{\R^d}|\theta_n(x)- \eta(x)|dx = 0. 
\end{equation}
Then
there exists a $\mathbb{P}$-full set $\Omega^\prime \subset \Omega$ such that for every $\omega \in \Omega^\prime$,
 for every test function  $\theta\in \mathcal{A}$, condition  (iii)  is satisfied.    
 Moreover we can assume that for any $\omega \in \Omega^\prime$ property  \eqref{eq-esb0} holds.

Let us choose and fix an arbitrary $\eta \in L^1(\mathbb{R}^d)$  and consider the $\mathbb{F}$-adapted process
\begin{equation}
u_{\cdot}(\eta)= \{[0,T] \ni  t \mapsto u_t(\eta)=\int_{\mathbb{R}^d}\eta(x)u( t,x,\cdot)\,dx  \in \mathbb{R}\}
\end{equation}
defined on our stochastic basis.  This process exists by the assumptions of this lemma and Remark \ref{rem-def-transport-weak-2}. Let us choose  {$(\theta_n)_{n=1}^\infty \subset \mathcal{A}$} such that \eqref{app1} holds.
 For any $\omega \in \Omega^{\prime}$ using \eqref{eq-esb0} we have:
\begin{equation}\label{eq 2025-07-16 01}
\begin{aligned}
\sup_{t \in [0,T]} |u_t(\theta_n)(\omega) - u_t(\eta)(\omega)|
& \le \sup_{t \in [0,T]}  \int_{\mathbb{R}^d} |u( t,x,\omega)|\,  |\theta_n(x)-\eta(x)|dx
\\
&\le 
M \int_{\mathbb{R}^d}|\theta_n(x)-\eta(x)|dx \to 0
\end{aligned}
\end{equation}
  as $n \to \infty$. Since $\big(u_t(\theta_n)(\omega) \big)_{n\geq1}$ converges  uniformly w.r.t. $t\in [0,T]$ to $u_t(\eta)(\omega)$ as $n \to \infty$ and each function $[0,T] \ni t \mapsto u_t(\theta_n)(\omega) $ is  c\`{a}dl\`{a}g on $[0,T]$
  we infer that $t \mapsto u_t(\eta)(\omega)$ is  c\`{a}dl\`{a}g as well.
   \end{proof}

We prepare the next lemma with some considerations.

\begin{remark} \label{d88}
The previous Lemma \ref{lem-transport-weak-2-(i)} implies that  for any $\omega \in \Omega^\prime$, any test function $\theta \in C_c^{\infty}(\mathbb{R}^d)$ and 
 any $t \in [0,T)$:
\begin{gather*}
\lim_{s \to t^+} u_{s}(\theta)(\omega)=\int_{\mathbb{R}^d}\theta(x)u(t,x,\omega) dx = u_t(\theta)
\end{gather*}
and %Moreover, 
for any $\omega \in \Omega^\prime$,  any test function $\theta \in C_c^{\infty}(\mathbb{R}^d)$ and 
 any $t \in (0,T]$, there exists a finite limit
\begin{equation} \label{dqq}
\lim_{s \to t^-} u_{s}(\theta)(\omega).
%\int_{\mathbb{R}^d}\theta(x) g(t, \omega x)dx
%u(\omega,t,x) dx.
\end{equation}
We denote such limit by $ u_{t-}(\theta)(\omega)$.  To simplify,  with \textbf{some abuse of notation} we will  sometimes write
$$
u_{t-}(\theta) = \int_{\mathbb{R}^d}\theta(x)u(t-,x) dx.
$$
However, note that $u( t-, x,\omega)$ is not well-defined.

It is clear that, for any test function $\theta \in C_c^{\infty}(\mathbb{R}^d)$ the process $ (u_{t-}(\theta))_{t\in (0,T]}$ is $\mathbb{F}$-adapted and left-continuous with finite right-limits, ${\mathbb P}$-a.s.. In fact,  for the path regularity one can consider $\omega \in \Omega^\prime$.

Let us fix $t\in (0,T]$  and $\omega\in \Omega^\prime$. By Lemma \ref{lem-transport-weak-2-(i)} again,  
 for any $\theta \in L^1(\R^d)$, there exists a finite limit
\begin{gather*}
\lim_{s \to t^-} u_{s}(\theta)(\omega)
%\gamma_{t, \omega}(\theta).
\end{gather*}
which we still denote by $ u_{t-}(\theta)(\omega).$

%\vskip 0.2cm

Note that identifying the dual space $(L^1(\R^d))^\prime$ with $L^{\infty}(\R^d)$ we infer  that 
\[\{  u_{s}(\cdot )(\omega): 0 \leq  s  < t\} \subset (L^1(\R^d))^\prime.\]

Hence, by the Banach–Steinhaus Theorem we deduce that $ u_{t-}(\cdot)(\omega)\in (L^1(\R^d))^\prime$. It follows, see Theorem 6.16 in \cite{Rudin_RCA_1987}, that there exists a unique element 
$g_{t, \omega} \in L^{\infty}(\R^d)$
such that 
\begin{equation} \label{gt33}
u_{t-}(\theta)(\omega) = \int_{\mathbb{R}^d}\theta(x) g_{t,\omega}(x) dx,\;\; \forall\theta \in L^1(\R^d).
\end{equation}
In particular, \eqref{gt33} holds for any test function $\theta \in C_c^{\infty}(\mathbb{R}^d)$, $t \in (0,T]$, $\omega \in \Omega^\prime.$ 
But we do not assert any measurability  of the function 
$(0,T]\times \Omega \ni (t,\omega) \mapsto  g_{t, \omega} \in L^{\infty}(\R^d)$. On the other hand, $\Vert g_{t, \omega}\Vert_{L^{\infty}(\R^d)} \leq M$. 
%}
\end{remark}

   In order to give a precise meaning to stochastic integrals as in \eqref{s55}
   %$$
%\int_0^t\lint_{B} \Bigl(\int_{\mathbb{R}^d}u(s-,x)\,[\theta(x+z)-\theta(x)]\, %dx\Bigr)\tilde{\prm}(ds,dz)
%$$
%appearing in \eqref{eqn-transport-Marcus}
we prove
%also need
the following lemma. We will use the notation $\theta (\cdot + z)=\tau_{-z}\theta$ and  
 the predictable $\sigma$-field $\mathcal P$ on $[0,T] \times \Omega$.

 \begin{lemma} \label{zz} Let $T>0.$ Assume that $\Omega^\prime$ is the event from Lemma \ref{lem-transport-weak-2-(i)} and  Remark \ref{d88}. 
 Let $\theta \in C_c^{\infty}(\mathbb{R}^d)$. Then the following assertions are satisfied. 
\begin{trivlist}
 \item[(i)] For every $(t, \omega) \in (0,T] \times \Omega^\prime$, the function 
 \[\R^d \ni z \mapsto  u_{t-} (\theta (\cdot + z)) \in \mathbb{R} \] is continuous.  
 \item[(ii)] The map 
 \begin{equation}\label{eqn-u(omega,t-,z)}
  [0,T] \times \R^d \times \Omega^\prime \ni (t,z,\omega) \mapsto  u_{t-}(\theta (\cdot + z))(\omega)  \in \mathbb{R}
 \end{equation}
 is $\mathcal P\otimes {\mathcal B}(\mathbb{R}^d)/{\mathcal B}(\mathbb{R})$-measurable.
\end{trivlist}
 \end{lemma}

 \begin{proof} To prove (i), let us choose and fix $(t, \omega) \in (0,T] \times \Omega^\prime$.  Using  \eqref{gt33} the assertion follows  by the  dominated convergence theorem.  Indeed, we have
 $$
 u_{t-}(\theta (\cdot + z)) = \int_{\mathbb{R}^d}  g_{t,\omega}(x) \theta(x +z) dx.
 $$
Now let us prove (ii). We first note  that  for any fixed $z \in \R^d$, the process    : $[0,T] \times \Omega^\prime \to u_{t-} (\theta (\cdot + z))(\omega)$  is $\mathbb{F}$-adapted and left-continuous with finite right-limits; 
 see Remark \ref{d88}.
 Hence, by the definition of predictable $\sigma$-field, see for instance, page 21 in \cite{Ikeda+Watanabe_1989}, 
 we 
 infer  that for any $z \in \R^d$,  the function  
 \[
\ [0,T]\times \Omega^\prime \ni (t,\omega)
\mapsto  u_{t-}(\theta (\cdot + z))(\omega) \in \mathbb{R} \]
   is  $\mathcal P$-measurable. Thus, by  (i) and Definition 4.50 in \cite{Aliprantis},  this function is a Carath\'eodory function and therefore, by 
 Lemma 4.51 in \cite{Aliprantis}  to the map \eqref{eqn-u(omega,t-,z)}
  $(\omega,t,z) \mapsto  u_{t-}(\theta (\cdot + z)) (\omega) $,
 we deduce  the assertion.
\end{proof}

After this preparation, we introduce our definition of solutions.

%---------------------
\subsection{Definition of  solutions to the stochastic transport equation}\label{subsection-def-solution}
  The definition involves  a general L\'evy process $(L_t)_{t\geq 0}$.
  We take into account the discussion in Remark \ref{d88}.

\begin{definition}\label{def-transport-weak-1}
 Assume that  $b\in L_{loc}^{1}(\mathbb{R}^{d};\mathbb{R}^{d})$ satisfies $\divv\,b\in L_{loc}^{1}(\mathbb{R}^{d})
$.  If  $u_{0}\in L^{\infty }(\mathbb{R}^{d})$ then 
a weak$^\ast$-$\mathrm{L}^{\infty}$-solution to the stochastic transport equation problem \eqref{eqn-transport-strong}
is a  
{measurable and essentially bounded}  stochastic process 
\begin{equation}\label{eqn-u2}
u:
 [0,T]\times\mathbb{R}^d\times \Omega \to \mathbb{R}
\end{equation}
in the sense that 
{ for some $M>0 $, 
  \begin{align}\label{eq-esb1}
    \sup_{t\in[0,T]}\| u( t,  \cdot,\omega) \|_{L^\infty(\mathbb{R}^d)} \leq M,\quad\mathbb{P}\text{-a.s.}
\end{align}
 (cf. \eqref{eq-esb03})} and the following two conditions are satisfied. 
 \begin{trivlist}
 \item[(ds1)] For every test function $\theta\in C_c^{\infty}(\mathbb{R}^d)$, the process 
 $u_{\cdot}(\theta):  [0,T]\times \Omega \to \mathbb{R}$ defined by 
 \begin{equation}\label{utt-2}
u_{t}(\theta)=\int_{\mathbb{R}^d}\theta(x)u(t,x)\,dx, \;\; t \in [0,T], 
\end{equation}
   is $\mathbb{F}$-adapted and c\`{a}dl\`{a}g,  $\mathbb{P}$-almost surely. 
  \item[(ds2)] For every test function $\theta\in C_c^{\infty}(\mathbb{R}^d)$, for every $t \in [0,T]$, the following identity is satisfied $\mathbb{P}$-almost surely, 
\begin{equation}\label{eqn-transport-weak}
\begin{aligned}
       \int_{\mathbb{R}^d}u(t,x)\,\theta(x)\,dx=&\int_{\mathbb{R}^d}u_{0}(x)\,\theta(x)\,dx+\int_0^t\int_{\mathbb{R}^d}u(s,x)\,[b(x)\cdot  D\theta(x)+\divv\, b(x)\,\theta(x)]\,dx\,d s\nonumber\\
       &+\sum_{i=1}^d\int_0^t\Bigl( \int_{\mathbb{R}^d}u(s-,x)D_i\theta(x)\,dx      \Bigr)\diamond d L_s^i.
       \end{aligned}
\end{equation}
\end{trivlist}
\end{definition}
Let us emphasise that the exceptional set in item (ds2) may depend on both $\theta$ and $t \in [0,T]$.  On the other hand, the exceptional set in item (ds1) does not depend on $\theta$ (cf. Lemma \ref{lem-transport-weak-2-(i)}). 

\begin{remark}\label{rem-defi}
Let us observe that
given a Borel measurable function $f:\mathbb{R}^d\to \mathbb{R}$, a function
\begin{equation}\label{eqn-Phi}
\Phi(f):[0,T]\times\mathbb{R}^d\times\mathbb{R}^d \ni
(t,z, y)  \to \Phi(t,z,f)(y):=  f(y-tz) \in  \mathbb{R}
\end{equation}
 is a 
 {weak} solution to the following transport equation
\begin{align}
&\frac{ \partial \Phi}{ \partial t}(t,z,f)=-\sum_{i=1}^d z_iD_i\Phi(t,z,f),\\
&\Phi(0,z,f)=f \mbox{ on }\mathbb{R}^d,
\end{align}
where  
{$D_i\Phi(t,z,f)(y)$ is the weak derivative of $\Phi(t,z,f)$ at $y$  in the $i$-th coordinate direction. }

Let us denote
\begin{equation}\label{eqn-Phi-1}
\Phi(z,f)(\cdot)=\Phi(1,z,f)(\cdot)=f(\cdot-z).
\end{equation}

Thus,   a strong version of the Marcus canonical integral in the stochastic transport equation
 \eqref{eqn-transport-strong}  is of the following form
\begin{align}
-\sum_{i=1}^de_i\cdot Du(s-,x)\diamond d L_s^i=&\lint_{B} \bigl[ \Phi(z,u(s-))(x)-u(s-,x)\bigr]\, \tilde{\prm}(ds,dz)\\
&+\int_{B^{\mathrm{c}}} \bigl[\Phi(z,u(s-))(x)-u(s-,x)\bigr]\, \prm(ds,dz)\\
&+\lint_{B}\bigl[ \Phi(z,u(s))(x)-u(s,x)+\sum_{i=1}^d z_iD_iu(s,x)\bigr]\, \nu(dz)\, ds,
\end{align}
where $B$ is the unit ball in $\mathbb{R}^d$, cf. \cite{HartmannPavlyukevich23}.
%, see \eqref{eqn-B_1}.
In other words, in case when  $u$ is a classical regular solution, equation   \eqref{eqn-transport-strong} should be understood in the following way

\begin{align}\label{eqn-transport-strong-2}
\begin{aligned}
  &\frac{\partial{u(s,x)}}{\partial s}+b(x)\cdot { D} u(s,x)\,ds\\
 =&\lint_{B} \bigl[ \Phi(z,u(s-))(x)-u(s-,x)\bigr]\, \tilde{\prm}(ds,dz)
 +\int_{B^{\mathrm{c}}} \bigl[\Phi(z,u(s-))(x)-u(s-,x)\bigr]\, \prm(ds,dz)\\
&+\lint_{B}\bigl[ \Phi(z,u(s))(x)-u(s,x)+\sum_{i=1}^d z_iD_iu(s,x)\bigr]\, \nu(dz)\, ds,
 \\
  & u(0,x)=u_0(x),\ \ x\in\mathbb{R}^d.
  \end{aligned}
\end{align}
\end{remark}

In what follows,  we  rewrite equation \eqref{eqn-transport-strong-2}  in the weak form, and thus  we get an  equivalent version of Definition \ref{def-transport-weak-1}.

\begin{definition}\label{def-transport-weak-2}
A weak$^\ast$-$\mathrm{L}^{\infty}$-solution to the stochastic transport equation
\eqref{eqn-transport-strong}  is
\begin{trivlist}
\item[(o)] a measurable and essentially bounded function
\begin{equation}\label{eqn-u is bounded}
u:
[0,T]\times\mathbb{R}^d\times \Omega \to \mathbb{R}
\end{equation}
\end{trivlist}
{
in the sense  that there exists $M>0 $ satisfying
  \begin{align}\label{eq-esb2}
    \sup_{t\in[0,T]}\| u(t, \cdot, \omega) \|_{L^\infty(\mathbb{R}^d)} \leq M,\quad\mathbb{P}\text{-a.s.}
\end{align}
}
together with the following two additional conditions
\begin{trivlist}
\item[(i)] for every test function  $\theta\in C_c^{\infty}(\mathbb{R}^d)$, the process $u_t(\theta)=\int_{\mathbb{R}^d}\theta(x)u(t,x)\,dx$, $t\in [0,T]$ is  $\mathbb{F}$-adapted and it has   c\`{a}dl\`{a}g paths, $\mathbb{P}$-a.s.;
\item[(ii)] for every test function  $\theta\in C_c^{\infty}(\mathbb{R}^d)$, the process $\int_{\mathbb{R}^d}\theta(x)u(t,x)\,dx$,  $t\in [0,T]$,  satisfies that, for every $t\in [0,T]$, $\mathbb{P}$-a.s.,
\begin{align}\label{eqn-transport-Marcus}
      &\int_{\mathbb{R}^d}u(t,x)\,\theta(x)\,dx\\
   & =
    \int_{\mathbb{R}^d}u_{0}(x)\,\theta(x)\,dx
       +\int_0^t\int_{\mathbb{R}^d}u(s,x)\,[b(x)\cdot  D\theta(x)+\divv\, b(x)\,\theta(x)]\,dx\,ds\\
       &\quad +\int_0^t\lint_{B} \Bigl(\int_{\mathbb{R}^d}u(s-,x)\,[\theta(x+z)-\theta(x)]\, dx\Bigr)\tilde{\prm}(ds,dz)\\
       &\quad+\int_0^t\int_{B^{\mathrm{c}}} \Bigl(\int_{\mathbb{R}^d}u(s-,x)\,[\theta(x+z)-\theta(x)]\, dx\Bigr) \prm(ds,dz)\\
       &\quad+\int_0^t\lint_{B}\Bigl(\int_{\mathbb{R}^d} u(s,x)\,[\theta(x+z)-\theta(x)-\sum_{i=1}^dz_iD_i\theta(x)]\,dx\Bigr)\,\nu(dz)\, ds.
\end{align}
The $\mathbb{P}$-a.s. means as usual that we have a 
$\mathbb{P}$-full  set involved. 
Actually, one could make this $\mathbb{P}$-full set independent of $t$ and $\theta$ by Lemma \ref{lem-transport-weak-2-(i)}.
\end{trivlist}
\end{definition}

Recall that
according to Remark \ref{d88} we should write more precisely
$$
\int_{\mathbb{R}^d}u(s-,x)\,[\theta(x+z)-\theta(x)]\, dx = u_{s-}(\theta (\cdot +z) - \theta).
$$

\addtocounter{theorem}{-1}

\begin{remark}\label{rem-weak solution} Following the terminology used in \cite{FGP_2010-Inventiones} we should have used a name ``weak-$\mathrm{L}^\infty$ solution. However we have decided to use a different name, i.e., weak$^\ast\text{-}\mathrm{L}^\infty$, since our $L^\infty$-valued process $u$ is c\`{a}dl\`{a}g   with respect to the weak$^\ast$ topology on $L^\infty$.
\end{remark}

We also have the following notion of uniqueness of the stochastic transport equation.

\begin{definition}\label{def-transport-weak-uni} We say that uniqueness holds for the stochastic transport equation if for a fixed $u_0 \in L^{\infty}(\R^d)$
 given any two weak$^\ast$-$\mathrm{L}^{\infty}$-solutions $u^1$ and $u^2$
 we have:

 for any  test function  $\theta\in C_c^{\infty}(\mathbb{R}^d)$, it holds:
$$
u^{1}_t(\theta) = u^{2}_t(\theta),
$$
$\mathbb{P}$-a.s., for any $t \in [0,T]$.

\end{definition}

%-------------------

%-------------------------

 %--------------------
 
Let us recall that  by  $\theta^\prime(a)=d_a\theta\in\mathscr{L}(\mathbb{R}^d,\mathbb{R})$ we denote  the Fr\'echet derivative of $\theta$ at $a\in \mathbb{R}^d$, see \cite{Cartan_1971-DC}. By $D\theta(a)$ or $\nabla \theta (a)$, we denote the gradient  of $\theta$ at $a\in \mathbb{R}^d$, i.e., the unique element of $\mathbb{R}^d$ such that
\[
\theta^\prime(a)(y)=\langle D \theta (a), y \rangle=\langle \nabla \theta (a), y \rangle, \;\; y \in \mathbb{R}^d.
\]

\subsection{Existence of weak solution to the stochastic transport equation}

The following theorem is one of the main results  of our paper. It is a result about the existence
 of weak solutions given by the stochastic flow. We will also assume some additional integrability condition on $\divv b$
(interpreted in the distributional sense).

\begin{theorem}\label{thm-transport equation}
Assume that  $\alpha \in  [1,2)$ or $\alpha \in (0,1)$. Assume that $L = (L_t)_{t\geq 0}$ is a L\'evy  process  satisfying  Hypotheses \ref{hyp-nondeg1} in the former case or Hypothesis \ref{hyp-nondeg3} in the latter case. Assume that $\beta\in(0,1)$ satisfies condition \eqref{eqn-beta+alpha half} and
 $b\in C_{\mathrm{b}}^{\beta}(\mathbb{R}^d,\mathbb{R}^d)$ is  such that  \[\divv b \in L^1_{\loc}(\mathbb{R}^d).\]
 If a function $u_0:\mathbb{R}^d \to \mathbb{R}$ is  Borel measurable and essentially bounded, then the  stochastic process $u$ defined by the following formula 
\begin{equation}\label{eqn-def-u}
 u(t,x,\omega):=u_0(\phi_t^{-1}(\omega)x), \;\; t\in [0,T], x \in\mathbb{R}^d
\end{equation} %
is a weak$^\ast$-$\mathrm{L}^{\infty}$-solution to the problem \eqref{eqn-transport-Markus} in the sense of Definition \ref{def-transport-weak-2}.
 \end{theorem}

 \begin{remark}
     By applying arguments similar to those in Lemma  \ref{ef}, our solutions to the stochastic transport equation and the associated stochastic flow can be extended to the infinite time horizon $t\geq0$. However, to simplify the presentation, we will focus on  the interval $[0,T]$, for an arbitrary $T>0$.
 \end{remark}

 \begin{proof} [Proof of Theorem \ref{thm-transport equation}] 

 \def\cia2{Let us choose and fix an initial data
  $u_0:\mathbb{R}^d \to \mathbb{R}$ which is a  Borel measurable and bounded function. We also choose and fix   a  test function  $\theta\in C_c^{\infty}(\mathbb{R}^d)$. Put
  \begin{equation}\label{eqn-K=supp theta}
   K:=\supp \theta.
  \end{equation}
} 
  
 Let $T>0.$ Let us define a random field $u$ by formula \eqref{eqn-def-u}. 
  It is obvious    that 
  $$
 u : [0,T] \times \mathbb{R}^d \times \Omega   \to \mathbb{R}
 $$
  is a $  {\mathcal B}([0,T]) \otimes {\mathcal B}(\mathbb{R}^d)\otimes {\call F}/\mathcal{B}(\mathbb{R})$-measurable function using  the measurability properties of our stochastic flow given in Section \ref{sec-regular flow}.

 We prove directly that 
  \begin{align} \label{si1}
    \sup_{t\in[0,T]}\| u(t, \cdot, \omega) \|_{L^\infty(\mathbb{R}^d)} \leq M,\quad\mathbb{P}\text{-a.s.}
\end{align}
 holds. This gives that 
 condition  (o) of Definition \ref{def-transport-weak-2} is verified.

 \def\cia2{ It is obvious that $u$
is bounded. Because for all $(t,x) \in \coma{[0,\infty)} \times \mathbb{R}^d$, the map $\Omega \ni \omega \mapsto \phi_t^{-1}(\omega)x \in \mathbb{R}^d$ is $\mathscr{F}_t$ measurable, so
is the function $ u(t,x,\cdot):\Omega \ni \omega \mapsto u(t,x,\omega) \in \mathbb{R}$. Hence, for every $T>0$, the function  $u: [0,T] \times \mathbb{R}^d \times \Omega \to \mathbb{R}$ is $\mathscr{B}([0,\infty) \times \mathbb{R}^d)\otimes \mathscr{F}_T$ measurable.
}

Since $u_0:\mathbb{R}^d\rightarrow \mathbb{R}^d$ is essentially bounded, we have 
$$|u_0(y)| \leq M \quad \text{for a.e. } y \in \mathbb{R}^d.$$
Consider the set where $ u_0 $ is bounded:
$$
E = \{ y \in \mathbb{R}^d : |u_0(y)| \leq M \}.
$$
Then $\Leb(E^c) = 0$. To simplify notation let us define 
$$
h_t(x)(\omega) = \phi_t^{-1}(x)(\omega),
$$
so that we are considering $u(\omega,t,x):=u_0(h_t(x)(\omega)), \;\; t\in [0,T], x \in\mathbb{R}^d, \omega \in \Omega.
$

Since by Theorem \ref{thm-homeomorphism}, there exists an almost sure event $\Omega'$ such that for every $\omega\in\Omega'$, $t\in[0,T]$, the mapping $x\mapsto h_t(x)(\omega)$ is a homeomorphism, we infer  the preimage 
$$
B_{t,\omega}:=[h_t(\omega)]^{-1}(E^c) = \{ x \in \mathbb{R}^d : h_t(\omega) (x) \in E^c \}
$$
has Lebesgue measure zero. That is $\Leb(B_{t,\omega})=0$, for every $\omega\in\Omega'$ and $t\in[0,T]$. For each $t\in[0,T]$, $\omega\in\Omega'$, let $A_{t,\omega}:=B_{t,\omega}^c$. Then for all $x\in A_{t,\omega}$, $|u(t,x,\omega)| \le M$. Consequently, $\sup_{t\in[0,T]}\| u(t,\cdot,\omega)\|_{L^{\infty}(\mathbb{R}^d)}\leq M$, for all $\omega\in\Omega'$. So
$u$ satisfies condition (o) of Definition \ref{def-transport-weak-2}.

\def\cia3{Let us fix $\omega\in\Omega'$. For every $t\in[0,T]$, every test function $\theta\in C_c^{\infty}(\mathbb{R}^d)$ with $\|\theta\|_{L^1}\leq 1$, we deduce
\begin{align*}
   \left| \int_{\mathbb{R}^d} u(t,\omega,x)\theta(x)dx\right|&=\left|\int_{\mathbb{R}^d} u_0(\phi_t^{-1}(\omega)x)\theta(x)dx\right|\\
    &\leq \left|\int_{A_{t,\omega}^c}u_0(\phi_t^{-1}(\omega)x)\theta(x)dx\right|+\left|\int_{A_{t,\omega}}u_0(\phi_t^{-1}(\omega)x)\theta(x)dx\right|\\
    &\leq M\|\theta\|_{L^1}.
\end{align*}
By the duality between $L^1(\mathbb{R}^d)$ and $L^\infty(\mathbb{R}^d)$, 
$$\| u(t, \omega, \cdot) \|_{L^\infty(\mathbb{R}^d)} = \sup_{\theta \in L^1(\mathbb{R}^d), \|\theta\|_{L^1} \leq 1} \left| \int_{\mathbb{R}^d} u(t, \omega, x) \theta(x) \, dx \right|.$$
Since $ C_c^\infty(\mathbb{R}^d) $ is dense in $L^1(\mathbb{R}^d)$, we have
$$\| u(t, \omega, \cdot) \|_{L^\infty(\mathbb{R}^d)} = \sup_{\theta \in C_c^{\infty}(\mathbb{R}^d), \|\theta\|_{L^1} \leq 1} \left| \int_{\mathbb{R}^d} u(t, \omega, x) \theta(x) \, dx \right|.$$
Hence  we obtain 
Then   $u$ satisfies condition (o) of Definition \ref{def-transport-weak-2}.
}
 Moreover,
 by the Change of  Measure Theorem,  we infer that
\begin{equation}\label{eqn-change of variables}
\int_{\mathbb{R}^d}u(t,x)\theta(x)\,dx=\int_{\mathbb{R}^d}
u_{0}(y)\theta(\phi_{t}(y))J_{t}(y)\,dy,
\end{equation}
for any test function $\theta\in C_c^\infty(\R^d)$,
where $J_t$ is the corresponding Jacobian determinant
\begin{equation}\label{eqn-J_t}
J_t(y)=\det[D\phi_{t}(y)], \;\; (t,y) \in [0,T]\times \mathbb{R}^d.
\end{equation}
By using the properties of the processes  $\phi_{t}(y)$ and $J_{t}(y)$ proved before, see Theorems \ref{d32} and \ref{ww1}, this easily will yield condition (i) in Definition \ref{def-transport-weak-2}.

  We will show  the $\mathbb{R}$-valued process
  \[
  \int_{\mathbb{R}^d}\theta(x)u(t,x)\,dx, \;\;\; t\geq 0, \]
  satisfies  equality \eqref{eqn-transport-Marcus}. For this purpose, let us introduce some standard techniques and notation. First of all we set
\[B(r):=\{x\in\mathbb{R}^d:\; |x|<r\}, \;\; r>0.
\] Secondly, let  $\bar{\vartheta}:\mathbb{R}^d\to [0,1] \subset  \mathbb{R}$ be a symmetric   $C^\infty_c$-class function  such that
\begin{equation}\label{eqn-vartheta-properties}
    \mathds{1}_{B(\frac{1}{4})} \leq \bar{\vartheta} \leq \mathds{1}_{B(2)}
 \mbox{ and }\quad c_d\int_{\mathbb{R}^d}\bar{\vartheta}(x)dx=1,
\end{equation}
where $c_d$ is a positive constant dependent only on $d$.

Define  $\vartheta(x)=c_d\, \bar{\vartheta}(x)$, $x\in\mathbb{R}^d$.  We then introduce an approximation of functions $(\vartheta_\eps)_{\eps >0}$ defined by the following formula
\begin{align}\label{def-vartheta-eps}
\vartheta_\eps(x)=\,\eps^{-d}\vartheta(\frac{x}{\eps}), \;\; x \in \mathbb{R}^d,\;\; \eps >0.
\end{align}
Next we  define a $C_{\mathrm{b}}^\infty$ approximation $b^\eps$ of the vector field $b$ by formula
\begin{align}\label{def-b^eps}
b^\eps=b\ast \vartheta_{\eps}, \mbox{ i.e., }     b^{\eps}(x)=\int_{\R^d}b(x-y)\,\vartheta_{\eps}(y)\,dy,\;\;\; x\in \mathbb{R}^d.
\end{align}
For the completeness of our exposition, let us state the following result, which holds because $b\in C_{\mathrm{bu}}(\mathbb{R}^d)$, i.e., $b$ is bounded and uniformly continuous,  and $\divv b \in L^1_{\loc}(\mathbb{R}^d)$.
\begin{lemma}\label{lem-b^eps-to-b}
Using the notation introduced above, we have
\begin{equation} \label{eqn-b^eps-to b L^infty}
b^\eps  \to b \mbox{ in } \; C_{\mathrm{b}}^{\beta^\prime}(\mathbb{R}^d, \R^d).
\end{equation}
 with $0< \beta^\prime < \beta$  
and 
\begin{equation} \label{eqn-b^eps-to-b-L^1}
\divv b^\eps  \to  \divv b \mbox{ in } L^1_{\loc}(\mathbb{R}^d).
\end{equation} 
\end{lemma}
In the sequel, we will assume that, cf. Remark \ref{okk}, 
\begin{gather*}
 \beta^\prime + \frac{\alpha}{2} >1.
\end{gather*}
Let $\phi^{\eps}_{s,t}$ be the stochastic flow associated with the equation  \eqref{eqn-SDE}  %{flow_eq1}
 with the drift  $b$ replaced by $b^{\eps}$, i.e.,
 \begin{align} \label{eqn-phi^eps}
\phi^{\eps}_{s,t}(y)=y+\int_{s}^{t}b^{\eps}(\phi_{s,r}^{\eps}(y))\, dr+L_{t}-L_{s}, \;\; t\geq s, \;\;  y \in\mathbb{R}^d.
\end{align}
Denote by $(\phi^{\eps}_{s,t})^{-1}$ the inverse of $\phi^{\eps}_{s,t}$. Then similar arguments as proving \eqref{de}, we have
\begin{equation} \label{de 2025 0102 01}
(\phi^{\eps}_{s,t})^{-1}(y) =  y - \int_s^t b^\eps((\phi^{\eps}_{r,t})^{-1}(y)) dr - (L_t - L_s), \;\;  0\le s \le t.
\end{equation}

Finally, let $J_{s,t}^{\eps}$ and $(J^{\eps}_{s,t})^{-1}$ be the corresponding Jacobian to $\phi^{\eps}_{s,t}$ and $(\phi^{\eps}_{s,t})^{-1}$, i.e.,
\begin{equation}\label{eqn-J_t^eps 2025}
J^\eps_{s,t}(y)=\det[D\phi^\eps_{s,t}(y)], \;\; (t,y) \in [s,\infty) \times \mathbb{R}^d,
\end{equation}
\begin{equation}\label{eqn-J_t^eps 2025 01}
(J^{\eps}_{s,t})^{-1}(y)=\det[D(\phi^{\eps}_{s,t})^{-1}(y)], \;\; (s,y) \in [0,t] \times \mathbb{R}^d.
\end{equation}

We begin with the following lemma which is the  stochastic version of   \cite[Theorem 3.6.1]{Cartan_1971-DC}.
\begin{lemma}\label{lem-A1}
There exists  a $\mathbb{P}$-full set  $\Omega^{(0)}$ such that
the following equality holds  for every $\omega \in \Omega^{(0)}$ and every $\eps \in (0,1]$, for any fixed $s\geq0$,
     \begin{align}\label{eqn-jacobian}
      J_{s,t}^{\eps}(x)= 1 + \int_s^t \divv b^{\eps}(\phi^{\eps}_{s,r}(x))\,J_{s,r}^{\eps}(x)\,dr, \;\; t\geq s, \; x\in \mathbb{R}^d,
     \end{align}
     and for any fixed $t\geq0$,
     \begin{align}\label{eqn-jacobian 2025 01 05}
      (J^{\eps}_{s,t})^{-1}(x)= 1 - \int_s^t \divv b^{\eps}((\phi^{\eps}_{r,t})^{-1}(x))\,(J^{\eps}_{r,t})^{-1}(x)\,dr, \;\; s\in[0,t], \; x\in \mathbb{R}^d.
     \end{align}
\end{lemma}
\begin{proof}
Here we only prove \eqref{eqn-jacobian} with $s=0$, and the other cases are similar.

Since the noise is additive, this lemma can be  proved in an elementary way by
the $\omega$-wise application of the classical deterministic results to the equation
\begin{equation}
\rho_t^\eps(x)=x+\int_0^tg^\eps(s,\rho_s^\eps(x))ds,
\end{equation}
where
\begin{equation}
\rho_t^\eps(x):=\phi^{\eps}_{0,t}(x)-L_t,  \;\; g^\eps(t,x):=b^\eps(x+L_t).
\end{equation}
We can deduce  from the above that
\begin{equation}
J_{0,t}^{\eps}(x)= \det [D\rho_t^\eps(x)] \mbox{ and }\divv g^\eps(t,x)=\divv b^\eps(x+L_t).
\end{equation}
and hence the result follows.
\end{proof}

Now, let $\phi^\eps_{t}=\phi^\eps_{0,t}$  and $J_t^{\eps}=J_{0,t}^{\eps}$, i.e.,
\begin{equation}\label{eqn-J_t^eps}
J^\eps_t(y)=\det[D\phi^\eps_{t}(y)], \;\; (t,y) \in [0,\infty) \times \mathbb{R}^d.
\end{equation}

   In  the following result    we assume that $\Upsilon$  is  countable subset      of $(0,1]$ whose  boundary $\partial(\Upsilon)\ni 0$, and the boundary of every infinite subset
 of $\Upsilon$ also contains $\{0\}$. When we write $\lim_{\eps \in \Upsilon}$ we mean $\lim_{ \Upsilon \ni \eps  \to 0}$.

\begin{lemma}\label{lem-A11}
There exists  a $\mathbb{P}$-full set  $\Omega^{(1)} \subset \Omega^{(0)} $ such that
the following equality holds  for every $\omega \in \Omega^{(1)}$,
and all $t\geq 0$, $y\in\mathbb{R}^d$ and $\eps \in \Upsilon$,
\begin{align}
 \theta (\phi^{\eps}_t(y))J_t^{\eps}(y)=&\,\theta(y)
    +\int_{0}^{t}b^{\eps}_{s}(\phi_{s}^{\eps}(y))\cdot  D \theta(\phi_{s}^{\eps}(y))J_s^{\eps}(y)\, ds
    \\&+\int_{0}^{t}\,\theta(\phi_{s}^{\eps}(y)) \divv b_{s}^{\eps}(\phi_{s}^{\eps}(y))J_{s}^{\eps}(y)
     ds\\
         &+\int_0^t\lint_{B}\big[  \theta(\phi_{s}^{\eps}(y)+z)-\theta(\phi_{s}^{\eps}(y)) -\sum_{i=1}^d z_iD_i\theta(\phi_{s}^{\eps}(y))\big]{J_s^{\eps}(y)} \; \nu(dz)\, ds\label{lem-A11-Ito}\\
         &+\int_0^t\int_{B^{c}}\big[  \theta(\phi_{s-}^{\eps}(y)+z)-\theta(\phi_{s-}^{\eps}(y)) \big]{J_{s-}^{\eps}(y)}  \prm(ds,dz)\\
    &+\int_0^t\lint_{B}   \big[ \theta(\phi_{s-}^{\eps}(y)+z)-\theta(\phi_{s-}^{\eps}(y)) \big ]{J_{s-}^{\eps}(y)}\; \tilde{\prm}(ds,dz).
     \end{align}
The set $\Omega^{(1)}$ depends on the function $\theta$.
     \end{lemma}
\begin{proof}[Proof of Lemma \ref{lem-A11}]
It is enough to apply  the classical It\^{o} formula, see, for instance, \cite[Theorem II.5.1]{Ikeda+Watanabe_1989},
%in the form of  Theorem \ref{theo-Ito},  to the 
to the $ \mathbb{R}^d\times \mathbb{R}$-valued  process     $(\phi^{\eps}_t(y),J_t^{\eps}(y))$ 
          and a function $F$ defined by
     \[F: \mathbb{R}^d \times \mathbb{R} \ni (x,p) \mapsto  \theta(x)p \in \mathbb{R}.
     \]   
\end{proof}

Since the set  $K:=\supp(\theta)$ is compact in  $\R^d$, by Corollary \ref{cor-homeomorphism} and Theorem  \ref{thm-homeomorphism},  there is a $\mathbb{P}$-full set (without loss of generality, denoted by $\Omega^{(1)}$, which is the same set as in Lemma \ref{lem-A11}) such that, for all $(\eps, T,\omega) \in \Upsilon \times [0,\infty) \times \Omega^{(1)}$, $\cup_{t\in[0,T]}\supp(\theta\circ\phi^\eps_t)=\cup_{t\in[0,T]}(\phi_t^\eps)^{-1}(K)$ is bounded.
  Also notice that Theorem \ref{thm-ww} implies that
\begin{equation}
[0,T]\times\mathbb{R}^d \ni (s,y) \longmapsto J_s^{\eps}(y)\in\mathbb{R}
\end{equation}
is continuous, and $b^\eps \in C_{\mathrm{b}}^\infty(\mathbb{R}^d,\mathbb{R}^d)$.
Therefore, every term of the equality \eqref{lem-A11-Ito}  multiplying $u_0(y)$ is integrable over $\R^d$, since $u_{0}\in L^{\infty }(\mathbb{R}^{d})$.

Next, by multiplying the equality above by $u_0(y)$ and then integrating over $\R^d$, using some classical arguments  we deduce that, for every $\omega \in \Omega^{(1)}$, and all $t\geq 0$, $y\in\mathbb{R}^d$ and $\eps \in \Upsilon$,
     \begin{align}\label{eq 20251105 V1}
    &\int_{\R^d} u_0(y)\,\theta (\phi^{\eps}_t(y))J_t^{\eps}(y)\,dy\\
&=\int_{\R^d}u_0(y)\,\theta(y)\,dy 
    +\int_0^t\Bigl(\int_{\R^d}u_0(y)\big[b^{\eps}(\phi^{\eps}_s(y))\cdot D \theta(\phi^{\eps}_s(y))+\theta(\phi_{s}^{\eps}(y)) \divv b^{\eps}(\phi^{\eps}_s(y))\big]J_s^{\eps}(y)\,dy \Bigr) ds\\
     &\quad+\int_0^t\lint_{B}\int_{\R^d}u_0(y)\big[  \theta(\phi_{s}^{\eps}(y)+z)-\theta(\phi_{s}^{\eps}(y)) -\sum_{i=1}^d z_iD_i\theta(\phi_{s}^{\eps}(y))\big] J_{s}^{\eps}(y)dy\; \nu(dz)\, ds\\
         &\quad+\int_0^t\int_{B^{c}}\int_{\R^d}u_0(y)\big[  \theta(\phi_{s-}^{\eps}(y)+z)-\theta(\phi_{s-}^{\eps}(y)) \big] J_{s-}^{\eps}(y)dy\; \prm(ds,dz)\\
    &\quad+\int_0^t\lint_{B}\int_{\R^d}u_0(y)  \big[ \theta(\phi_{s-}^{\eps}(y)+z)-\theta(\phi_{s-}^{\eps}(y)) \big ]J_{s-}^{\eps}(y)dy\; \tilde{\prm}(ds,dz),
     \end{align}
with $ D \theta(x) \in \mathbb{R}^d$ being the gradient of $\theta$ at $x\in \mathbb{R}^d$.\\

\indent Next, we will list and prove a series of Lemmata that we will need in order to pass with $\eps \todown 0$.
 All the convergence results from  Lemmata \ref{lem-A2}-\ref{lem-A7} are then applied  along a limit over an  infinite subset $\tilde{\Upsilon} \subset
 \Upsilon$ to obtain the  following limit
 \begin{align*}
    &\int_{\R^d} u_0(y)\,\theta (\phi_t(y))J_t(y)\,dy\\ &=\int_{\R^d}u_0(y)\,\theta(y)\,dy 
    +\int_0^t\int_{\R^d}u_0(y) \Bigl(b(\phi_s(y))\cdot  D \theta(\phi_s(y))+
    \theta(\phi_s(y)) \divv\;b(\phi_s(y))\Bigr)J_s(y)\,dyds\\
   &\quad+  \int_0^t\lint_{B}\int_{\R^d}u_0(y)\big[  \theta(\phi_{s}(y)+z)-\theta(\phi_{s}(y)) -\sum_{i=1}^d z_iD_i\theta(\phi_{s}(y))\big] J_s(y)dy\; \nu(dz)\, ds\\
         &\quad+\int_0^t\lint_{B}\int_{\R^d}u_0(y)  \big[ \theta(\phi_{s-}(y)+z)-\theta(\phi_{s-}(y)) \big ]J_{s-}(y)dy\; \tilde{\prm}(ds,dz)\\
    &\quad+\int_0^t\int_{B^c}\int_{\R^d}u_0(y)\big[  \theta(\phi_{s-}(y)+z)-\theta(\phi_{s-}(y)) \big] J_{s-}(y)dy\; \prm(ds,dz).
     \end{align*}
By changing variable $y=\phi_{t}^{-1}(x)$ or $y=\phi_{s}^{-1}(x)$, we obtain
 \begin{align}
   &\int_{\R^d} u_0(\phi_{t}^{-1}(x))\,\theta (x)\,dx\\ &=\int_{\R^d}u_0(x)\,\theta(x)\,dx
   +\int_0^t\int_{\R^d}u_0(\phi_{s}^{-1}(x)) \Bigl(b(x)\cdot  D \theta(x)+
    \theta(x) \divv\;b(x)\Bigr)\,dx ds \\
   &\quad+  \int_0^t\lint_{B}\int_{\R^d}u_0(\phi_{s}^{-1}(x))\big[  \theta(x+z)-\theta(x) -\sum_{i=1}^d z_iD_i\theta(x)\big] dx\; \nu(dz)\, ds\\
         &\quad+\int_0^t\lint_{B}\int_{\R^d}u_0(\phi_{s-}^{-1}(x))  \big[ \theta(x+z)-\theta(x) \big ]dx\; \tilde{\prm}(ds,dz)\\
    &\quad+\int_0^t\int_{B^c}\int_{\R^d}u_0(\phi_{s-}^{-1}(x))\big[  \theta(x+z)-\theta(x) \big] dx\; \prm(ds,dz)
     \end{align}
     which shows that $\int_{\R^d} u_0(\phi_{t}^{-1}(x))\,\theta (x)\,dx$ satisfies \eqref{eqn-transport-Marcus}.
\end{proof}
Of course, in order to fully complete the proof of Theorem \ref{thm-transport equation} we need to prove the remaining results.

\begin{notation}
 In what follows we put $\eps_\infty=0$ so that in particular $\phi^{\eps_\infty}_t=\phi_t^0=\phi_t$ and $b^{\eps_\infty}=b^0=b$.
\end{notation}

The proofs of the following two Lemmata are postponed till   Appendix  \ref{sec-proofs of 3 lemmata}.

Let us recall that in the following parts of our paper, whenever we mention $\mathbb{P}$-full set, we mean a set which may depend on $\theta$. If this is not the case, we will mention this explicitly.

\begin{lemma}\label{lem-A2}
There exists  a $\mathbb{P}$-full set  $\Omega^{(iv)} \subset \Omega^{(1)}$ and an infinite set $\Upsilon_{1} \subset \Upsilon$  such that for
every $\omega \in \Omega^{(iv)}$ and every $T>0$,
  \begin{equation}\label{eqn-A2}
\lim_{ \eps \in \Upsilon_{1}\rightarrow 0 }  \sup_{t \in [0,T]} \Big\vert   \int_{\R^d} u_0(y)\,\theta (\phi^{\eps}_t(y))J_t^{\eps}(y)\,dy- \int_{\R^d} u_0(y)\,\theta (\phi_t(y))J_t(y)\,dy \Big\vert=0.
    \end{equation}
\end{lemma}

\begin{proof}[Proof of Lemma \ref{lem-A2}] We refer to Appendix \ref{sec-proofs of 3 lemmata} which is based on Corollary \ref{cor-civuole}, see  also Remark \ref{okk}.
\end{proof} 

\begin{remark} \label{rem-ciao}
Let us observe that  Lemma \ref{lem-A2} implies that  for every $\omega \in \Omega^{(iv)}$ and every $T>0$,
  \[u_0 \circ (\phi^\eps_t)^{-1} \to u_0 \circ \phi_t^{-1}, \text{ as }\eps\in \Upsilon_{1}\rightarrow0,
  \]
weakly$^\ast$ in $L^{\infty}(\mathbb{R}^d)$, uniformly in $t \in [0,T]$. \\
Indeed, for every $\theta \in C_c^\infty(\mathbb{R} ^d)$ we
have
\begin{align}
\label{eqn-weak-star}
& \sup_{t \in [0,T]}
\Big | \int_{{}{\mathbb{R}}^{d}} (u_0 \circ (\phi^\eps_t)^{-1}
)(x) \theta(x) dx  -
\int_{{}{\mathbb{R}}^{d}} (u_0 \circ (\phi_t)^{-1}
)(x) \theta(x) dx \Big|
\\
&=\sup_{t \in [0,T]}\Big | \int_{{}{\mathbb{R}}^{d}} u_0(y)
\theta(\phi^{{\eps}}_t(y)) J_t^\eps(y) dy  -
\int_{{}{\mathbb{R}}^{d}} u_0(y) \theta(\phi_t(y)) J_t(y) dy \Big|
\to 0 \text{ as }\eps\in \Upsilon_{1}\rightarrow0.
\end{align}
Since the space $C_c^\infty(\mathbb{R}^d)$ is dense in $L^1(\mathbb{R} ^d)$,
by employing the classical density argument we can prove that  the
convergence \eqref{eqn-weak-star}  holds for every  $\theta \in L^1(\mathbb{R} ^d)$.
 \end{remark}

\begin{lemma}\label{lem-A3}
 For   every $T>0$ and   every $\omega \in \Omega^{(iv)}$,
where  $\Omega^{(iv)}$ and $\Upsilon_1$ have been introduced in   Lemma \ref{lem-A2},  the following holds
\begin{equation}\label{eqn-A3}
\begin{aligned}
\lim_{\eps \in \Upsilon_1} \sup_{t\in [0,T]}
&\Bigl\vert \int_0^t\Bigl(\int_{\R^d}u_0(y) b^{\eps}(\phi^{\eps}_s(y))\cdot  D \theta(\phi^{\eps}_s(y))J_s^{\eps}(y)\,dy \Bigr)\, ds
\\
&\hspace{1truecm}-
\int_0^t\Bigl(\int_{\R^d}u_0(y) b(\phi_s(y))\cdot D \theta(\phi_s(y))J_s(y)\,dy \Bigr)\, ds \Bigr\vert =0.
\end{aligned}
  \end{equation}
\end{lemma}

\begin{proof}[Proof of Lemma \ref{lem-A3}]

 See Appendix \ref{sec-proofs of 3 lemmata}.

\end{proof}

\begin{lemma}\label{lem-A4} There exists  a $\mathbb{P}$-full set  $\Omega^{(v)} \subset \Omega^{(iv)}$ and an infinite set $\Upsilon_{2} \subset \Upsilon_{1}$  such that
for every $\omega \in \Omega^{(v)}$ and every $T>0$,
\begin{align}\label{eqn-A4}
&\lim_{ \eps \in \Upsilon_{2} } \sup_{t \in [0,T]}  \Big\vert \int_0^t\int_{\R^d}u_0(y)\theta(\phi^{\eps}_s(y)) \divv\,b^{\eps}(\phi^{\eps}_s(y))J_s^{\eps}(y)\,dyds\\
& \hspace{1truecm} -\int_0^t\int_{\R^d}u_0(y)\theta(\phi_s(y)) \divv\,b(\phi_s(y))J_s(y)\,dyds
\Big\vert =0.
\end{align}
\end{lemma}

\begin{proof}[Proof of Lemma \ref{lem-A4}]
In the proof below we set 
\begin{equation}\label{eqn-K=supp theta}
   K:=\supp \theta.
  \end{equation}
%notation \eqref{eqn-K=supp theta}.
 Obverse that
\begin{align*}
&\Big|\int_0^t\int_{\R^d}u_0(y)\theta(\phi^{\eps}_s(y)) \divv\,b^{\eps}(\phi^{\eps}_s(y))J_s^{\eps}(y)\,dyds
 -    \int_0^t\int_{\R^d}u_0(y)\theta(\phi^{}_s(y)) \divv\,b(\phi^{}_s(y))J_s(y)\,dyds
 \Big|
\\
&= \Big|\int_0^t\int_{K}u_0((\phi^{\eps}_s)^{-1}(x))\theta(x) \divv\,b^{\eps}(x)\,dxds -
      \int_0^t\int_{K}u_0((\phi^{\eps}_s)^{-1}(x))\theta(x) \divv\,b(x)\,dxds
 \Big|
 \\
 &\quad+ \Big|\int_0^t\int_{K}[ u_0((\phi^{\eps}_s)^{-1}(x))
 - u_0 ((\phi_s)^{-1}(x))]
  \theta(x)   \divv\,b(x)\,dxds
 \Big|.
\end{align*}
Now since $u_0$ is bounded and  $\divv b^\eps \to \divv b $ in
$L^1_{\loc}(\mathbb{R}^d)$
by \eqref{eqn-b^eps-to-b-L^1} from  Lemma \ref{lem-b^eps-to-b}, we infer  that
\begin{align}
\sup_{t \in [0,T]} \Big|\int_0^t\int_{K}u_0((\phi^{\eps}_s)^{-1}(x))\theta(x) \divv\,b^{\eps}(x)\,dxds-
      \int_0^t\int_{K}u_0((\phi^{\eps}_s)^{-1}(x))\theta(x) \divv\,b(x)\,dxds
 \Big| \to 0.
\end{align}
 On the other hand, using Remark \ref{rem-ciao}, we find that
\begin{align}
\sup_{t \in [0,T]} \Big|\int_0^t\int_{K}[ u_0((\phi^{\eps}_s)^{-1}(x))
 - u_0 ((\phi_s)^{-1}(x))]
  \theta(x)   \divv\,b(x)\,dxds
 \Big| \to 0.
\end{align}
The proof is complete.
 \end{proof}
\begin{lemma}\label{lem-A5}
There exists   $\mathbb{P}$-full set  $\Omega^{(vi)} \subset \Omega^{(v)}$ and an infinite set $\Upsilon_{3} \subset \Upsilon_{2}$ such that for every  $T\geq 0$ and every $\omega \in \Omega^{(vi)}$,
\begin{equation}\label{eqn-A5}
\begin{aligned}
&\lim_{\eps \in \Upsilon_{3} }\sup_{t\in[0,T]} \Big\vert \int_0^t\lint_{B}\int_{\R^d}u_0(y)  \big[ \theta(\phi_{s-}^{\eps}(y)+z)-\theta(\phi_{s-}^{\eps}(y)) \big ]J_{s-}^{\eps}(y)dy\; \tilde{\prm}(ds,dz)
\nonumber\\
&\hspace{2truecm}-
\int_0^t\lint_{B}\int_{\R^d}u_0(y)  \big[ \theta(\phi_{s-}(y)+z)-\theta(\phi_{s-}(y)) \big ]J_{s-}(y)dy\; \tilde{\prm}(ds,dz) \Big\vert =0.
\end{aligned}
\end{equation}
\end{lemma}

\begin{proof}[Proof of Lemma \ref{lem-A5}]
 Set $K^{1}=\{x\in\R^d:\dist(K,\{x\})\leq 1\}$. Then we can find  $R_0>0$ such that $K^{1}\subset B_{R_0}:=\{y\in\mathbb{R}^d:|y|\leq R_0\}$. By the Burkholder inequality, see Proposition \ref{prop-Burkholder inequality},
\begin{align}\label{eqn-A5 01}
&\E\sup_{t\in[0,T]}\Big|\int_0^t\lint_{B}\int_{\R^d}u_0(y)  \big[ \theta(\phi_{s-}^{\eps}(y)+z)-\theta(\phi_{s-}^{\eps}(y)) \big ]J_{s-}^{\eps}(y)dy\; \tilde{\prm}(ds,dz)\\
&\quad\quad\quad\quad-
\int_0^t\lint_{B}\int_{\R^d}u_0(y)  \big[ \theta(\phi_{s-}(y)+z)-\theta(\phi_{s-}(y)) \big ]J_{s-}(y)dy\; \tilde{\prm}(ds,dz)
\Big|^2\nonumber\\
&=
\E\sup_{t\in[0,T]}\Big|\int_0^t\lint_{B}\int_{K^{1}}u_0((\phi_{s-}^{\eps})^{-1}(x))  \big[ \theta(x+z)-\theta(x) \big ]dx\; \tilde{\prm}(ds,dz)\nonumber\\
&\quad\quad\quad\quad\quad-
\int_0^t\lint_{B}\int_{K^{1}}u_0(\phi_{s-}^{-1}(x))  \big[ \theta(x+z)-\theta(x) \big ]dx\; \tilde{\prm}(ds,dz)
\Big|^2\nonumber\\
&\leq C
\E\Big(\int_0^T\lint_{B}\Big|\int_{K^{1}}\big[u_0((\phi_{s}^{\eps})^{-1}(x))-u_0(\phi_{s}^{-1}(x))\big]  \big[ \theta(x+z)-\theta(x) \big ]dx\Big|^2\nu(dz)ds\Big).\nonumber
\end{align}
 By  Remark \ref{rem-ciao},   on  $\Omega^{(iv)}$
\begin{align}
\int_{K^{1}}\big[u_0((\phi_{s}^{\eps})^{-1}(x))-u_0(\phi_{s}^{-1}(x))\big]  \big[ \theta(x+z)-\theta(x) \big ]dx \to 0,
\end{align}
for every $z \in B$,  uniformly w.r.t.  $s \in [0,T]$. \\
\indent Moreover, with $C_0 = \Leb (B_{R_0})$,  we also have the following estimate: for any $(z,s)\in B\times [0,T]$,
\begin{align}
\Big |\int_{K^{1}}\big[u_0((\phi_{s}^{\eps})^{-1}(x))-u_0(\phi_{s}^{-1}(x))\big]  \big[ \theta(x+z)-\theta(x) \big ]dx \Big|^2 \le 4 C_0^2 \| u_0\|_{L^\infty(\R^d)}^2 \|  D \theta \|^2_{0}
 |z|^2.
\end{align}
Since $\lint_B |z|^2 \nu (dz) < \infty$, we can finally apply the Lebesgue  DCT and  deduce that
\begin{align}
\E\Bigl[\int_0^T\lint_{B}\Big|\int_{K^{1}}\big[u_0((\phi_{s}^{\eps})^{-1}(x))-u_0(\phi_{s}^{-1}(x))\big]
\big[ \theta(x+z)-\theta(x) \big ]dx\Bigr|^2\nu(dz)ds\Big) \to 0.
\end{align}
Since convergence in the mean implies the a.s. convergence of a subsequence,
the result follows.
\end{proof}
%------------
\begin{lemma}\label{lem-A6}
There exists   $\mathbb{P}$-full set $\Omega^{(vii)} \subset \Omega^{(vi)}$ and  an infinite set $\Upsilon_{4} \subset \Upsilon_{3}$  such that for all $T\geq 0$ and every $\omega \in \Omega^{(vii)}$,
\begin{equation}\label{eqn-A6}
\begin{aligned}
&\lim_{\eps  \in \Upsilon_{4} }\sup_{t\in[0,T]}\Big| \int_0^t\int_{B^c}\int_{\R^d}u_0(y)\big[  \theta(\phi_{s-}^{\eps}(y)+z)-\theta(\phi_{s-}^{\eps}(y)) \big] J_{s-}^{\eps}(y)dy\; \prm(ds,dz)
\\
& \hspace{3truecm} - \int_0^t\int_{B^c}\int_{\R^d}u_0(y)\big[  \theta(\phi_{s-}(y)+z)-\theta(\phi_{s-}(y)) \big] J_{s-}(y)dy\; \prm(ds,dz)\Big|=0.\nonumber
\end{aligned}
\end{equation}
\end{lemma}
\begin{proof}[Proof of Lemma \ref{lem-A6}]
Using similar argument of proving \eqref{eqn-A5 01}, we have
\begin{align}\label{eqn-A6 01}
&\sup_{t\in[0,T]}\Big|
\int_0^t\int_{B^c}\int_{\R^d}u_0(y)\big[  \theta(\phi_{s-}^{\eps}(y)+z)-\theta(\phi_{s-}^{\eps}(y)) \big] J_{s-}^{\eps}(y)dy\; \prm(ds,dz)\\
&\quad\quad\quad\quad\quad\quad-
\int_0^t\int_{B^c}\int_{\R^d}u_0(y)\big[  \theta(\phi_{s-}(y)+z)-\theta(\phi_{s-}(y)) \big] J_{s-}(y)dy\; \prm(ds,dz)
\Big| \\
&=\sup_{t\in[0,T]}
\Big|
\int_0^t\int_{B^c}\int_{\R^d}u_0((\phi_{s-}^{\eps})^{-1}(x))\big[  \theta(x+z)-\theta(x) \big]dx\; \prm(ds,dz)\\
&\quad\quad\quad\quad\quad\quad-
\int_0^t\int_{B^c}\int_{\R^d}u_0(\phi_{s-}^{-1}(x))\big[  \theta(x+z)-\theta(x) \big]dx\; \prm(ds,dz)
\Big|\\
&\leq \int_0^T\int_{B^c} \Big|\int_{\R^d}u_0((\phi_{s-}^{\eps})^{-1}(x))\big[  \theta(x+z)-\theta(x) \big]dx\\
&\quad\quad\quad\quad\quad\quad-
\int_{\R^d}u_0(\phi_{s-}^{-1}(x))\big[  \theta(x+z)-\theta(x) \big]dx\Big| \; \prm(ds,dz).
\end{align}
Let us fix $z \in B^c$, and set $K^{1}_z=\{x\in\R^d:\dist(K,\{x\})\leq |z|+1\}$. Since by  Remark \ref{rem-ciao},  on  $\Omega^{(iv)}$,
\begin{align}
&\int_{\R^d}\big[u_0((\phi_{s}^{\eps})^{-1}(x))- u_0(\phi_{s}^{-1}(x))\big]  \big[ \theta(x+z)-\theta(x) \big ]dx\\
&=\int_{K^{1}_z}\big[u_0((\phi_{s}^{\eps})^{-1}(x))- u_0(\phi_{s}^{-1}(x))\big]  \big[ \theta(x+z)-\theta(x) \big ]dx
\to 0,
\end{align}
uniformly w.r.t.  $s \in [0,T]$,  by recalling  that
 $$
\int_0^t \int_{B^c} |F_\eps(s,z) | \mu(ds, dz)  = \sum_{0 < s \le t
  \; }
  |F_\eps(s, \triangle L _s)| \1_{B^c} (\triangle L _s)
$$  as a random finite sum, with
\begin{align}
F_\eps(s,z) = \int_{\R^d}\big[u_0((\phi_{s}^{\eps})^{-1}(x))- u_0(\phi_{s}^{-1}(x))\big]  \big[ \theta(x+ z)-\theta(x) \big ]dx,
\end{align}
we deduce
 the assertion.
\end{proof}

\begin{lemma}\label{lem-A7}
There exists   $\mathbb{P}$-full set  $\Omega^{(viii)} \subset \Omega^{(vii)}$ and  an infinite set $\Upsilon_{5} \subset \Upsilon_{4}$
 such that for all $T\geq 0$ and every $\omega \in \Omega^{(viii)}$,
\begin{equation}\label{eqn-A7}
\begin{aligned}
&
\lim_{\eps\in \Upsilon_{5} }\sup_{t\in[0,T]}\Big| \int_0^t\lint_{B}\int_{\R^d}u_0(y)\big[  \theta(\phi_{s}^{\eps}(y)+z)-\theta(\phi_{s}^{\eps}(y)) -\sum_{i=1}^d z_iD_i\theta(\phi_{s}^{\eps}(y))\big] J_s^{\eps}(y)dy\, \nu(dz)\, ds
\\
& \hspace{1.9truecm} - \int_0^t\lint_{B}\int_{\R^d}u_0(y)\big[  \theta(\phi_{s}(y)+z)-\theta(\phi_{s}(y)) -\sum_{i=1}^d z_iD_i\theta(\phi_{s}(y))\big] J_s(y)dy\, \nu(dz)\, ds\Big|=0.
\end{aligned}
\end{equation}
\end{lemma}
\begin{proof}[Proof of Lemma \ref{lem-A7}]  We use notations introduced in the proof of Lemma \ref{lem-A5}. Arguing as before we have
\begin{align}\label{eqn-A7 01}
 &\sup_{t\in[0,T]}\Big|\int_0^t\lint_{B}\int_{\R^d}u_0(y)\big[  \theta(\phi_{s}^{\eps}(y)+z)-\theta(\phi_{s}^{\eps}(y)) -\sum_{i=1}^d z_iD_i\theta(\phi_{s}^{\eps}(y))\big] J_s^{\eps_n}(y)dy\; \nu(dz)\, ds\nonumber\\
&\quad\quad\quad\quad\quad\quad-
 \int_0^t\lint_{B}\int_{\R^d}u_0(y)\big[  \theta(\phi_{s}(y)+z)-\theta(\phi_{s}(y)) -\sum_{i=1}^d z_iD_i\theta(\phi_{s}(y))\big] J_s(y)dy\; \nu(dz)\, ds
\Big|\nonumber\\
&=
 \sup_{t\in[0,T]}\Big|\int_0^t\lint_{B}\int_{\R^d}u_0((\phi_{s}^{\eps})^{-1}(x))\big[  \theta(x+z)-\theta(x) -\sum_{i=1}^d z_iD_i\theta(x)\big] dx\; \nu(dz)\, ds\nonumber\\
&\quad\quad\quad\quad\quad\quad-
 \int_0^t\lint_{B}\int_{\R^d}u_0(\phi_{s}^{-1}(x))\big[  \theta(x+z)-\theta(x) -\sum_{i=1}^d z_iD_i\theta(x)\big] dx\; \nu(dz)\, ds
\Big|\nonumber\\
&\leq
\int_0^T\lint_{B}\Big|\int_{K^1}[u_0((\phi_{s}^{\eps})^{-1}(x))-u_0(\phi_{s}^{-1}(x))]\big[  \theta(x+z)-\theta(x) -\sum_{i=1}^d z_iD_i\theta(x)\big] dx\Big| \nu(dz)\, ds
\nonumber.
\end{align}
 Let us fix $z \in B$.   By  Remark \ref{rem-ciao}, we have  on  $\Omega^{(iv)}$,
\begin{align}
\int_{K^{1}}\big[u_0((\phi_{s}^{\eps})^{-1}(x))-u_0(\phi_{s}^{-1}(x))\big]  \big[ \theta(x+z)-\theta(x) - \sum_{i=1}^d z_iD_i\theta(x) \big ]dx \to 0,
\end{align}
uniformly w.r.t.  $s \in [0,T]$.
 We also have the estimate, for $z \in B$ and $s\in[0,T]$,
\begin{align*}
\Big |\int_{K^{1}}\big[u_0((\phi_{s}^{\eps})^{-1}(x))-u_0(\phi_{s}^{-1}(x))\big]  \big[ \theta(x+z)-\theta(x) - \sum_{i=1}^d z_iD_i\theta(x) \big ]dx \Big|
\le 2 C_0 \| u_0\|_{L^\infty(\R^d)}\| D^2 \theta \|_{0}
 |z|^2,
\end{align*}
where $C_0 = Leb (B_{R_0})$.
Since $\lint_B |z|^2 \nu (dz) < \infty$, we can finally apply the Lebesgue DCT and get, on  $\Omega^{(iv)}$,
\begin{align*}
\int_0^T\lint_{B}\Big|\int_{K^{1}}\big[u_0((\phi_{s}^{\eps})^{-1}(x))-u_0(\phi_{s}^{-1}(x))\big]  \big[ \theta(x+z)-\theta(x) -  \sum_{i=1}^d z_iD_i\theta(x) \big ]dx\Big|\nu(dz)ds  \to 0.
\end{align*}
 The proof of this lemma is complete.
\end{proof}

Now the proof of Theorem \ref{thm-transport equation} is complete.

\section{Uniqueness of weak solutions to the stochastic transport equation}\label{uni1}

{The two main results in this section are Theorems \ref{thm-uniqueness-1} and \ref{thm-uniqueness} about the 
    uniqueness of  weak$^\ast$-$\mathrm{L}^{\infty}$-solutions to the problem \eqref{eqn-transport-Markus} 
 in the sense of Definition \ref{def-transport-weak-2} under different assumptions. }

\subsection{ New auxiliary regularity type  results }\label{subsec 6.1}

We begin with recalling a basic definition. 
\begin{definition}\label{def-weak derivative}
Suppose that $u:\mathbb{R}^d \to  \mathbb{R}$ is locally Lebesgue essentially bounded  and $b:\mathbb{R}^d \to  \mathbb{R}^d$ is locally Lebesgue integrable  such that $\divv b :\mathbb{R}^d \to  \mathbb{R}$ is locally Lebesgue integrable. \\
We define a  distribution $b \cdot  D u \in \call{D}^\prime(\mathbb{R}^d)$ by the following formula
\begin{equation}\label{eqn-v nabla g}
\lb b\cdot  D u, \rho \rb:=-\lb u  b  ,  D \rho \rb   -\lb   u \,\divv b, \rho \rb, \;\; \rho \in \call{D}(\mathbb{R}^d),
\end{equation}
where $\call{D}^\prime(\mathbb{R}^d)$ is the space of distributions.
\end{definition}
\begin{remark}\label{rem-weak derivative}
 The above   formula is true when the functions $b$ and $u$ are of $C^1$ class. Let us also observe that the bilinear map
 \begin{equation}\label{eqn-???}
 (b,u) \mapsto b\cdot  D u \in \call{D}^\prime(\mathbb{R}^d)
 \end{equation}
 is continuous with respect to the topology on $(b,u)$ as required above.
    \end{remark}
Below we present an equivalent form of the definition of $b\cdot D u$.
\begin{proposition}\label{prop-weak derivative}
Suppose that $u:\mathbb{R}^d \to  \mathbb{R}$ is a locally Lebesgue essentially bounded function  and $b:\mathbb{R}^d \to  \mathbb{R}^d$ is a locally Lebesgue integrable vector field  such that the distributional derivative  $\divv b :\mathbb{R}^d \to  \mathbb{R}$ is locally Lebesgue integrable.\\
 Then the   distribution $b\cdot D u $ defined in Definition \ref{def-weak derivative} satisfies the following identity
\begin{equation}\label{eqn-v nabla g-2}
 b\cdot  D u :=  -  u \,\divv b +\divv( u  b),
\end{equation}
where \\
(i) $u \,\divv b$ is a point-wise product of $u$ and  $\divv b $ and so is a locally Lebesgue integrable function,\\
(ii)  $u  b$ is a locally Lebesgue integrable vector field,\\
(iii) $\divv( u  b)$ is the distributional divergence of $ u  b$.
\end{proposition}
\begin{remark}\label{rem-divergence}
Let us note that if a locally Lebesgue integrable  vector field $b:\mathbb{R}^d \to  \mathbb{R}^d$ is   such that $\divv b :\mathbb{R}^d \to  \mathbb{R}$ is locally Lebesgue integrable, then for every $C^1$-class function $u:\mathbb{R}^d \to \mathbb{R}$ the following formula  holds in the sense of distributions
\begin{equation}\label{eqn-divergence}
 \divv( u  b)=b\cdot  D u +  u \,\divv b,
\end{equation}
where now each term is a well defined distribution. Indeed, by assumptions $ub$ is a locally Lebesgue integrable  vector field and functions $u \,\divv b$ and $b\cdot D u$ are  locally Lebesgue integrable functions.
    \end{remark}

\begin{definition}\label{def-commutaror}
Suppose that $u:\mathbb{R}^d \to  \mathbb{R}$ is locally Lebesgue essentially bounded and $b:\mathbb{R}^d \to  \mathbb{R}^d$ is locally Lebesgue integrable  such that $\divv b :\mathbb{R}^d \to  \mathbb{R}$ is locally Lebesgue integrable. We define  an expression $\mathcal{R}_\eps[b,u]$  by the following formula
\begin{equation}\label{eqn-commutator R_eps}
 \mathcal{R}_\eps[b,u]:= \vartheta_\eps \ast ( b\cdot  D u) -  b \cdot   D ( \vartheta_\eps \ast u),
\end{equation}
where $b\cdot  D u$ is defined in Definition \ref{def-weak derivative} and $\vartheta_\eps$ is defined by \eqref{def-vartheta-eps},  and we call $ \mathcal{R}_\eps[b,u]$ the  $\eps$-commutator.
\end{definition}
Let us observe that although the first term on the RHS of \eqref{eqn-commutator R_eps} is a $C^\infty$ function,  the second is not even $C^1$.

Let $\mathcal{S}(\mathbb{R}^{d})$ be the space of smooth functions on $\mathbb{R}^d$ rapidly decreasing at infinity.
For $s \in \mathbb{R}$, we define the Bessel potential of order $s$ to be the continuous bijective linear operator $\mathcal{J}^{s}: \mathcal{S}\left(\mathbb{R}^{d}\right) \to  \mathcal{S}(\mathbb{R}^{d})$ by
\begin{equation}
\mathcal{J}^{s} u=\mathscr{F}^{-1}\left(1+|\cdot\vert^{2}\right)^{s / 2} \mathscr{F} u,
\end{equation}
where $\mathscr{F}$ is the Fourier transform and $\mathscr{F}^{-1}$ is the inverse Fourier transform.

For $p \in(1, \infty)$ and $s\geq0$, we use $H^{s,p}\left(\mathbb{R}^{d}\right)$ to denote the Bessel potential space
\begin{equation}
H^{s,p}(\mathbb{R}^{d})=\left\{u \in L^{p}(\mathbb{R}^{d}):\mathcal{J}^s u \in L^{p}(\mathbb{R}^{d})\right\}
\end{equation}
which is equipped with the norm
\begin{equation*}
\Vert u\Vert_{H^{s,p}}:=\left\Vert\mathcal{J}^s u\right\Vert_{L^{p}}.
\end{equation*}
{
For positive integers $m$ and $1\leq p<\infty$, let $$W^{m, p}(\mathbb{R}^d) =\left\{u \in L^p(\mathbb{R}^d): D^\alpha u \in L^p(\mathbb{R}^d), \text{ for }0 \leq|\alpha| \leq m\right\}.$$
Then $W^{m,p}(\mathbb{R}^d)=H^{m,p}(\mathbb{R}^d)$. }
Let $W^{s, p}\left(\mathbb{R}^{d}\right)$, $0\leq s<1$, $p\geq 1$ consist of all functions $u$ in $L^{p}\left(\mathbb{R}^{d}\right)$ for which the norm

\begin{equation}
\Vert u\Vert_{W^{s,p}}:=\Vert u\Vert_{L^p}+\left(\int_{\mathbb{R}^{d}} \int_{\mathbb{R}^{d}} \frac{|u(x)-u(y)\vert^{p}}{|x-y\vert^{d+s p}} d x d y\right)^{1 / p}
:=
\Vert u\Vert_{L^p}
+
[u]_{W^{s,p}}
\end{equation}
is finite. {Then $W^{s, p}(\mathbb{R}^d)=(L^p(\R^d), W^{1, p}(\mathbb{R}^d))_{s , p}$ by the real interpolation.} By $C_r^{\infty}:= C_{0}^{\infty}(B_r)$ we denote the space of smooth functions with compact support in $B_r=\{x\in\R^d: |x|\leq r\}$. 
Similarly we use the following  notations. For any $0<s<1$ and $p \in[1, \infty]$,
\begin{align}
 L_{r}^{p}&\coloneq L^{p}(B_r),  \\ 
% C_{r}^{\infty}&\coloneq C_{0}^{\infty}(B(0,r)), % \mbox{i.e. of compact support in } B(0,r);
%\\
W_{r}^{s, p}&\coloneq W^{s, p}(B_r).
 \end{align}
Let $\Vert\cdot\Vert_{L_r^p}$ and $\Vert\cdot\Vert_{W_{r}^{s, p}}$ denote the $L^{p}$-norm and the $W^{s, p}$-norm on $B_r$, $0<s<1$, $p \in[1, \infty]$, respectively.

We will analyze the regularity of the Jacobian of the stochastic flow in Theorem \ref{Reg-est-Jaco} and Theorem \ref{Reg-est-Jaco-2}. Under some integrability assumption on $\mathrm{div }\,b$, we will show that the Jacobian $J\phi_\cdot(\cdot,\omega)$ of the stochastic flow is in $L^2(0,T;W_r^{\delta,p})$, $\mathbb{P}$-a.s. for some $p>1$.
 Before moving to the regularity of Jacobian, let us establish first some results on the distributional commutator. 
 
 \begin{remark}
   Comparing Proposition \ref{comu-prop-1} below with \cite[Lemma 1]{{FGP-2012}}, we observe that the latter requires $J\phi_t \in W^{1-\beta,1}_{r}$ which is not sufficient for our purpose.  Because even for bounded domain $\mathcal{O}$ in $\mathbb{R}^d$,  the embedding $W^{\delta, p}(\mathcal{O}) \subset W^{\delta, 1}(\mathcal{O})$, $p>1$ no longer holds when  $\delta$ is not an integer, see \cite{Mir+Sic}.
So we generalize the result in Proposition \ref{comu-prop-1} to prove that $J\phi_t \in W^{1-\beta,p}_{r}$, $p\geq 1$.
\end{remark}

The proofs of the following Lemma \ref{Lem-commu} and Proposition \ref{comu-prop-1} are similar in spirit to \cite[Corollary 1 and Lemma 1]{{FGP-2012}}.

\begin{lemma}\label{Lem-commu}
Assume $v \in L_{\mathrm{loc}}^{\infty}(\mathbb{R}^{d}, \mathbb{R}^{d}), \,\divv v \in L_{\mathrm{loc}}^{1}(\mathbb{R}^{d}),\, g \in L_{\mathrm{loc}}^{\infty}(\mathbb{R}^{d})$. If there exists $\beta \in(0,1)$ such that $v \in C_{\mathrm{loc}}^{\beta}(\mathbb{R}^{d}, \mathbb{R}^{d})$, then for any $\rho \in C_{r}^{\infty}$ and for sufficiently small $\eps$, we have, for $p\geq 1$,
\begin{align}
\left|\int_{\R^d} \mathcal{R}_{\eps}[v, g](x) \rho(x) d x\right|\leq  C_{r} \Vert g\Vert_{L_{r+1}^{\infty}}\left(\Vert\rho\Vert_{L_{r}^{\infty}}\Vert\divv v\Vert_{L_{r+1}^{1}}+[v]_{C_{r+1}^{\beta}}[\rho]_{W_{r}^{1-\beta, p}}\right).
\end{align}
\end{lemma}
\begin{proof}The case $p=1$ is proved in \cite[Corollary 1]{{FGP-2012}}. We only  need to consider $p>1$.
On account of  \cite[Lemma 22]{FGP_2010-Inventiones} and for sufficiently small $\eps$, we have
\begin{align}
&\left|\int_{\R^d} \mathcal{R}_{\eps}[v, g](x) \rho(x) d x\right|\\
\leq& C_{r}  \Vert g\Vert_{L_{r+1}^{\infty}}\Vert\rho\Vert_{L_{r}^{\infty}}\Vert\divv v\Vert_{L_{r+1}^{1}}
+\left|\iint_{\R^d\times \R^d} g(x^{\prime}) D_{x} \vartheta_{\eps}(x-x^{\prime})(\rho(x)-\rho(x^{\prime}))\left[v(x)-v(x^{\prime})\right] d x d x^{\prime}\right|.
\end{align}
Let $q=\frac{p}{p-1}$ be the H\"older conjugate of $p$.
By applying the H\"older inequality, we find
\begin{align*}
&\left|\iint_{\R^d\times \R^d} g(x^{\prime}) D_{x} \vartheta_{\eps}(x-x^{\prime})\left(\rho(x)-\rho(x^{\prime})\right)\left[v(x)-v(x^{\prime})\right] d x d x^{\prime}\right| \\
&\leq  \frac{1}{\eps}[v]_{C_{r+1}^{\beta}}\Vert g\Vert_{L_{r+1}^{\infty}}  \frac{1}{\eps^{d}} \Big( \iint_{B_{r+1}\times B_{r+1}}\left|D_{x} \vartheta(\frac{x-x^{\prime}}{\eps})\right\vert^q d x d x^{\prime}\Big)^{\frac1q}\\
&\hspace{3cm}\cdot\Big(  \iint_{B_{r+1}\times B_{r+1}}\left|\rho (x)-\rho(x^{\prime})\right\vert^p |x-x'\vert^{p\beta} \1_{[0,2\eps]}(|x-x'|) d x d x^{\prime}\Big)^{\frac 1p} \\
&\leq C_r \frac{\eps^{\beta}}{\eps}[v]_{C_{r+1}^{\beta}}\Vert g\Vert_{L_{r+1}^{\infty}}  \frac{1}{\eps^{d}} \Big( \iint_{B_{r+1}\times B_{r+1}}\left|D_{x} \vartheta(\frac{x-x^{\prime}}{\eps})\right\vert^q d x d x^{\prime}\Big)^{\frac1q}\\
&\hspace{3cm}\cdot\Big( \iint_{B_{r+1}\times B_{r+1}} \frac{\left|\rho(x)-\rho(x^{\prime})\right\vert^p}{\left|x-x^{\prime}\right\vert^{(1-\beta) p+d}}\left|x-x^{\prime}\right\vert^{(1-\beta) p+d}\1_{[0,2\eps]}(|x-x'|) d x d x^{\prime}\Big)^{\frac1p} \\
&= C_r[v]_{C_{r+1}^{\beta}}\Vert g\Vert_{L_{r+1}^{\infty}}    \Big( \iint_{B_{r+1}\times B_{r+1}}\left|D_{x} \vartheta(\frac{x-x^{\prime}}{\eps})\right\vert^q \frac{1}{\eps^d}d x d x^{\prime}\Big)^{\frac1q}\\
&\hspace{1cm}\cdot\Big( \iint_{B_{r+1}\times B_{r+1}} \frac{|\rho(x)-\rho(x^{\prime})\vert^p}{|x-x^{\prime}\vert^{(1-\beta) p+d}}\left|x-x^{\prime}\right\vert^{(1-\beta) p+d} \frac{1}{\eps^{(1-\beta) p+d}}\1_{[0,2\eps]}(|x-x'|)d x d x^{\prime}\Big)^{\frac1p} \\
&\leq  C_r   [v]_{C_{r+1}^{\beta}}\Vert D\vartheta\Vert_{0}\Vert g\Vert_{L_{r+1}^{\infty}}[\rho]_{W_{r+1}^{1-\beta, p}},
\end{align*}
where the constant $C_r$ is independent on $\eps$ and its value may change from line to line. 
For the last inequality, we have used that
\begin{align*}
&\quad\ \iint_{B_{r+1}\times B_{r+1}}\left|D_{x} \vartheta(\frac{x-x^{\prime}}{\eps})\right\vert^q \frac{1}{\eps^d}d x d x^{\prime}\\
&=\iint_{B_{r+1}\times B_{r+1}}\left|D_{x} \vartheta(\frac{x-x^{\prime}}{\eps})\right\vert^q \frac{1}{\eps^d}\1_{[0,2\eps]}(|x-x'|)d x d x^{\prime}\\
&\leq
\Vert D\vartheta\Vert^q_{0}\iint_{B_{r+1}\times B_{r+1}} \frac{1}{\eps^d}\1_{[0,2\eps]}(|x-x'|)d x d x^{\prime} \leq C_r\Vert D\vartheta\Vert^q_{0}.
\end{align*}

\end{proof}

In the following, we always assume that $\eps$ is sufficiently small. The constants appearing in the following two propositions are independent of $\eps$.
\begin{proposition}\label{comu-prop-1}
 Let $\phi$ be a $C^{1}$-diffeomorphism of $\mathbb{R}^{d}$.  Assume that there exist $\beta \in(0,1)$ and $p\geq 1$ such that 
 \begin{trivlist}
\item[(i)] $v \in C_{\mathrm{loc }}^{\beta}(\mathbb{R}^{d}, \mathbb{R}^{d})$,
\item[(ii)]
$\divv v \in L_{\mathrm{loc}}^{1}(\mathbb{R}^{d})$,
\item[(iii)]
$g \in L_{\mathrm{loc}}^{\infty}(\mathbb{R}^{d})$ and 
\item[(iv)]
$J \phi^{-1} \in W_{\mathrm{loc}}^{1-\beta,p}\left(\mathbb{R}^{d}\right)$.
  \end{trivlist}
 Then for  all $r>0$ and any $R>0$, 
 there exist  constants $C_R>0$ and $C_{R,p}>0$ such that if 
 $\rho \in C_{r}^{\infty}$  such that $\operatorname{supp}\left(\rho \circ \phi^{-1}\right) \subseteq B_R$, then 
\begin{align}
&\quad\left|\int_{\R^d} \mathcal{R}_{\eps}[ v,g](\phi(x)) \rho(x) d x\right|\\
&\leq C_{R}\Vert g\Vert_{L_{R+1}^{\infty}}\Vert\rho\Vert_{L_{r}^{\infty}}
\Vert\divv v\Vert_{L_{R+1}^{1}}\Vert J \phi^{-1}\Vert_{L_{R}^{\infty}}\\
&\quad+C_{R,p}\Vert g\Vert_{L_{R+1}^{\infty}}[v]_{C_{R+1}^{\beta}}\left(\Vert D \rho\Vert_{L_{r}^{p}}\left\Vert D \phi^{-1}\right\Vert_{L_{R}^{\infty}}\left\Vert J \phi^{-1}\right\Vert_{L_{R}^{\infty}}+\left[J \phi^{-1}\right]_{W_{R}^{1-\beta, p}}\Vert\rho\Vert_{L_{r}^{\infty}}\right).
\end{align}
Moreover, 
\begin{equation}
\lim _{\eps \to  0} \int_{\R^d} \mathcal{R}_{\eps}[v, g](\phi(x)) \rho(x) d x=0.
\end{equation}

\end{proposition}
\begin{proof}
 By changing variable, we obtain
 \begin{equation}\int_{\R^d} \mathcal{R}_{\eps}[v, g](\phi(x)) \rho(x) d x=\int_{\R^d} \mathcal{R}_{\eps}[v, g](y) \rho_{\phi}(y) d y,\end{equation} where the function $\rho_{\phi}$, which is defined by 
\begin{equation}
\rho_{\phi}(y)=\rho\left(\phi^{-1}(y)\right) J \phi^{-1}(y),
\end{equation}
has the support strictly contained in the ball of radius $R$. Observe that
\begin{equation*}
\Vert\rho_{\phi}\Vert_{L^{\infty}_R} \leq \Vert\rho \Vert_{L^{\infty}_r}\Vert J\phi^{-1}\Vert_{L^{\infty}_R}.
\end{equation*}
Now let us check that $\rho_{\phi}\in W^{1-\beta,p}_{\mathrm{loc}}(\R^d)$. By \cite[Corollary 2.2]{Brz+Millet_2014},
\begin{align}
[\rho_{\phi}]_{W_{R}^{1-\beta, p}} \leq C_p \Big([\rho \circ \phi^{-1}]_{W_{R}^{1-\beta, p}} \Vert J \phi^{-1}\Vert_{L_{R}^{\infty}}+[J \phi^{-1}]_{W_{R}^{1-\beta, p}}\Vert\rho\Vert_{L_{r}^{\infty}}\Big).
\end{align}
In view of the Sobolev embedding $W^{1,p}_R\subset W^{1-\beta,p}_R$, we infer
\begin{equation}
[\rho \circ \phi^{-1}]_{W_{R}^{1-\beta, p}} \leq C_R\left\Vert D\left(\rho \circ \phi^{-1}\right)\right\Vert_{L_{R}^{p}} \leq C_R\Vert D \rho\Vert_{L_{r}^{p}}\left\Vert D \phi^{-1}\right\Vert_{L_{R}^{\infty}}.
\end{equation}
Combining the above estimates gives
\begin{equation}
[\rho_{\phi}]_{W_{R}^{1-\beta, p}} \leq C_{R,p}\Big(\Vert D \rho\Vert_{L_{r}^{p}}\left\Vert D \phi^{-1}\right\Vert_{L_{R}^{\infty}}\left\Vert J \phi^{-1}\right\Vert_{L_{R}^{\infty}}+\left[J \phi^{-1}\right]_{W_{R}^{1-\beta, p}}\Vert\rho\Vert_{L_{r}^{\infty}}\Big).
\end{equation}

Applying Lemma \ref{Lem-commu}, the proof is complete.
\end{proof}

We recall the following commutator estimates result from \cite{FGP_2010-Inventiones}. As a minor remark, we note that the commutator should be bounded by $\Vert D \phi^{-1}\Vert^2_{C_{R+1}^{1-\theta}}$ instead of $\Vert D \phi^{-1}\Vert_{C_{R}^{1-\theta}}$ by examining the proof of \cite[Corollary 23]{FGP_2010-Inventiones}.

\begin{proposition}\label{commu-pro}\cite[Corollary 23]{FGP_2010-Inventiones} Let $\phi$ be a $C^{1}$-diffeomorphism of $\mathbb{R}^{d}$.  Assume that there exists $\theta \in(0,1)$ such that $v \in C_{\mathrm{loc }}^{\theta}(\mathbb{R}^{d}, \mathbb{R}^{d})$, $\divv v \in$ $L_{\mathrm{loc}}^{1}(\mathbb{R}^{d}), g \in L_{\mathrm{loc}}^{\infty}(\mathbb{R}^{d})$. Moreover, assume that $J \phi \in C_{\mathrm{loc}}^{1-\theta}\left(\mathbb{R}^{d}\right)$. Then, for any $\rho \in C_{r}^{\infty}$ and any $R>0$ such that $\operatorname{supp}\left(\rho \circ \phi^{-1}\right) \subseteq B_R$, we have the uniform bound
\begin{align*}
\left|\int_{\R^d} \mathcal{R}_{\eps}[ v,g](\phi(x)) \rho(x) d x\right|
&\leq C_{R}\Vert g\Vert_{L_{R+1}^{\infty}}\Vert\rho\Vert_{L_{r}^{\infty}}\left\Vert J \phi^{-1}\right\Vert_{L_{R}^{\infty}}\Vert\divv v\Vert_{L_{R+1}^{1}}\\
&\quad+C_{R}\Vert g\Vert_{L_{R+1}^{\infty}}[v]_{C_{R+1}^{\theta}}\left(\Vert D \phi^{-1}\Vert^2_{C_{R+1}^{1-\theta}}\Vert D \rho\Vert_{L_{r}^{\infty}}+\Vert\rho\Vert_{L_{r}^{\infty}}\left[D \phi^{-1}\right]_{C_{R+1}^{1-\theta}}\right),
\end{align*}
and
in addition,
\begin{equation}
\lim _{\eps \to  0} \int_{\R^d} \mathcal{R}_{\eps}[v, g](\phi(x)) \rho(x) d x=0.
\end{equation}
\end{proposition}

Next we will prove some Sobolev type estimates of the Jacobian of the stochastic flow.
We begin with a result  which can be deduced from   \cite{Zhang_2013}; see
\cite[Theorem 2.1]{Dong+Kim_2012} for a related result.

\begin{theorem}\label{TH-Lp-est}%\cite{Zhang_2013}
Let   $L = (L_t)_{t\geq 0}$ be a L\'evy  process with L\'evy triplet $(0,0,\nu)$ satisfying Hypotheses  \ref{hyp-nondeg2} with $\alpha \ge 1$. Let us consider
 \begin{align}\label{eqn-resolvent0}
 \lambda v -  \gen{A} v  - b \cdot Dv = f,\;\; \lambda>0,
 \end{align}
 with $b \in C_{\mathrm{b}}^{\infty}(\mathbb{R}^d,\mathbb{R}^d)$ and $f \in L^p(\R^d)\cap C_{\mathrm{b}}^{\infty}(\mathbb{R}^d)$, for some $p>1$. Let $v \in C^2_\mathrm{b}(\R^d)$ be the classical bounded solution to \eqref{eqn-resolvent0}.  
  Then, the following assertions hold.      
\\
1) If  $\alpha \in(1,2)$ there exists  $c = c(d,\alpha,\lambda,p,\Vert b\Vert_{\infty}) $ such that 
%for every  $\lambda>0$,  $p>1$, $f \in L^p (\R^d)$,  there exists a unique strong solution $v$ in the space $H^{\alpha,p}(\mathbb{R}^d)$ to
%  \begin{align}\label{eqn-resolvent}
%& \lambda v -  \gen{A} v  - b \cdot Dv = f,
% \end{align} 
%and 
the solution $v$ satisfies
 \begin{equation}\label{eqn-resolvent-estimate}
 \Vert v\Vert_{H^{\alpha,p}}
  \le c \Vert f\Vert_{L^p}. 
 \end{equation}
2) If $\alpha =1$, $\beta \in (0,1)$  there exists  $c = c(d,\alpha,\lambda,p,\Vert b\Vert_{\beta}) $
 such that 
\begin{equation}\label{eqn-resolvent-estimate1}
 \Vert v\Vert_{H^{1,p}}
  \le c \Vert f\Vert_{L^p}.   
 \end{equation} 
  \end{theorem} 
\begin{proof} Recall that $H^{1,p}(\mathbb{R}^d)$ $=W^{1,p}(\mathbb{R}^d)$. The existence of $v$ follows by the classical approach of pages 173-177 in  \cite{skor} which are dedicated to  the  study of non-local Kolmogorov equations with smooth coefficients.

The first assertion follows easily from Corollary 4.4 in \cite{Zhang_2013} noting that $b \cdot Dv$ is a ``small perturbation'' of
$ \gen{A} v $.  Indeed, by  \cite[Corollary 4.4]{Zhang_2013} the domain $\mathcal{D}(\gen{A})$ of the operator $\gen{A}$ on the $L^p(\mathbb{R}^d)$ is equal to the Sobolev space $H^{\alpha,p}(\mathbb{R}^d)$ with equivalent norms.

The second assertion can be deduced by  \cite[Theorem 5.3]{Zhang_2013}, which deals with a  related parabolic problem. Let us  explain the proof.
 We first consider  
\begin{gather*}
\partial_t u (t,x) = \gen{A} u(t,x) +  b(x) \cdot Du(t,x) + e^{\lambda t} f(x),\;\;
t \in ]0,T],
\end{gather*} 
 $u(0,x)=0$, $x \in \R^d$. By \cite[Theorem 5.3]{Zhang_2013} 
 %there exists a 
 the unique bounded classical solution $u$ belongs to
 $C([0,T]; W^{1 -1/p,p}(\R^d)) \cap L^p([0,T]; W^{1,p}(\R^d))$.
 Moreover it holds:  
 \begin{equation} \label{s33aa} 
  \int_0^T \|  D u(t)\|_{L^p}^p dt \le C_T  \| f\|_{L^p}^p.
\end{equation}  
 Let us define $v(t,x) = e^{-\lambda t }u(t,x)$. We find that $v$ solves
 \begin{gather*}
\partial_t v (t,x) = \gen{A} v(t,x) +  b(x) \cdot Dv(t,x) +  f(x) - \lambda v(t,x),\;\;
t \in ]0,T], \;\; v(0,x)=0. 
\end{gather*}
By the maximum principle   it follows that $v(t,x) = v(x)$, $(t,x) \in [0,T] \times \R^d$ where $v$ solves  \eqref{eqn-resolvent0}. We find by \eqref{s33aa}
\begin{gather*}
\int_0^T \|  D u(t)\|_{L^p}^pdt = \int_0^T e^{\lambda p t }dt \, \|  D v\|_{L^p}^p 
\le  C_T  \| f\|_{L^p}^p.
\end{gather*}
  The assertion is proved. 
  \end{proof}

The next result is a consequence of \cite[Theorem 1.1]{Felix-2020} or \cite[Lemma A.1]{Bra+Lin+Par}\label{tran-est-frac}.
\begin{proposition}
There  exists  a constant $C>0$, depending  only on $d$, such that for every  $p \in[1,+\infty]$,  $s \in[0,1]$,  and $u \in W^{s, p}(\mathbb{R}^{d})$,
the following estimate holds
\begin{align}\label{tran-est-frac-eq}
\Vert u(\cdot+y)-u\Vert_{L^{p}} \leq C[u]_{W^{s, p}}\vert y\vert^{s} .
\end{align}
\end{proposition}

\begin{lemma}\label{lem-Sob-space-comp}

 Assume the assumptions of  Theorem \ref{d32}.
Assume that $p> 1$,    $\delta \in (0,1)$ and $\eps \in (0,1)$. Let  $\phi^{\eps}=\left( \phi_{t}^{\eps}\right)_{t\geq 0}$ be the flow corresponding to equation \eqref{eqn-SDE} with the vector field  $b$ replaced by $b^{\eps}=b \ast \vartheta_{\eps}$. Assume that  $\tilde{p}>p$, let  $\tilde{\delta}=\delta+d(\frac1p-\frac1{\tilde{p}})$. Then, the following assertion holds. \\
 If   $f\in W^{\tilde{\delta},
\tilde{p}}(\mathbb{R}^d)$,   then for all  $r>0$ and  $s\in[0,T]$, $f \circ \phi_s^{\eps}\in L^p(\Omega; W_r^{\delta,p})$, and there exists  a constant $M_{r,T,\tilde{p},d,\delta,p}$  independent of $f$ and $\eps$, such that 
\begin{align}\label{eq 20250103 01}
   \mathbb{E}\Vert f(\phi_s^{\eps})\Vert^p_{W_r^{\delta,p}}\leq M_{r,T,\tilde{p},d,\delta,p}\|f\|^{p}_{W^{\tilde{\delta},\tilde{p}}},
\end{align}
and,  for all $0\leq s \leq t\leq T$, 
\begin{align}\label{eq 20250103 02}
\mathbb{E}\Vert f((\phi^{\eps}_{s,t})^{-1})\Vert^p_{W_r^{\delta,p}}\leq M_{r,T,\tilde{p},d,\delta,p}\|f\|^{p}_{W^{\tilde{\delta},\tilde{p}}},
\end{align}
where  $(\phi^{\eps}_{s,t})^{-1}$ satisfies  \eqref{de 2025 0102 01}.
\end{lemma}
\begin{proof}
Here we only prove \eqref{eq 20250103 01} because  we can  prove  \eqref{eq 20250103 02} by a  similar argument.
Let us choose and  fix  $p> 1$,    $\delta \in (0,1)$,  $\eps \in (0,1)$,   $\tilde{p}>p$ and    $\tilde{\delta}:=\delta+d(\frac1p-\frac1{\tilde{p}})$.

By using  the definition of the space $W^{\delta,p}_r$, we infer
\begin{align*}
 \mathbb{E}\Vert f(\phi_{s}^{\eps}(\cdot)) \Vert^p_{W^{\delta,p}_r}
\leq C_p\mathbb{E}\int_{B(r)} | f(\phi_{s}^{\eps}(x)) \vert^p dx
+C_p
\mathbb{E}\int_{B(r)} \int_{B(r)} \frac{\big|f( \phi_{s}^{\eps}(x))-f\big(\phi_{s}^{\eps}(y))\big\vert^p}{|x-y\vert^{d+\delta p}}\,dxdy.
\end{align*}

Since $\tilde{p}>p$ by the H\"older inequality and the change of measure theorem   we have
\begin{align}
    &\quad\mathbb{E}\int_{B(r)} \int_{B(r)} \frac{\big|f( \phi_{s}^{\eps}(x))-f\big(\phi_{s}^{\eps}(y))\big\vert^p}{|x-y\vert^{d+\delta p}}\,dxdy\\
    &=
    \mathbb{E}\int_{B(r)} \int_{B(r)} \frac{\big|f( \phi_{s}^{\eps}(x))-f\big(\phi_{s}^{\eps}(y))\big\vert^p}{|\phi_{s}^{\eps}(x)-\phi_{s}^{\eps}(y)\vert^{d+\delta p}}\frac{|\phi_{s}^{\eps}(x)-\phi_{s}^{\eps}(y)\vert^{d+\delta p}}{|x-y\vert^{d+\delta p}}\,dxdy\\
    &\leq
    \Big(\mathbb{E}\int_{B(r)} \int_{B(r)} \frac{\big|f( \phi_{s}^{\eps}(x))-f\big(\phi_{s}^{\eps}(y))\big\vert^{\tilde{p}}}{|\phi_{s}^{\eps}(x)-\phi_{s}^{\eps}(y)\vert^{d+\tilde{\delta}\tilde{p}}}\,dxdy\Big)^{\frac{p}{\tilde{p}}}\\
    &\quad\cdot\Big(\mathbb{E}\int_{B(r)} \int_{B(r)} \Big(\frac{|\phi_{s}^{\eps}(x)-\phi_{s}^{\eps}(y)\vert^{d+\delta p}}{|x-y\vert^{d+\delta p}}\Big)^{\frac{\tilde{p}}{\tilde{p}-p}}\,dxdy\Big)^{\frac{\tilde{p}-p}{\tilde{p}}}.
\end{align}
The first factor (inside the bracket) above we estimate by  
    \begin{align}
    &\quad\mathbb{E}\int_{\phi_{s}^{\eps}(B(r))} \int_{\phi_{s}^{\eps}(B(r))} \frac{\big|f(x')-f\big(y')\big\vert^{\tilde{p}}}{|x'-y'\vert^{d+\tilde{\delta}\tilde{p}}}J (\phi_s^{\eps})^{-1}(x')J (\phi_s^{\eps})^{-1}(y')\,dx'dy'\\
    &=
    \int_{\mathbb{R}^d} \int_{\mathbb{R}^d} \frac{\big|f(x')-f\big(y')\big\vert^{\tilde{p}}}{|x'-y'\vert^{d+\tilde{\delta}\tilde{p}}}\mathbb{E}\Big(J (\phi_s^{\eps})^{-1}(x')J (\phi_s^{\eps})^{-1}(y')\1_{\phi_{s}^{\eps}(B(r))}(x')\1_{\phi_{s}^{\eps}(B(r))}(y')\Big)\,dx'dy'
    \\
&\leq 
   [f]^{\tilde{p}}_{W^{\tilde{\delta},\tilde{p}}}\sup_{(\eps,s,x) \in(0,1)\times [0,T] \times \mathbb{R}^d  }\mathbb{E}\Big( \vert J (\phi_s^{\eps})^{-1}(x) \vert^2\Big),    \end{align}
while,   using $\phi_{s}^{\eps}(x)-\phi_{s}^{\eps}(y)=\int_0^1D\phi_{s}^{\eps}(y+\theta(x-y))(x-y)d\theta$, we get the following estimate for  second factor
\begin{align}
    &\Big(\mathbb{E}\int_{B(r)} \int_{B(r)} \Big(\frac{|\phi_{s}^{\eps}(x)-\phi_{s}^{\eps}(y)\vert^{d+\delta p}}{|x-y\vert^{d+\delta p}}\Big)^{\frac{\tilde{p}}{\tilde{p}-p}}\,dxdy\Big)^{\frac{\tilde{p}-p}{\tilde{p}}}\\
     \leq&
    \Big(\int_{B(r)} \int_{B(r)} \int_0^1\mathbb{E}\Big(\|D\phi_{s}^{\eps}(y+\theta(x-y))\|^{\frac{(d+\delta p)\tilde{p}}{\tilde{p}-p}}\Big)d\theta\,dxdy\Big)^{\frac{\tilde{p}-p}{\tilde{p}}}\\
\leq&
  \Big(\sup_{\eps\in(0,1)} \sup_{s\in[0,T]}\sup_{x\in \mathbb{R}^d}\mathbb{E}\Big(\|D\phi_{s}^{\eps}(x)\|^{\frac{(d+\delta p)\tilde{p}}{\tilde{p}-p}}\Big)\Big)^{\frac{\tilde{p}-p}{\tilde{p}}}
   \Big(\int_{B(r)} \int_{B(r)} 1\,dxdy\Big)^{\frac{\tilde{p}-p}{\tilde{p}}}.
\end{align}

Hence
    \begin{align}
    \mathbb{E}\int_{B(r)} \int_{B(r)} \frac{\big|f( \phi_{s}^{\eps}(x))-f\big(\phi_{s}^{\eps}(y))\big\vert^p}{|x-y\vert^{d+\delta p}}\,dxdy
\leq
C_{r,T,\tilde{p},d,\delta,p} [f]^{p}_{W^{\tilde{\delta},\tilde{p}}}.
\end{align}
Here 
\begin{align}
C_{r,T,\tilde{p},d,\delta,p}:=&
\Big(\sup_{\eps\in(0,1)} \sup_{s\in[0,T]}\sup_{x\in \mathbb{R}^d}\mathbb{E}|J (\phi_s^{\eps})^{-1}(x) \vert^2\Big)^{\frac{p}{\tilde{p}}}\\
&\Big(\sup_{\eps\in(0,1)} \sup_{s\in[0,T]}\sup_{x\in \mathbb{R}^d}\mathbb{E}\Big(\|D\phi_{s}^{\eps}(x)\|^{\frac{(d+\delta p)\tilde{p}}{\tilde{p}-p}}\Big)\Big)^{\frac{\tilde{p}-p}{\tilde{p}}}\Big(\int_{B(r)} \int_{B(r)} 1\,dxdy\Big)^{\frac{\tilde{p}-p}{\tilde{p}}}.
   \end{align}
Meanwhile, we have
\begin{align*}
\mathbb{E}\int_{B(r)} | f(\phi_{s}^{\eps}(x)) \vert^p dx
&\leq
C_{p,\tilde{p},r}\Big(\mathbb{E}\int_{B(r)} | f(\phi_{s}^{\eps}(x)) \vert^{\tilde{p}} dx\Big)^{\frac{p}{\tilde{p}}}\\
&=
C_{p,\tilde{p},r}\Big(\mathbb{E}\int_{\phi_s^{\eps}(B(r))} |f(y)\vert^{\tilde{p}} J (\phi_s^{\eps})^{-1}(y) dy\Big)^{\frac{p}{\tilde{p}}}\\
&\leq
C_{p,\tilde{p},r}\Big(\sup_{\eps\in(0,1)} \sup_{s\in[0,T]}\sup_{x\in \mathbb{R}^d}\mathbb{E} |J (\phi_s^{\eps})^{-1}(x) |\Big)^{\frac{p}{\tilde{p}}}
\Vert f  \Vert^{p}_{L^{\tilde{p}}}\\
&=
C_{p,\tilde{p},r,T}
\Vert f  \Vert^{p}_{L^{\tilde{p}}}.
\end{align*}
Here $C_{p,\tilde{p},r,T}:=C_{p,\tilde{p},r}\Big(\sup_{\eps\in(0,1)} \sup_{s\in[0,T]}\sup_{x\in \mathbb{R}^d}\mathbb{E} |J (\phi_s^{\eps})^{-1}(x) |\Big)^{\frac{p}{\tilde{p}}}.$

   Note that according to Theorem \ref{thm-stability} and Remark \ref{rem Sec 5 000},
  $C_{r,T,\tilde{p},d,\delta,p}+C_{\tilde{p},r,T}<\infty$. Obviously $C_{r,T,\tilde{p},d,\delta,p}+C_{\tilde{p},r,T}$  is independent on $f$ and $\eps$.

  The result follows.
\end{proof}

%-----------------------------------------------

\begin{theorem}\label{Reg-est-Jaco} Let us assume that $p\geq 2$ and that the parameters $\alpha\in[1,2)$  and
 $\beta  \in (0, 1)  $ satisfy condition \eqref{eqn-beta+alpha half}. Let   $L = (L_t)_{t\geq 0}$ be a L\'evy  process with L\'evy triplet $(0,0,\nu)$ satisfying Hypothesis \ref{hyp-nondeg2}, i.e., 
$L$ is a symmetric $\alpha$-stable  process, which is non-degenerate in the sense that it satisfies \eqref{stim}.

 Assume that $b\in  C_{\mathrm{b}}^{\beta}( \mathbb{R}^{d}, \R^{d})$ and  $\divv b \in L^{p}( \mathbb{R}^{d})$.  Then,
           for all $\delta \in [0,\frac{\alpha}2)$ and  $r>0$,
     \begin{align}
     J\phi\in L^2(0,T;W_r^{\delta,p}), \;\; \mathbb{P}\mbox{-a.s.}
     \end{align}
\end{theorem}

\begin{proof}
Let us choose and fix numbers  $p\geq 2$,
 $\alpha\in[1,2)$  and
 $\beta  \in (0, 1)  $ satisfying condition
\eqref{eqn-beta+alpha half}.
 Assume also that that $b\in C_{\mathrm{b}}^{\beta}( \mathbb{R}^{d}, \R^{d})$  such that  $\divv b \in L^{p}( \mathbb{R}^{d})$.
 Let us choose and fix $r>0$ and $\delta \in [0,\frac{\alpha}2)$.\\
 \indent \textbf{Step 1}  In view of the chain rule for fractional Sobolev spaces, we infer that 
 \begin{equation}
[J \phi_{t}]_{W_{r}^{\delta, p}}=[e^{\log(J \phi_{t})}]_{W_{r}^{\delta, p}} \leq\Big(\sup _{x \in B(r)}\big|J \phi_{t}(x)\big|\Big)[\log J \phi_{t}]_{W_{r}^{\delta, p}}, \;\; t \in [0,T].
\end{equation}
Since by 
{Theorem \ref{ww1}} $\sup_{t\in[0,T],\,x\in B(r)}|J\phi_t(x)|<\infty$,  it is enough to prove that
\begin{equation*}
\log J \phi_{\cdot}(\cdot)\in L^2(0,T;W_r^{\delta,p}).
\end{equation*}
\indent \textbf{Step 2}:  Define, for $0<\eps<1$,  as that in the proof of Theorem \ref{thm-transport equation}, $b^{\eps}=b*\vartheta_{\eps} $. Let us also set $b^{0}=b$. Let $\phi_{t}^{\eps}$ be the flow corresponding to \eqref{eqn-SDE} with $b$ replaced by $b^{\eps}$. According to Lemma \ref{lem-A1},  we have
\begin{equation}
\log J \phi_{t}^{\eps}(x)=\int_{0}^{t} \divv b^{\eps}\left(\phi_{s}^{\eps}(x)\right) d s.
\end{equation}
Define
\begin{equation}
\psi_{\eps}(t, x)=\int_{0}^{t} \divv b^{\eps}\left(\phi_{s}^{\eps}(x)\right) d s .
\end{equation}
By Theorem \ref{thm-stability}, we have for every $R>0$, $p\geq 1$,
\begin{align}
\sup_{\eps \in (0,1)}\mathbb{E}\int_0^T  \Vert D\phi_{t}^{\eps}\Vert^{p}_{L_R^{p}} dt+\mathbb{E}\int_0^T  \Vert D\phi_{t} \Vert^{p}_{L_R^{p}} dt<\infty,
\end{align}
and
\begin{align}
\lim_{\eps\to 0}\mathbb{E}\int_0^T  \Vert D\phi_{t}^{\eps}-D\phi_{t} \Vert^{p}_{L_R^{p}} dt=0.
\end{align}
Then we can extract a sequence $\{\eps_n\}$ such that
\begin{align}
D\phi_{t}^{\eps_n}(x) \to  D\phi_{t}(x),\quad \text{a.e. in }t,\,x,\,\omega,\;\text{as }\eps_n\to 0.
\end{align}
By the continuity of the determinant function,  we have
\begin{align}\label{eq FT}
\log J \phi_{t}^{\eps_n}(x)=\psi_{\eps_{n}}(t, x) \to  \log J \phi_{t}(x),\quad \text{a.e. in }t,\,x,\,\omega,\;\text{as }\eps_n\to 0.
\end{align}
In view of Remark \ref{rem Sec 5 000}, we see, for $p\geq 2$
\begin{align*}
&\quad\ \sup_{\eps \in (0,1)}\Big(\mathbb{E} \int_{0}^{T} \int_{B(r)}|\psi_{\eps}(t, x)\vert^{p} d x d t\Big) \\
&=\sup_{\eps \in (0,1)}\mathbb{E} \int_{0}^{T}\ \int_{B(r)}\Big|\int_{0}^{t} \divv b^{\eps}( \phi_{s}^{\eps}(x)) d s\Big|^{p} d x d t\\
& \leq C_{p,T}\sup_{\eps \in (0,1)} \mathbb{E} \int_{0}^{T}  \int_{B(r)}\big|\divv b^{\eps}\big(\phi_{s}^{\eps}(x)\big)\big\vert^{p} d x d s \\
&\leq C_{p,T}\sup_{\eps \in (0,1)} \mathbb{E} \int_{0}^{T}  \int_{\mathbb{R}^{d}}\big|\divv b^{\eps}( y)\big\vert^{p} J(\phi_{s}^{\eps})^{-1}(y)d y  ds\\
& \leq C_{p,T} \sup_{\eps \in (0,1)}  \sup _{s \in[0, T], y \in \mathbb{R}^{d}} \mathbb{E}\Big[J\big(\phi_{s}^{\eps}\big)^{-1}(y)\Big] \int_{0}^{T}  \int_{\mathbb{R}^{d}}\Big|\divv b^{\eps}( y)\Big\vert^{p} d y d s  \leq C_{p,T}<\infty.
\end{align*}
This shows that $\left(\psi_{\eps}\right)_{\eps \in (0,1)}$ is uniformly bounded in $L^p(\Omega\times(0,T);L^p_r)$ and hence $\left(\psi_{\eps}\right)_{\eps \in (0,1)}$ is uniformly integrable on $\Omega \times(0, T) \times B(r)$. Up to extraction of a subsequence $\psi_{\eps_n}$, 
the Vitali theorem ensures that $\psi_{\eps_n}$ converges strongly in $L^1(\Omega\times(0,T)\times B(r))$ to $\log J\phi$. Notice that the dual space of the Banach space $W_r^{-\delta,p'}$, $p'=\frac{p}{p-1}$  is $W_r^{\delta,p}$. If we can prove that $(\psi_{\eps})_{\eps \in (0,1)}$ is uniformly bounded in $L^2(\Omega\times(0,T);W_r^{\delta,p})$, what  will be done in \textbf{Step 3}, then according to the Banach-Alaoglu Theorem, we can extract a subsequence $(\psi_{\eps_{n_k}})$  of  $(\psi_{\eps_n})$ such that $\psi_{\eps_{n_k}}$ converges weakly star in $L^2(\Omega\times(0,T);W_r^{\delta,p})$ to some $\psi'\in L^2(\Omega\times(0,T);W_r^{\delta,p})$. Consequently,  $\psi_{\eps_{n_k}}$ converges weakly star in $L^2(\Omega\times(0,T);L^p_r)$ to $\psi'$. By the uniqueness of the limit, we infer $\psi'=\log J\phi$. Hence $\log J\phi \in L^2(\Omega\times(0,T);W_r^{\delta,p})$.

\textbf{Step 3}: In this step, we shall prove that $(\psi_{\eps})_{\eps \in (0,1)}$ is uniformly bounded in $L^2(\Omega\times(0,T);W_r^{\delta,p})$.

 Consider first the case when 
  { $\delta\in(\alpha-1,\frac{\alpha}2)$.}
Recall that $b^{\eps} \in C_{\mathrm{b}}^{\infty}(\R^d,\R^d)$.  We introduce the following Kolmogorov problem, for $0<\eps <1$,
\begin{align}\label{eq-Cauchy-pro-1}
&\lambda v^{\eps} -  \gen{A} v^{\eps}  - b^{\eps} \cdot Dv^{\eps} = \divv b^{\eps}.
\end{align}
By Theorem \ref{TH-Lp-est}, there exists $c>0$ such that
\begin{align}\label{est-v-divb}
 \Vert v^{\eps}\Vert_{H^{\alpha,p}}
  \le c \Vert \divv b^\eps  \Vert_{L^p}  \le c \Vert \divv b \Vert_{L^p}, \;\;  \eps \in (0,1]. 
\end{align} 
  Moreover,  
  by applying the   It\^o formula,
  we obtain 
\begin{align}\label{psi-Ito-iden}
&v^{\eps}\left(\phi_{t}^{\eps}(x)\right)-v^{\eps}(x)-\int_{0}^{t}\lint_{\mathbb{R}^d} v^{\eps}\left( \phi_{s-}^{\eps}(x)+z\right)-v^{\eps}\left( \phi_{s-}^{\eps}(x)\right) \tilde{\mu}(ds,dz)-\lambda \int_0^t v^{\eps}(\phi_s^{\eps}(x))ds\nonumber\\
=&-\int_{0}^{t} \divv b^{\eps}\left(\phi_{s}^{\eps}(x)\right) d s=-\psi_{\eps}(t, x).
\end{align} 
To this purpose, note that since $b^{\eps} $ and div $ b^{\eps}$ are $ C_{\mathrm{b}}^{\infty}$-functions, the solutions $v^{\eps}$ are $C_{\mathrm{b}}^2$-functions, see  the proof of Theorem
 \ref{TH-Lp-est} where we mention 
\cite{skor}.  
Alternatively, one can remark that by Theorem \ref{reg},  we have $v^\eps \in C^{\alpha+\beta}_{\mathrm{b}}(\mathbb{R}^d)$ and we can apply the   It\^o formula as in  \cite[(4.6)]{Pr12}.

Since the space $W_r^{\delta,p}$, $p\geq 2$ is a martingale type $2$ Banach space, we infer (cf. \cite{Zhu_2010} or \cite{Zhu+Brz+Liu-2019}) 
 \begin{align*} 
&\mathbb{E}\int_{0}^{T} \Big\Vert\int_{0}^{t}\lint_{\mathbb{R}^d} v^{\eps}\big( \phi_{s-}^{\eps}(\cdot)+z\big)-v^{\eps}( \phi_{s-}^{\eps}(\cdot)) \tilde{\mu}(ds,dz)\Big\Vert^2_{W^{\delta,p}_r} dt\\
\leq &C_{p,T}\mathbb{E}\Big(\int_{0}^{T}\lint_{\mathbb{R}^d} \Vert v^{\eps}\big( \phi_{s}^{\eps}(\cdot)+z\big)-v^{\eps}( \phi_{s}^{\eps}(\cdot))\Vert^2_{W_r^{\delta,p}}\nu(dz)ds\Big)\\
= &C_{p,T}\mathbb{E}\Big(\int_{0}^{T}\lint_{B} \Vert v^{\eps}( \phi_{s}^{\eps}(\cdot)+z)-v^{\eps}( \phi_{s}^{\eps}(\cdot))\Vert^2_{W_r^{\delta,p}}\nu(dz)ds\Big)\\
&+C_{p,T}\mathbb{E}\Big(\int_{0}^{T}\int_{B^c} \Vert v^{\eps}( \phi_{s}^{\eps}(\cdot)+z)-v^{\eps}( \phi_{s}^{\eps}(\cdot))\Vert^2_{W_r^{\delta,p}}\nu(dz)ds\Big)\\
:=&I_1+I_2.
\end{align*}
Since $\alpha\in[1,2)$ and $\delta\in (\alpha-1,\frac{\alpha}2)$, we can always choose $\tilde{p}>p$ such that $\frac{\alpha-\delta-1}2<d(\frac1p-\frac1{\tilde{p}})<\frac{\alpha-2\delta}4$. It follows that $s:=\alpha-\delta-2d(\frac1p-\frac1{\tilde{p}})\in(0,1)$.
By applying Lemma \ref{lem-Sob-space-comp},  we obtain
\begin{align}\label{eq 20251108 01}
 \mathbb{E}\Vert v^{\eps}( \phi_{s}^{\eps}(\cdot)+z)-v^{\eps}( \phi_{s}^{\eps}(\cdot))\Vert^2_{W_r^{\delta,p}}\leq M_{r,T,\tilde{p},d,\delta,p} \Vert v^{\eps}(\cdot+z)-v^{\eps}(\cdot)    \Vert^2_{W^{\tilde{\delta},\tilde{p}}}
\end{align}
where $\tilde{\delta}=\delta+d(\frac1p-\frac1{\tilde{p}})$.
Using the Sobolev embedding $H^{\tilde{\delta},\tilde{p}}(\mathbb{R}^d)\subset W^{\tilde{\delta},\tilde{p}}(\mathbb{R}^d)$, for $\tilde{p}> 2$ (c.f.  \cite[p. 155, Theorem 5]{Stein}, \cite[p. 49 (9)]{Triebel-10}) gives
\begin{align}\label{eq 20251108 02}
I_1&\leq C_{p,T}M_{r,T,\tilde{p},d,\delta,p} \Big(\int_0^T \lint_{B}\Vert v^{\eps}(\cdot+z)-v^{\eps}(\cdot)    \Vert^2_{W^{\tilde{\delta},\tilde{p}}} \nu(dz)ds\Big)
\\
&\leq C_{r,T,\tilde{p},d,\delta,p}\Big(\int_0^T \lint_{B}\Vert v^{\eps}(\cdot+z)-v^{\eps}(\cdot)    \Vert^2_{H^{\tilde{\delta},\tilde{p}}} \nu(dz)ds\Big)
\\
&= C_{r,T,\tilde{p},d,\delta,p}\Big(
\int_0^T \lint_{B}
\big\Vert \mathcal{J}^{\tilde{\delta}} \big(v^{\eps}\left( \cdot+z\right)\big)-
\mathcal{J}^{\tilde{\delta}}\big(
v^{\eps}
( \cdot)\big)
\Vert^2_{L^{\tilde{p}}} \nu(dz)ds\Big).
\end{align}
By using  the translation estimate in the fractional Sobolev space \eqref{tran-est-frac-eq} with
 $s=\alpha-\delta-2d(\frac1p-\frac1{\tilde{p}})\in(0,1)$, the Sobolev embedding $H^{s,\tilde{p}}(\mathbb{R}^d)\subset W^{s,\tilde{p}}(\mathbb{R}^d)$, $\tilde{p}> 2$, the fact that $\mathcal{J}^{\tilde{\delta}}$ is an isomorphism of $H^{s+\tilde{\delta},\tilde{p}}(\mathbb{R}^d)$ to $H^{s,\tilde{p}}(\mathbb{R}^d)$ (c.f. \cite[Section 2.3.8, Theorem]{Triebel-10}) and also the Sobolev embedding $H^{\alpha,p}(\mathbb{R}^d)\subset H^{\tilde{\delta}+s,\tilde{p}}(\mathbb{R}^d)$ (which holds because the parameters satisfy the relation $\alpha-\frac{d}p=\tilde{\delta}+s-\frac{d}{\tilde{p}}$), we infer
\begin{align}\label{eq 20251108 03}
\big\Vert \mathcal{J}^{\tilde{\delta}} \big(v^{\eps}\left( \cdot+z\right)\big)- \mathcal{J}^{\tilde{\delta}}\big(v^{\eps}( \cdot)\big)
\Vert^2_{L^{\tilde{p}}}
&=\big\Vert \mathcal{J}^{\tilde{\delta}} (v^{\eps})( \cdot+z)- \mathcal{J}^{\tilde{\delta}}(v^{\eps})(\cdot)\Vert^2_{L^{\tilde{p}}}\\
&\leq C_d\big\Vert \mathcal{J}^{\tilde{\delta}}(v^{\eps})\Vert^2_{W^{s,\tilde{p}}}|z\vert^{2s}\\
&\leq C_d\big\Vert \mathcal{J}^{\tilde{\delta}}(v^{\eps})\Vert^2_{H^{s,\tilde{p}}}|z\vert^{2s}\\
&= C_d\big\Vert v^{\eps}\Vert^2_{H^{s+\tilde{\delta},\tilde{p}}}|z\vert^{2s}\\
&\leq C_d\big\Vert v^{\eps}\Vert^2_{H^{\alpha,p}}|z\vert^{2s}\\
&\leq C_d\big\Vert \divv b \Vert^2_{L^{p}}|z\vert^{2s},
\end{align}
where we used \eqref{est-v-divb} in the last inequality.
Observe that $2s>\alpha$. 
 By Hypothesis \ref{hyp-nondeg2},  we have $\lint_{B}|z\vert^{2s}\nu(dz)<\infty$. It follows that
\begin{align*}
I_1&\leq C_{r,T,\tilde{p},d,\delta,p}\Vert  \divv b  \Vert_{L^p}^2 \Big(\lint_{B}|z\vert^{2s}\nu(dz)
\Big)\leq C<\infty,
\end{align*}
where $C_{r,T,\tilde{p},d,\delta,p}$ is independent of $\eps$.\\
\indent For the other term $I_2$, choose $\tilde{p}>p$ such that $\alpha-\frac{d}p\geq\tilde{\delta}-\frac{d}{\tilde{p}}$, where $\tilde{\delta}=\delta+d(\frac1p-\frac1{\tilde{p}})$   and then apply Lemma \ref{lem-Sob-space-comp}   and the Sobolev embedding $H^{\alpha,p}(\mathbb{R}^d)\subset H^{\tilde{\delta},\tilde{p}}(\mathbb{R}^d)$ to get
\begin{align*}
   I_2&=C_{p,T}\mathbb{E}\Big(\int_{0}^{T}\int_{B^c} \Vert v^{\eps}( \phi_{s}^{\eps}(\cdot)+z)-v^{\eps}( \phi_{s}^{\eps}(\cdot))\Vert^2_{W_r^{\delta,p}}\nu(dz)ds\Big)\\
   & \leq C_{r,T,\tilde{p},d,\delta,p}\Big(\int_{0}^{T}\int_{B^c} \Vert v^{\eps}( \cdot+z)-v^{\eps}( \cdot)\Vert^2_{W^{\tilde{\delta},\tilde{p}}}\nu(dz)ds\Big)\\
   & \leq C_{r,T,\tilde{p},d,\delta,p}\Vert v^{\eps}\Vert^2_{W^{\tilde{\delta},\tilde{p}}} \big(\nu(\{|z|>1\})\big)
   \\
   &\leq C_{r,T,\tilde{p},d,\delta,p} \Vert v^{\eps}\Vert_{H^{\alpha,p}}^2\big(\nu(\{|z|>1\})\big)
   \\
   &\leq C_{r,T,\tilde{p},d,\delta,p} \Vert  \divv b  \Vert_{L^p}^2\big(\nu(\{|z|>1\})\big)<\infty.
\end{align*}
   Combining the above estimates, we find
   \begin{align*}
&\mathbb{E}\int_{0}^{T} \Big\Vert\int_{0}^{t}\lint_{\mathbb{R}^d} v^{\eps}
\left( \phi_{s-}^{\eps}(\cdot)+z\right)-v^{\eps}
\left( \phi_{s-}^{\eps}(\cdot)\right) \tilde{\mu}(ds,dz)\Big\Vert^2_{W^{\delta,p}_r} dt\leq C<\infty,
\end{align*}
where the constant $C$ is independent of $\eps$.\\
Similarly  to $I_2$,  we  apply Lemma \ref{lem-Sob-space-comp} to get
\begin{align}
\mathbb{E}\int_{0}^{T} \Big\Vert \lambda \int_0^t v^{\eps}(\phi_s^{\eps}(\cdot))ds \Big\Vert_{W_r^{\delta,p}}^2 d t
&\leq  \lambda^2 T \mathbb{E}\int_{0}^{T}  \Vert v^{\eps}(\phi_t^{\eps}(\cdot))\Vert_{W_r^{\delta,p}}^2 d t\\
\leq C_{\lambda,r,T,\tilde{p},d,\delta,p} \Vert v^{\eps} \Vert^2_{W^{\tilde{\delta},\tilde{p}}}
&\leq  C_{\lambda,r,T,\tilde{p},d,\delta,p}  \Vert  \divv b  \Vert_{L^p}^2\leq C<\infty.
\end{align}
Consequently,
$
(\psi_{\eps})_{\eps \in (0,1)}
$
is bounded in $L^2((\Omega\times(0,T);W_r^{\delta,p})$, for 
 {$\delta\in(\alpha-1,\frac{\alpha}2)$}. Since $W_r^{\delta_2,p} \subset W_r^{\delta_1,p}$ when $\delta_1<\delta_2$, the uniform boundedness of $ (\psi_{\eps})_{\eps \in (0,1)}$ also holds for all $\delta\in[0,\frac{\alpha}2)$.

\end{proof}

%-----------------------------------------------
\begin{theorem}\label{Reg-est-Jaco-2}
 
{ Let us assume that   $\alpha\in[1,2)$  and
 $\beta  \in (0, 1)  $ be such that \eqref{eqn-beta+alpha half} holds. Let   $L = (L_t)_{t\geq 0}$ be a L\'evy  process with L\'evy triplet $(0,0,\nu)$ satisfying Hypothesis \ref{hyp-nondeg2}.}
     Assume that $b\in C_{\mathrm{b}}^{\beta}(\mathbb{R}^d,\mathbb{R}^d)$ and there exists
     {$p \in (\alpha,2]$} such that $\divv b \in L^{p}( \mathbb{R}^{d})$.\\ Then for any $\delta \in [0,\frac{p-1}{p}\alpha)$ and $r>0$,
     \begin{align*}
     J\phi\in L^p(0,T;W_r^{\delta,p}), \;\; \mathbb{P}\mbox{-a.s.}
     \end{align*}
\end{theorem}

The first two steps of the proof are similar to the proof of Theorem \ref{Reg-est-Jaco}. The crucial part is step 3.
\begin{proof}
\textbf{Step 1}  We may apply similar arguments as in the proof of Theorem \ref{Reg-est-Jaco} to get
 \begin{equation}
[J \phi_{t}]_{W_{r}^{\delta, p}}^{p}\leq\Big(\sup _{x \in B(r)}|J \phi_{t}(x)|\Big)^{p}[\log J \phi_{t}]_{W_{r}^{\delta, p}}^{p}.
\end{equation}
Hence it is enough to prove that
\begin{equation}
\log J \phi_{\cdot}(\cdot)\in L^p(0,T;W_r^{\delta,p}).
\end{equation}
\textbf{Step 2.} Similarly, we define as in Theorem \ref{Reg-est-Jaco},  $b^{\eps}(x)=\left(b * \vartheta_{\eps}\right)(x), \eps \in (0,1)$ and \begin{equation}
\psi_{\eps}(t, x)=\int_{0}^{t} \divv b^{\eps}\left(\phi_{s}^{\eps}(x)\right) d s .
\end{equation}
Observe that
\begin{align*}
&\quad\ \sup_{\eps \in (0,1)}\Big(\mathbb{E} \int_{0}^{T} \int_{B(r)}|\psi_{\eps}(t, x)\vert^{p} d x d t\Big)\\
& \leq C_{p,T} \sup_{\eps \in (0,1)}  \sup _{s \in[0, T], y \in \mathbb{R}^{d}} \mathbb{E}\left[J\left(\phi_{s}^{\eps}\right)^{-1}(y)\right] \int_{0}^{T}  \int_{\mathbb{R}^{d}}\left|\divv b^{\eps}( y)\right\vert^{p} d y d s  \leq C_{p,T}<\infty.
\end{align*}
By the Vitali theorem, there exists a subsequence $\psi_{\eps_n}$ converges strongly in $L^1(\Omega\times(0,T)\times B(r))$ to $\log J\phi$.  Therefore, we only need to prove that $(\psi_{\eps})_{\eps \in (0,1)}$ is uniformly bounded in $L^p(\Omega\times(0,T);W_r^{\delta,p})$.

\textbf{Step 3} To prove the  uniform boundedness of $(\psi_{\eps})_{\eps \in (0,1)}$ in $L^p(\Omega\times(0,T);W_r^{\delta,p})$, in view of \eqref{psi-Ito-iden}, we now need only to show
   \begin{align}\label{uni-boundedness-ST}
&\mathbb{E}\int_{0}^{T} \Big\Vert\int_{0}^{t}\lint_{\mathbb{R}^d} v^{\eps}\left( \phi_{s-}^{\eps}(\cdot)+z\right)-v^{\eps}\left( \phi_{s-}^{\eps}(\cdot)\right) \tilde{\mu}(ds,dz)\Big\Vert^p_{W^{\delta,p}_r} dt\leq C<\infty.
\end{align}
Since the space $W_r^{\delta,p}$  is a martingale type $p$ Banach space,  by using \cite[Propositon 2.2]{Zhu+Brz+Liu-2019}, we infer
\begin{align*}
&\mathbb{E}\int_{0}^{T} \Big\Vert\int_{0}^{t}\lint_{\mathbb{R}^d} v^{\eps}( \phi_{s-}^{\eps}(\cdot)+z)-v^{\eps}( \phi_{s-}^{\eps}(\cdot)) \tilde{\mu}(ds,dz)\Big\Vert^p_{W^{\delta,p}_r} dt\\
\leq &C_{T}\mathbb{E}\Big(\int_{0}^{T}\lint_{\mathbb{R}^d} \Vert v^{\eps}( \phi_{s}^{\eps}(\cdot)+z)-v^{\eps}( \phi_{s}^{\eps}(\cdot))
\Vert^p_{W_r^{\delta,p}}\nu(dz)ds\Big)\\
= &C_{T}\mathbb{E}\Big(\int_{0}^{T}\lint_{B} \Vert v^{\eps}( \phi_{s}^{\eps}(\cdot)+z)-v^{\eps}( \phi_{s}^{\eps}(\cdot))\Vert^p_{W_r^{\delta,p}}\nu(dz)ds\Big)\\
&+C_{T}\mathbb{E}\Big(\int_{0}^{T}\int_{B^c} \Vert v^{\eps}( \phi_{s}^{\eps}(\cdot)+z)-v^{\eps}
( \phi_{s}^{\eps}(\cdot))\Vert^p_{W_r^{\delta,p}}\nu(dz)ds\Big)\\
:=&I_1+I_2.
\end{align*}
Let's consider $I_1$ first. Since $\alpha\in[1,2)$, $p \in (\alpha,2]$ and $\delta\in[0,\frac{p-1}p\alpha)$, we can always choose $\tilde{p}>p$ such that $\frac{\alpha-\delta-1}2<d(\frac1p-\frac1{\tilde{p}})<\frac{\frac{p-1}{p}\alpha-\delta}2$. It follows that $s:=\alpha-\delta-2d(\frac1p-\frac1{\tilde{p}})\in(0,1)$.
By applying Lemma \ref{lem-Sob-space-comp},  we obtain
\begin{align*}
I_1
&\leq C_TM_{r,T,\tilde{p},d,\delta,p}\int_0^T\lint_{B} \Vert v^{\eps}(\cdot+z)-v^{\eps}(\cdot)    \Vert^p_{W^{\tilde{\delta},\tilde{p}}}\nu(dz)ds.
\end{align*}
Here $\tilde{\delta}=\delta+d(\frac1p-\frac1{\tilde{p}})$.
Since $d(\frac1p-\frac1{\tilde{p}})<\frac{\frac{p-1}{p}\alpha-\delta}2$, we have $sp>\alpha$. Choosing $\eps_0>0$ such that 
$(s-2\eps_0)p>\alpha$,
which implies that $ \lint_{B} |z\vert^{(s-2\eps_0)p}\nu(dz) <\infty$ by 
 {Hypothesis \ref{hyp-nondeg2}}.
Since $H^{s+\eps',\kappa}(\mathbb{R}^d) \subset W^{s,\kappa}(\mathbb{R}^d) \subset H^{s-\eps',\kappa}(\mathbb{R}^d) $ for any $s>0$, $\kappa\in(1,+\infty)$ and $0<\eps'<s$ (cf. \cite[Section 2.3.3]{Triebel}),
we obtain
\begin{align}\label{eq-20251108-1}
\big\Vert  v^{\eps}( \cdot+z)- v^{\eps}( \cdot)\Vert^p_{W^{\tilde{\delta},\tilde{p}}}&\leq C
   \big\Vert  v^{\eps}( \cdot+z)- v^{\eps}( \cdot)\Vert^p_{H^{\tilde{\delta} +\eps_0,\tilde{p}}}\\
   &= C \Vert J^{\tilde{\delta}+\eps_0}v^{\eps}( \cdot+z)- J^{\tilde{\delta}+\eps_0}v^{\eps}( \cdot)\Vert^{p}_{L^{\tilde{p}}}\\
   &\leq C_d\Vert J^{\tilde{\delta}+\eps_0}v^{\eps}\Vert_{W^{s-2\eps_0,\tilde{p}}}^p|z\vert^{(s-2\eps_0)p}\\
   &\leq C_d\Vert J^{\tilde{\delta}+\eps_0}v^{\eps}\Vert_{H^{s-\eps_0,\tilde{p}}}^p|z\vert^{(s-2\eps_0)p}\\
   &=C_d\Vert v^{\eps}\Vert_{H^{\tilde{\delta}+s,\tilde{p}}}^p|z\vert^{(s-2\eps_0)p}\\
      &\leq C_d\Vert v^{\eps}\Vert_{H^{\alpha,p}}^p|z\vert^{(s-2\eps_0)p}\\
     &\leq C_d\Vert \divv b\Vert_{L^p}^p|z\vert^{(s-2\eps_0)p},
\end{align}
where we used the translation estimate \eqref{tran-est-frac-eq} in the second inequality, the Sobolev embedding $H^{\alpha,p}(\mathbb{R}^d)\subset H^{\tilde{\delta}+s,\tilde{p}}(\mathbb{R}^d)$, and \eqref{est-v-divb}.  It follows that
\begin{align}
I_1\leq C_{T,d}M_{r,T,\tilde{p},d,\delta,p}  \Vert \divv b\Vert_{L^p}^p \lint_{B} |z\vert^{(s-2\eps_0)p}\nu(dz) \leq C<\infty,
\end{align}
where the constant $C$ is independent of $\eps$.

For $\delta\in[0,\frac{p-1}{p}\alpha)$, choose $\tilde{p}$ such that $p<\tilde{p}$ and $\alpha-\frac{d}p\geq\tilde{\delta}-\frac{d}{\tilde{p}}$, where $\tilde{\delta}=\delta+d(\frac1p-\frac1{\tilde{p}})$.   Therefore, for the other term $I_2$, we can apply Lemma \ref{lem-Sob-space-comp}    and the embedding $H^{\alpha,p}(\mathbb{R}^d)\subset W^{\tilde{\delta},\tilde{p}}(\mathbb{R}^d)$ to get
\begin{align}
I_2
&\leq C_TM_{r,T,\tilde{p},d,\delta,p}\int_0^T\int_{B^c} \Vert v^{\eps}(\cdot+z)-v^{\eps}(\cdot)    \Vert^p_{W^{\tilde{\delta},\tilde{p}}}\nu(dz)ds\\
   & \leq C_{r,T,\tilde{p},d,\delta,p}\mathbb{E}\Big(\int_{0}^{T}\int_{B^c} \Vert v^{\eps}( \cdot+z)-v^{\eps}( \cdot)\Vert^p_{H^{\alpha,p}}\nu(dz)ds\Big)\\
   &\leq C_{r,T,\tilde{p},d,\delta,p}\Vert v^{\eps}\Vert_{H^{\alpha,p}}^p\big(\nu(\{|z|>1\})\big)\\
   &\leq C_{r,T,\tilde{p},d,\delta,p} \Vert  \divv b  \Vert_{L^p}^p\big(\nu(\{|z|>1\})\big)\leq C<\infty.
\end{align}
Combining the above estimates for  $I_1$ and $I_2$ gives \eqref{uni-boundedness-ST}.

\end{proof}

\begin{remark}\label{Remark 202501}
    Using similar arguments, Theorems \ref{Reg-est-Jaco} and \ref{Reg-est-Jaco-2} also hold with $\phi$ replaced by $\phi^{-1}$. In the process of proofs, we need to pay special attention to the following results/claim:
\begin{enumerate}
    \item[(1)] By Theorem \ref{ww1}, $\sup_{t\in[0,T],\,x\in B(r)}|J\phi^{-1}_t(x)|<\infty$, $\mathbb{P}$-a.s..

    \item[(2)] Recall $(\phi^{\eps}_{s,t})^{-1}$ in \eqref{de 2025 0102 01}. By Lemma \ref{lem-A1}, we have,
    for any $0\leq s\leq t\leq T$
\begin{equation}
\log J (\phi^{\eps}_{s,t})^{-1}(x)=-\int_{s}^{t} \divv b^{\eps}\left((\phi^{\eps}_{r,t})^{-1}(x)\right) d r.
\end{equation}
\end{enumerate}
Alternatively, one can use that $D\phi^{-1}_t(x) = (D\phi_t( \phi^{-1}_t(x)))^{-1}$ and so
$$
J \phi^{-1}_t(x) =\det[D\phi^{-1}_t(x)]  = ( J \phi_t( \phi^{-1}_t(x))  )^{-1}. 
$$
Therefore the fractional  Sobolev regularity of  $\log (J \phi_t^{-1}(x))$ follows from the one of  
$y \mapsto \log (J \phi_t(y)).$ 
\end{remark}

%-----------------------------------------------

\subsection{Uniqueness results for our SPDEs }
Now we prove the uniqueness results for our SPDE.  To this purpose we will use the regularity results of Subsection \ref{subsec 6.1} together with an application of the Ito-Wentzell formula  in the jump case, see  Theorem \ref{thm-IW-applied}. This is a delicate issue; see our  Appendix  
\ref{sec-IWF}.

\begin{theorem}\label{thm-uniqueness-1}
{Let us assume that   $\alpha\in[1,2)$  and
 $\beta  \in (0, 1)  $ be such that \eqref{eqn-beta+alpha half} holds.
Let   $L = (L_t)_{t\geq 0}$ be a L\'evy  process with L\'evy triplet $(0,0,\nu)$ satisfying Hypothesis \ref{hyp-nondeg2}.}
 Assume that $b\in C_{\mathrm{b}}^{\beta}(\R^d,\R^d)$.
 Assume  also that one of the following conditions hold. \\
 (1)Assume
$\divv b \in L^{p}( \mathbb{R}^{d})$ for some $p\geq2$ when $d\geq 1$ or simply $Db\in L^1_{loc}(\mathbb{R})$ in the case $d=1$.
\\
(2) Assume $\beta >1-\alpha+\frac{\alpha}p$ and $\divv b \in L^{p}( \mathbb{R}^{d})$ for some 
{$p\in(\alpha,2]$} when $d\geq 1$ or simply $Db\in L^1_{loc}(\mathbb{R})$ in the case $d=1$.\\
Then, for every $u_{0} \in L^{\infty}(\mathbb{R}^{d})$, there exists a unique weak$^\ast$-$\mathrm{L}^{\infty}$-solution to the problem \eqref{eqn-transport-Markus} of the form $u(\omega,t,x)=u_0(\phi_t^{-1}(\omega)x), \;\; t\in [0,T], x \in\mathbb{R}^d,
$ in the sense of Definition \ref{def-transport-weak-2}.
\end{theorem}

\begin{remark}
    The uniqueness result from Theorem \ref{thm-uniqueness-1} is comparable to that of  from \cite[Theorem 20]{FGP_2010-Inventiones}, which requires the hypothesis that there exists $p\in (2,\infty)$ such that $\divv b \in L^{p}( \mathbb{R}^{d})$. On the other hand, our result extends \cite[Theorem 20]{FGP_2010-Inventiones} by allowing $p$ to range also in $(\alpha,2]$. 
\end{remark}

Our   next uniqueness result  holds in particular under the general Hypothesis  \ref{hyp-nondeg1} when $\alpha \ge 1$. It does not require any global integrability condition on $\divv b$. On the other hand \eqref{eqn-alphaover4+beta>1} is stronger than condition \eqref{eqn-beta+alpha half}.

It is quite unexpected that the result  even holds when $\alpha \in (0,1)$ in the case of rotationally invariant $\alpha$-stable L\'evy processes.

\begin{theorem}\label{thm-uniqueness}
 Assume Hypothesis \ref{hyp-nondeg1}
 if $\alpha\in [1,2)$ or assume  Hypothesis \ref{hyp-nondeg3}
 if $\alpha\in (0,1)$.    Let  $\beta  \in (0, 1)  $ be such that
\begin{align}\label{eqn-alphaover4+beta>1}
    &\frac{\alpha}4+\beta>1.
    \end{align}
   Assume  also that $b\in C_{\mathrm{b}}^{\beta}(\mathbb{R}^d,\mathbb{R}^d)$  and $\divv b \in L_{\mathrm{loc}}^{1}( \mathbb{R}^{d})$, $d\geq 1$.
Then, for every $u_{0} \in L^{\infty}(\mathbb{R}^{d})$, there exists a unique weak$^\ast$-$\mathrm{L}^{\infty}$-solution to the problem \eqref{eqn-transport-Markus} of the form \[u(\omega,t,x)=u_0(\phi_t^{-1}(\omega)x), \;\; t\in[0,T], x \in\mathbb{R}^d,
\]
 in the sense of Definition \ref{def-transport-weak-2}.

\end{theorem}
\begin{remark}\label{rem-thm-uniqueness}
The uniqueness from Theorem \ref{thm-uniqueness}  is comparable with the paper \cite[Theorem 21]{FGP_2010-Inventiones}, whose  uniqueness result in the case of Brownian Motion, i.e., $\alpha=2$, requires $b\in C_{\mathrm{b}}^{\beta}(\R^d,\R^d)$ with $\beta>\frac12$.\\
Let us emphasize that Theorem \ref{thm-uniqueness}  does not rely on the paper \cite{Zhang_2013} (the previous result, i.e., Theorem \ref{thm-uniqueness-1}, depends on \cite{Zhang_2013}).
\end{remark}

\begin{proof}[Proof of Theorem \ref{thm-uniqueness-1}]
Let us assume that  $u$ is a weak$^\ast$-$\mathrm{L}^{\infty}$-solution of equation \eqref{eqn-transport-Markus} in the sense of Definition \ref{def-transport-weak-2} with   $u_0=0$. 
By linearity of the equation, the uniqueness follows from showing that for any test function $\theta\in C_c^{\infty}(\R^d)$, 
\begin{align}\label{eq-uni101}
\int_{\R^d} u(t,x)\theta(x) dx=0,\quad \mathbb{P}\text{-a.s., for any }t\in[0,T].
\end{align}
\indent In what follows, we establish the claim \eqref{eq-uni101}. Define an auxiliary $\mathbb{R}$-valued process $u^\eps(t)$, $t\in [0,T]$ by
\begin{equation}\label{eqn-u^eps}
u^\eps(t,\cdot )=(\vartheta_\eps \ast u)(t, \cdot ), \;\; t \in [0,T],
\end{equation}
i.e.,
\[
u^\eps(t,y )=\int_{\mathbb{R}^d} \vartheta_\eps(y-x)  u(t,x )\; dx,
\]
where $\vartheta_\eps$ is defined by \eqref{def-vartheta-eps}.
By \eqref{eqn-u^eps}, we have
\[
u^\eps(0,\cdot )=0.
\]
Note that for every $y \in \mathbb{R}^d$, the real-valued process $u^\eps(t,y)$ is well defined.

Let us fix $y \in \mathbb{R}^d$ and introduce an auxiliary test function $\theta$  defined  by
\begin{equation}\label{eqn-theta}
\theta_y =\vartheta_\eps(y-\cdot) \;\; \mbox{ i.e., for }x\in \R^d, \,  \;\;\theta_y(x)=\vartheta_\eps(y-x) \in \mathbb{R}.
\end{equation}
Then, by the definition \eqref{eqn-u^eps} of $u^\eps$,   we have
\begin{align}\label{eqn-u^eps-1}
u^\eps(t,y)&= \int_{\mathbb{R}^d} \vartheta_\eps(y-x) u(t,x)\,dx
=     \int_{\mathbb{R}^d}  u(t,x)\theta_y(x) \,dx.
\end{align}
Therefore, by (i) of Definition \ref{def-transport-weak-2} and Lemma \ref{lem-transport-weak-2-(i)}, we know that the process $u^\eps(t,y)$, $t\geq 0$, is well defined and for every $y$, $u^\eps(t,y)$ has  c\`{a}dl\`{a}g paths $\mathbb{P}$-a.s. Observe that
by the above equality \eqref{eqn-u^eps-1} and the equality \eqref{eqn-transport-Marcus} from Definition \ref{def-transport-weak-2},   we have
\begin{align}
\begin{aligned}
       u^\eps(t,y)
      =& \int_0^t\; \int_{\mathbb{R}^d}u(s,x)\,[b(x)\cdot  { D}\theta_y(x)+\divv\, b(x)\,\theta_y(x)]\,dx\, ds\\
       &+\int_0^t\lint_{B} \Bigl(\int_{\mathbb{R}^d}u(s-,x)\,[\theta_y(x+z)-\theta_y(x)]\, dx\Bigr)\tilde{\prm}(ds,dz)\\
       &+\int_0^t\int_{B^{\mathrm{c}}} \Bigl(\int_{\mathbb{R}^d}u(s-,x)\,[\theta_y(x+z)-\theta_y(x)]\, dx\Bigr) \prm(ds,dz)\\
       &+\int_0^t\lint_{B}\Bigl(\int_{\mathbb{R}^d} u(s,x)\,[\theta_y(x+z)-\theta_y(x)-\sum_{i=1}^dz_iD_{i}\theta_y(x)]\,dx\Bigr)\,\nu(dz)\, ds.
                    \end{aligned}
\end{align}
It follows that for $y \in \mathbb{R}^d$, the real-valued process $u^\eps(t,y)$ satisfies $t\geq 0$
\begin{align}\label{eqn-transport-Marcus-2}
       u^\eps(t,y)=&\int_0^t\; \int_{\mathbb{R}^d}u(s,x)\,[b(x)\cdot D_x \bigl(\vartheta_\eps(y-x)\bigr) +\divv\, b(x) \vartheta_\eps(y-x)]\,dx\, ds
       \\
       &+\int_0^t\lint_{B} \Bigl(\int_{\mathbb{R}^d}u(s-,x)\,[ \vartheta_\eps(y-z-x) -\vartheta_\eps(y-x) ] \,dx\Bigr)\tilde{\prm}(ds,dz)\nonumber \\
       &+\int_0^t\int_{B^{\mathrm{c}}} \Bigl(\int_{\mathbb{R}^d}u(s-,x)\,[\vartheta_\eps(y-z-x) -\vartheta_\eps(y-x) ]\, dx\Bigr) \prm(ds,dz) \nonumber\\
       &+\int_0^t\lint_{B}\Bigl(\int_{\mathbb{R}^d} u(s,x) \bigl[\vartheta_\eps(y-z-x) -\vartheta_\eps(y-x)
      -\sum_{i=1}^dz_iD_{i} \bigl( \vartheta_\eps(y-x)\bigr) \bigr]\, dx\Bigr)\,\nu(dz)\, ds.\nonumber
\end{align}
 By $D_x (\vartheta_\eps(y-x))$ we mean the Fr{\'e}chet derivative at point $x$ of the function $\theta: x\mapsto \vartheta_\eps(y-x)$, i.e.,
\[
D_x (\vartheta_\eps(y-x))= D_{x} \theta_y,
\]
where the function $\theta_y$ is defined by the formula \eqref{eqn-theta}.
\\
Let us observe that the following identity holds: \[D_{i} (\vartheta_\eps(y-x)) =-D_{i} \vartheta_\eps(y-x),\]
where $D \vartheta_\eps(a)$ is the directional Fr{\'e}chet derivative of $ \vartheta_\eps$ at $a$ and $D_i \vartheta_\eps(a)=\langle D \vartheta_\eps(a),e_i\rangle$. Similarly,  when we write $D\vartheta_\eps(y-x)$ we mean the Fr{\'e}chet derivative  of $\vartheta_\eps$ at the point $y-x$.

By Theorem \ref{thm-IW-applied}, we have 
 $\mathbb P$-a.s., for any $t \in [0,T]$,
 \begin{align}
 u^\eps(t,\phi_t(x))
=& \int_0^t Du^{\eps}(s,\phi_{s}(x))\cdot b(\phi_{s}(x))\,ds\label{eqn-u^eps-final-1}
\\
&+\int_0^t \int_{\mathbb{R}^d}u(s,y)\,\bigl[
b(y)\cdot D_y\bigl(\vartheta_\eps(\phi_{s}(x)-y)\bigr)
+\divv\, b(y) \vartheta_\eps(\phi_{s}(x)-y)
\bigr]\,d y\, ds.\nonumber
\end{align}

\textbf{Step 2} Let us fix $x \in \mathbb{R}$. Since  the function $\theta: y \mapsto \vartheta_\eps(\phi_{s}(x)-y)$ is of $C_c^\infty$-class, according to Definition \ref{def-weak derivative} we have
\begin{align*}
& \int_0^t \int_{\mathbb{R}^d}u(s,y)\,\bigl[
b(y)\cdot D_y \bigl(\vartheta_\eps(\phi_{s}(x)-y) \bigr)
+\divv\, b(y) \vartheta_\eps(\phi_{s}(x)-y)
\bigr]\,d y\, ds
\\
&=\int_0^t \Bigl[  \int_{\mathbb{R}^d}u(s,y)\,
b(y)\cdot D_y \theta
\,d y\,  + \int_{\mathbb{R}^d}u(s,y)\, \divv\, b(y) \theta (y)  \,d y \bigr]\, ds
\\
&=-\int_0^t \int_{\mathbb{R}^d} b(y)  \cdot  D_y u(s,y) \, \theta(y)  \,d y\, ds
\\
&=-\int_0^t \int_{\mathbb{R}^d} b(y)  \cdot  D_y u(s,y) \, \vartheta_\eps(\phi_{s}(x)-y)  \,d y\, ds\\
&=-\int_0^t  [\vartheta_\eps \ast (b\cdot D u(s,\cdot ))](\phi_{s}(x))  \, ds,
\end{align*}
where $b(y)  \cdot  D_y u(s,y)$ is the object defined in \eqref{eqn-v nabla g}.
Hence, we infer, using notation $u_s=u(s,\cdot)$,
\begin{align}
u^{\eps}(t,\phi_t(x))=& \int_0^t b(\phi_{s}(x)) \cdot Du^{\eps}(s,\phi_{s}(x))\,ds\label{sec-6-u-commutator-1}
-\int_{0}^{t} \bigl[\vartheta_{\eps} \ast (b \cdot D u_{s})\bigr](\phi_{s}(x)) \, dx\nonumber\\
=&-\int_{0}^{t}\bigl[\vartheta_{\eps} \ast (b \cdot D u_{s})-b \cdot D(\vartheta_{\eps} \ast  u_{s})\bigr](\phi_{s}(x)) d s\nonumber\\
=&-\int_{0}^{t} \mathcal{R}_{\eps}\bigl[b, u_{s}\bigr](\phi_{s}(x)) d s.
\end{align}
Let $\rho\in C_c^{\infty}(\mathbb{R}^d)$. By changing the variable we find
\begin{equation}
\int_{\mathbb{R}^d} u^{\eps}(t, \phi_{t}(x)) \rho(x) d x = \int_{\mathbb{R}^d} u^{\eps}(t, y) \rho(\phi_{t}^{-1}(y)) J \phi_{t}^{-1}(y) d y.
\end{equation}
Since $u$ satisfies (o) of Definition \ref{def-transport-weak-2} and $\rho\in C_c^{\infty}(\mathbb{R}^d)$, for any $t\in[0,T]$,  there exists a subsequence $\eps_n$ such that $\mathbb{P}$-a.s.
\begin{align*}
\lim_{\eps_n\to  0}\int_{\mathbb{R}^d} u^{\eps_n}(t, \phi_{t}(x)) \rho(x) d x &=\lim_{\eps_n\to  0} \int_{\mathbb{R}^d} u^{\eps_n}(t, y) \rho(\phi_{t}^{-1}(y)) J \phi_{t}^{-1}(y) d y\\
&=\int_{\mathbb{R}^d} u(t, y) \rho(\phi_{t}^{-1}(y)) J \phi_{t}^{-1}(y) d y\\
&=\int_{\mathbb{R}^d} u(t, \phi_{t}(x)) \rho(x) d x.
\end{align*}
Therefore, in view of \eqref{sec-6-u-commutator-1}, we infer $\mathbb{P}$-a.s.
\begin{align}\label{sec-6-conv-u-R-1}
\int_{\mathbb{R}^d} u(t, \phi_{t}(x)) \rho(x) d x=-\lim_{\eps_n\to  0}\int_{\mathbb{R}^d} \Big(\int_{0}^{t} \mathcal{R}_{\eps_n}\bigl[b, u_{s}\bigr]\left(\phi_{s}(x)\right) d s\Big) \rho(x) dx.
\end{align}
\textbf{Step 3} Now assume $\rho\in C_r^{\infty}(\mathbb{R}^d)$ for some $r>0$. Define a random variable \begin{equation}R=\sup_{s\in[0,T],\,x\in B(r)}|\phi_s(x)|\end{equation}
so that the map $\phi_s$, $s\in[0,T]$ sends the support of $\rho$  in the ball of radius $R$.
By Theorem \ref{d32}, we have $R\leq r+T\|b\|_0+\sup_{s\in[0,T]}|L_s|<\infty$, $\mathbb{P}$-a.s. In view of Theorem \ref{ww1}, Theorem \ref{thm-stability} and  Remark \ref{rem Sec 5 000}, we get, $\mathbb{P}$-a.s.,
\begin{align}
&\sup _{s \in[0, T],\, x \in B(R)}\left|J \phi_{s}^{-1}(x)\right|<\infty,\\
&\sup _{s \in[0, T],\, x \in B(R)}\Vert D\phi_s^{-1}(x)\Vert<\infty.
\end{align}
Since $1-\frac{\alpha}2<\beta<1$,  by applying Theorem \ref{Reg-est-Jaco} with $\delta=1-\beta \in (0,\frac{\alpha}{2})$ we have  $J\phi\in L^2(0,T;W_r^{1-\beta,p})$, $\mathbb{P}$-a.s. for any $r>0$, when $p\geq 2$. In view of Remark \ref{Remark 202501}, we infer $J\phi^{-1}\in L^2(0,T;W_r^{1-\beta,p})$, $\mathbb{P}$-a.s. for any $r>0$, when $p\geq 2$.\\
\indent
 If
 $p\in(\alpha,2]$ and $1-\alpha+\frac{\alpha}p<\beta<1$,  we may apply Theorem \ref{Reg-est-Jaco-2} with $\delta=1-\beta\in (0,\frac{(p-1)\alpha}{p})$ to yield that $J\phi\in L^p(0,T;W_r^{1-\beta,p})$, $\mathbb{P}$-a.s. for any $r>0$ and hence $J\phi^{-1}\in L^p(0,T;W_r^{1-\beta,p})$, $\mathbb{P}$-a.s. for $r>0$, by Remark \ref{Remark 202501}.
Now we can invoke Proposition \ref{comu-prop-1} to deduce that for $p>1$ and $r>0$,
\begin{align}
\begin{aligned}\label{sec-6-conv-u-R-2}
&\Big|\int_{\mathbb{R}^d} \mathcal{R}_{\eps_n}\bigl[b, u_{s}\bigr]\left(\phi_{s}(x)\right)  \rho(x) dx\Big|\\
&\leq C_{R}\Vert u_s\Vert_{L_{R+1}^{\infty}}\Vert\rho\Vert_{L_{r}^{\infty}}\Vert\divv b\Vert_{L_{R+1}^{1}}\Vert J \phi^{-1}_s\Vert_{L_{R}^{\infty}}\\
&\quad+C_{R,p}\Vert u_s\Vert_{L_{R+1}^{\infty}}[b]_{C_{R+1}^{\beta}}\left(\left\Vert D \phi^{-1}_s\right\Vert_{L_{R}^{\infty}}\Vert D \rho\Vert_{L_{r}^{p}}\left\Vert J \phi^{-1}_s\right\Vert_{L_{R}^{\infty}}+\left[J \phi^{-1}_s\right]_{W_{R}^{1-\beta, p}}\Vert\rho\Vert_{L_{r}^{\infty}}\right)\\
&\in L^1(0,T),\quad\mathbb{P}\text{-a.s.}
\end{aligned}
\end{align}
and
\begin{align}
\begin{aligned}\label{sec-6-conv-u-R-3}
\lim _{\eps_n \to  0}\int_{\mathbb{R}^d} \mathcal{R}_{\eps_n}\bigl[b, u_{s}\bigr]\left(\phi_{s}(x)\right)  \rho(x) dx=0,\quad\mathbb{P}\text{-a.s.}
\end{aligned}
\end{align}
(the limit is in $L^1(0,T)$ by the dominated convergence theorem). 
 In particular, when $d=1$, we can apply \cite[Corollary 19]{FGP_2010-Inventiones} to obtain
\begin{align}
\begin{aligned}\label{sec-6-conv-u-R-4}
\Big|\int_{\mathbb{R}^d} \mathcal{R}_{\eps_n}\bigl[b, u_{s}\bigr]\left(\phi_{s}(x)\right)  \rho(x) dx\Big|
\leq C_{\rho}\Vert u_s\Vert_{L_{R+2}^{\infty}}\left\Vert J \phi^{-1}_s\right\Vert_{L_{R}^{\infty}}\Vert b\Vert_{W^{1,1}_{R+2}}\in L^1(0,T)
\quad\mathbb{P}\text{-a.s.}
\end{aligned}
\end{align}
and
\begin{align}
\begin{aligned}\label{sec-6-conv-u-R-5}
\lim _{\eps_n \to  0}\int_{\mathbb{R}^d} \mathcal{R}_{\eps_n}\bigl[b, u_{s}\bigr]\left(\phi_{s}(x)\right)  \rho(x) dx=0,\quad\mathbb{P}\text{-a.s.}
\end{aligned}
\end{align}
(the limit is in $L^1(0,T$)).

 Consequently, we 
conclude from \eqref{sec-6-conv-u-R-1}--\eqref{sec-6-conv-u-R-5} that, for any test function $\rho\in C_c^\infty(\mathbb{R}^d)$, $\mathbb{P}$\text{-a.s.}, 
\begin{align} \label{dee}
\int_{\mathbb{R}^d} u(t, \phi_{t}(x)) \rho(x) d x=0,  \quad  t \in [0,T].
\end{align}
{
 For any test function $\theta\in  C_c^\infty(\mathbb{R}^d)$,  
 \begin{align}
\int_{\mathbb{R}^d} u(t, y) \theta(y) d y=\int_{\R^d}u(t,\phi_t(x)) \tilde{\rho}(x) dx,  \, t \in [0,T],
\end{align}
where $\tilde{\rho}=(\theta\circ \phi_t )J\phi_t$ is a continuous function with compact support. Since $C_c^{\infty}(\R^d)$ is dense in $C_c(\R^d)$ in $L^1$-norm  and \eqref{dee} holds for any test function, by a standard density argument, we can conclude that
for any test function $\theta\in  C_c^\infty(\mathbb{R}^d)$,  $\mathbb{P}$-a.s.
 \begin{align}
\int_{\mathbb{R}^d} u(t, y) \theta(y) d y=0, \quad t \in [0,T],
\end{align}
which completes the proof of uniqueness.}
\end{proof} 
\begin{proof}[Proof of Theorem \ref{thm-uniqueness}]
The first two steps are the same as that in the proof of Theorem \ref{thm-uniqueness-1}.

\textbf{Step 3} 
According to Theorem \ref{ww1}, $D\phi^{-1}_s$ is $\mathbb{P}$-a.s.  locally $\theta$-H\"older continuous, uniformly in $s\in[0,T]$, for any $0<\theta<\frac\alpha2+\beta-1$. Since by assumption $1>\beta>1-\frac{\alpha}4$, we have  $1-\beta<\frac\alpha2+\beta-1$. It follows that $D\phi^{-1}_s$ and $J\phi^{-1}_s$ are $\mathbb{P}$-a.s.  locally $(1-\beta)$-H\"older continuous in $x$, uniformly in $s\in[0,T]$. Observe also that
\begin{align}
&\sup _{s \in[0, T]}\Vert D\phi_s^{-1}\Vert_{C_R^{1-\beta}}<\infty,\quad\mathbb{P}\text{-a.s.}
\end{align}
 Now we can apply Proposition \ref{commu-pro} to get
\begin{align}
\begin{aligned}\label{sec-6-conv-u-R-6}
&\Big|\int_{\R^d} \mathcal{R}_{\eps_n}\bigl[b, u_{s}\bigr]\left(\phi_{s}(x)\right)  \rho(x) dx\Big|\\
&\leq C_{R}\Vert u_s\Vert_{L_{R+1}^{\infty}}\Vert\rho\Vert_{L_{r}^{\infty}}\left\Vert J \phi^{-1}_s\right\Vert_{L_{R}^{\infty}}\Vert\divv b\Vert_{L_{R+1}^{1}}\\
&\quad+C_{R}\Vert u_s\Vert_{L_{R+1}^{\infty}}[b]_{C_{R+1}^{\beta}}\left(\left\Vert D \phi^{-1}_s\right\Vert_{C_{R+1}^{1-\beta}}^2\Vert D \rho\Vert_{L_{r}^{\infty}}+\Vert\rho\Vert_{L_{r}^{\infty}}\left[D \phi^{-1}_s\right]_{C_{R+1}^{1-\beta}}\right)\\
&\in L^1(0,T),\quad\mathbb{P}\text{-a.s.}
\end{aligned}
\end{align}
and
\begin{align}
\begin{aligned}\label{sec-6-conv-u-R-7}
\lim _{\eps_n \to  0}\int_{\R^d} \mathcal{R}_{\eps_n}\bigl[b, u_{s}\bigr]\left(\phi_{s}(x)\right)  \rho(x) dx=0,\quad\mathbb{P}\text{-a.s.}
\end{aligned}
\end{align}
(the limit is in $L^1(0,T)$ by the dominated convergence theorem).  Consequently, using \eqref{sec-6-conv-u-R-1},  \eqref{sec-6-conv-u-R-6} and \eqref{sec-6-conv-u-R-7}, an application of dominated convergence theorem shows
\begin{align}
\int_{\R^d} u(t, \phi_{t}(x)) \rho(x) d x=0.
\end{align}
 Now we conclude as in the proof of 
Theorem \ref{thm-uniqueness-1}. 
We have thus proved the theorem.

\end{proof}

\appendix

%\section{It\^o-Wentzel formula from \cite{Mik_1983}}\label{sec-Ito-Wentzel formula}

\newenvironment{tightenumerate}{
\begin{enumerate}[(1)]
  \setlength{\itemsep}{3pt}
  \setlength{\parskip}{0pt}
}{\end{enumerate}}

\newcommand{\rE}{\bR^{d}}%{{\mathrm{E}}}
\newcommand\bR{\mathbb{R}}%{\mathbf{R}}
\newcommand\bE{\mathbb{E}}%{\mathbf{E}}
\newcommand\bF{\mathbb{F}}%{\mathbf{F}}
\newcommand\bP{\mathbb{P}}%{\mathbf{P}}
\newcommand\cC{\mathcal{C}}
\newcommand\cF{\mathscr{F}}%{\mathcal{F}}
\newcommand\cZ{\mathcal{Z}}

%---------------
\section{Appendix: The It\^{o}-Wentzell formula in the jump case }
%by  Leahy-Mikulevicius \cite{Leahy+Mik_2016}
\label{sec-IWF}

In this section, we present in  Theorem
\ref{thm-Ito+Wentzel-LM} the  It\^o-Wentzell formula in the  form taken from
\cite{Leahy+Mik_2016}, see their Proposition 4.16, see also \cite{Mik_1983}. This concerns the computation of the stochastic differential of $F_t(\xi_t)$ where $F_t(y)$ is an adapted  random field with jumps and $\xi_t$ is in general
an   adapted quasi-left continuous  semimartingale. To the best of our knowledge,  only the references \cite{Leahy+Mik_2016} and \cite{Mik_1983} contain complete proofs of such  It\^o-Wentzell formula.

We will simplify the assumptions on $\xi_t$ and $F_t(y)$ (see  below) 
with respect to \cite{Leahy+Mik_2016}. 
The obtained simplified  It\^o-Wentzell formula  can still cover  our situation, cf. Appendix \ref{app}. However, we stress  that even if such a simplified  form  contains  some natural assumptions taken from \cite{Leahy+Mik_2016}, it is still   quite difficult to apply in concrete
examples related to SDEs. Indeed, it is a difficult task to check when  the several  assumptions
 hold  and to understand the consequences of the   It\^o-Wentzell formula in the jump case, see  for instance our computations in Appendix \ref{app}. Therefore, the  situation appears to be very different with respect to the one  of the well-known It\^o-Wentzell formula in  the continuous case, c.f. \cite{Kunita_1990}.

\vskip 2mm

We fix a stochastic basis or  a complete
probability space  $\left( \Omega ,\cF, \mathbb{F}\right) $. Let $\newL= (\newL_t)_{t\ge 0}$ be an adapted pure-jump L\'evy process with values in $\R^d$. 
 Moreover, $A = (A_t)_{t\ge 0}$  is an $\R^d$-valued continuous adapted process of finite variation with $A_0 = 0$.

 Assume that  $\oldL=(\oldL_{t}(x))_{t\ge 0}$, $x \in \R^d$
 verifies $\mathbb{P}$-a.s.\ for all $t\ge 0$
\begin{align} \label{f4}
  \oldL_{t}(x) =x+A_{t} + \newL_t.
\end{align}%
Then $\oldL = (\oldL_{t}(x))_{t\geq 0}$ is a special  It\^o process with jumps. {\it  In particular, our
stochastic flow $\phi_t(x)$ is a particular process of type \eqref{f4}}. In the sequel,  we will often omit the dependence of $\xi$ by $x$ and simply write $\xi_t$.

According to \cite{Applebaum_2009}, it is known that
 $\mathbb{P}$-a.s.\ for all $t> 0$,
\begin{equation}
\sum_{0<s\leq t}|\Delta \oldL_{s}|^{2}<\infty.
\end{equation}
Indeed,  using the Poisson random measure $\mu^S$  associated to $S$, we have by the L\'evy-It\^o formula:
\begin{align}
    \sum_{s\in(0,t]}|\Delta \xi_s|^2 = \sum_{s\in(0,t]}|\Delta S_s|^2
   =
    \int_0^t \lint_{B} |z|^2 \mu^S(ds, dz) +
 \int_0^t \int_{B^{\mathrm{c}}} |z|^2  \mu^S(ds, dz)<\infty.
\end{align}
Here $B=\{z \in \mathbb{R}^d: | z|  \leq 1\}$ and $B^c=\{z \in \mathbb{R}^d : | z| >1\}$ as before.
 Let us now introduce the hypotheses on the random field $F_t(y)$ in a simplified form with respect to  \cite{Leahy+Mik_2016}. In the following, the L\'evy measure of $S$ is denoted by $\nu^S$. We point out that
 $$
 S_t=\int_0^t \lint_{B} z \tilde{\mu}^S(ds, dz) +
 \int_0^t \int_{B^{\mathrm{c}}} z  \mu^S(ds, dz),
 $$
here $\tilde{\mu}^S(ds, dz)=\mu^S(ds, dz)-\nu^S(dz)ds$.

 To this purpose,  we introduce  another pure-jump L\'evy process $\newS= (\newS_t)_{t\geq 0}$ with values in $\R^d$ and adapted to the previous stochastic basis:
 \begin{equation}\label{ds}
\newS_t = \int_0^t\lint_{B}z \newq(ds,dz)+\int_0^t\int_{B^c}z
\newp(ds,dz),
\end{equation}
 where $\newp$ is the related Poisson random measure and  $\tilde \mu$ is the compensated   Poisson random measure. Moreover, the L\'evy measure of $\newS$ is denoted by $\nu$.

Let us  fix $T>0$.
 In the sequel, as in  \cite{Leahy+Mik_2016}, we will use the predictable  $\mathscr{P}=\mathscr{P}_T$
   and the optional, i.e., well-measurable,    $\mathscr{O}=\mathscr{O}_T$, $\sigma$-fields, see   \cite[Definition I.5.2, page 20]{Ikeda+Watanabe_1989}.
    The former one is generated by all real adapted  left-continuous  processes on $[0,T] \times \Omega$; while the latter one is  generated by all real adapted  right-continuous  processes on $[0,T] \times \Omega$.
    It is well-known
\begin{align*}
\mathscr{P} \subset   \mathscr{O}.
\end{align*}

We first assume
\begin{hypothesis} \label{ftx}

 \begin{trivlist}
\item[(i)] a function
 \[f:\Omega\times [0,T] \times \rE\rightarrow \bR^{m}\]  is
 $\mathscr{O}\otimes \mathcal{B}(\rE)/\mathcal{B}(\R^m)$-measurable,  and
\item[(ii)] a function
\[%
h:\Omega\times [0,T] \times \rE\times \R^d \rightarrow \bR^{m}
\]
is  $\mathscr{P}\otimes \mathcal{B}(\rE)\otimes  \mathcal{B}(\rE)/\mathcal{B}(\R^m)$-measurable.
\end{trivlist}
\item[(iii)] Moreover, assume that, $\mathbb{P}$-a.s.\ for all $y\in \rE$,
\begin{equation*}
\int_0^T|f_t(y)|dt +\int_0^T\lint_{B}|h_t(y,z)|^{2}\nu(dz)dt+\int_0^T\int_{B^c}|h_t(y,z)|\nu(dz)dt<\infty.
\end{equation*}%
\item[(iv)] Let
\begin{equation}
F:\Omega\times [0,T]\times \rE\rightarrow \bR^{m}
\end{equation}
 be $\mathscr{O}\otimes \mathcal{B}(\rE)/\mathcal{B}(\R^m)$-measurable and  assume that for each $y\in \mathbb{R}^d$, $\mathbb{P}$-a.s.,\ for all $t\in[0,T]$,
\begin{equation}\label{eqn-F}
F(t,y) =F_{0}(y)+\int_0^tf_s(y)ds +\int_0^t\int_{\mathbb{R}^d}h_s(y,z)[\1_{B\setminus\{0\}}(z)\newq(ds,dz) +\1_{B^c}(z)\newp(ds,dz)].
\end{equation}

\end{hypothesis}

For each $n\in \{1,2\}$, let ${C}^{n}(\rE;\bR^{m})$ be the space of $n$-times continuously differentiable functions $g:\rE\rightarrow \bR^{m}$.  The topology is defined via suitable semi-norms over balls of natural radii.

We now state our version of the It{\^o}-Wentzell formula. For each $\omega,t$ and $y$, we denote
\[
\Delta F({t},y, \omega)=F(t,y, \omega)-F(t-,y, \omega)
\]
(as usual in the sequel we omit  the dependence on $\omega$).
The following result is a  reformulation  of Proposition 4.16  \cite{Leahy+Mik_2016} (with $\alpha=2$ in the notation of \cite{Leahy+Mik_2016}) which in turn is based on Proposition 1 in \cite{Mik_1983}.

Note that for a Fr\'echet space ${\mathcal X}$, $D([0, T ];{\mathcal X}  )$ is the space of ${\mathcal X}$-valued c\`adl\`ag
functions on $[0, T ];$  it is endowed with
 semi-norms involving $\sup_{t \in [0,T]}$.

    \begin{theorem}
\label{thm-Ito+Wentzel-LM} Let $(\oldL_t)_{t\ge 0}$ be as before, i.e., in \eqref{f4}. In addition to  Hypothesis \ref{ftx}, let us assume   that:
\begin{enumerate}
%\begin{enumerate}[(a)]
\item[(1a)]  $\mathbb{P}$-a.s.\  $F\in D([0,T];{C}^{2}(\bR^d;\bR^m))$;
\item[(1b)]   for $\mathbb{P}\otimes \Leb$-almost-all $(\omega,t)\in \Omega\times [0,T]$,  the function
\[
\bR^d \ni y \mapsto   f_t(y) \in \bR^m
\]
is continuous  and, for any $x \in \R^d$, 
\begin{equation}
\lim_{y\rightarrow x}\left[\lint_{B}|h_t(y,z)-h_t(x,z)|
^{2}\nu (dz) +\int_{B^c}|h_t(y,z)-h_t(x,z)|\nu (dz)\right]= 0;
\end{equation}
\item[(2)] for  each compact and convex subset $K$ of $\rE$, $\mathbb{P}$-a.s.\
\begin{equation}\label{eqn-jumps supp 2025 06}
\int_0^T\sup_{y\in K} \left(|f_t(y)| +\lint_{B}|h_t(y,z)|^2\nu(dz)+ \int_{B^c}|h_t(y,z)|\nu(dz)\right)dt<\infty,
\end{equation}
and
\begin{equation}\label{eqn-jumps do not overlap}
\sum_{t\le T}[\Delta F(t,\cdot)]_{{1};K}|\Delta \oldL_t|
<\infty,
\end{equation}
where the notation $[u]_{{1};K}$ used above is defined  %\eqref{eqn-norm-n}.
at the beginning of the paper. We recall that
\begin{align*}
[\Delta F(t,\cdot)]_{{1};K} &= \sup_{x, y \in K,\, x \not = y} \frac{|\Delta F(t,x)- \Delta F(t,y)|}{|x-y|}\\
&\leq\sup_{x\in K}|\Delta DF(t,x)|=\sup_{x\in K}|DF(t,x)-DF(t-,x)|.
\end{align*}
\end{enumerate}
Then $\mathbb{P}$-a.s., for all $t\in [0,T]$,
\begin{align}\label{eqn-Ito+Wentzel-LM}
F(t,\oldL_t)&=F_{0}(\oldL_{0})+\int_0^tf_s(\oldL_{s})ds\\
&\quad +\int_0^t\int_{\mathbb{R}^d}h_s(\oldL_{s-},z)[\1_{B\setminus\{0\}}(z)\newq(ds,dz) +\1_{B^c}(z)\newp(ds,dz)]\notag\\
&\quad  +{\int_0^t D F(s-,\oldL_{s-})
d {\xi_s}}\\
&\quad +\sum_{s\le t}\Big(F(s-,\oldL_{s})-F(s-,\oldL_{s-})- D F(s-,\oldL_{s-})\cdot\Delta \oldL_s\Big)\notag\\
&\quad +\sum_{s\leq t}\Big(\Delta F(s,\oldL_{s})-\Delta F(s,\oldL_{s-})\Big).
\notag
\end{align}
\end{theorem}

\begin{remark}\label{rem-comparison}
In the special case when $F=F_0$, the It\^o-Wentzell formula \eqref{eqn-Ito+Wentzel-LM} reduces to the It\^o formula as in Theorem 1.12 in the paper \cite{Kunita_2004} rather then
the It\^o formula as in Theorem II.5.1 in the monograph \cite{Ikeda+Watanabe_1989}. However, when $\xi$ is a semimartigale of the form (5.1) in \cite[Section II]{Ikeda+Watanabe_1989}, then one can give a form of
equality \eqref{eqn-Ito+Wentzel-LM} in a form analogous to formula (5.2) therein. 

Calculations in this directions are presented in the following Remark \ref{rem-ooo}.
\end{remark}

\begin{remark} \label{rem-ooo} In our understanding of the It\^o-Wentzell formula from \cite{Leahy+Mik_2016}, the term ${\int_0^t D F(s-,\oldL_{s-})
d {\xi_s}}$ in line 3 of  \eqref{eqn-Ito+Wentzel-LM} should be as it is and not  ${\int_0^t D F(s-,\oldL_{s-})
d {A_s}}$ as it is  written in paper \cite{Leahy+Mik_2016}. This agrees with the corresponding formula  from   the original 1983 paper \cite{Mik_1983}. Let us show, by  an additional argument,   why this is the correct form.

Let us consider the simplified case \[F(t,y)=F_0(y),\]
where $F_0(y)\in C^2(\mathbb{R}^d)$. In this case $f_s=h_s=0$. Then the above It\^o-Wentzell formula \eqref{eqn-Ito+Wentzel-LM} (with {$d\xi_s$}) becomes
\begin{align*}
F_0(\oldL_t)&=F_{0}(\oldL_{0}) +{\int_0^t D F_0(\oldL_{s-})
d {\xi_s}}+\sum_{s\le t}\Big(F_0(\oldL_{s})-F_0(\oldL_{s-})- D F_0(\oldL_{s-})\cdot\Delta \oldL_s\Big)\notag\\
&=F_{0}(\oldL_{0})+{\int_0^t D F_0(\oldL_{s})
d A_s + \int_0^t D F_0(\oldL_{s-}) dS_s}\\
&\quad+\int_0^t \lint_{B} \Big(F_0(\oldL_{s-}+z)-F_0(\oldL_{s-})- D F_0(\oldL_{s-})\cdot z\Big) \mu^S(ds,dz)\\
&\quad+\int_0^t \int_{B^c} \Big(F_0(\oldL_{s-}+z)-F_0(\oldL_{s-})- D F_0(\oldL_{s-})\cdot z\Big) \mu^S(ds,dz)\\
&=F_{0}(\oldL_{0})+{\int_0^t D F_0(\oldL_{s})
d A_s + \int_0^t\lint_B D F_0(\oldL_{s-})\cdot z \tilde{\mu}^S(ds,dz)
}\\
&\quad+{ \int_0^t\int_{B^c} D F_0(\oldL_{s-})\cdot z \mu^S(ds,dz)}\\
&\quad+\int_0^t \lint_{B} \Big(F_0(\oldL_{s-}+z)-F_0(\oldL_{s-})- D F_0(\oldL_{s-})\cdot z\Big) \mu^S(ds,dz)\\
&\quad+\int_0^t \int_{B^c}\Big(F_0(\oldL_{s-}+z)-F_0(\oldL_{s-}) - D F_0(\oldL_{s-})\cdot z\Big) \mu^S(ds,dz)\\
&=F_{0}(\oldL_{0})+{\int_0^t D F_0(\oldL_{s})
d A_s + \int_0^t\lint_B D F_0(\oldL_{s-})\cdot z \tilde{\mu}^S(ds,dz)}\\
&\quad+\int_0^t \lint_{B} \Big(F_0(\oldL_{s-}+z)-F_0(\oldL_{s-})- D F_0(\oldL_{s-})\cdot z \Big)\mu^S(ds,dz)\\
&\quad+\int_0^t \int_{B^c} \Big(F_0(\oldL_{s-}+z)-F_0(\oldL_{s-})\Big)  \mu^S(ds,dz),
\end{align*}
where   the two terms $ \int_0^t\int_{B^c} D F_0(\oldL_{s-})\cdot z \mu^S(ds,dz)$ cancel each other out.

Next, let us observe that by the Taylor formula in the integral form, see e.g., \cite{Cartan_1971-DC},
\begin{align*}
 \Big| F_0(\oldL_{s-}+z)-F_0(\oldL_{s-})- D F_0(\oldL_{s-})\cdot z \Big| \leq \int_0^1|D^2 F_0(\xi_{s-}+\delta z) ||z|^2\, d \delta .
\end{align*}
Hence by applying formula \cite[II. (3.8) on page 62]{Ikeda+Watanabe_1989} we can rewrite the fourth term on the right-hand side of the above formula as
\begin{align*}
&\int_0^t \lint_{B} \Big(F_0(\oldL_{s-}+z)-F_0(\oldL_{s-})- D F_0(\oldL_{s-})\cdot z \Big)\mu^S(ds,dz)\\
&= \int_0^t \lint_{B} \Big(F_0(\oldL_{s-}+z)-F_0(\oldL_{s-})- D F_0(\oldL_{s-})\cdot z \Big)\tilde{\mu}^S(ds,dz)\\
&\quad +\int_0^t \lint_{B} \Big(F_0(\oldL_{s-}+z)-F_0(\oldL_{s-})- D F_0(\oldL_{s-})\cdot z \Big)\nu^S(dz)ds.
\end{align*}
Thus, by  simplifying through cancellations,
we infer 
\begin{align*}
F_0(\oldL_t)=&F_{0}(\oldL_{0})+{\int_0^t D F_0(\oldL_{s})
d A_s }\\
&+\int_0^t \lint_{B} \Big(F_0(\oldL_{s-}+z)-F_0(\oldL_{s-}) \Big) \tilde{\mu}^S(ds,dz)\\
&+\int_0^t \lint_{B} \Big(F_0(\oldL_{s-}+z)-F_0(\oldL_{s-})- D F_0(\oldL_{s-})\cdot z\Big) \nu^S(dz)ds\\
&+\int_0^t \int_{B^c} \Big(F_0(\oldL_{s-}+z)-F_0(\oldL_{s-})\Big) \mu^S(ds,dz).
\end{align*}
This is consistent with the It\^o formula given in \cite{Ikeda+Watanabe_1989}, see Theorem II.5.1 on page 66. 
The above  argument can be reversed.
\end{remark}

\begin{proof}[Proof of Theorem \ref{thm-Ito+Wentzel-LM}] We give a sketch of  proof for the sake of completeness, assuming $m=1$. We also provide some additional details and  clarify the point indicated in Remark \ref{rem-ooo}.

Recall that, for fixed $y \in \R^d$, in particular, the process $(F(t,y),\xi_t)_{t \in [0,T]}$ is 
a semimartingale, see , for instance,  \cite{Protter_2004} and \cite{Kunita_2004}.

{Since both sides of formula \eqref{eqn-Ito+Wentzel-LM} have identical jumps (more precisely, we are considering the jumps of both $t \mapsto \xi_t\in \R^d$ and $t \mapsto F(t, \cdot) \in C^{2}(\R^d, \R)$) and we can always interlace a finite set of jumps. }
 In particular  we may assume that  $|\Delta \oldL_{t}|\le 1$ for all $t\in [0,\infty)$.
 It also suffices to assume that for some $K>0$,
 \[|\oldL_0|\le K \mbox{ for all $\omega$ }.\]
 For each  $R>K$, let us define \begin{equation}
\tau_R=\inf \Big( t\in [0,\infty):|A|_t +\sum_{s\le t}|\Delta \oldL_s|^{2}+|\oldL_{t}|>
R\Big) \wedge T.\end{equation}  
 Here $|A|_t$ indicates the total variation of the random function: $s \to A_s$ on $[0,t]$.
Then we note that $
\tau_R$ is a stopping time and   $\mathbb{P}$-a.s.
\[
\tau_R\toup  T \mbox{ as } R \to \infty.\]
Moreover, since the jumps are limited to size 1, $|\Delta \oldL_{t}|\le 1$ for all $t\in [0,\infty)$, we have for every $t\in[0,\tau_n]$,
\begin{equation}\label{eq-1031}
|A|_t+\sum_{s\le t}|\Delta \oldL_s|^{2}+|\oldL_{t}|\le R+1.
 \end{equation}
%If instead of $\xi, f,  h,$ and $F$,
Next we observe that the assumptions  of the theorem hold for the stopped processes
 $\oldL_{\cdot\wedge \tau_R}$, $f\1_{(0,\tau_R]}$,   $h\1_{(0,\tau_R]}$, $F\1_{(0,\tau_R]}$.
In particular,  the process $h\1_{(0,\tau_R]}$ is predictable, c.f., Proposition I.4.4 in \cite{Metivier_1982}.

 Moreover, if we can prove \eqref{eqn-Ito+Wentzel-LM} for this new set of processes, then by taking the limit as $R$ tends to infinity, we obtain \eqref{eqn-Ito+Wentzel-LM}. Therefore, without loss of generality, we may assume that for some $R>0$, $\mathbb{P}$-a.s.\ for all $t\in [0,\infty)$,
\begin{equation}\label{ineq:boundsfubini}
|A|_t+\sum_{s\le t}|\Delta \oldL_s|^{2}+|\oldL_{t}|\le R.
 \end{equation}
 In fact, this $R$ can be taken to be the previous bound $R+1$ in \eqref{eq-1031} (enlarging the constant if necessary). 

Let $\phi\in C_{c}^{\infty }( \rE;\bR)$ have support in the unit ball in $\rE$ and  satisfy  $\int_{%
\rE}\phi(x)dx=1,\phi(x)=\phi(-x),$ and $\phi(x)\geq 0$, for all $x\in\rE$. For
each $\varepsilon \in(0,1)$, let $\phi_{\varepsilon }(x)=\varepsilon ^{-d}\phi\left(
x/\varepsilon \right) ,x\in \rE$.

Let us first fix $y \in \R^d$. By applying It\^{o}'s formula to the semimartingale  $(F(t,y),\xi_t)$ using the smooth function $(u,v) \mapsto G(u,v)= u \phi_{\eps} (y -v)$ we deduce, see  for instance,  Theorem 1.12 \cite{Kunita_2004} or  Section II.7 in \cite{Protter_2004}, that 
  $\mathbb{P}$-a.s.\ for all $t\in [0,T]$,
\begin{align*}
G(F(t,y),  \oldL_{t}) &= F(t,y) \phi_{\varepsilon }(y-\oldL_{t})
\\
&= F_{0}(y)\phi_{\varepsilon }(y-\oldL_0) + \int_0^t \phi_{\varepsilon }\left( y-\oldL_{s-}\right) dF(s,y) - \int_0^t F(s-,y)D\phi_{\varepsilon }\left( y-\oldL_{s-}\right) d\xi_s
\\
&\ \ \ + \sum_{s \le t} \Big( F(s,y) \phi_{\eps}(y- \xi_s) -  F(s-,y) \phi_{\eps}(y- \xi_{s-})
  \\&\ \ \ \ \ \ \ \quad\quad\quad\quad\quad\quad  - \phi_{\eps}(y- \xi_{s-}) \Delta F(s,y)  + F(s-,y) D\phi_{\varepsilon}\left( y-\oldL_{s-}\right) \Delta \xi_s \Big).
\end{align*}
   Now, writing
 \begin{align*}
& F(s,y) \phi_{\eps}(y- \xi_s) -  F(s-,y) \phi_{\eps}(y- \xi_{s-})
 \\
 =& F(s,y) \phi_{\eps}(y- \xi_s) -  F(s-,y) \phi_{\eps}(y- \xi_s) + F(s-,y) \phi_{\eps}(y- \xi_s)- F(s-,y) \phi_{\eps}(y- \xi_{s-})
 \\=& \Delta F(s,y) \phi_{\eps}(y- \xi_s)   + F(s-,y) \phi_{\eps}(y- \xi_s)- F(s-,y) \phi_{\eps}(y- \xi_{s-}),
\end{align*}
 we infer that
  \begin{align}\label{eq:productruleIto-Wentzell}
F(t,y) \phi_{\varepsilon }(y-\oldL_{t}) &= F_{0}(y)\phi_{\varepsilon }(y-\oldL_0){-\sum_{i=1}^d\int_0^tF(s-,y) \partial _{i}\phi_{\varepsilon
}(y-\oldL_{s-})d\oldL_{s}^{i}}\\
&\quad  +\int_0^t\phi_{\varepsilon }\left( y-\oldL_{s}\right)
f_s(y)ds\notag\\
&\quad +\int_0^t\int_{\mathbb{R}^d}\phi_{\varepsilon }\left( y-\oldL_{s-}\right)
h_s(y,z)[\1_{B\setminus\{0\}}(z)\newq(ds,dz)+\1_{B^c}(z)\newp(ds,dz)]\notag\\
&\quad +\sum_{s\leq t}\Delta F(s,y) \left(\phi_{\varepsilon
}(y-\oldL_{s})-\phi_{\varepsilon }\left( y-\oldL_{s-}\right) \right) \notag\\
&\quad +\sum_{s\leq t}F(s-,y)\big( \phi_{\varepsilon }(y-\oldL_{s})-\phi_{\varepsilon }(y-\oldL_{s-})+\sum_{i=1}^d\partial _{i}\phi_{\varepsilon }\left(y-
\oldL_{s-}\right) \Delta \oldL_{s}^i\big).
\end{align}
Recall that  (cf. \eqref{f4})
\begin{align*}
\int_0^tF(s-,y) \partial _{i}\phi_{\varepsilon
}(y-\oldL_{s-})d\oldL_{s}^{i} = \int_0^tF(s,y) \partial _{i}\phi_{\varepsilon
}(y-\oldL_{s})dA_{s}^{i}
 + \int_0^tF(s-,y) \partial _{i}\phi_{\varepsilon
}(y-\oldL_{s-})dS_{s}^{i}.
\end{align*}
Let us define, for each $\omega,t,y,$ and $z$,
\begin{align*}
F^{(\varepsilon )}(t,y)&:=[\phi_{\varepsilon }\ast F(t,\cdot)](y),
\\
f^{(\varepsilon)}_t(y)&=[\phi_{\varepsilon }\ast f_t](y),
\\
h^{(\varepsilon) }_t(y,z)&=[\phi_{\varepsilon }\ast h_t( \cdot ,z)](y),
\end{align*}%
where $\ast $ denotes the convolution operator on $\rE$.
Appealing to assumption (2) and \eqref{ineq:boundsfubini} (i.e., for the integrals against $F$),  we  integrate both sides of the above \eqref{eq:productruleIto-Wentzell} in $y$, apply the stochastic Fubini theorem,
see  Corollary 4.13  from \cite{Leahy+Mik_2016}, see also, Remark 4.14  therein,  and the deterministic Fubini Theorem, and then integrate by parts  to get that   $\mathbb{P}$-a.s.\ for all $t\in [0,T]$,
\begin{align}\label{eq:Ito-Wentzellapprox}
F^{(\varepsilon) }(t,\oldL_{t})
&=F^{(\varepsilon)}_0(\oldL_0)
{+\int_0^tD F^{(\varepsilon)}(s-,\oldL_{s-})d\xi_s}
+\int_0^tf^{(\varepsilon) }_s(\oldL_{s})ds \notag\\
&\quad +\int_0^t\int_{\mathbb{R}^d}h^{(\varepsilon) }_s(\oldL_{s-},z)[\1_{B\setminus\{0\}}(z)\newq(ds,dz)+\1_{B^c}(z)\newp(ds,dz)]\notag\\
&\quad +\sum_{s\leq t}\left(\Delta F^{(\varepsilon) }(s,\oldL_{s})-\Delta F^{(\varepsilon)}(s,\oldL_{s-})\right)\notag\\
&\quad +\sum_{s\leq t}\left(F^{(\varepsilon) }(s-,\oldL_{s})-F^{(\varepsilon)
}(s-,\oldL_{s-})-D F^{(\varepsilon) }(s-,\oldL_{s-})\Delta \oldL_s\right).
\end{align}
  Let
\begin{equation}
\label{eqn-B_R}
B_{r}=\{x\in\rE: |x|\le r\}, \;\; r>0.
\end{equation}
   Owing to Assumption (1a) and standard properties of mollifiers,  we infer that for each multi-index $\gamma$ with $|\gamma|\le 2$, $\mathbb{P}$-a.s.,
\begin{equation}
\sup_{t\le T}|\partial^{\gamma}F^{(\varepsilon )}(t,\oldL_t)|\le \sup_{t\le T}\sup_{x\in B_{R+1}}| \partial^{\gamma}F(t,x)|<\infty,
\end{equation}
and for each  $K\in\mathbb{N}$,
\begin{equation}\label{eq:convergenceFepsilon}
\lim_{\varepsilon\todown 0}\sup_{t\in[0,T]}\sup_{|y|\leq K}| \partial^{\gamma}F^{(\varepsilon )}(t,y)-\partial^{\gamma}F(t,y)|=0.
\end{equation}
Similarly, by Assumptions (1b) and (2), $\mathbb{P}\otimes \Leb$-almost-all $(\omega,t)\in \Omega\times [0,T]$,
\begin{align*}
|f^{(\varepsilon)}_t(\oldL_t)|\le \sup_{x\in B_{R+1}}|f_t(x)|<\infty, \quad   \\
\lint_{B}|h^{\varepsilon}_t(\oldL_t,z)|^2\nu(dz)\le  \sup_{x\in B_{R+1}}\lint_{B}|h_t(x,z)|^2\nu(dz),\\
\int_{B^c}|h^{\varepsilon}_t(\oldL_t,z)|\nu(dz)\le \sup_{x\in B_{R+1}}\int_{B^c}|h_t(x,z)|\nu(dz),
\end{align*}
and for each $K\in\mathbb{N}$,
\begin{equation}
\mathbb{P}\otimes \Leb-\lim_{\varepsilon\todown 0} \sup_{|y|\leq K}|f^{(\varepsilon)}_t(y) -f_t(y)|=0,
\end{equation}
and, $\mathbb{P}\otimes \Leb$-almost everywhere,
\begin{equation}\label{eq:hconvergstochasticfubini}
\lim_{\varepsilon \todown  0}\sup_{|y|\leq K}\int_{\mathbb{R}^d}[\1_{B\setminus\{0\}}(z)|h^{(\varepsilon )}_t(y,z)-h_t(y,z)|^2+\1_{B^c}(z)|h^{(\varepsilon )}_t(y,z)-h_t(y,z)|]\nu(dz)=0,
\end{equation}
where in the last-line we have also used 
a standard mollifying convergence argument.

  Owing to Assumptions  (1a) and \eqref{ineq:boundsfubini}, $\mathbb{P}$-a.s.\, for any $s\in[0,T]$,
  $$
  |F^{(\varepsilon) }(s-,\oldL_{s})-F^{(\varepsilon)
}(s-\oldL_{s-})-D F^{(\varepsilon) }(s-,\oldL_{s-})\cdot\Delta \oldL_s| \le  \sup_{t\le T}[F_{t}]_{2;B_{R+1}}
|\Delta \oldL_s|^{2 }
  $$
   and $\sum_{s\le T}|\Delta \oldL_s|^{2 }\le R$.

Since $\mathbb{P}$-a.s.\ $F\in D([0,T];C^{2}(\bR^d;\bR))$, it follows that { $\mathbb{P}$-a.s.}, for each $K\in\mathbb{N}$,
\begin{equation}
\lim_{\varepsilon\todown 0}\sup_{t\in[0,T]}\sup_{|y|\leq K}|\Delta F^{\varepsilon}(t,y)-\Delta F(t,y)|=0.
\end{equation}
By Assumption (2), $\mathbb{P}$-a.s for all $s\in[0,T]$, we have
$$
\left|\Delta F^{(\varepsilon) }(s,\oldL_{s-}+\Delta \oldL_{s})-\Delta F^{(\varepsilon)
}(s,\oldL_{s-})\right|\le  [\Delta F({s},\cdot)]_{{1};B_{R+1}}|\Delta \oldL_s|
$$
and, 
\begin{equation}
\sum_{s\leq T}[\Delta F({s},\cdot)]_{{1};B_{R+1}}|\Delta \oldL_s|< \infty.
\end{equation}
Combining the above and using Assumptions (1a) and (2) and the bounds given in \eqref{ineq:boundsfubini}  and the 
dominated convergence theorem, we obtain convergence of all the terms in  \eqref{eq:Ito-Wentzellapprox}, which completes the proof.
\end{proof}

%---------------
\section{A crucial  application of the  It\^o-Wentzell formula}\label{app}

Recall (cf. \eqref{eqn-u^eps}) that
\begin{align} \label{v22}
u^\eps(t,y)= \int_{\mathbb{R}^d} \vartheta_\eps(y-x) u(t,x)\,dx
=     \int_{\mathbb{ R}^d}  u(t,x)\theta_y(x) \,dx,
\end{align}
with  $u^{}(0, \cdot)=0$,   where the function $\theta_y$ is defined by
$$
\theta_y =\vartheta_\eps(y-\cdot) \;\; \mbox{ i.e., for }x\in \R^d, \,  \;\;\theta_y(x)=\vartheta_\eps(y-x) \in \mathbb{R}.
$$
For any $y \in \mathbb{R}^d$, the real-valued process $u^\eps(t,y)$ satisfies: $\mathbb{P}$-a.s., for any  $t\geq 0$,
\begin{align}\label{eqn-transport-Marcus-2-1}
       u^\eps(t,y)=&\int_0^t\; \int_{\mathbb{R}^d}u(s,x)\,[b(x)\cdot D_x \bigl(\vartheta_\eps(y-x)\bigr) +\divv\, b(x) \vartheta_\eps(y-x)]\,dx\, ds
       \\
       &+\int_0^t\lint_{B} \Bigl(\int_{\mathbb{R}^d}u(s-,x)\,[ \vartheta_\eps(y-x-z) -\vartheta_\eps(y-x) ] \,dx\Bigr)\tilde{\prm}(ds,dz)\nonumber \\
       &+\int_0^t\int_{B^{\mathrm{c}}} \Bigl(\int_{\mathbb{R}^d}u(s-,x)\,[\vartheta_\eps(y-x-z) -\vartheta_\eps(y-x) ]\, dx\Bigr) \prm(ds,dz) \nonumber\\
       &+\int_0^t\lint_{B}\Bigl(\int_{\mathbb{R}^d} u(s,x) \bigl[\vartheta_\eps(y-x-z) -\vartheta_\eps(y-x)
      -\sum_{i=1}^dz_iD_{i} \bigl( \vartheta_\eps(y-x)\bigr) \bigr]\, dx\Bigr)\,\nu(dz)\, ds.\nonumber
\end{align}
Formula \eqref{eqn-transport-Marcus-2-1} can be rewritten in the following compact form,  for each $y \in \mathbb{R}^d$, $t \ge 0,$
\begin{equation} \label{eqn-u^eps-2}
   u^\eps(t,y)=\int_0^tG^\eps_s(y) \,ds +\int_0^t\int_{B^c}\newh^\eps_s(z,y)\, \prm(ds,dz)
+\int_0^t\lint_{B}\newH^\eps_s(z,y)\tilde{\prm}(ds,dz),
\end{equation}
where    the functions
\begin{align*}
G^\eps &:\Omega \times [0,\infty)\times    \mathbb{R}^d \to {\mathbb{R}},
\\
 \newH^\eps&: \Omega \times [0,\infty)\times  \mathbb{R}^d \times    \mathbb{R}^d\to {\mathbb{R}},
\end{align*}
are defined by the following formulae, for $s\in [0,\infty)$,  $y \in \mathbb{R}^d$ and $z \in \mathbb{R}^d$,
\begin{align}\label{eqn-G^eps}
   G^\eps_s(y)&= \int_{\mathbb{R}^d}u(s,x)\,[b(x)\cdot D_x \bigl(\vartheta_\eps(y-x)\bigr) +\divv\, b(x) \vartheta_\eps(y-x)]\,dx
   \\
  &\quad+ \lint_{B}\Bigl(\int_{\mathbb{R}^d} u(s,x)\,[\vartheta_\eps(y-x-z) -\vartheta_\eps(y-x) -\sum_{i=1}^dz_iD_{i} \bigl( \vartheta_\eps(y-x)\bigr) ]\, dx\Bigr)\,\nu(dz),\nonumber
   \\
\newh^\eps_s(y,z) &= \int_{\mathbb{R}^d}u(s-,x)\,[\vartheta_\eps(y-x-z) -\vartheta_\eps(y-x) ]\, dx.
\label{eqn-p^eps}
\end{align}
Note that,  for $s\in [0,\infty)$,  $y \in \mathbb{R}^d$ and $z \in \mathbb{R}^d$, we have
\begin{align}\label{eqn-p^eps-2}
\newh^\eps_s(y,z) &=
u^{\eps}(s-,y-z)-u^{\eps}(s-,y),
\end{align}
and
\begin{align}\label{eqn-G^eps-21}
 G^\eps_s(y)=& \int_{\mathbb{R}^d}u(s,x)\,[b(x)\cdot D_x \bigl(\vartheta_\eps(y-x)\bigr) +\divv\, b(x) \vartheta_\eps(y-x)]\,dx\nonumber\\
 &+\lint_B\Big[ u^{\eps}(s,y-z)-u^{\eps}(s,y)+\int_{\mathbb{R}^d}u(s,x)z \cdot D\vartheta_\eps(y-x) \, dx \Big]\nu(dz)
 \\
 =&\int_{\mathbb{R}^d}u(s,x)\,[b(x)\cdot D_x \bigl(\vartheta_\eps(y-x)\bigr) +\divv b(x) \vartheta_\eps(y-x)]\,dx\nonumber\\
 &+ \lint_B\Big[ u^{\eps}(s,y-z)-u^{\eps}(s,y)+ Du^{\eps}(s,y) \cdot z  \Big]\nu(dz).
\end{align}
We have a series of important lemmas.

\begin{lemma}
\label{lem-u^eps satisfies assumptions ITWF}    Let $T>0.$
The process $F=u^\eps$   satisfies all assumptions of Theorem \ref{thm-Ito+Wentzel-LM}
 with
 $$
\xi_t(x) = \phi(t,x)= \phi_t(x),\;\;\; x \in \R^d, \;\; t \in [0,T],
 $$
 where $\phi$ is the stochastic flow we constructed in Section \ref{sec-regular flow}, i.e.,
 \begin{equation}\label{Ito-nxi-psi}
\nxi_{t}(x)=x+ \int_{0}^{t}b(\nxi_{r}( x) )\,dr+ L_{t}.
\end{equation}
\end{lemma}

\begin{proof}[Proof of Lemma \ref{lem-u^eps satisfies assumptions ITWF}] 
 Note that we take $A_t = A_t(x)=\int_{0}^{t}b(\nxi_{r}( x) )\,dr$, $S_t=L_t$ in \eqref{f4} and  $F(t,y)=u^\eps(t,y)$ with
\begin{align*}
f_t(y) &=  G^\eps_t(y),
\\
h_s(y,z) &= \newh^\eps_s(y,z)
= u^{\eps}(s-,y-z)-u^{\eps}(s-,y).
\end{align*}

First, we have to check that (i), (ii), (iii) and (iv) in Hypothesis \ref{ftx} hold.

We will use formula  \eqref{v22}:
$$
u^\eps(t,y)= \int_{\mathbb{R}^d} \vartheta_\eps(y-x) u(t,x)\,dx
=     \int_{\mathbb{ R}^d}  u(t,x)\theta_y(x) \,dx.
$$
By Lemma \ref{lem-transport-weak-2-(i)} and (o) in Definition \ref{def-transport-weak-2}, we know that there exists a $\mathbb{P}$-full set $\Omega^\prime \subset \Omega$  such that  
 for any fixed $y \in \R^d$, 
 \begin{align}\label{eq-1029}
 u^{\eps}(t, \omega,y): [0,T] \times \Omega \to \R\text{ is  c\`{a}dl\`{a}g on }\Omega'\text{ and  adapted,}
 \end{align}
 and
\begin{align}\label{eq Center 2025-10-30}
    \text{for every }\omega\in\Omega', \sup_{t\in[0,T]}\|u(\omega,t,\cdot)\|_{L^{\infty}(\mathbb{R}^d)}\leq M<\infty.
\end{align}
Moreover, by the  dominated convergence theorem, we find that, for any $(t, \omega) \in [0,T] \times \Omega^\prime:$
 \begin{align} \label{s44}
 y \mapsto u^{\eps}(t, \omega,y) \;\; \text{is continuous on $\R^d$.}
 \end{align}
 Hence, we know in particular  that for any $y \in \R^d$,  $u^{\eps}(\cdot , \cdot,y)$ is
 $\mathscr{O}$-measurable and the function $(t,\omega)\mapsto u^{\eps}(t- ,\omega,y)$ is $\mathscr{P}$-measurable.  Furthermore,
 for any $y,z \in \R^d$,
 \begin{equation}\label{c44}
(t, \omega) \mapsto \newh^\eps_t(y,z, \omega) \;\; \text{is predictable on $[0,T]\times \Omega$.}
\end{equation}
{Using \eqref{eqn-p^eps}  and Remark \ref{d88}, we deduce that, for a fixed   $(t, \omega) \in (0,T] \times \Omega^\prime$ and  any
  $z \in \R^d$, we have 
  \begin{equation}\label{c441}
y \mapsto \newh^\eps_t(y,z) \;\; \text{is continuous on $\R^d$}.
\end{equation}
}
{
 Applying Lemma 4.51 in \cite{Aliprantis} we deduce that the function $(t, \omega,y) \mapsto u^{\eps}(t, \omega,y) \in \mathbb{R}$ is $\mathscr{O}\otimes \mathcal{B}(\rE)$-measurable and the function $(t, \omega,y,z) \mapsto \newh^\eps_t(y,z, \omega)$ is $\mathscr{P}\otimes \mathcal{B}(\rE)\otimes \mathcal{B}(\rE)$-measurable. Similar arguments yield that the function $(t, \omega,y) \mapsto G^\eps_t(y, \omega)$ is $\mathscr{O}\otimes \mathcal{B}(\rE)$-measurable. 
 It remains to check (iii), which follows by observing that $\mathbb{P}$-a.s. for every $y,z\in\mathbb{R}^d$, 
 \begin{align*}
     | \newh^\eps_t(y,z) |&=|u^{\eps}(t-,y-z)-u^{\eps}(t-,y)|\\
     &\leq  \lim_{s \uparrow t}\left( \left| \int_{\mathbb{R}^d} \theta_{y-z}(x)u(s,x)dx\right|+\left| \int_{\mathbb{R}^d} \theta_{y}(x)u(s,x)dx\right|\right)
     \leq 2M \Leb(B(2\eps)) \|\vartheta_{\varepsilon}\|_0, \\
      | \newh^\eps_t(y,z) |&=|u^{\eps}(t-,y-z)-u^{\eps}(t-,y)|\\
      & \leq M\int_{\mathbb{R}^d}\int_0^1|D\vartheta_\eps(y-\rho z-x)|d\rho\; dx\,|z| \leq M\Leb(B(2\eps))\|\vartheta_\eps\|_{1}|z|,
 \end{align*}
 where $D \vartheta_\eps(a)$ is the directional Fr{\'e}chet derivative of $ \vartheta_\eps$ at $a$ as before, and we used \eqref{eq Center 2025-10-30} and $\mathds{1}_{B(\frac{1}{4}\eps )} \leq \eps^{d}\vartheta_\eps \leq \mathds{1}_{B({2}{\eps})}$.
 }

Next, we shall prove that the other assumptions in Theorem \ref{thm-Ito+Wentzel-LM} hold. First we will show  that Assumption (1a) of Theorem \ref{thm-Ito+Wentzel-LM} is satisfied. We have to check that
\begin{equation}\label{x3}
 \text{ $\mathbb{P}$-a.s.\  $u^{\eps} \in D([0,T];{C}^{2}(\bR^d; \bR))$.}
\end{equation}
Keep in mind that  the topology on ${C}^{2}(\bR^d; \bR)$ is defined via suitable semi-norms over balls of natural radii.
    Recall \eqref{eq Center 2025-10-30}.
  In the sequel, we set  $\Omega' = \Omega$ to simplify the notation. 
  Note that for any $y_1,y_2\in\mathbb{R}^d$, we have
\begin{align}\label{eq coninuous 01 2025 V1}
&\sup_{(t,\omega)\in[0,T]\times \Omega}|u^\eps(t,y_1,\omega)-u^\eps(t,y_2,\omega)|\nonumber\\
    &=
\sup_{(t,\omega)\in[0,T]\times\Omega}\Big|\int_{\mathbb{R}^d} \vartheta_\eps(y_1-x)  u(t,x,\omega )\; dx-\int_{\mathbb{R}^d} \vartheta_\eps(y_2-x)  u(t,x,\omega )\; dx\Big|\nonumber\\
&\leq
M\int_{\mathbb{R}^d}\int_0^1|D\vartheta_\eps(y_2+\rho(y_1-y_2)-x)|d\rho\; dx\,|y_1-y_2|\\
&\leq
M\,\Leb(B(2\eps))\|\vartheta_\eps\|_{1}|y_1-y_2|
    \leq
 {  M_{\eps}}  \|\vartheta_\eps\|_{1}
 |y_1-y_2|.
\end{align}

If $K\in\mathbb{N}$,  we  consider the process $ {u}^\eps|_\mathbb{K}$, the restriction of the process $ {u}^\eps$ to the set $[0,T]\times \mathbb{K}\times \Omega$, with   {$\mathbb{K}=[-K,K]^d$,} 
defined  by
   \[
   {u}^\eps|_\mathbb{K}:[0,T]\times \mathbb{K}\times \Omega\ni(t,x,\omega)\mapsto {u}^\eps(t,x,\omega)\in\mathbb{R}.
    \]
We claim that\begin{enumerate}
    \item[] Claim 1. $\quad\quad\quad$  for any $\omega\in\Omega,$ ${u}^\eps(\cdot,\cdot,\omega)\in D([0,T];C_{}(\mathbb{R}^d,\mathbb{R}))$.
\end{enumerate}
To prove Claim 1, it is sufficient to prove that for any $\omega\in\Omega$, $ {u}^\eps|_\mathbb{K}(\cdot,\cdot,\omega)\in D([0,T];C(\mathbb{K},\mathbb{R}))$, for an arbitrary $K\in \mathbb{N}$. 
  Let us choose and fix $K\in\mathbb{N}$ and  $n\in\mathbb{N}$. Let $\Pi_n$ be the set of all lattice points in $\mathbb{K}$ of the form
\[
(\frac{i_1K}{2^n}, \dots, \frac{i_dK}{2^n}), \mbox{ where }-2^n\leq i_1, \dots, i_d\leq 2^n \mbox{ are integers}.
\]
 Then by \eqref{eq coninuous 01 2025 V1} and \eqref{eq-1029}, for any $\omega\in\Omega$, $s\geq0$ and $n\in\mathbb{N}$, we have
 \begin{align*}
&\lim_{t\searrow s}\Big| {u}^\eps|_\mathbb{K}(t,\cdot,\omega)- {u}^\eps|_\mathbb{K}(s,\cdot,\omega)\Big|_{C(\mathbb{K},\mathbb{R})}\\
&\leq
\lim_{t\searrow s}\sup_{y\in\Pi_n}| {u}^\eps(t,y,\omega)- {u}^\eps(s,y,\omega)|+2dM\|\vartheta_\eps\|_{1}\Leb(B(2\eps))\frac{K}{2^n}\\
&\leq
2dM\|\vartheta_\eps\|_{1}\Leb(B(2\eps))\frac{K}{2^n}.
 \end{align*}
{As   $n$ is arbitrary,  this implies that
\begin{eqnarray}
\lim_{t\searrow s}\Big| {u}^\eps|_\mathbb{K}(t,\cdot,\omega)- {u}^\eps|_\mathbb{K}(s,\cdot,\omega)\Big|_{C(\mathbb{K},\mathbb{R})}=0.
 \end{eqnarray}}
Hence, for every $\omega\in\Omega$ and $s_0\geq0$,
the function $ [0,\infty) \ni s \mapsto {u}^\eps|_\mathbb{K}(s,\cdot,\omega)\in C(\mathbb{K},\mathbb{R})$ is right continuous at  $s_0\geq0$.

Now for any $\omega\in\Omega$, $s_0>0$,  using similar arguments as above, for any $0<s,t<s_0$ and $n\in\mathbb{N}$,
 \begin{align*}
&\lim_{s,t\nearrow s_0}\Big| {u}^\eps|_\mathbb{K}(t,\cdot,\omega)- {u}^\eps|_\mathbb{K}(s,\cdot,\omega)\Big|_{C(\mathbb{K},\mathbb{R})}\\
&\leq
\lim_{s,t\nearrow s_0}\sup_{y\in\Pi_n}| {u}^\eps(t,y,\omega)- {u}^\eps(s,y,\omega)|+2dM\|\vartheta_\eps\|_{1}\Leb(B(2\eps))\frac{K}{2^n}\\
&\leq
2dM\|\vartheta_\eps\|_{1}\Leb(B(2\eps))\frac{K}{2^n}.
 \end{align*}
As $n$ is  arbitrary, this   implies that
  \begin{eqnarray}
\lim_{s,t\nearrow s_0}\Big| {u}^\eps|_\mathbb{K}(t,\cdot,\omega)- {u}^\eps|_\mathbb{K}(s,\cdot,\omega)\Big|_{C(\mathbb{K},\mathbb{R})}=0.
 \end{eqnarray}
 Hence, for any $\omega\in\Omega$ and $s_0>0$, $ {u}^\eps|_\mathbb{K}(s_0-,\cdot,\omega)$ is well-defined in $C(\mathbb{K},\mathbb{R})$. This proves Claim 1. 
By using the fact that
\begin{align}\label{eq the fact 20251031}
D u^{\eps}(s,y)&= D \int_{\R^{d}}u(s,x)\vartheta_{\eps}(y-x)\,dx =\int_{\R^{d}}u(s,x)D\bigl( \vartheta_{\eps}(y-x) \bigr)\,dx
\nonumber\\
&=-\int_{\R^{d}}u(s,x)D_{x}\bigl(\vartheta_{\eps}(y-x)\bigr)\,dx,\;\; s\in [0,\infty),\;y \in \mathbb{R}^d,
\end{align}
and applying similar arguments as above, we can prove that for any $\omega\in\Omega$, $ {u}^\eps(\cdot,\cdot,\omega)\in D([0,T];C^1_{}(\mathbb{R}^d,\mathbb{R}))$, and then for any $\omega\in\Omega$, $ {u}^\eps(\cdot,\cdot,\omega)\in D([0,T];C^2_{}(\mathbb{R}^d,\mathbb{R}))$.
Then the verification of Assumption (1a) of Theorem \ref{thm-Ito+Wentzel-LM} is complete.

Now we will verify   Assumption (1b) of Theorem \ref{thm-Ito+Wentzel-LM}, that is,
 for $\mathbb{P}\otimes \Leb$-almost-all $(\omega,t)\in \Omega\times [0,T]$,  the function
\[
\mathbb{R}^d \ni y \mapsto   G_t^\eps(y) \in \mathbb{R}
\]
 is continuous  and for any $y_0\in\mathbb{R}^d$, 
\begin{equation}
\lim_{y\rightarrow y_0}\left[\lint_{B}|\newh^\eps_t(y,z)-\newh^\eps_t(y_0,z)|
^{2}\nu (dz) +\int_{B^c}|\newh^\eps_t(y,z)-\newh^\eps_t(y_0,z)|\nu (dz)\right]= 0.
\end{equation}
Let us choose and fix   $y_0\in\mathbb{R}^d$. We choose  $K>0$ such that $y_0\in \mathbb{K}:=[-K,K]^d$. 
 Then, { if  $y\in [-K-1,K+1]^d $,} we have
\begin{align}
    | G_t^\eps(y)- G_t^\eps(y_0)|
    &\leq \Bigl\{
    M\|\vartheta_\eps\|_2\int_{\mathbb{K}_\eps}|b(x)|+|\divv b(x)|dx
    \\
    &\quad +
    M\|\vartheta_\eps\|_3\lint_B|z|^2\nu(dz)\Leb(B(2\eps))\Bigr\} \times |y-y_0|,
\end{align}
where  {$\mathbb{K}_\eps=[-K-1-2\eps,K+1+2\eps]^d$,}  and
\begin{align}
    &\lint_{B}|\newh^\eps_t(y,z)-\newh^\eps_t(y_0,z)|
^{2}\nu (dz) +\int_{B^c}|\newh^\eps_t(y,z)-\newh^\eps_t(y_0,z)|\nu (dz)\\
&\leq\;
\Bigl\{ M^2\lint_B|z|^2\nu(dz)\|\vartheta_\eps\|^2_2\Big(\Leb(B(2\eps))\Big)^2|y-y_0| +
2M\nu(B^c)\|\vartheta_\eps\|_1\Leb(B(2\eps))\Bigr\} \times |y-y_0|.
\end{align}
The above two inequalities imply that  Assumption (1b) of Theorem \ref{thm-Ito+Wentzel-LM} is satisfied.\\
\indent We will now verify condition \eqref{eqn-jumps supp 2025 06}. It is sufficient to   prove that, for every $K,T>0$ and  $\mathbb{K}=[-K,K]^d$, $\mathbb{P}$-a.s.
\begin{equation}\label{eq-Ass-101}
    \int_0^T\sup_{y\in\mathbb{K}} \left(|G_s^\eps(y)| +\lint_{B}|\newH_s^\eps(y,z)|^2\nu(dz)+ \int_{B^c}|\newh_s^\eps(y,z)|\nu(dz)\right)ds<\infty.
\end{equation}
Note that
\begin{align}
\int_0^T\sup_{y\in\mathbb{K}} |G_s^\eps(y)|ds
&\leq
TM\|\vartheta_\eps\|_1\int_{[-K-2\eps,K+2\eps]^d}|b(x)|+|\divv b(x)|dx
\\
&+
TM\|\vartheta_\eps\|_2\lint_{B}|z|^2\nu(dz)\Leb(B(2\eps))
< \infty,
\end{align}
and
\begin{align}
\int_0^T\sup_{y\in\mathbb{K}} \lint_{B}|\newH_s^\eps(y,z)|^2\nu(dz)ds
\leq
TM^2\|\vartheta_\eps\|_1^2\lint_{B}|z|^2\nu(dz)\Big(\Leb(B(2\eps))\Big)^2
< \infty,
\end{align}
and
\begin{align}
\int_0^T\sup_{y\in\mathbb{K}} \int_{B^c}|\newh_s^\eps(y,z)|\nu(dz)ds
\leq
2TM\|\vartheta_\eps\|_0\Leb(B(2\eps))\nu(B^c)
<\infty.
\end{align}
Hence \eqref{eq-Ass-101} follows.\\
\indent Finally we will verify  condition \eqref{eqn-jumps do not overlap}. Let us recall     that  $F={u}^\eps$  and  $\xi_t=\phi_t$ as in the initial statement of the lemma.
For this aim,    it is sufficient to   prove that, for every $K,T>0$, $\mathbb{P}$-a.s.
\begin{align}\label{eq-1029-1}
	\sum_{s\leq T}\;\sup_{y\in \mathbb{K}}|DF(s,y)-DF(s-,y)||\xi_{s}-\xi_{s-}|<\infty,
	\end{align}
where as before $\mathbb{K}=[-K,K]^d$, and $DF(s,y)=D_{{y}}F(s,y)$.
	
 Observe that
\begin{align}
D_{{y}}F(s,y)= D \int_{\R^{d}}u(s,x)\vartheta_{\eps}(y-x)\,dx
=-\int_{\R^{d}}u(s,x)D_{x}\bigl(\vartheta_{\eps}(y-x)\bigr)\,dx,\;\; s\in [0,T],\;y \in \mathbb{R}^d.
\end{align}
We also claim that $\mathbb{P}$-a.s., for any $y\in\mathbb{R}^d$ and $s\in(0,T]$,
\begin{align}\label{eq 20250731 01}
    D_{{y}} F(s,y)-D_{{y}}F(s-,y)&=\tilde{h}^{\varepsilon}_s(y,\Delta L_s)\mathds{1}_{B^c}(z)(\Delta L_s)+
 \tilde{h}^\eps_s(y,\Delta L_s)\1_{B\backslash\{0\}}(\Delta L_s),\\
    \xi_s-\xi_{s-}&=\Delta L_s\mathds{1}_{B^c}(z)(\Delta L_s)+\Delta L_s\1_{B\backslash \{0\}}(\Delta L_s),
\end{align}
where for $s\in (0,T]$,  $y \in \mathbb{R}^d$ and $z \in \mathbb{R}^d$,
\begin{align}\label{eqn-G^eps-2025}
\tilde{\newh}^\eps_s(y,z) &= -\int_{\mathbb{R}^d}u(s-,x)\,[D_x(\vartheta_\eps(y-x-z)) -D_x(\vartheta_\eps(y-x)) ]\, dx.\nonumber
\end{align}
It follows directly from \eqref{eq the fact 20251031} and Lemma \ref{lemma 20250731 01} that this claim holds. 
 
Then we have
	\begin{align}
	    &\sum_{s\leq T}\;\sup_{y\in \mathbb{K}}|DF(s,y)-DF(s-,y)||\xi_{s}-\xi_{s-}|\\
        &=\sum_{s\leq T}\;\sup_{y\in \mathbb{K}}\Big[ | \tilde{\newh}^\eps_s(y,\Delta L_s)| |\Delta L_s| \mathds{1}_{B^c}(z)(\Delta L_s)+ |
 \tilde{\newh}^\eps_s(y,\Delta L_s)||\Delta L_s| \1_{B\backslash\{0\}}(\Delta L_s)\Big]\\
        &=
\int_0^T\int_{B^c}\sup_{y\in \mathbb{K}}\big|\tilde{\newh}^{\varepsilon}_s(y,z)\big||z|\, \prm(ds,dz)
+\int_0^T\lint_{B}\sup_{y\in \mathbb{K}}\big|
 \tilde{\newh}^\eps_s(y,z)\big||z|{\prm}(ds,dz).
	\end{align}
Since $\nu(B^c)<\infty$ and
\[\sup_{\omega\in\Omega}\sup_{s\in[0,T]}\sup_{y\in \mathbb{K}}\big|\tilde{\newh}^{\varepsilon}_s(y,z)\big|\leq 2M\|\vartheta_\eps\|_1 \Leb(B(2\eps))<\infty,\]
 we infer that, $\mathbb{P}$-a.s.,
$$
\int_0^T\int_{B^c}\sup_{y\in \mathbb{K}}\big|\tilde{\newh}^{\varepsilon}_s(y,z)\big||z|\, \prm(ds,dz)<\infty.
$$

Define a function
\[
f_{\eps,y,x}: \mathbb{R}^d \ni v \mapsto D_x(\vartheta_\eps(y-x-v)) \in   \mathbb{R}^d.
\]
Then
\begin{eqnarray}
    \sup_{\omega\in\Omega}\sup_{s\in[0,T]}\sup_{y\in \mathbb{K}}|
 \tilde{\newh}^\eps_s(y,z)||z|
    &\leq&
    \sup_{\omega\in\Omega}\sup_{s\in[0,T]}\sup_{y\in \mathbb{K}}\Big(\int_{\mathbb{R}^d}|u(s,x)|\,\Big|\int_0^1D_vf_{\eps,y,x}(\delta z)\cdot z\,d\delta\Big| \,dx\Big)|z|\\
    &\leq&
    M\|\vartheta_\eps\|_2\, \Leb(B(2\eps))|z|^2.
\end{eqnarray}
Combining with $\lint_{B}|z|^2\nu(dz)<\infty$, we have, $\mathbb{P}$-a.s.,
$$
\int_0^T\lint_{B}\sup_{y\in \mathbb{K}}\big|
 \tilde{\newh}^\eps_s(y,z)\big||z|{\prm}(ds,dz)<\infty.
$$
Hence, \eqref{eq-1029-1} holds, which verifies condition  \eqref{eqn-jumps do not overlap}.  Therefore, all the assumptions of Theorem  \ref{thm-Ito+Wentzel-LM} are fulfilled, and the proof is complete.
\end{proof}

We now prove the following lemma, which implies \eqref{eq 20250731 01} and helps clarify the following term appearing in the It\^o-Wentzell formula
\begin{equation}%\label{s33}
\sum_{s\leq t}\Big(\Delta  u^\eps(s,\phi_s(x))-\Delta  u^\eps(s,\phi_{s-}(x))\Big).
\end{equation}
\begin{lemma}\label{lemma 20250731 01}  Let $T>0$.
    There exists a $\mathbb{P}$-full set  $\Omega_1\subset \Omega$ such that for any $\omega\in\Omega_1$, $y\in\mathbb{R}^d$ and $t\in (0,T]$,
    \begin{align}
     u^\eps(t,y)-u^\eps(t-,y)
     =&
   \nonumber   \int_{\{t\}}\int_{B^c}\newh_s^\eps(y,z)\mu(ds,dz) +
\int_{\{t\}}\lint_{B}\newh_s^\eps(y,z)\mu(ds,dz)\\
=&\1_{\{ | \Delta L_t|\neq 0 \}}\,
\newh_t^{\eps}(y, \Delta L_t),
 \end{align}
 where $
 \newh^\eps_s(y,z)
= u^{\eps}(s-,y-z)-u^{\eps}(s-,y)$.

\end{lemma}
\begin{proof}
Recall from \eqref{eqn-u^eps-2} that
\begin{equation}\label{eqn-u^eps(t,y)-new}
 u^\eps(t,y)=\int_0^tG^\eps_s(y) \,ds +\int_0^t\int_{B^c}\newh^\eps_s(y,z)\, \prm(ds,dz)
+\int_0^t\lint_{B}\newh^\eps_s(y,z)\tilde{\prm}(ds,dz).
\end{equation}
Then, for any $y \in \R^d$, $\mathbb{P}$-a.s., for any $t \in (0,T]$, we have
\begin{align}\label{eqn-ZB-doubt-01}
      & u^\eps(t,y)-u^\eps(t-,y)\\
      &=
    \int_0^t\int_{B^c}\newh_s^\eps(z,y)\mu(ds,dz)
    +
    \int_0^t\lint_{B}\newH_s^\eps(z,y)\tilde\mu(ds,dz)\\
    &\quad -
    \lim_{l\nearrow t}\Big(\int_{0}^l \int_{B^c}\newh_s^\eps(z,y)\mu(ds,dz)
    +
    \int_{0}^l\lint_{B}\newH_s^\eps(z,y)\tilde\mu(ds,dz)\Big)\\
    &= \int_{\{t\}}\int_{B^c}\newh_s^\eps(y,z)\mu(ds,dz)+\int_0^t\lint_{B}\newH_s^\eps(z,y)\tilde\mu(ds,dz)-\lim_{l\nearrow t} \int_{0}^l\lint_{B}\newH_s^\eps(z,y)\tilde\mu(ds,dz).
\end{align}
Denote $B_{\delta} = \{x \in \R^d \, :\,  \delta < |x| \le 1\}$, $\delta \in (0,1).$ 
Let us fix $y \in \R^d.$ By combining the argument  in the definition of the It\^o integral for $f \in \mathbf{F}^2_p$ given in \cite[page 63]{Ikeda+Watanabe_1989} 
with the argument  from the proof of    Theorem 41 in \cite[page 31]{Protter_2004}, we deduce that we can find a subsequence $\delta_n \to 0^+$ (possibly depending on $y$) such that $\mathbb{P}$-a.s.
\begin{align}
\label{eqn-convergence 0731}
&\lim_{\delta_n \to 0}  \Big[\int_{0}^t \int_{B_{\delta_n}}\newh_s^\eps(y,z) \mu(ds,dz)
- \int_{0}^t \int_{B_{\delta_n}}\newh_s^\eps(y,z) \nu(dz) ds\Big]
\\ \nonumber
&=\int_{0}^t \lint_{B_{}}\newh_s^\eps(y,z)\tilde \mu(ds,dz), \mbox{ uniformly  w.r.t. }  t \in [0,T].
 \end{align}
  Let us choose and fix such a subsequence. In what follows, instead of writing $\lim_{\delta_n \to 0}$, we will write $\lim_{\delta \to 0^+}$.

By the uniform convergence \eqref{eqn-convergence 0731}, we can change the order of limits $\lim_{\delta \to 0+}\lim_{l \toup t}=\lim_{l \toup t}\lim_{\delta \to 0+} $, and therefore, for any $y \in \R^d$, $\mathbb P$-a.s., for any $t\in[0,T]$,
\begin{align*}
     &u^\eps(t,y)-u^\eps(t-,y)\\
      &=
    \int_{\{t\}}\int_{B^c}\newh_s^\eps(y,z)\mu(ds,dz)
    +
\lim_{\delta \to 0^+}   \Big[ \int_0^{t}\int_{B_{\delta}}\newh_s^\eps(y,z)\mu(ds,dz)
- \lim_{l \toup t}\int_0^{l}\int_{B_{\delta}}\newh_s^\eps(y,z)\mu(ds,dz) \Big]
\\
&=  \int_{\{t\}}\int_{B^c}\newh_s^\eps(y,z)\mu(ds,dz) +
\lim_{\delta \to 0^+}
\int_{\{t\}}\int_{B_{\delta}}\newh_s^\eps(y,z)\mu(ds,dz).
 \end{align*}
 We have, $\mathbb P$-a.s.,  for any $\delta \in (0,1)$ and $t\in[0,T]$,
\begin{align}
    \int_{\{t\}}\int_{B_{\delta}} \newh_s^{\eps}(y,z) \mu(ds,dz)
   & = \sum_{0<s \le t} \1_{ \{t\}}(s) \, \1_{\{ \delta < | \Delta L_s|\le 1 \}}\,
\newh_s^{\eps}(y, \Delta L_s)\\
&=  \1_{\{ \delta < | \Delta L_t|\le 1 \}}\,
\newh_t^{\eps}(y, \Delta L_t).
 \end{align}
  Hence,  for any $y \in \R^d$, $\mathbb P$-a.s., for any $t\in[0,T]$,
 \begin{equation*}
    \lim_{\delta \to 0^+} \int_{\{t\}}\int_{B_{\delta}}\newh_s^{\eps}(y,z) \mu(ds,dz) = \int_{\{t\}}\lint_{B}\newh_s^{\eps}(y,z)\mu(ds,dz)=\1_{\{ 0 < | \Delta L_t|\le 1 \}}\,
\newh_t^{\eps}(y, \Delta L_t).
\end{equation*}
 We arrive at, for any $t\in(0,T]$,  
 \begin{align} \label{d66 0731}
     u^\eps(t,y)-u^\eps(t-,y)
      =&
    \int_{\{t\}}\int_{B^c}\newh_s^\eps(y,z)\mu(ds,dz) +
\int_{\{t\}}\lint_{B}\newh_s^\eps(y,z)\mu(ds,dz) \nonumber\\
=&\1_{\{ | \Delta L_t|\neq 0 \}}\,
\newh_t^{\eps}(y, \Delta L_t).
 \end{align}

{Next, we shall prove that,  $\mathbb{P}$-a.s., \eqref{d66 0731} holds for any $y \in \mathbb{R}^d$.}
Let $D$ be a countable dense set in $\mathbb{R}^d$. From the above argument, we can find a  $\mathbb{P}$-full event $\Omega'\subset \Omega$  such that for any $\omega\in\Omega'$, \eqref{d66 0731} holds for any $y \in D$. Note that $\Omega'$ is  independence of $t\in[0,T]$.

Also, by \eqref{x3}, we  know that,  $\mathbb{P}$-a.s.,
$ {u}^\eps(\cdot,\cdot)\in D([0,T];C_{}(\mathbb{R}^d,\mathbb{R}))$. It follows that,
 $\mathbb{P}$-a.s., for any $t\in(0,T]$,
\begin{equation} \label{u11 0731}
y \mapsto u^\eps(t,y)-u^\eps(t-,y) \;\; \text{is continuous on $\mathbb{R}^d$}.
\end{equation}
On the other hand, in view of \eqref{c441},  $\mathbb{P}$-a.s., for any $t\in(0,T]$,
\begin{equation} \label{u12 0731}
 y \mapsto
 \int_{\{t\}}\int_{B^c}\newh_s^\eps(y,z)\mu(ds,dz) +
\int_{\{t\}}\lint_{B}\newh_s^\eps(y,z)\mu(ds,dz) = \1_{\{ | \Delta L_t|\neq 0 \}}\newh_t^\eps(y,\triangle L_t)
\end{equation}
is continuous on $\mathbb{R}^d$. 
Taking both \eqref{u11 0731} and \eqref{u12 0731} into account, we can conclude that, there exists a $\mathbb{P}$-full set $\Omega_1 \subset \Omega' \subset \Omega$ such that 
 for any 
$\omega \in \Omega_1$, $y \in \mathbb{R}^d$,   \eqref{d66 0731} holds. The lemma has been proved.
\end{proof}

Now we are in a position to  apply  the  It\^o-Wentzell formula to $u^\eps(t,\phi_t(x))$.

\begin{theorem}\label{thm-IW-applied}  Assume that  $x \in \R^d$. Then,
 $\mathbb P$-a.s., for any $t \in [0,T]$,
 \begin{align}
 u^\eps(t,\phi_t(x))
=& \int_0^t Du^{\eps}(s,\phi_{s}(x))\cdot b(\phi_{s}(x))\,ds\label{eqn-u^eps-final}
\\
&+\int_0^t \int_{\mathbb{R}^d}u(s,y)\,\bigl[
b(y)\cdot D_y\bigl(\vartheta_\eps(\phi_{s}(x)-y)\bigr)
+\divv\, b(y) \vartheta_\eps(\phi_{s}(x)-y)
\bigr]\,d y\, ds.
\end{align}
 \end{theorem}
 \begin{proof}[Prof of Theorem \ref{thm-IW-applied}]
 Since we have verified all the assumptions of Theorem \ref{thm-Ito+Wentzel-LM} in Lemma \ref{lem-u^eps satisfies assumptions ITWF}, we can apply the   It\^o-Wentzell formula \eqref{eqn-Ito+Wentzel-LM} stated in Theorem \ref{thm-Ito+Wentzel-LM} with $F=u^\eps$ given by \eqref{eqn-u^eps-2} and $\xi_t=\phi_t$ given by \eqref{Ito-nxi-psi}. The application of It\^o-Wentzell formula \eqref{eqn-Ito+Wentzel-LM} must be handled carefully, as it is not entirely straightforward.
 
 We now calculate each term explicitly. 
Keep in mind that
$$
\xi_t = \xi_{t-}+  \Delta L_t = \phi(t-, x) + \Delta L_t.
$$
First, note that 
\begin{align}
    F(t,\xi_t)&=u^\eps(t,\xi_t),\\
    F(0,\xi_0)&=u^\eps(0,\xi_0)=0,
\\
    \int_0^tf_s(\oldL_{s})ds&=\int_0^tG^{\eps}_{s}(\xi_{s})\,ds,
\end{align}
where by  \eqref{eqn-G^eps-21},
 \begin{align*}
 G^\eps_s(y)=& \int_{\mathbb{R}^d}u(s,x)\,[b(x)\cdot D_x \bigl(\vartheta_\eps(y-x)\bigr) +\divv\, b(x) \vartheta_\eps(y-x)]\,dx\nonumber
 \\
 &+ \lint_B\Big[ u^{\eps}(s,y-z)-u^{\eps}(s,y)+ Du^{\eps}(s,y) \cdot z  \Big]\nu(dz).
\end{align*}
Since
$
\newh^\eps_s(y,z) =  u^{\eps}(s-,y-z)-u^{\eps}(s-,y)$, see  \eqref{eqn-p^eps-2}, we get
\begin{align} \label{b4}
    &\int_0^t\int_{\R^d}h_s(\oldL_{s-},z)[\1_{B\setminus\{0\}}(z)\newq(ds,dz) +\1_{B^c}(z)\mu(ds,dz)]\\
    &=
     \int_0^t\int_{B^c} \Bigl( u^\eps(s-,\xi_{s-}-z)-u^\eps(s-,\xi_{s-}) \Bigr) \,\mu(ds,dz)\\
    &\quad+
    \int_0^t\lint_{B} \Bigl( u^\eps(s-,\xi_{s-}-z)-u^\eps(s-,\xi_{s-})\Bigr) \,  \tilde\mu(ds,dz).
\end{align}
Recall that  $A_t =\int_{0}^{t}b(\nxi_{r}( x) )\,dr$ and $S_t=L_t$ in \eqref{f4}. It follows that 
\begin{align} \label{b3}
    \int_0^t D F(s-,\oldL_{s-})d\xi_s
   =& \int_0^t D F(s-,\oldL_{s-})dA_s + \int_0^t D F(s-,\oldL_{s-})dS_s\\
    =&\int_0^t b(\xi_s)\cdot D u^\eps(s,\xi_s)ds
    +
    \int_0^t\lint_{B}  D u^\eps(s-,\xi_{s-})\cdot z\, \tilde \prm(ds,dz)
    \\
    &+
    \int_0^t\int_{B^c}  D u^\eps(s-,\xi_{s-})\cdot z\, \prm(ds,dz).
\end{align}
 Moreover,
 \begin{align}\label{b2}
   &\sum_{s\le t}\Big(F(s-,\oldL_{s})-F(s-,\oldL_{s-})-D F(s-,\oldL_{s-})\Delta \oldL_s\Big)\\
   &=
   \int_0^t\int_{B^c}\big(u^\eps(s-,\xi_{s-}+z)-u^\eps(s-,\xi_{s-})- D u^\eps(s-,\xi_{s-})\cdot z\big) \mu(ds,dz)\\
    &\quad+
    \int_0^t\lint_{B} \big( u^\eps(s-,\xi_{s-}+z)-u^\eps(s-,\xi_{s-})- D u^\eps(s-,\xi_{s-})\cdot z\big)\mu(ds,dz),
\end{align}
which is meaningful since, for $z \in B,$ $s \in (0,t]$, $\mathbb{P}$-a.s.,
$$
|u^\eps(s-,\xi_{s-}+z)-u^\eps(s-,\xi_{s-})- D u^\eps(s-,\xi_{s-})\cdot z|
\le C{\|u^\eps\|_2}|z|^2.
$$
The most delicate term to understand in the It\^o-Wentzell formula is
\begin{equation}\label{s33}
\sum_{s\leq t}\Big(\Delta F(s,\oldL_{s})-\Delta F(s,\oldL_{s-})\Big).
\end{equation}
By Lemma \ref{lemma 20250731 01}, we have, $\mathbb{P}$-a.s., for any $y\in\mathbb{R}^d$ and $s\in(0,T]$,  
 \begin{align} \label{d66}
    F(s,y)-F(s-,y)= &u^\eps(s,y)-u^\eps(s-,y)
    \\  =&
   \nonumber   \int_{\{s\}}\int_{B^c}\newh_r^\eps(y,z)\mu(dr,dz) +
\int_{\{s\}}\lint_{B}\newh_r^\eps(y,z)\mu(dr,dz)\\
=&\1_{\{ | \Delta L_s|\neq 0 \}}\,
\newh_s^{\eps}(y, \Delta L_s).
 \end{align}
By substituting $y=\oldL_s$ or $\oldL_{s-}$ in \eqref{d66} and combining
\begin{align*}
 \Delta F(s,\oldL_{s})-\Delta F(s,\oldL_{s-})
        :=\Big( F(s,\oldL_{s})-F(s-,\oldL_{s})\Big)-\Big(F(s,\oldL_{s-})-F(s-,\oldL_{s-})\Big),
       \end{align*}
we infer that
\begin{align}\label{eq-1029-2}
    &\Delta F(s,\oldL_{s})-\Delta F(s,\oldL_{s-})\nonumber\\
&= \1_{\{ | \Delta L_s|\neq 0 \}}\,
\newh_s^{\eps}(\oldL_{s}, \Delta L_s)
-
\1_{\{ | \Delta L_s|\neq 0 \}}\,
\newh_s^{\eps}(\oldL_{s-}, \Delta L_s)\nonumber\\
&=
\1_{\{ | \Delta L_s|\neq 0 \}}\,
\newh_s^{\eps}(\oldL_{s-}+\Delta L_s, \Delta L_s)
-
\1_{\{ | \Delta L_s|\neq 0 \}}\,
\newh_s^{\eps}(\oldL_{s-}, \Delta L_s)\nonumber\\
&=
\1_{\{ 0 < | \Delta L_s|\le 1 \}}\,
[\newh_s^{\eps}(\xi_{s-} + \Delta L_s, \Delta L_s)-\newh_s^\eps( \xi_{s-}, \Delta L_s )]\nonumber\\
&\quad +
\1_{\{ | \Delta L_s|> 1 \}}\,
[\newh_s^{\eps}(\xi_{s-} + \Delta L_s, \Delta L_s)-\newh_s^\eps( \xi_{s-}, \Delta L_s )]\nonumber\\
&=\int_{\{s\}}\lint_{B}[\newh_r^{\eps}(\xi_{r-} + z,z)-\newh_r^\eps(  \xi_{r-},z)]\,\mu(dr,dz)+
\int_{\{s\}}\int_{B^c}[\newh_r^{\eps}(\xi_{r-} + z,z)-\newh_r^\eps(  \xi_{r-},z)]\,\mu(dr,dz)\\
&=
\int_{\{s\}}\lint_{B} ( [u^\eps(r-,\xi_{r-})-u^\eps(r-,\xi_{r-}+z)]
  -
    [u^\eps(r-,\xi_{r-}-z)-u^\eps(r-,\xi_{r-})])\mu(dr,dz)\nonumber\\
    &\quad +
    \int_{\{s\}}\int_{B^c} ( [u^\eps(r-,\xi_{r-})-u^\eps(r-,\xi_{r-}+z)]
  -
    [u^\eps(r-,\xi_{r-}-z)-u^\eps(r-,\xi_{r-})])\mu(dr,dz).
\end{align}
Note that
\begin{align*}
  & \sum_{s\leq t}  \int_{\{s\}}\lint_{B} \Big| [u^\eps(r-,\xi_{r-})-u^\eps(r-,\xi_{r-}+z)]
  -
    [u^\eps(r-,\xi_{r-}-z)-u^\eps(r-,\xi_{r-})]\Big|\mu(dr,dz)\\
&=\int_0^t\lint_{B}\Big|u^\eps(s-,\xi_{s-})-u^\eps(s-,\xi_{s-}+z)
    -
    \Big(u^\eps(s-,\xi_{s-}-z)-u^\eps(s-,\xi_{s-})\Big)\Big|\, \mu(ds,dz)\\
    &\leq
    \|u^\eps\|_2\int_0^t\lint_{B}|z|^2\mu(ds,dz),
\end{align*}
and
\begin{align*}
  & \sum_{s\leq t}  \int_{\{s\}}\int_{B^c} \Big| u^\eps(r-,\xi_{r-})\Big|+\Big|u^\eps(r-,\xi_{r-}+z)\Big|
  +
    \Big|u^\eps(r-,\xi_{r-}-z)\Big|+\Big|u^\eps(r-,\xi_{r-})\Big|\mu(dr,dz)\\
&=\int_0^t\int_{B^c}\Big|u^\eps(s-,\xi_{s-})\Big|+\Big|u^\eps(s-,\xi_{s-}+z)\Big|
    +
    \Big|u^\eps(s-,\xi_{s-}-z)\Big|+\Big|u^\eps(s-,\xi_{s-})\Big|\, \mu(ds,dz)\\
    &\leq
    4\|u^\eps\|_0\int_0^t\int_{B^c}1\mu(ds,dz),
\end{align*}
both of which are finite,  $\mathbb P$-a.s., because
\begin{align}
    \mathbb{E}\int_0^t\lint_{B}|z|^2\mu(ds,dz)
    = t\lint_{B}|z|^2\nu(dz)
    <\infty,
\end{align}
and
\begin{align}
    \mathbb{E}\int_0^t\int_{B^c}1\mu(ds,dz)
    = t\nu(B^c)
    <\infty.
\end{align}
Hence, the term \eqref{s33} is well defined and we have
 \begin{align} \label{b1}
     &\sum_{s\leq t}\Big(\Delta F(s,\oldL_{s})-\Delta F(s,\oldL_{s-})\Big)\\
    &= \int_0^t\int_{B^c}[u^\eps(s-,\xi_{s-})-u^\eps(s-,\xi_{s-}+z)]\,\mu(ds,dz)
     -\int_0^t\int_{B^c}{ [u^\eps(s-,\xi_{s-}-z)-u^\eps(s-,\xi_{s-})]}\, \mu(ds,dz)\\
    &\quad+
    \int_0^t\lint_{B}\{ [u^\eps(s-,\xi_{s-})-u^\eps(s-,\xi_{s-}+z)]
     {-}
    \big[u^\eps(s-,\xi_{s-}-z)-u^\eps(s-,\xi_{s-})\big]\}\mu(ds,dz).
\end{align}
      Applying the It\^o-Wentzell formula and then using \eqref{b4}, \eqref{b3}, \eqref{b2} and \eqref{b1}, we obtain
\begin{align}\label{eqn-IWFapplication}
 u^\eps(t,\xi_t)=& \int_0^tG^{\eps}_{s}(\xi_{s})\,ds+ \int_0^t b(\xi_{s}) \cdot Du^{\eps}(s,\xi_{s})\,ds
 \\&
  +
    \int_0^t\lint_{B}  D u^\eps(s-,\xi_{s-})\cdot z\, \tilde \prm(ds,dz)
    \\
    & +
    \int_0^t\int_{B^c}  D u^\eps(s-,\xi_{s-})\cdot z\, \prm(ds,dz)
 \\
 %&+\int_0^t \lint_{\R^d}\Bigl[u^{\eps}(s-,\xi_{s-}+z)-u^{\eps}(s-,\xi_{s-})-Du^{\eps}(s-,\xi_{s-})\cdot z\Bigr]\;\prm(ds,dz)\\
 &+ \int_0^t\int_{B^c}\big(u^\eps(s-,\xi_{s-}+z)-u^\eps(s-,\xi_{s-})\big) \mu(ds,dz)\\
  &- \int_0^t\int_{B^c} D u^\eps(s-,\xi_{s-})\cdot z\mu(ds,dz)\\
    &+
    \int_0^t\lint_{B} \big( u^\eps(s-,\xi_{s-}+z)-u^\eps(s-,\xi_{s-})- D u^\eps(s-,\xi_{s-})\cdot z\big)\mu(ds,dz)\\
 & +\int_0^t \int_{B^c}\Bigl[u^{\eps}(s-,\xi_{s-}-z)-u^{\eps}(s-,\xi_{s-}) \Bigr]\;\prm(ds,dz)\\
 &+\int_0^t \lint_{B}\Bigl[u^{\eps}(s-,\xi_{s-}-z)-u^{\eps}(s-,\xi_{s-})\Bigr]\, \tilde\prm(ds,dz)
\\
 &+ \int_0^t\int_{B^c}{ [u^\eps(s-,\xi_{s-})-u^\eps(s-,\xi_{s-} +z)]}\, \mu(ds,dz)
 \\
 &{ -\int_0^t\int_{B^c}{ [u^\eps(s-,\xi_{s-}-z)-u^\eps(s-,\xi_{s-})]}\, \mu(ds,dz)}
 \\ &+\int_0^t \lint_{B}\Bigl[ [u^{\eps}(s-,\xi_{s-})-u^{\eps}(s-,\xi_{s-}+z)]
 \\
&\hspace{3truecm} -[u^{\eps}(s-,\xi_{s-}-z)-u^{\eps}(s-,\xi_{s-})]\Bigr]\,\prm(ds,dz).
\end{align}
 By combining the third term on the right hand side of \eqref{eqn-IWFapplication} with the ninth term, canceling the fourth term with the sixth term, the eighth term with the eleventh term, and the fifth term with the tenth term, we arrive at
  \begin{align}\label{eqn-IWFapplication-1}
 u^\eps(t,\xi_t)=& \int_0^tG^{\eps}_{s}(\xi_{s})\,ds+ \int_0^t b(\xi_{s}) \cdot Du^{\eps}(s,\xi_{s})\,ds
    \\
 &
 +{\int_0^t \lint_{B}\Bigl[u^{\eps}(s-,\xi_{s-}+z)-u^{\eps}(s-,\xi_{s-})-Du^{\eps}(s-,\xi_{s-})\cdot z\Bigr]\;\prm(ds,dz) }
 \\
 &+{ \int_0^t \lint_{B}\Bigl[u^{\eps}(s-,\xi_{s-}-z)-u^{\eps}(s-,\xi_{s-}) +  D u^\eps(s-,\xi_{s-})\cdot z\Bigr]\, \tilde\prm(ds,dz)}
  \\
 &+ \int_0^t \lint_{B}\Bigl[ u^{\eps}(s-,\xi_{s-})-u^{\eps}(s-,\xi_{s-}+z)
 \\
&\hspace{3truecm} -
[u^{\eps}(s-,\xi_{s-}-z)- u^{\eps}(s-,\xi_{s-})] \Big]\,\prm(ds,dz).
\end{align}
Now we use the expression of $G^{\eps}$ in \eqref{eqn-G^eps-21} to obtain
 \begin{align}\label{eqn-IWFapplication-2}
 u^\eps(t,\xi_t)=&  \int_0^t b(\xi_{s}) \cdot Du^{\eps}(s,\xi_{s})\,ds
   \\ &+\int_0^t\int_{\mathbb{R}^d}u(s,x)\,[b(x)\cdot D_x \bigl(\vartheta_\eps(\xi_s-x)\bigr) +\divv\, b(x) \vartheta_\eps(\xi_s-x)]\,dx\,ds\nonumber\\
 &+ { \int_0^t \lint_B\Big[ u^{\eps}(s,\xi_s-z)-u^{\eps}(s,\xi_s)+ Du^{\eps}(s,\xi_s) \cdot z  \Big]\nu(dz)\,ds}
   \\
 &
 +{\int_0^t \lint_{B}\Bigl[u^{\eps}(s-,\xi_{s-}+z)-u^{\eps}(s-,\xi_{s-})-Du^{\eps}(s-,\xi_{s-})\cdot z\Bigr]\;\prm(ds,dz) }
 \\
 &+{ \int_0^t \lint_{B}\Bigl[u^{\eps}(s-,\xi_{s-}-z)-u^{\eps}(s-,\xi_{s-}) +  D u^\eps(s-,\xi_{s-})\cdot z\Bigr]\, \tilde\prm(ds,dz)}
  \\
 &+ \int_0^t \lint_{B}\Bigl[ u^{\eps}(s-,\xi_{s-})-u^{\eps}(s-,\xi_{s-}+z)
[u^{\eps}(s-,\xi_{s-}-z)- u^{\eps}(s-,\xi_{s-})] \Big]\,\prm(ds,dz).
\end{align}
 Note that the last integral is meaningful because $u^{\eps} \in D([0,T];C^2(\R^d)).$
We find, using  $\tilde \mu = \mu - \nu$, that
 \begin{align}\label{eqn-IWFapplication-3}
 u^\eps(t,\xi_t)=&  \int_0^t b(\xi_{s}) \cdot Du^{\eps}(s,\xi_{s})\,ds
   \\ &+\int_0^t\int_{\mathbb{R}^d}u(s,x)\,[b(x)\cdot D_x \bigl(\vartheta_\eps(\xi_s-x)\bigr) +\divv\, b(x) \vartheta_\eps(\xi_s-x)]\,dx\, ds\nonumber
   \\
 &
 +{ \int_0^t \lint_{B}\Bigl[u^{\eps}(s-,\xi_{s-}+z)-u^{\eps}(s-,\xi_{s-})-Du^{\eps}(s-,\xi_{s-})\cdot z\Bigr]\;\prm(ds,dz) }
 \\
 &+ { \int_0^t \lint_{B}\Bigl[u^{\eps}(s-,\xi_{s-}-z)-u^{\eps}(s-,\xi_{s-}) +  D u^\eps(s-,\xi_{s-})\cdot z\Bigr] \, \prm(ds,dz)}
  \\
 &-  \int_0^t \lint_{B}\Bigl[u^{\eps}(s-,\xi_{s-}+z) - u^{\eps}(s-,\xi_{s-})+
 u^{\eps}(s-,\xi_{s-}-z) - u^{\eps}(s-,\xi_{s-}) \Big]\,\prm(ds,dz).
\end{align}
After performing some cancellations, we obtain
\begin{align}\label{eqn-IWFapplication-4}
 { u^\eps(t,\xi_t)} =    \int_0^t b(\xi_{s}) \cdot Du^{\eps}(s,\xi_{s})\,ds
   +\int_0^t\int_{\mathbb{R}^d}u(s,x)\,[b(x)\cdot D_x \bigl(\vartheta_\eps(\xi_s-x)\bigr) +\divv\, b(x) \vartheta_\eps(\xi_s-x)]\,dx\,ds.
\end{align}
This proves equality \eqref{eqn-u^eps-final}  and thus
the proof of Theorem \ref{thm-IW-applied} is complete.
\end{proof}

%----------------------------

\section{No blow-up of $C^1$-class solutions}\label{sec-no blow up}

\begin{definition}\label{def-transport-C1} Assume $b\in L^1_{loc}(\mathbb{R}^d;\mathbb{R}^d)$ and $u_0 \in C_\mathrm{b}^1(\mathbb{R}^d)$. 
A measurable and essentially bounded function  (in the sense of \eqref{eq-esb0}) $u:
\Omega\times[0,\infty)\times\mathbb{R}^d \to \mathbb{R}
 $ is a $C^1$-solution to the problem \eqref{eqn-transport-Markus}  if it
satisfies the following four conditions: for any $T>0$,
\begin{trivlist}
\item[(i)] for almost every $\omega\in\Omega$, $u(\omega, t,\cdot)\in C^1(\mathbb{R}^d)$, 
for all $t\in[0,T]$; 
{\item[(ii)] for  almost every $\omega\in\Omega$ and every $R>0$, the gradient $Du(\omega,t,x)$ satisfies $$\sup_{t\in[0,T]}\|D u(\omega, t,\cdot)\|_{L^\infty(B(R))}<\infty.$$}
 This means that there exists a $\mathbb{P}$-full set $\Omega^\prime \subset \Omega$ such that for any $R>0$, there exists $M_R\in(0,\infty)$ such that 
 for any $t \in [0,T]$, $\omega \in \Omega^\prime,$
\begin{align*}
\|Du(\omega,t , x)\| \le M_R, \;\;  \forall x \in A^R_{t, \omega},
\end{align*}
where $A^R_{t, \omega}$ is a Borel set of $B(R)$ such that $\Leb( (A^R_{t, \omega})^c)=0$. Recall that $\|\cdot\|_{L^\infty(B(R))}$ denotes the essential supremum norm on $B(R)$, the closed ball of radius $R$ centered at $0$ in $\mathbb{R}^d$;
\item[(iii)] for every test function  $\theta\in C_c^{\infty}(\mathbb{R}^d)$, the process $\int_{\mathbb{R}^d}\theta(x)u(t,x)\,dx$, $t\geq 0$ is $\mathbb{F}$-adapted  and it has c\`{a}dl\`{a}g paths $\mathbb{P}$-a.s.;
\item[(iv)] for every test function  $\theta\in C_c^{\infty}(\mathbb{R}^d)$, the process $\int_{\mathbb{R}^d}\theta(x)u(t,x)\,dx$,  $t\geq 0$,  satisfies that, for every $t\geq 0$, $\mathbb{P}$-a.s.
\begin{align}\label{eqn-transport-Marcus-C1}
      \int_{\mathbb{R}^d}u(t,x)\,\theta(x)\,dx
    =&
    \int_{\mathbb{R}^d}u_{0}(x)\,\theta(x)\,dx
       -\int_0^t\Big(\int_{\mathbb{R}^d}b(x)\cdot D u(s,x)\theta(x)\,dx\Big) ds\nonumber\\
       &+\int_0^t\lint_{B} \Bigl(\int_{\mathbb{R}^d}u(s-,x)\,[\theta(x+z)-\theta(x)]\, dx\Bigr)\tilde{\prm}(ds,dz)\nonumber\\
       &+\int_0^t\int_{B^{\mathrm{c}}} \Bigl(\int_{\mathbb{R}^d}u(s-,x)\,[\theta(x+z)-\theta(x)]\, dx\Bigr) \prm(ds,dz)\\
       &+\int_0^t\lint_{B}\Bigl(\int_{\mathbb{R}^d} u(s,x)\,[\theta(x+z)-\theta(x)-\sum_{i=1}^dz_iD_i\theta(x)]\,dx\Bigr)\,\nu(dz)\, ds.\nonumber
\end{align}
The $\mathbb{P}$-a.s. means that our $\mathbb{P}$-full set depends on $t \in [0,\infty)$ and $\theta$.
\end{trivlist}

\end{definition}

\begin{theorem}\label{th-app-C1-reg}	
 {Assume Hypothesis \ref{hyp-nondeg1}
 if $\alpha\in [1,2)$ or assume  Hypothesis \ref{hyp-nondeg3}
 if $\alpha\in (0,1)$. Assume that $\beta\in(0,1)$ satisfies condition \eqref{eqn-beta+alpha half} and
 $b\in C_{\mathrm{b}}^{\beta}(\mathbb{R}^d,\mathbb{R}^d)$.} Then for every  $u_0 \in C_\mathrm{b}^1\left(\mathbb{R}^d\right)$,  there exists a unique classical $C^1$-solution for the transport equation \eqref{eqn-transport-Markus} of the form
 \[u(\omega,t,x)=u_0(\phi_t^{-1}(\omega)x), \;\; t\in [0,T], x \in\mathbb{R}^d,
\]
  with probability one, in the sense of Definition \eqref{def-transport-C1}.
\end{theorem}

\begin{proof}
Step 1 (Existence). We will prove that $u_0(\phi_t^{-1}(\omega)x),  t\geq 0, x \in\mathbb{R}^d$ is a $C^1$-solution to the problem \eqref{eqn-transport-Markus}, that is, $u_0(\phi_t^{-1}(\omega)x),  t\geq 0, x \in\mathbb{R}^d$ satisfies all conditions in Definition \eqref{def-transport-C1}.

Recall from Theorem \ref{ww1}  both $\phi_t$ and its inverse $\phi^{-1}_t$ are diffeomorphisms of $C^1$ class, and for any $R>0$, $\sup_{t\in[0,T]}\sup_{x\in B(R)}\|D\phi^{-1}_t(x)\|<\infty$ $\mathbb{P}$-a.s.. Since $u_0 \in C_\mathrm{b}^1(\mathbb{R}^d)$,  the composition $u_0\circ \phi^{-1}_t \in C^1(\mathbb{R}^d)$ and for any $x\in\mathbb{R}^d$ and $t\geq0$, $D \Big(u_0(\phi^{-1}_t(x))\Big)=Du_0(\phi^{-1}_t(x))D\phi^{-1}_t(x)$. Hence Conditions (i) and (ii) in Definition \ref{def-transport-C1} is satisfied. For $\theta\in C_c^{\infty}(\mathbb{R}^d)$,  using \eqref{eqn-change of variables} and the properties of the processes  $\phi_{t}(y)$ and $J_{t}(y)$  proved in Theorems \ref{d32} and \ref{ww1}, it follows straightforwardly that Condition (iii) in Definition \ref{def-transport-C1} is satisfied. 

Now let us verify Condition (iv). Recall $b^\eps$ defined in \eqref{def-b^eps} and the set $\Upsilon$  introduced prior to  Lemma \ref{lem-A11}. 
Then $b^\eps$, $\eps \in \Upsilon$, is a $C_{\mathrm{b}}^\infty$ approximation  of the vector field $b$. Let  $\phi^{\eps}_{s,t}$ be the stochastic flow associated with equation  \eqref{eqn-SDE}  %{flow_eq1}
 with the drift  $b$ replaced by $b^{\eps}$. Then following the arguments in the proof of Theorem \ref{thm-transport equation}, for every $\omega \in \Omega^{(1)}$ with $\mathbb{P}(\Omega^{(1)})=1$, and all $t\geq 0$, $y\in\mathbb{R}^d$ and $\eps \in \Upsilon$, we have, see \eqref{eq 20251105 V1},
     \begin{align*}
    &\int_{\R^d} u_0(y)\,\theta (\phi^{\eps}_t(y))J_t^{\eps}(y)\,dy\\
&=\int_{\R^d}u_0(y)\,\theta(y)\,dy 
    +\int_0^t\Bigl(\int_{\R^d}u_0(y)\big[b^{\eps}(\phi^{\eps}_s(y))\cdot D \theta(\phi^{\eps}_s(y))+\theta(\phi_{s}^{\eps}(y)) \divv b^{\eps}(\phi^{\eps}_s(y))]J_s^{\eps}(y)\,dy \Bigr) ds\\
     &\quad +\int_0^t\lint_{B}\int_{\R^d}u_0(y)\big[  \theta(\phi_{s}^{\eps}(y)+z)-\theta(\phi_{s}^{\eps}(y)) -\sum_{i=1}^d z_iD_i\theta(\phi_{s}^{\eps}(y))\big] J_{s}^{\eps}(y)dy\; \nu(dz)\, ds\\
         &\quad +\int_0^t\int_{B^{c}}\int_{\R^d}u_0(y)\big[  \theta(\phi_{s-}^{\eps}(y)+z)-\theta(\phi_{s-}^{\eps}(y)) \big] J_{s-}^{\eps}(y)dy\; \prm(ds,dz)\\
    &\quad +\int_0^t\lint_{B}\int_{\R^d}u_0(y)  \big[ \theta(\phi_{s-}^{\eps}(y)+z)-\theta(\phi_{s-}^{\eps}(y)) \big ]J_{s-}^{\eps}(y)dy\; \tilde{\prm}(ds,dz).
     \end{align*}
Since $u_0 \in C_{\mathrm{b}}^1(\mathbb{R}^d)$, by changing variable $y=(\phi_{t}^{\eps})^{-1}(x)$ or $x=\phi_{t}^{\eps}(y)$,  and  integration by parts, we infer
\begin{align*}
&\int_0^t\Bigl(\int_{\R^d}u_0(y)\big[b^{\eps}(\phi^{\eps}_s(y))\cdot D \theta(\phi^{\eps}_s(y))+\theta(\phi_{s}^{\eps}(y)) \divv b^{\eps}(\phi^{\eps}_s(y))\big]J_s^{\eps}(y)\,dy \Bigr) ds\\
&=\int_0^t\Bigl(\int_{\R^d}u_0(y) b^{\eps}(\phi^{\eps}_s(y))\cdot D \theta(\phi^{\eps}_s(y))J_s^{\eps}(y)\,dy \Bigr) ds +\int_0^t\Bigl(\int_{\R^d}u_0((\phi_{s}^{\eps})^{-1}(x))\theta(x) \divv b^{\eps}(x)\,dx \Bigr) ds\\
&=\int_0^t\Bigl(\int_{\R^d}u_0(y) b^{\eps}(\phi^{\eps}_s(y))\cdot D \theta(\phi^{\eps}_s(y))J_s^{\eps}(y)\,dy \Bigr) ds -\int_0^t\Bigl(\int_{\R^d}u_0((\phi_{s}^{\eps})^{-1}(x)) D\theta(x) \cdot b^{\eps}(x)\,dx \Bigr) ds\\
&\quad -\int_0^t\Bigl(\int_{\R^d} D u_0((\phi_{s}^{\eps})^{-1}(x))D(\phi_{s}^{\eps})^{-1}(x) \cdot b^{\eps}(x)\theta(x) \,dx \Bigr) ds\\
&=-\int_0^t\Bigl(\int_{\R^d} D u_0(y) \cdot b^{\eps}( \phi_{s}^{\eps}(y))\theta(\phi_{s}^{\eps}(y)) J_s^{\eps}(y)\,dy \Bigr) ds.
\end{align*}
Applying similar arguments as in \eqref{lem-A3}, we have 
\begin{equation}\label{eqn-App-e1}
\begin{aligned}
\sup_{t\in [0,T]}
&\Bigl\vert \int_0^t\Bigl(\int_{\R^d} D u_0(y)\cdot b^{\eps}(\phi^{\eps}_s(y))  \theta(\phi^{\eps}_s(y))J_s^{\eps}(y)\,dy \Bigr)\, ds
\\
&\hspace{1truecm}-
\int_0^t\Bigl(\int_{\R^d} D u_0(y)\cdot b(\phi_s(y)) \theta(\phi_s(y))J_s(y)\,dy \Bigr)\, ds \Bigr\vert \rightarrow0,
\end{aligned}
  \end{equation}
  as $\varepsilon\rightarrow 0$, $\mathbb{P}$-a.s. 
By combining this with Lemmata  \ref{lem-A2}, \ref{lem-A5}, \ref{lem-A6} and \ref{lem-A7}, taking the limit, and then by changing variable $y=(\phi_{t})^{-1}(x)$ or $x=\phi_{t}(y)$, we derive the identity \eqref{eqn-transport-Marcus-C1}.\\
\indent Step 2 (Uniqueness). Assume that $u_0=0$. Following similar arguments as in Step 1 and 2 in the proof of Theorem \ref{thm-uniqueness-1}, see \eqref{sec-6-conv-u-R-1}, for $\rho\in C_c^{\infty}(\mathbb{R}^d)$, we have 
 $\mathbb{P}$-a.s.
\begin{align}\label{sec-6-conv-u-R-1-1}
\int_{\mathbb{R}^d} u(t, \phi_{t}(x)) \rho(x) d x&=-\lim_{\eps\to  0}\int_{0}^{t} \Big( \int_{\mathbb{R}^d}  \mathcal{R}_{\eps}\bigl[b, u_{s}\bigr](\phi_{s}(x))\Big) \rho(x) dx\Big)d s\\
&=-\lim _{\varepsilon \rightarrow 0} \int_0^t\Big(\int_{\mathbb{R}^d} \mathcal{R}_{\eps}[ b,u_s](y) \rho\left(\phi_s^{-1}(y)\right) J \phi_s^{-1}(y) d y\Big) d s.
\end{align}

We will prove that the above limit equals 0. First, we provide some preliminary results.

Similar as in the proof of Theorem \ref{thm-uniqueness-1},  assume now $\rho\in C_r^{\infty}(\mathbb{R}^d)$ for some $r>0$.  By (ii) of Theorem \ref{d32}, 
we have
\begin{equation}
R=\sup_{s\in[0,T],\,x\in B(r)}|\phi_s(x)|\leq r+T\|b\|_0+\sup_{s\in[0,T]}|L_s|<\infty,\quad \mathbb{P}\text{-a.s.}
\end{equation}
 In view of Theorem \ref{ww1}, Theorem \ref{thm-stability} and  Remark \ref{rem Sec 5 000}, we get
\begin{align}
&\sup _{s \in[0, T],\, x \in B(R)}\left|J \phi_{s}^{-1}(x)\right|<\infty, \quad \sup _{s \in[0, T],\, x \in B(R)}\Vert \phi_s^{-1}(x)\Vert<\infty,\quad\mathbb{P}\text{-a.s.}
\end{align}
Put  $\theta_s(x)=\rho\left(\phi_s^{-1}(x)\right) J \phi_s^{-1}(x)$. 
It follows that
\begin{align*}
% C(\omega):=
   \|\theta_{\cdot}(\omega,\cdot)\|_{L^{\infty}([0,T]\times B(R))}<\infty,\quad \text{for }\mathbb{P}\text{-a.e.} \;\omega\in\Omega.
\end{align*}
Observe that 
\begin{align*}
\int_0^t\left(\int_{\mathbb{R}^d} \mathcal{R}_{\eps}\left[ b,u_s\right](y) \theta_s(y) d y\right) d s
=\int_0^t\left(\int_{\mathbb{R}^d} \Big( \int_{\mathbb{R}^d} \vartheta_{\varepsilon}(y-z)\left[b(z)-b(y)\right] \theta_s( y) d y\Big)\cdot D u_s( z) dz\right) d s.
\end{align*}
 Recall the definition of $\vartheta_\eps$ in \eqref{def-vartheta-eps} and note that $\rho\left(\phi_s^{-1}(\cdot)\right)\in C_R(\mathbb{R}^d)$ for any $s\in[0,T]$ $\mathbb{P}$-a.s.. Then, $\mathbb{P}$-a.s.,
 \begin{align*}
    &\int_0^T\left(\int_{\mathbb{R}^d} \Big( \int_{\mathbb{R}^d} |\vartheta_{\varepsilon}(y-z)|\|b(z)-b(y)\| |\theta_s( y)| d y\Big)\| D u_s( z) \|dz\right) d s\\
    &\leq 
    \|\theta_{\cdot}(\omega,\cdot)\|_{L^{\infty}([0,T]\times B(R))}\int_0^T\left(\int_{\mathbb{R}^d} \Big( \int_{B(R)} |\vartheta_{\varepsilon}(y-z)|\|b(z)-b(y)\| d y\Big)\| D u_s( z) \|dz\right) d s\\
   & = 
    \|\theta_{\cdot}(\omega,\cdot)\|_{L^{\infty}([0,T]\times B(R))}\int_0^T\left(\int_{B(R+2\eps)} \Big( \int_{B(R)} |\vartheta_{\varepsilon}(y-z)|\|b(z)-b(y)\| d y\Big)\| D u_s( z) \|dz\right) d s\\
    &\leq 
    T\|\theta_{\cdot}(\omega,\cdot)\|_{L^{\infty}([0,T]\times B(R))}
    \sup_{s\in[0,T]}\| D u_s( z) \|_{L^\infty(B(R+2\eps))}
    {[b]_\beta}
    \int_{B(R+2\eps)} \Big( \int_{B(R)} |\vartheta_{\varepsilon}(y-z)||y-z|^\beta d y\Big)dz\\
    &= 
    T\|\theta_{\cdot}(\omega,\cdot)\|_{L^{\infty}([0,T]\times B(R))}
    \sup_{s\in[0,T]}\!\!\| D u_s( z) \|_{L^\infty(B(R+2\eps))}
    {[b]_\beta}\!\!
    \int_{B(R+2\eps)}\!\! \Big( \int_{B(R)\cap B(z,2\eps)}\!\!\! |\vartheta_{\varepsilon}(y-z)||y-z|^\beta d y\Big)dz\\
   & \leq
    (2\eps)^\beta T\|\theta_{\cdot}(\omega,\cdot)\|_{L^{\infty}([0,T]\times B(R))}
    \sup_{s\in[0,T]}\| D u_s( z) \|_{L^\infty(B(R+2\eps))}
    {[b]_\beta}
    \int_{B(R+2\eps)} 1dz\rightarrow0
 \end{align*}
 as $\eps\rightarrow0$.

Hence we obtain
\begin{align}\label{con-201}
\lim _{\varepsilon \rightarrow 0} \int_0^t\Big(\int_{\mathbb{R}^d} \mathcal{R}_{\eps}[ b,u_s](y) \rho\left(\phi_s^{-1}(y)\right) J \phi_s^{-1}(y) d y\Big) d s=0,
\end{align}
which proves $\int_{\mathbb{R}^d} u(t, \phi_{t}(x)) \rho(x) d x=0$. The proof is now complete.
\end{proof}

\section{The perturbative equation} \label{perturb}

 The next result has been proved in Appendix B of  \cite{FGP_2010-Inventiones}
 in the case of the transport equation driven by a Brownian motion. 
 The proof in our jump case is more delicate. We believe that the result is of independent interest. It shows that the definition of weak$^\ast$-$\mathrm{L}^{\infty}$-solution  can be formulated  equivalently without using the Marcus integration form. This result  is in the spirit of rough path theory (for recent developments of rough path theory for SDEs driven by  L\'evy processes see \cite{KrempPerk} and the references therein).

\begin{theorem} \label{pert}
Let $b \in L_{\mathrm{loc}}^1(\mathbb{R}^d ; \mathbb{R}^d), \operatorname{div} b \in L_{\mathrm{loc}}^1(\mathbb{R}^d)$ and $u_0 \in L^{\infty}(\mathbb{R}^d)$. A  measurable process $u \in L^{\infty}([0,T] \times\mathbb{R}^d\times\Omega)$ is a  weak$^\ast$-$\mathrm{L}^{\infty}$-solution of the SPDE \eqref{eqn-transport-Markus} (in the sense of Definition \ref{def-transport-weak-2}) if and only if, 
{there exists $M>0 $ such that 
\begin{align}
    \sup_{t\in[0,T]}\| u(t, \cdot,\omega) \|_{L^\infty(\mathbb{R}^d)} \leq M,\quad\mathbb{P}\text{-a.s.},
\end{align}
}
 for every $\theta \in C_c^{\infty}(\mathbb{R}^d)$, $u_t(\theta):=\int_{\mathbb{R}^d} \theta(x) u(t, x) d x$ is $\mathbb{F}$-adapted and it has  c\`adl\`ag paths $\mathbb{P}$-a.s.  and the perturbative equation
\begin{align}
u_t(\theta)= &\int_{\mathbb{R}^d} \theta(x+L_t) u_0( x) d x\nonumber \\
& +\int_0^t  \int_{\mathbb{R}^d}\Big[b(x) \cdot D \theta(x+L_t-L_r)+\operatorname{div} b(x) \theta(x+L_t-L_r)\Big] u(r,x) d x\, d r\label{per-equ-app}
\end{align}
holds almost surely in $\omega \in \Omega$, for every $t \in[0, T]$.
\end{theorem}

\begin{remark} \label{v33} Arguing as in Lemma \ref{lem-transport-weak-2-(i)}  one can show that   there exists a $\mathbb{P}$-full set $\Omega^\prime \subset \Omega$ such that for every $\omega \in \Omega^\prime$,
 for every 
 $\theta\in L^1(\mathbb{R}^d)$, 
 {the solution of the perturbative equation \eqref{per-equ-app} verifies:}
 \begin{equation}
\label{eqn-cadlag-001-1}
u_t(\theta)(\omega): [0,T]  \ni t \mapsto  \int_{\mathbb{R}^d}\theta(x)u(t,x,\omega)\,dx 
\end{equation}
is c\`{a}dl\`{a}g and the corresponding process is adapted.

 It follows, using also the dominated convergence theorem, that for every $\omega \in \Omega^\prime$,
 for any  $\theta\in C_c^{\infty}(\mathbb{R}^d)$, the mapping:
$$
 (r,y) \mapsto
 \int_{\mathbb{R}^d} \theta(x+ y) u_0( x) d x + \int_{\mathbb{R}^d}\Big[b(x) \cdot D \theta(x+y)+\operatorname{div} b(x) \theta(x+y)\Big] u(r,x) d x
$$
is c\`{a}dl\`{a}g in $r \in [0,T]$ for any fixed $y \in \mathbb{R}^d$ and continuous in $y \in \mathbb{R}^d$ for any fixed $r \in [0,T]$. Hence it is easy to check that formula \eqref{per-equ-app} is meaningful. 
 \end{remark}

\begin{proof}

Definition $\eqref{per-equ-app}\Longrightarrow  $ Definition \ref{def-transport-weak-2}.

Let us fix $\theta \in  C_c^{\infty}(\mathbb{R}^d).$ Taking into account Remark \ref{v33}, 
 define the function $F:\Omega\times [0,T]\times \rE\rightarrow \bR^{},$
      \begin{align} \label{opq}
F(t,y)= & \int_{\mathbb{R}^d} \theta(x+y) u_0( x) d x \\
 &+ \int_0^t \int_{\mathbb{R}^d}\Big[b(x) \cdot D \theta(x+y-L_r)+\operatorname{div} b(x) \theta(x+y-L_r)\Big] u(r,x) d x\,  d r.
\end{align}
Clearly, $\mathbb{P}$-a.s., for any $t \in [0,T]$,  $F(t,y)$ is smooth in $y$ and, for any $y \in \R^d,$  $F(\cdot,y)$ has no jumps.

Assume that $u_t(\theta)$ satisfies the perturbation equation \eqref{per-equ-app}.
Then, $\mathbb{P}\text{-a.s},$ for every $t\in [0,T]$,
\begin{gather}\label{eq 20250801 01}
   u_t(\theta)=   \int_{\mathbb{R}^d} \theta(x) u(t, x) d x=F(t,L_t).
      \end{gather}
Moreover,  for any $t \in (0,T]$,  using also the dominated convergence theorem
\begin{gather} \label{sff}
 F(t,L_{t-})= \lim_{s \to t-}F(t,L_{s}) = \lim_{s \to t-}\int_{\mathbb{R}^d} \theta(x) u(s, x) d x = \int_{\mathbb{R}^d} \theta(x) u(t-, x) d x.
\end{gather}
In order to clarify the term $\int_{\mathbb{R}^d} \theta(x) u(t-, x) d x$, see the discussion around formula \eqref{gt33} in Remark \ref{d88}.

To proceed, we need to apply  the It\^o-Wentzell formula, see  Theorem \ref{thm-Ito+Wentzel-LM}) in a  special case since
\begin{gather*}
F(t,y)=  F_{0}(y)+\int_0^tf_r(y)dr,
\end{gather*}
where $F_0(y) = \int_{\mathbb{R}^d} \theta(x+y) u_0( x) d x$ and $f_r(y) = \int_{\mathbb{R}^d}\Big[b(x) \cdot D \theta(x+y-L_r)+\operatorname{div} b(x) \theta(x+y-L_r)\Big] u(r,x) d x$.

Let us check the assumptions of the It\^o-Wentzell formula.
First it is straightforward to check that $f:\Omega\times [0,T] \times \rE\rightarrow \bR^{}$ and $F$  are  both
 $\mathscr{O}\otimes \mathcal{B}(\rE)$-measurable.
 Moreover, it is clear that $\mathbb{P}$-a.s.\ for all $y\in \rE$,
$
\int_0^T|f_r(y)|dr < \infty.
$

Hence the assumptions of Hypothesis \ref{ftx} hold. Let us check the other assumptions in Theorem \ref{thm-Ito+Wentzel-LM}.
We deduce easily  that
 $\mathbb{P}$-a.s.\  $F\in C([0,T];{C}^{2}(\bR^d;\bR))$.

We also have by the dominated convergence theorem,
 for $\mathbb{P}\otimes \Leb$-almost-all $(\omega,r)\in \Omega\times [0,T]$,  the function
\[
\bR^d \ni y \mapsto   f_r(y) \in \bR
\]
is continuous. It remains to prove that
 for every $T>0$ and  each compact subset $K$ of $\rE$, $\mathbb{P}$-a.s.,
\begin{equation}\label{ma1}
\int_0^T\sup_{y\in K} |f_r(y)| dr<\infty.
\end{equation}
Recall that 
\begin{align}
    \sup_{r\in[0,T]}\| u(r, \omega, \cdot) \|_{L^\infty(\mathbb{R}^d)} \leq M,\quad\mathbb{P}\text{-a.s.}
\end{align}
Let us fix $K$ and consider $Supp(\theta)$. For each $\omega $, $\mathbb{P}$-a.s,
we know that the image of $r\in[0,T] \mapsto L_r(\omega)$ belongs to a compact set $N^T_{\omega}$ of $\R^d$.

Therefore, defining the compact set $H^T_{\omega} = Supp(\theta) - K  + N^T_{\omega}$ we have, for any $y\in K$,
\begin{align*}
&\int_{\mathbb{R}^d}\Big[b(x) \cdot D \theta(x+y-L_r)+\operatorname{div} b(x) \theta(x+y-L_r)\Big] u(r,x) d x
\\ &=  \int_{H^T_{\omega}}\Big[b(x) \cdot D \theta(x+y-L_r)+\operatorname{div} b(x) \theta(x+y-L_r)\Big] u(r,x) d x.
\end{align*}
We find,  $\mathbb{P}$-a.s.,
\begin{gather*}
\int_0^T\sup_{y\in {K}} |f_r(y)|dr
\leq
 T M
\|\theta\|_1 \int_{H^T_{\omega}}(|b(x)|+|\divv b(x)|)dx < \infty.
 \end{gather*}
In view of the It\^o-Wentzell formula (Theorem \ref{thm-Ito+Wentzel-LM}) and $\mathbb{P}$-a.s.\  $F\in C([0,T];{C}^{2}(\bR^d;\bR))$,  we infer,  $\mathbb{P}$-a.s., for any $t\geq0$, 
\begin{align*}
 F(t,L_t)
=& \int_{\mathbb{R}^d} \theta(x) u_0( x) d x+\int_0^t \int_{\mathbb{R}^d}\Big[b(x ) \cdot D \theta(x+L_r-L_r)+\operatorname{div} b(x) \theta(x+L_r-L_r)\Big] u(r,x) d x\,dr
\\ &+\int_0^t DF(s,L_{s-})dL_s + \sum_{s \le t} \Big[  F(s,L_{s})-F(s,L_{s-}) -\sum_{i=1}^d D_i F(s, L_{s-}) \Delta L^i_s       \Big].
\end{align*}
Hence
\begin{align}\label{eq 2025-11-06-06}
 F(t,L_t)=& \int_{\mathbb{R}^d}\!\!\!\! \theta(x) u_0( x) d x+\int_0^t\!\!\!\! \int_{\mathbb{R}^d}\!\!\!\!\Big[b(x ) \cdot D \theta(x)+\operatorname{div} b(x) \theta(x)\Big] u(r,x) d x\,dr
\\
 &+\int_0^t \lint_{B} DF(s,L_{s-})\cdot z \, \tilde  \mu (ds, dz) +\int_0^t \int_{B^c} DF(s,L_{s-})\cdot z  \, \mu (ds, dz)
 \\
 &+ \int_0^t \lint_{B} \Big[  F(s,L_{s-}+z)-F(s,L_{s-}) -  D F(s, L_{s-})\cdot z       \Big] \mu(ds,dz)
 \\ &+
  \int_0^t \int_{B^c} \Big[  F(s,L_{s-}+z)-F(s,L_{s-}) -  D F(s, L_{s-})\cdot z       \Big] \mu(ds,dz).
\end{align}

Now, we will prove that 
\begin{align}\label{eq 2025-11-06-01}
\int_0^t \lint_{B} \Big[  F(s,L_{s-}+z)-F(s,L_{s-}) -  D F(s, L_{s-})\cdot z       \Big] \tilde{\mu}(ds,dz)
\end{align}
and 
\begin{align}\label{eq 2025-11-06-02}
\int_0^t \lint_{B} \Big[  F(s,L_{s}+z)-F(s,L_{s}) -  D F(s, L_{s})\cdot z       \Big]\nu(dz) ds.
\end{align}
are well-defined.

Observe that, differentiating \eqref{opq} and similar to \eqref{eq 20250801 01}, we get
\begin{align}\label{eq 2025-11-06-03}
 D_i F(s, L_s)&=D_i F(s, y)|_{y=L_s}=    \int_{\mathbb{R}^d} D_i\theta(x) u(s, x) d x.
 \end{align}
Moreover (cf. \eqref{sff}) $\mathbb{P}$-a.s.\ for all $z\in \rE$ and $s\geq0$,
 \begin{align}\label{eq 2025-11-06-04}
  F(s,L_{s-}+z)-F(s,L_{s-})&=\int_{\mathbb{R}^d} \theta(x+z) u(s-, x) d x-\int_{\mathbb{R}^d} \theta(x) u(s-, x) d x.
 \end{align}

 {Applying the Taylor expansion formula yields for some $\rho\in(0,1)$, 
\begin{align*}
 \big| F(s,L_{s-}+z)-F(s,L_{s-})-  D F(s, L_{s-})\cdot z \big|& \leq \Big| \sum_{i=1}^d D_{i}F(s,L_{s-}+\rho z)z_i\Big| +\Big| \sum_{i=1}^d D_{i}F(s,L_{s-})z_i\Big|, \\
 \big| F(s,L_{s-}+z)-F(s,L_{s-}) -  D F(s, L_{s-})\cdot z\big|& \leq \frac12 \Big| \sum_{i,j=1}^d D_{i,j}F(s,L_{s-}+\rho z)z_iz_j \Big|,
\end{align*}
where 
\begin{align*}
  D_{i}F(s,y)=& \int_{\mathbb{R}^d} D_{i} \theta(x + y) u_0(x) \, dx \\
  &+ \int_0^s \int_{\mathbb{R}^d} \Big[ \sum_{k=1}^d b_k(x) D_{ik} \theta(x + y - L_r) + \operatorname{div} b(x) D_{i} \theta(x + y - L_r) \Big] u(r, x) \, dx \, dr,\\
  D_{i,j}F(s,y)=& \int_{\mathbb{R}^d} D_{ij} \theta(x + y) u_0(x) \, dx \\
  &+ \int_0^s \int_{\mathbb{R}^d} \Big[ \sum_{k=1}^d b_k(x) D_{ijk} \theta(x + y - L_r) + \operatorname{div} b(x) D_{ij} \theta(x + y - L_r) \Big] u(r, x) \, dx \, dr.
\end{align*}
Notice that the image of $[0,T]\ni s \mapsto L_{s-}(\omega)\in \mathbb{R}^d$ belongs to a compact set $K^T_{\omega}$ of $\R^d$. Define a compact set $J_{\omega}=\supp(\theta)-K^T_{\omega}-B+N^T_{\omega}$. 
It follows that, $\mathbb{P}$-a.s., for every $s\in [0,T]$, $z\in B$,
\begin{align*}
 \Big| D_{i}F(s,L_{s-}+\rho z) \Big|
\leq M\|\theta\|_{1}\Leb(\supp(\theta))+
  MT
\|\theta\|_2 \int_{J_{\omega}}(|b(x)|+|\divv b(x)|)dx < \infty,\\
 \Big| D_{i,j}F(s,L_{s-}+\rho z) \Big|
\leq M\|\theta\|_{2}\Leb(\supp(\theta))+
  MT
\|\theta\|_3 \int_{J_{\omega}}(|b(x)|+|\divv b(x)|)dx < \infty.
\end{align*}
Hence for any $z\in B$ and $s\in[0,T]$,
\begin{align*}
 &\big| F(s,L_{s-}+z)-F(s,L_{s-}) -  D F(s, L_{s-})\cdot z\big| \leq C_1|z|^2, \\
 & \big| F(s,L_{s-}+z)-F(s,L_{s-})  -  D F(s, L_{s-})\cdot z \big| \leq C_2 |z|. 
\end{align*}
where random variables $C_1,C_2<\infty$ $\mathbb{P}$-a.s. depend on  $M$ and $\|\theta\|_3$.  
Then we have, $\mathbb{P}$-a.s.,
\begin{align}
 &\int_0^T\!\lint_B \big| F(s,L_{s-}+z)-F(s,L_{s-}) -  D F(s, L_{s-})\cdot z\big| \nu(dz)ds  \leq C_1 T\; \lint_B |z|^2 \nu(dz)<\infty,\label{eq-2015-701} \\
 & \int_0^T\lint_B \big| F(s,L_{s-}+z)-F(s,L_{s-})  -  D F(s, L_{s-})\cdot z \big|^2  \nu(dz)ds\label{eq-2015-702}\\
 &=\int_0^T\lint_B \big| F(s,L_{s-}+z)-F(s,L_{s-})  -  D F(s, L_{s-})\cdot z \big|^2  \nu(dz)ds  \leq C_2^2T\, \lint_B |z|^2 \nu(dz)<\infty.\nonumber 
\end{align}
Hence, \eqref{eq 2025-11-06-01} and \eqref{eq 2025-11-06-02} are well-defined.
Therefore,
\begin{align}\label{eq 2025-11-06-00}
&\int_0^t \lint_{B} \Big[  F(s,L_{s-}+z)-F(s,L_{s-}) -  D F(s, L_{s-})\cdot z       \Big] \mu(ds,dz) \\
&= \int_0^t \lint_{B} \Big[  F(s,L_{s-}+z)-F(s,L_{s-}) -  D F(s, L_{s-})\cdot z       \Big] \tilde{\mu}(ds,dz)\\
&\quad+\int_0^t \lint_{B} \Big[  F(s,L_{s}+z)-F(s,L_{s}) -  D F(s, L_{s})\cdot z       \Big]\nu(dz) ds.
\end{align}
After inserting \eqref{eq 2025-11-06-00} to \eqref{eq 2025-11-06-06} and performing some cancellations, we get
\begin{align*}
 F(t,L_t)= & \int_{\mathbb{R}^d} \theta(x) u_0( x) d x+\int_0^t dr\int_{\mathbb{R}^d}\Big[b(x ) \cdot D \theta(x)+\operatorname{div} b(x) \theta(x)\Big] u(r,x) d x
\\
 &+\int_0^t \lint_{B}  \Big[  F(s,L_{s-}+z)-F(s,L_{s-})       \Big]  \, \tilde\mu (ds, dz)  \\
 &+
  \int_0^t \int_{B^c} \Big[  F(s,L_{s-}+z)-F(s,L_{s-})       \Big] \mu(ds,dz)\\
  &+\int_0^t \lint_{B} \Big[  F(s,L_{s}+z)-F(s,L_{s}) -  D F(s, L_{s})\cdot z       \Big]\nu(dz) ds.
\end{align*}
Combining \eqref{eq 20250801 01}, \eqref{eq 2025-11-06-03} and \eqref{eq 2025-11-06-04}, the above equality  yields that $u$ satisfies \eqref{eqn-transport-Marcus} in the Definition \ref{def-transport-weak-2}.
}

Definition \ref{def-transport-weak-2} $ \Longrightarrow  $ Definition $\eqref{per-equ-app} $

 On the other hand, assume that $u$ is a   weak$^\ast$-$\mathrm{L}^{\infty}$ solution of the SPDE \eqref{eqn-transport-Markus} in the sense of Definition \ref{def-transport-weak-2}. Given $\theta \in C_c^{\infty}(\mathbb{R}^d)$ and $y\in\mathbb{R}^d$, consider the test function $\theta_y(x)=\theta(x+y)$.
 Then $u$ satisfies
 \begin{align}
      &\int_{\mathbb{R}^d}u(t,x)\,\theta(x+y)\,dx\\
    &=
    \int_{\mathbb{R}^d}u_{0}(x)\,\theta(x+y)\,dx
       +\int_0^t\int_{\mathbb{R}^d}u(s,x)\,[b(x)\cdot D\theta(x+y)+\divv\, b(x)\,\theta(x+y)]\,dx\,ds\\
       &\quad +\int_0^t\lint_{B} \Bigl(\int_{\mathbb{R}^d}u(s-,x)\,[\theta(x+z+y)-\theta(x+y)]\, dx\Bigr)\tilde{\prm}(ds,dz)\\
       &\quad +\int_0^t\int_{B^{\mathrm{c}}} \Bigl(\int_{\mathbb{R}^d}u(s-,x)\,[\theta(x+z+y)-\theta(x+y)]\, dx\Bigr) \prm(ds,dz)\\
       &\quad +\int_0^t\lint_{B}\Bigl(\int_{\mathbb{R}^d} u(s,x)\,[\theta(x+z+y)-\theta(x+y)-\sum_{i=1}^dz_iD_i\theta(x+y)]\,dx\Bigr)\,\nu(dz)\, ds.\end{align}

Next consider the random field $F(t,y):=\int_{\mathbb{R}^d}u(t,x)\,\theta(x+y)\,dx $.  Note that
$$
D F(t,y)=\int_{\mathbb{R}^d}u(t,x)\,D\theta(x+y)\,dx .
$$
Applying similar arguments as in the proof of Lemma \ref{lem-u^eps satisfies assumptions ITWF}, we can verify that $F(t,y)$ satisfies the  assumptions of  It\^o-Wentzell Theorem \ref{thm-Ito+Wentzel-LM}.

Let $t\in(0,T]$. Now we apply It\^o-Wentzell Theorem \ref{thm-Ito+Wentzel-LM} to the function 
$$F(t,y)=F_0(y)+\int_0^t f_s(y)ds+\int_0^t\lint_B h_s(y,z)\tilde\mu(ds,dz)+\int_0^t\int_{B^c} h_s(y,z) \mu(ds,dz)$$ 
and $\xi_t=y-L_t$, $t\geq 0$, 
where 
\begin{align*}
F_0(y)=& \int_{\mathbb{R}^d}u_{0}(x)\,\theta(x+y)\,dx,\\
 f_s(y)=& \int_{\mathbb{R}^d}u(s,x)\,[b(x)\cdot  D\theta(x+y)+\divv\, b(x)\,\theta(x+y)]\,dx\\
 &+\lint_{B}\Bigl(\int_{\mathbb{R}^d} u(s,x)\,[\theta(x+z+y)-\theta(x+y)-\sum_{i=1}^dz_iD_i\theta(x+y)]\,dx\Bigr)\,\nu(dz),\\
 h_s(y,z)=&\int_{\mathbb{R}^d}u(s-,x)\,[\theta(x+z+y)-\theta(x+y)]\, dx.
\end{align*}
We obtain,  $\mathbb{P}$-a.s. for any $t \in [0,T]$, 
\begin{align}
F(t,\oldL_t)&=F_{0}(\oldL_{0})+\int_0^tf_s(\oldL_{s})ds  +\int_0^t\int_{\mathbb{R}^d}h_s(\oldL_{s-},z)[\1_{B\setminus\{0\}}(z)\newq(ds,dz) +\1_{B^c}(z)\newp(ds,dz)]\\
&\quad  +{\int_0^t D F(s-,\oldL_{s-})
d {\xi_s}}
+\sum_{s\le t}\Big(F(s-,\oldL_{s})-F(s-,\oldL_{s-})- D F(s-,\oldL_{s-})\cdot\Delta \oldL_s\Big)\nonumber\\
&\quad +\sum_{s\leq t}\Big(\Delta F(s,\oldL_{s})-\Delta F(s,\oldL_{s-})\Big).\label{IW-per-eq-703}
\end{align}
First observe that
\begin{align*}
&\int_0^tf_s(\oldL_{s})ds\\
&=
 \int_0^t \int_{\mathbb{R}^d}u(s,x)\,[b(x)\cdot D\theta(x+y-L_s)+\divv\, b(x)\,\theta(x+y-L_s)]\,dx\,ds \\
 &\quad+\int_0^t \lint_{B}\Bigl(\int_{\mathbb{R}^d} u(s,x)\,[\theta(x+z+y-L_s)-\theta(x+y-L_s)-\sum_{i=1}^dz_iD_i\theta(x+y-L_s)]\,dx\Bigr)\,\nu(dz)ds\\
&=
 \int_0^t \int_{\mathbb{R}^d}u(s,x)\,[b(x)\cdot D\theta(x+y-L_s)+\divv\, b(x)\,\theta(x+y-L_s)]\,dx\,ds \\
 &\quad+\int_0^t \lint_{B}\Bigl(\int_{\mathbb{R}^d} u(s-,x)\,[\theta(x+z+y-L_{s-})-\theta(x+y-L_{s-})-\sum_{i=1}^dz_iD_i\theta(x+y-L_{s-})]\,dx\Bigr)\,\nu(dz)ds\\
&=  \int_0^t \int_{\mathbb{R}^d}u(s,x)\,[b(x)\cdot D\theta(x+y-L_s)+\divv\, b(x)\,\theta(x+y-L_s)]\,dx\,ds \\
 &\quad-\int_0^t \lint_{B}\Bigl(\int_{\mathbb{R}^d} u(s-,x)\,[\theta(x+z+y-L_{s-})-\theta(x+y-L_{s-})-\sum_{i=1}^dz_iD_i\theta(x+y-L_{s-})]\,dx\Bigr)\,\tilde{\mu}(ds,dz)\\
  &\quad+\int_0^t \lint_{B}\Bigl(\int_{\mathbb{R}^d} u(s-,x)\,[\theta(x+z+y-L_{s-})-\theta(x+y-L_{s-})-\sum_{i=1}^dz_iD_i\theta(x+y-L_{s-})]\,dx\Bigr)\,\mu(ds,dz),
    \end{align*}  
 for the second equality we have used the relation: for any $z\in B$ and almost all $s\in(0,T]$,
 \begin{align}
 &\int_{\mathbb{R}^d}u(s,x)\,[\theta(x+z+y-L_{s})-\theta(x+y-L_{s})-\sum_{i=1}^dz_iD_i\theta(x+y-L_{s})]dx\\
 &=
     \int_{\mathbb{R}^d}u(s,x)\,[\theta(x+z+y-L_{s-})-\theta(x+y-L_{s-})-\sum_{i=1}^dz_iD_i\theta(x+y-L_{s-})]dx\\
       &=
     \int_{\mathbb{R}^d}u(s-,x)\,[\theta(x+z+y-L_{s-})-\theta(x+y-L_{s-})-\sum_{i=1}^dz_iD_i\theta(x+y-L_{s-})]dx,
 \end{align}
  and the stochastic integrals with respect to the compensated Poisson random measure are well defined because, analogous to equations \eqref{eq-2015-701} and \eqref{eq-2015-702}, we have $\mathbb{P}$-a.s.
 \begin{align*}
&\int_0^T \lint_{B} \Big| \int_{\mathbb{R}^d}u(s-,x)\,[\theta(x+z+y-L_{s-})-\theta(x+y-L_{s-})-\sum_{i=1}^dz_iD_i\theta(x+y-L_{s-})]\,dx\Big|\nu(dz)ds\\
& \leq M\|\theta\|_2 \Leb(supp(\theta)) T\lint_{B} |z|^2\nu(dz)<\infty,\\
&\int_0^T \lint_{B} \Big| \int_{\mathbb{R}^d} u(s-,x)\,[\theta(x+z+y-L_{s-})-\theta(x+y-L_{s-})-\sum_{i=1}^dz_iD_i\theta(x+y-L_{s-})]\,dx\Big|^2\nu(dz)ds\\
& \leq 4 M^2\|\theta\|^2_1 \Big(\Leb(supp(\theta))\Big)^2  T\lint_{B} |z|^2\nu(dz)<\infty.
     \end{align*}  
 For the third term on the right-hand side of \eqref{IW-per-eq-703}, by substituting $h_s$ and $\xi_s$ we find 
\begin{align*}
&\int_0^t\int_{\mathbb{R}^d}h_s(\oldL_{s-},z)[\1_{B\setminus\{0\}}(z)\newq(ds,dz) +\1_{B^c}(z)\newp(ds,dz)]\\
 &=\int_0^t \lint_B\Bigl(\int_{\mathbb{R}^d}u(s-,x)\,[\theta(x+z+y-L_{s-})-\theta(x+y-L_{s-})]\, dx \Bigr)\,\tilde{\mu}(ds,dz)\\
   & \quad+\int_0^t \int_{B^c} \Bigl(\int_{\mathbb{R}^d}u(s-,x)\,[\theta(x+z+y-L_{s-})-\theta(x+y-L_{s-})]\, dx\Bigr)\, \mu(ds,z).  
   \end{align*}   
  The fourth term  can be rewritten as
\begin{align*}
  {\int_0^t D F(s-,\oldL_{s-})
d {\xi_s}}=&-\int_0^t \lint_B \int_{\mathbb{R}^d}u(s-,x)\,D\theta(x+y-L_{s-})\cdot z\,dx \tilde{\mu}(ds,dz)\\
   &- \int_0^t \int_{B^c} \int_{\mathbb{R}^d}u(s-,x)\,D\theta(x+y-L_{s-})\cdot z\,dx \mu(ds,dz).
   \end{align*}
 Moreover, similar considerations as in the proof of Lemma \ref{lem-u^eps satisfies assumptions ITWF},  we can deduce
 \begin{align*}
& \sum_{s\le t}\Big(F(s-,\oldL_{s})-F(s-,\oldL_{s-})- D F(s-,\oldL_{s-})\Delta \oldL_s\Big)\\
 &= \int_0^t\lint_B \Big(F(s-,\oldL_{s-}-z)-F(s-,\oldL_{s-})- D F(s-,\oldL_{s-})\cdot (-z) \mu(ds,dz)\\
&\quad+ \int_0^t\int_{B^c} \Big(F(s-,\oldL_{s-}-z)-F(s-,\oldL_{s-})- D F(s-,\oldL_{s-})\cdot (-z) \mu(ds,dz)\\
 &= \int_0^t\lint_B \Bigl( \int_{\mathbb{R}^d}u(s-,x)\,\big(\theta(x-z+y-L_{s-})-\theta(x+y-L_{s-})+D\theta(x+y-L_{s-})\cdot z\big)dx\Bigr) \mu(ds,dz)\\
 &\quad+ \int_0^t\int_{B^c} \Bigl( \int_{\mathbb{R}^d}u(s-,x)\,\big(\theta(x-z+y-L_{s-})-\theta(x+y-L_{s-})+D\theta(x+y-L_{s-})\cdot z\big)dx\Bigr) \mu(ds,dz).
     \end{align*} 
     Furthermore, we have, $\mathbb{P}$-a.s. for any $t \in [0,T]$, 
     \begin{align*} 
     &\sum_{s\leq t}\Big(\Delta F(s,\oldL_{s})-\Delta F(s,\oldL_{s-})\Big)\\
     &=\sum_{s\leq t}
 \1_{\{ 0 < | \Delta L_s|\le 1 \}}\,
[\newh_s (\xi_{s-} - \Delta L_s, \Delta L_s)-\newh_s ( \xi_{s-}, \Delta L_s )]    \\
&\quad+\sum_{s\leq t} \1_{\{ | \Delta L_s|> 1 \}}\,
[\newh_s (\xi_{s-} - \Delta L_s, \Delta L_s)-\newh_s ( \xi_{s-}, \Delta L_s )]\\
    & =
    \int_0^t\lint_{B}[\newh_s( \xi_{s-} - z,z)-\newh_s(\xi_{s-} ,z )]\,\mu(ds,dz)+ \int_0^t\int_{B^c} [\newh_s(\xi_{s-} - z,z) -
\newh_s( \xi_{s-} ,z)]
    \mu(ds,dz)\\
        & =
    \int_0^t\lint_{B} \Big[ \int_{\mathbb{R}^d}u(s-,x)\,[\theta(x+\xi_{s-})-\theta(x+\xi_{s-}-z)-\theta(x+z+\xi_{s-})+\theta(x+\xi_{s-})]\, dx \Big]\,\mu(ds,dz)\\
  &\quad+    \int_0^t\int_{B^c} \Big[ \int_{\mathbb{R}^d}u(s-,x)\,[\theta(x+\xi_{s-})-\theta(x+\xi_{s-}-z)-\theta(x+z+\xi_{s-})+\theta(x+\xi_{s-})]\, dx \Big]\,\mu(ds,dz)\\
          & =
    \int_0^t\lint_{B} \Big[ \int_{\mathbb{R}^d}u(s-,x)\,[\theta(x+y-L_{s-})-\theta(x+y-L_{s-}-z)\\
    &\quad\quad\quad\quad \quad\quad\quad\quad\quad\quad\quad\quad\quad-\theta(x+z+y-L_{s-})+\theta(x+y-L_{s-})]\, dx \Big]\,\mu(ds,dz)\\
  &\quad+    \int_0^t\int_{B^c}  \Big[ \int_{\mathbb{R}^d}u(s-,x)\,[\theta(x+y-L_{s-})-\theta(x+y-L_{s-}-z)\\
      &\quad\quad\quad \quad\quad\quad\quad\quad\quad\quad\quad\quad\quad-\theta(x+z+y-L_{s-})+\theta(x+y-L_{s-})]\, dx \Big]\,\mu(ds,dz).
          \end{align*} 
         Substituting the above calculations into \eqref{IW-per-eq-703} and simplifying, we deduce that
  \begin{align*}
 F(t,y-L_t)
 =&\int_{\mathbb{R}^d}u_0(x)\,\theta(x+y)\,dx\\
 &+\int_0^tds \int_{\mathbb{R}^d}u(s,x)\,[b(x)\cdot D\theta(x+y-L_s)+\divv\, b(x)\,\theta(x+y-L_s)]\,dx.
  \end{align*}
Since both sides are smooth functions of $y$ $\mathbb{P}$-a.s. for any $t \in [0,T]$,   setting $y=L_t$ and substituting the definition of $F$ yields the perturbation equation \eqref{per-equ-app}. The proof has now been completed.

\end{proof}

\section{Proofs of some  lemmata from Section \ref{sec-existence}} \label{sec-proofs of 3 lemmata}

\begin{proof}[Proof of Lemma \ref{eqn-A2}] By possibly employing the Heine diagonal procedure we  can  fix $T\in\mathbb{N}$ and consider $t \in [0,T]$.
By Corollary \ref{cor-homeomorphism} and Theorem  \ref{thm-homeomorphism} applied to vector fields $b$ and $b^{\eps_n}$,  there exists an almost sure event $\Omega^{\prime\prime}$ such that
for every  $\omega \in \Omega^{\prime\prime}$  there exists $C_T(\omega)>0$ such that  for every $n\in \Nb:=\mathbb{N}\cup \{\infty\} $,  the following inequality holds
\begin{equation} \label{ineq-inverse-2}
\vert (\phi^{\eps_n}_s)^{-1}y -  y \vert \leq  s \Vert  b^{\eps_n} \Vert_{0}    + C_T(\omega), \;\; y \in \R^d, \; 0\le s  \le T.
\end{equation}
 Recall that $K:= \supp \theta   $. By the above, for very  $\omega \in \Omega^{\prime\prime}$  there exists a compact set $K_0(\omega) \subset \R^d$ such that
\[
\bigcup_{n\in \Nb, s\in [0,T]} (\phi^{\eps_n}_s)^{-1}(K) \subset K_0(\omega).
\]
Let us now fix $\omega \in \Omega^{\prime\prime}$. Then we have
\begin{align}\label{eq Sec 5 Lema 5.7 1}
&\vert \int_{\R^d} u_0(y)\,\theta (\phi^{\eps_n}_t(y))J_t^{\eps_n}(y)\,dy
    - \int_{\R^d} u_0(y)\,\theta (\phi_t(y))J_t(y)\,dy \vert\nonumber
    \\
     &\leq
     \int_{\R^d} \vert u_0(y) \vert \, \vert  \theta (\phi^{\eps_n}_t(y))J_t^{\eps_n}(y)
    - \,\theta (\phi_t(y))J_t(y) \vert \,dy\nonumber
    \\
    &=
 \int_{K_0(\omega)} \vert u_0(y) \vert \, \vert  \theta (\phi^{\eps_n}_t(y))J_t^{\eps_n}(y)
    - \,\theta (\phi_t(y))J_t(y) \vert \,dy\nonumber
\\
   & \leq \Vert u_0 \Vert_{L^\infty(\R^d)}
 \int_{K_0(\omega)}  \vert  \theta (\phi^{\eps_n}_t(y))J_t^{\eps_n}(y)
    - \,\theta (\phi_t(y))J_t(y) \vert \,dy.
          \end{align}
Moreover,
\begin{align}\label{eq Sec 5 Lema 5.7 2}
& \int_{K_0(\omega)}  \vert  \theta (\phi^{\eps_n}_t(y))J_t^{\eps_n}(y)
    - \theta (\phi_t(y))J_t(y) \vert \,dy\nonumber\\
&\leq  \int_{K_0(\omega)}  \vert  \theta (\phi^{\eps_n}_t(y))\vert\, \vert  J_t^{\eps_n}(y)
    - J_t(y) \vert \,dy
 +
\int_{K_0(\omega)}  \vert  \theta (\phi^{\eps_n}_t(y))
    - \,\theta (\phi_t(y))\vert  |J_t(y)| \,dy\nonumber\\
    &\leq
\Vert  \theta \Vert_{0} \int_{K_0(\omega)}   \vert  J_t^{\eps_n}(y)
    - J_t(y) \vert \,dy
 +
\sup_{y \in K_0(\omega)} |J_t(y)|
\int_{K_0(\omega)}  \vert  \theta( \phi^{\eps_n}_t(y))
    - \theta(  \phi_t(y))\vert   \,dy.
            \end{align}
Now applying Corollary \ref{cor-civuole}, there exists a subsequence of the sequence $\phi^{\eps_n}$,   still denoted by $\phi^{\eps_n}$, and a  $\mathbb{P}$-full set $\Omega^{1} \subset \Omega$    such that for all $\omega \in \Omega^1$, $p\geq 1$ and $M>0$,
\begin{equation}
\label{eq Sec 5 Lema 5.7 4}
\lim_{n\to\infty} \int_{B_M}\sup_{t \in [0,T] } \, \Bigl[ |\phi_{t}^{\eps_n}(x)-\phi_{t}(x)\vert^p +  \Vert D\phi_{t}^{\eps_n}(x)-D\phi_{t}(x) \Vert^p \Bigr]dx=0,
\end{equation}
here $B_M=\{y \in \R^d: \vert y|\leq M\}$.
By Theorem \ref{thm-ww}, there exists $\mathbb{P}$-full set $\Omega^2$ such that, for any $\omega\in\Omega^2$, and $M>0$,
\begin{align}\label{eq Sec 5 Lema 5.7 6}
\sup_{t\in[0,T]}\sup_{x\in\mathbb{R}^d,|x|\leq M}\Vert D\phi_t(x)\Vert<\infty.
\end{align}
Since the map $ \det: \R^{d\times d}\to \R$ is continuous, for any $\omega\in\Omega^2$, and $M>0$, we have
\begin{align}\label{eq Sec 5 Lema 5.7 7}
\sup_{t\in[0,T]}\sup_{x\in\mathbb{R}^d,|x|\leq M}|J_t(x)|<\infty.
\end{align}
Let us take  now an arbitrary $\omega \in  \Omega^{(iv)}:=\Omega^{\prime\prime} \cap\Omega^1\cap\Omega^2$, which is a $\mathbb{P}$-full set. Note that  $\theta\in C_c^{\infty}(\mathbb{R}^d)$  and there exists $M(\omega)<\infty$ such that $K_0(\omega)\subset B_{M(\omega)}$, we infer by \eqref{eq Sec 5 Lema 5.7 4}  and \eqref{eq Sec 5 Lema 5.7 7} that $$\sup_{t\in[0,T]}\sup_{y \in K_0(\omega)} |J_t(y)|<\infty,$$
\begin{align}
\label{eq Lemma 5.7 jia 1}
\hskip-0.6cm\!\lim_{n\to \infty} \!\sup_{t\in[0,T]}\!\int_{K_0(\omega)}  \!\!\!\!\!\!\!\!\vert  \theta( \phi^{\eps_n}_t(y))
    - \theta(  \phi_t(y))\vert   \,dy\!\leq\!\lim_{n\to \infty}\!\int_{B_{M(\omega)}} \!\sup_{t\in[0,T]}\vert  \theta( \phi^{\eps_n}_t(y))
    - \theta(  \phi_t(y))\vert   \,dy= 0,
\end{align}
and
\begin{align}\label{eq App E 20251008}
    \lim_{n\to \infty} \sup_{t\in[0,T]}\int_{K_0(\omega)}  \vert  J^{\eps_n}_t(y)
    - J_t(y)\vert   \,dy
\leq
\lim_{n\to \infty}\int_{B_{M(\omega)}}  \sup_{t\in[0,T]}\vert  J^{\eps_n}_t(y)
    - J_t(y)\vert   \,dy= 0.
\end{align}

To obtain the last inequality, we have used the fact again that the map $ \det: \R^{d\times d}\to \R$ is continuous, and the fact that there exists a constant $C_d$ such that for any $A,B\in \R^{d\times d}$, $|\det A-\det B|\leq C_d(\max\{\|A\|,\|B\|\})^{d-1}\|A-B\|$.
The above three inequalities combined with \eqref{eq Sec 5 Lema 5.7 1} and \eqref{eq Sec 5 Lema 5.7 2} complete the proof of this lemma.
\end{proof}

\begin{proof}[Proof of Lemma \ref{lem-A3}] We encounter here the same difficulty as in the proof of Lemma \ref{lem-A2}.

Recall from the proof of Lemma \ref{eqn-A2}, $\Omega^{(iv)}=\Omega^{\prime\prime} \cap\Omega^1\cap\Omega^2$, and recall $K_0(\omega)\subset B_{M(\omega)}$, for $\omega\in \Omega^{(iv)}$.
 Let us now fix $\omega \in \Omega^{(iv)}$. Then, for any $t\in[0,T]$,
\begin{align}\label{eq Lema 5.8 1}
&\Bigl\vert\int_0^t\Bigl(\int_{\R^d}u_0(y) b^{\eps_n}(\phi^{\eps_n}_s(y))\cdot  D \theta(\phi^{\eps_n}_s(y))J_s^{\eps_n}(y)\,dy \Bigr)ds
-
\int_0^t\Bigl(\int_{\R^d}u_0(y) b(\phi_s(y))\cdot  D \theta(\phi_s(y))J_s(y)\,dy \Bigr) ds\Big|\nonumber\\
&\leq
\Vert u_0 \Vert_{L^\infty(\R^d)}
   \int_0^T\int_{K_0(\omega)}|b^{\eps_n}(\phi^{\eps_n}_s(y))\cdot  D \theta(\phi^{\eps_n}_s(y))||J_s^{\eps_n}(y)-J_s(y)|\,dy \, ds\nonumber\\
    &\ \ \ +
    \Vert u_0 \Vert_{L^\infty(\R^d)}\int_0^T\int_{K_0(\omega)}  |b^{\eps_n}(\phi^{\eps_n}_s(y))|| D \theta(\phi^{\eps_n}_s(y))- D \theta(\phi_s(y))||J_s(y)|\,dy \, ds\nonumber\\
   &\ \ \ +
  \Vert u_0 \Vert_{L^\infty(\R^d)}\int_0^T\int_{K_0(\omega)}  |b^{\eps_n}(\phi^{\eps_n}_s(y))-b(\phi^{\eps_n}_s(y))|| D \theta(\phi_s(y))||J_s(y)|\,dyds\nonumber\\
   &\ \ \ +
   \Vert u_0 \Vert_{L^\infty(\R^d)}\int_0^T\int_{K_0(\omega)}  |b(\phi^{\eps_n}_s(y))-b(\phi_s(y))|| D \theta(\phi_s(y))||J_s(y)|\,dyds \nonumber\\
&\leq
\Vert u_0 \Vert_{L^\infty(\R^d)}\|b^{\eps_n}\|_{0}\| D \theta\|_{0}T
   \int_{B_{M(\omega)}}\sup_{s\in[0,T]}|J_s^{\eps_n}(y)-J_s(y)|\,dy\nonumber\\
    &\ \ \ +
    \Vert u_0 \Vert_{L^\infty(\R^d)}\|b^{\eps_n}\|_{0}\sup_{s\in[0,T]}\sup_{y\in B_{M(\omega)}}|J_s(y)|\int_0^T\int_{K_0(\omega)} | D \theta(\phi^{\eps_n}_s(y))- D \theta(\phi_s(y))|\,dy \, ds\nonumber\\
   &\ \ \ +
   \Vert u_0 \Vert_{L^\infty(\R^d)}\| D \theta\|_{0}\sup_{s\in[0,T]}\sup_{y\in B_{M(\omega)}}|J_s(y)|\int_0^T\int_{K_0(\omega)}  |b^{\eps_n}(\phi^{\eps_n}_s(y))-b(\phi^{\eps_n}_s(y))|\,dyds\nonumber\\
   &\ \ \ +
  \Vert u_0 \Vert_{L^\infty(\R^d)} \| D \theta\|_{0}\sup_{s\in[0,T]}\sup_{y\in B_{M(\omega)}}|J_s(y)|\int_0^T\int_{K_0(\omega)}  |b(\phi^{\eps_n}_s(y))-b(\phi_s(y))|\,dyds.
\end{align}

$\lim_{n\to \infty}\Vert b^{\eps_n}-b\Vert_0=0$ implies that, for any  $\omega \in \Omega^{\prime\prime}$,
\begin{align}\label{eq Lema 5.8 3}
\lim_{n\to \infty}\int_0^T\int_{K_0(\omega)}  |b^{\eps_n}(\phi^{\eps_n}_s(y))-b(\phi^{\eps_n}_s(y))|\,dyds=0.
\end{align}
Here we have used that for any  $\omega \in \Omega^{\prime\prime}$, $K_0(\omega)$ is a compact subset of $\mathbb{R}^d$.

Since $\theta\in C_c^{\infty}(\mathbb{R}^d)$ and $b\in C_{\mathrm{b}}^{\beta}(\R^d,\R^d)$, a
similar argument to that used for the proof of  \eqref{eq Lemma 5.7 jia 1} shows that, for any $\omega \in \Omega^{(iv)}$,
\begin{align}\label{eq Lema 5.8 4}
\lim_{n\to \infty}\int_0^T\int_{K_0(\omega)} | D \theta(\phi^{\eps_n}_s(y))- D \theta(\phi_s(y))|\,dyds=0.
\end{align}
and
\begin{align}\label{eq Lema 5.8 5}
\lim_{n\to \infty}\int_0^T\int_{K_0(\omega)}  |b(\phi^{\eps_n}_s(y))-b(\phi_s(y))|\,dyds=0.
\end{align}
Note that, by \eqref{eq Sec 5 Lema 5.7 7}, for any $\omega\in\Omega^2$,
\begin{align}\label{eq Lema 5.8 6}
\sup_{t\in[0,T]}\sup_{x\in\mathbb{R}^d,|x|\leq M(\omega)}|J_t(x)|<\infty.
\end{align}
Let us observe that assertions \eqref{eq App E 20251008},
\eqref{eq Lema 5.8 1}-\eqref{eq Lema 5.8 3}, \eqref{eq Lema 5.8 4}-\eqref{eq Lema 5.8 6} imply \eqref{eqn-A3}, completing the proof of this lemma.

\end{proof}

\section{Kolmogorov-Totoki-Cencov-type Theorem}
\label{App-sec-Kolm-Th}

The following result is known.

\begin{lemma}\label{lem-KTC-ZB}
 If $T>0$, then the map
\begin{equation}\label{eqn-sup norm}
D_{Skor}([0,T]; \mathbb{R}^d) \ni \omega \mapsto \sup_{t \in [0,T]} \vert \omega(t) \vert \in \mathbb{R},
\end{equation}
where $D_{Skor}([0,T]; \mathbb{R}^d)$ is
the space $D([0,T];\mathbb{R}^d)$ of $\mathbb{R}^d$-valued c\`adl\`ag functions,  endowed with  the Skorokhod metric,
is measurable.
\end{lemma}
\begin{proof}It is enough to apply Lemma 1 and Theorem 12.5 from Section 12 in \cite{Billingsley_1999}.
\end{proof}

The following theorem extends the Kolmogorov–Chentsov continuity criterion, see  Lemma A.2.37 in \cite{Bichteler_2002}, to a more general setting.

\begin{theorem}\label{thm-KTC-ZB}
Assume that  $(\Omega, \mathscr{F}, \mathbb{P})$ is a  complete probability space,  $X$ is an infinite dimensional vector space and let $\rho$ and $\Vert \cdot \Vert$  be a metric and norm on $X$. The space $X$ endowed with  the metric $\rho$ will be denoted by $E$ and
the space $X$ endowed with  the  norm $\Vert \cdot \Vert$  will be denoted by $B$. We assume that $E$ is separable, $B$ is complete and that the natural embedding
\[
i: B \embed E
\]
is continuous. We denote by $\mathscr{B}(E)$ the Borel $\sigma$-field on $E$.

Assume that $D \subset \mathbb{R}^d$ is an open subset of $\mathbb{R}^d$.  Assume that for every $x \in D$, $K(x)$ is an $E$-valued random element, i.e., $K(x):\Omega \to E$ is
$\mathscr{F}/ \mathscr{B}(E)$-measurable. Assume that $p\in (0,\infty)$ and $\delta,\gamma \in (0,1)$, $\delta<\gamma$. \\
Assume that  for all $x,y \in D$, the functions
\begin{equation}\label{eqn-measurability-K}
\begin{aligned}
\Omega \ni \omega &\mapsto \Vert K(x,\omega)- K(y,\omega) \Vert \in [0,\infty),
\\
\Omega \ni \omega &\mapsto \Vert K(x,\omega) \Vert \in [0,\infty)
\end{aligned}
\end{equation}
are $\mathscr{F}/ \mathscr{B}(\mathbb{R})$-measurable.
Assume that there exists $C>0$ such that for all $x,y \in D$,
\begin{equation}\label{eqn-KTC-01}
  \mathbb{E} \Vert K(x)- K(y) \Vert^p \leq C \left[ \vert x-y\vert^{d+\delta}+\vert x-y\vert^{d+\gamma}\right].
\end{equation}
Then, there exists a family $K^\prime(x)$,  $x \in D$, of  $E$-valued random elements, such that
\begin{trivlist}
\item[(i)] $K^\prime$ is a modification of $K$, i.e., for every $x \in D$,
\[
\mathbb{P}(K(x)\not=K^\prime(x))=0;
\]
\item[(ii)] for every $\omega \in \Omega$, the map
\[
D \ni x \mapsto K^\prime(x,\omega) \in B
\]
is continuous.
\end{trivlist}
Moreover, the assertion (ii) can be strengthened as follows. For any $\beta< \frac{\delta}{p}$ and  every $R>0$, the map
\[
D \cap B(0,R) \ni x \mapsto K^\prime(x,\omega) \in B
\]
where $B(0,R)$ is the closed ball of radius $R$ in $\mathbb{R}^d$,
is $\beta$-H\"older continuous (and hence bounded).
\end{theorem}
\begin{remark} 
The idea of the proof draws upon \cite[Theorem 4.1]{Kunita_2004} and \cite[Lemma A.2.37]{Bichteler_2002}, but both results require substantial adjustments to fit current framework. 
\end{remark}

\begin{proof}
We will prove the case that there exists $M\in\mathbb{N}:=\{1,2,3,\cdots\}$ such that $D= [0,M]^d$. For the general case, please see \cite{Bichteler_2002}.

In order to prove the theorem, we will introduce a modulus of continuity of a map from $D$ to $B$. Let $\Pi_n$ be the set of all lattice points in $D$ of the form $(i_1M/2^n, \dots, i_dM/2^n)$, where $0\leq i_1, \dots, i_d\leq 2^n$ are integers. We set $\Pi = \cup_{n\in\mathbb{N}} \Pi_n$. It is a dense subset of $D$. For any $x\in D$ and $n\in\mathbb{N}$, there exists a unique
$x_n \in \Pi_n   \text{ such that } 0 \leq x^i - x_n^i < 2^{-n}M \text{ for all } \, i = 1, \dots, d,
$
and we define $\Pi^D_n(x)=x_n$.
%Then $D_n$ is increasing with $n$ and $\cup_n D_n = D$ holds.

Given a map $f : \Pi \to B$, we define for each $n$ a modulus of continuity and the modulus of $\beta$-$\rm H\ddot{o}lder$ continuity of $f$ by

\[
\Delta_n(f) = \max_{x, y \in \Pi_n, |x-y|=2^{-n}M} \|f(x) - f(y)\|,
\]

\[
\Delta_n^\beta(f) = 2^{\beta n}M^{-\beta} \Delta_n(f).
\]

{\bf{Claim 1}: The inequality
\begin{eqnarray}\label{Revised Kol Lemma 4.2 01}
\|f(x) - f(y)\| \leq 2d \left( \sum_{n=1}^{\infty} \Delta_n^{\beta}(f) \right) |x - y|^\beta, \quad \forall x, y \in \Pi
\end{eqnarray}
holds for any map $f : \Pi \to B$.}

  If $\sum_{n=1}^{\infty} \Delta_n^{\beta}(f)=\infty$, then the claim is obvious.
  In the following, we assume that $\sum_{n=1}^{\infty} \Delta_n^{\beta}(f)<\infty$.

  Define a sequence of simple functions $g_n : D \to B, n = 1, 2, \ldots$ by $g_n(x) = f(x_n)$, where $x_n \in \Pi_n$ is the point such that $0 \leq x^i - x_n^i < 2^{-n}M$ holds for any $i = 1, \ldots, d$, that is, $x_n=\Pi_n^D(x)$. Then it holds

\[
\|g_{n+1}(x) - g_n(x)\| \leq \Delta_{n+1}(f)= 2^{-\beta (n+1)}M^\beta\Delta_{n+1}^\beta(f), \quad \forall x \in D.
\]

Then, for any $x \in D$ and $k\in\mathbb{N}$,

\[
\sum_{n=k}^{\infty} \|g_{n+1}(x) - g_n(x)\| \leq \sum_{n=k}^{\infty} \Delta_{n+1}(f) \leq 2^{-\beta (k+1)}M^\beta\sum_{n=k+1}^{\infty} \Delta_{n}^\beta(f) < \infty.
\]

Then the sequence of simple functions $g_n(x)$ converges on $D$. Let $g(x)$ be the limit function. Then $g(x) = f(x)$ holds valid for $x \in \Pi$, and for any $x\in D$ and $k\in\mathbb{N}$,
\begin{equation}\label{APP 2025 E2 6}
  \|g(x)-g_k(x)\|\leq 2^{-\beta (k+1)}M^\beta\sum_{n=k+1}^{\infty} \Delta_{n}^\beta(f) < \infty.
\end{equation}
 In the sequel, we prove this claim for $g$ instead of $f$.

Now let $x\neq y$ be any points in $D$. Without loss of generality, we assume that $x^1>y^1$ and $x^1-y^1=\sup_{i=1,2,\cdots,d}|x^i-y^i|$. Set $z_1^i=\min\{x^i,y^i\}$ and $z_2^i=\max\{x^i,y^i\}$, $i=1,2,\cdots,d$. Then there exist unique $k\in\mathbb{N}$ and just one $j_1\in\{0, 1,2,\cdots,2^k\}$ such that
\begin{equation}\label{APP 2025 E2 4}
\frac{j_1M}{2^k}\in(y^1,x^1)\text{ and }x^1-y^1\geq \frac{M}{2^k}.
\end{equation}
Note that $x^1$ or $y^1$ may belong to $\{\frac{jM}{2^k}; j=0,1,2,\cdots, 2^k\}$, that is why we consider the open set $(y^1,x^1)$.

Since $x^1-y^1=\sup_{i=1,2,\cdots,d}|x^i-y^i|$, then, for any $i=1, 2,3,\cdots, d$, the number of the elements of the set
$\{j\in \{0,1,2,\cdots, 2^k\}:\frac{jM}{2^k}\in(z_1^i,z_2^i)\}$ is 0 or 1. Set $H_i=\max\{\frac{jM}{2^k}<z_2^i:j\in \{0,1,2,\cdots, 2^k\}\}$, and $H=(H_1,\cdots,H_d)$. Then $H$  belongs to $\Pi_k$, and we have
\begin{equation}\label{APP 2025 E2 5}
  |x-H|\leq \frac{dM}{2^k}\text{ and }|y-H|\leq \frac{dM}{2^k}.
\end{equation}

By \eqref{APP 2025 E2 4} and \eqref{APP 2025 E2 5},
\begin{equation}\label{APP 2025 E2 7}
1\leq 2^kM^{-1}\sum_{i=1}^d|\Pi_k^D(x)^i-\Pi_k^D(y)^i|\leq 2d,
\end{equation}
and
\begin{equation}\label{2025 App E2 3}
  \frac{M}{2^k}\leq |x-y|\leq \frac{2dM}{2^k}.
\end{equation}

Hence, by  \eqref{APP 2025 E2 6} and \eqref{2025 App E2 3},
\begin{eqnarray}
  \|g(x)-g_k(x)\|\leq 2^{-\beta (k+1)}M^\beta\sum_{n=k+1}^{\infty} \Delta_{n}^\beta(f)
  \leq
  2^{-\beta}\sum_{n=k+1}^{\infty} \Delta_{n}^\beta(f)|x - y|^\beta,
\end{eqnarray}
and
\begin{eqnarray}
  \|g(y)-g_k(y)\|\leq 2^{-\beta (k+1)}M^\beta\sum_{n=k+1}^{\infty} \Delta_{n}^\beta(f)
  \leq
  2^{-\beta}\sum_{n=k+1}^{\infty} \Delta_{n}^\beta(f)|x - y|^\beta.
\end{eqnarray}

Further since \eqref{APP 2025 E2 7} and \eqref{2025 App E2 3} we have

\begin{eqnarray}
\|g_{k}(x) - g_{k}(y)\|&=&\|g_{k}(\Pi_k^D(x)) - g_{k}(\Pi_k^D(y))\| \\
&\leq& \sum_{i=1}^d \|g_{k}(\vec{x}^{i-1}) - g_{k}(\vec{x}^i)\|\\
&\leq& 2d \Delta_{k}(f)=2d \Delta_{k}^\beta(f)2^{-\beta k}M^\beta \leq 2d \Delta_{k}^\beta(f)|x - y|^\beta.
\end{eqnarray}

Here
\begin{eqnarray}
  &&\vec{x}^{0}=\Pi_k^D(x),\\
  &&\vec{x}^{1}=(\Pi_k^D(y)^1,\Pi_k^D(x)^2,\cdots,\Pi_k^D(x)^d),\\
  &&\vec{x}^{2}=(\Pi_k^D(y)^1,\Pi_k^D(y)^2,\Pi_k^D(x)^3,\cdots,\Pi_k^D(x)^d),\\
  &&\cdots,\\
  &&\vec{x}^{d}=\Pi_k^D(y).
\end{eqnarray}

Therefore

\begin{eqnarray*}
\|g(x) - g(y)\| &\leq& \|g(x) - g_{k}(x)\| + \|g_{k}(x) - g_{k}(y)\|+ \|g_{k}(y) - g(y)\|\\
&\leq& 2d \left( \sum_{k=1}^{\infty} \Delta_k^\beta(f) \right) |x - y|^\beta.
\end{eqnarray*}

The proof of Claim 1 is complete.

By the above claim, the map $f : \Pi  \to B$ satisfying $\sum_{k=1}^\infty \Delta_k^{\beta}(f) < \infty$ for some $\beta > 0$ is uniformly continuous on $\Pi$ and has a continuous extension $g : D \to B$ i.e., there exists a continuous map $g : D \to B$ such that $g(x) = f(x)$ holds for $x \in \Pi$. The function $g$ is $\beta$-$\rm H\ddot{o}lder$ continuous.

We shall apply the above claim to the random field $K(x)$. Observe that for each $\omega$, $K(\cdot, \omega)$ restricting $x$ to $\Pi$ can be regarded as a map from $\Pi$ to $B$. Then we have

\begin{equation}\label{Revised Kol Lemma 4.2 02}
\|K(x, \omega) - K(y, \omega)\| \leq 2d\left( \sum_{k=1}^\infty \Delta_k^{\beta}(K(\omega)) \right)|x-y|^\beta, \quad \forall x, y \in \Pi.
\end{equation}

{\bf Claim 2 } Let $\beta$ be a positive number satisfying $\beta p < \delta$. Then, there exists
a constant $C_{M,\gamma-\delta,p,\beta,\delta,d}$ such that
\[
\Big[\mathbb{E}(\sum_{k=1}^\infty \Delta_k^{\beta}(K))^p\Big]^{1/p} \leq \left( \sum_{k=1}^{\infty} 2^{-k(\delta/p-\beta)} \right) \cdot (2^d C_{M,\gamma-\delta,p,\beta,\delta,d})^{1/p} < \infty.
\]

{\bf Proof:} We will consider the case $p \geq 1$ only. Observe the inequality

\[
(\Delta_k^{\beta}(K))^p = M^{-\beta p}(\sup_{x,y \in \Pi_k, |x-y|=M/2^k} \|K(x) - K(y)\|{2^{k\beta}})^p \leq M^{-\beta p}\sum (\|K(x') - K(y')\|{2^{k\beta}})^p,
\]
where the summations are taken over all $x', y' \in \Pi_k$ such that $|x' - y'| = M/2^k$. Then the number of summations are at most $d(2^k+1)^d<d2^{(k+1)d}$. Therefore
\begin{align}
\mathbb{E}[\Delta_k^{\beta}(K)^p]
&\leq d2^{(k+1)d+kp\beta}M^{-\beta p} \mathbb{E}[\|K(x') - K(y')\|^p]\\
&\leq d2^{(k+1)d+kp\beta}M^{-\beta p} C_{M,\gamma-\delta}(M/2^k)^{d+\delta}\\
&\leq C_{M,\gamma-\delta,p,\beta,\delta,d}2^{d}2^{k(p\beta-\delta)}.
\end{align}
Here we have used that by \eqref{eqn-KTC-01} for any $x,y\in D=[0,M]^d$
\begin{equation}\label{eqn-KTC-02}
  \mathbb{E} \Vert K(x)- K(y) \Vert^p \leq C_{M,\gamma-\delta} \vert x-y\vert^{d+\delta}.
\end{equation}
 Therefore, by the Minkowski inequality we get

\[
\Big[\mathbb{E}(\sum_{k=1}^\infty \Delta_k^{\beta}(K))^p\Big]^{1/p} \leq \sum_{k=1}^{\infty} [\mathbb{E}\Delta_k^{\beta}(K)^p]^{1/p}
\leq  \sum_{k=1}^{\infty} C^{1/p}_{M,\gamma-\delta,p,\beta,\delta,d}2^{d/p}2^{k(\beta-\delta/p)} < \infty.
\]

The proof of Claim 2 is complete.

Now we are in position to finish the proof of Theorem \ref{thm-KTC-ZB} for the case $D=[0,M]^d$.

 The random field $K(x)$ restricting $x$ on $\Pi$ satisfies the inequality \eqref{Revised Kol Lemma 4.2 02}, where $\sum_{k=1}^\infty \Delta_k^{\beta}(K) < \infty$ holds a.s. Therefore $K(x), x \in \Pi$ is uniformly $\beta$-$\rm H\ddot{o}lder$ continuous in $B$ a.s.. Then there exists a continuous random field $\tilde{K}(x)$ in $E$ defined on $D$ such that $K(x) = \tilde{K}(x)$ holds a.s. for any $x \in \Pi$. Since $K(x)$ is continuous in probability, the equality $K(x) = \tilde{K}(x)$ holds a.s for any $x \in D$. Thus we have proved the theorem.

\end{proof}

\section*{Acknowledgment}
  This work is supported by National Key R\&D Program of China (Nos.  2024YFA1012301, 2022YFA1006001).

The  research activity of E. Priola was carried out as part of
the PRIN 2022 project ``Noise in fluid dynamics and related models'' (he is also  member of GNAMPA, which
is part of the ``Istituto Nazionale di Alta Matematica'' INdAM).

 The research of Jianliang Zhai is supported by National Natural Science Foundation of China (No. 12371151, 12131019),  the Fundamental Research Funds for the Central Universities(USTC) (WK3470000031, WK0010000081).

The research of Jiahui Zhu was supported by National Natural Science Foundation of China (No.12071433).

\section*{Declarations} 
	
\noindent 	\textbf{Ethical Approval:}   Not applicable 
	
%\noindent  \textbf{Competing interests: } The authors declare no competing interests. 
	
\noindent  \textbf{Conflict of Interest: } On behalf of all authors, the corresponding author states that there is no conflict of interest.
	
\noindent 	\textbf{Authors' contributions:} All authors have contributed equally.

\noindent \textbf{Data Availability Statement and materials:} Not applicable.

\end{document}